\documentclass[a4,9pt]{amsart}

\usepackage{amssymb}
\usepackage{amstext}
\usepackage{amsmath}
\usepackage{amscd}
\usepackage{amsthm}
\usepackage{amsfonts}
\usepackage{enumerate}
\usepackage{graphicx}
\usepackage{latexsym}
\usepackage[all]{xy}
\usepackage[usenames]{color}
\usepackage{comment}
\usepackage{bm}
\usepackage{algorithm}
\usepackage{algpseudocode}
\usepackage[total={6.5in,8.75in},
top=1.2in, right=1.0in, left=1.0in, bottom=1.0in,includefoot]{geometry} %

\newtheorem*{maintheorem*}{Main Theorem}
\newtheorem*{corollary*}{Corollary}
\newtheorem*{question*}{Question}
\newtheorem{theorem}{Theorem}[section]

\newtheorem{lemma}[theorem]{Lemma}
\newtheorem{proposition}[theorem]{Proposition}
\newtheorem*{proposition*}{Proposition}
\newtheorem{question}[theorem]{Question}
\newtheorem{conjecture}[theorem]{Conjecture}

\newtheorem{claim}{Claim}
\newtheorem*{claim*}{Claim}
\newtheorem*{keylemma*}{Key Lemma}
\newtheorem*{keyclaim*}{Key Claim}

\newtheorem{assumption}[theorem]{Assumption}
\theoremstyle{definition}
\newtheorem{definition}[theorem]{Definition}
\newtheorem{remark}[theorem]{Remark}

\newtheorem{example}[theorem]{Example}

\newtheorem{definitiontheorem}[theorem]{Definition-Theorem}

\theoremstyle{remark}

\numberwithin{equation}{theorem}

\newenvironment{pfclaim}{

\begin{proof}
}
{
\end{proof}

}

\renewcommand{\mod}{\operatorname{mod}}
\newcommand{\proj}{\operatorname{proj}}
\newcommand{\inje}{\operatorname{inj}}

\newcommand{\End}{\operatorname{End}}
\newcommand{\Hom}{\operatorname{Hom}}
\newcommand{\add}{\operatorname{\mathsf{add}}}
\newcommand{\Ext}{\operatorname{Ext}}

\renewcommand{\P}{\mathbb{P}}

\newcommand{\Mu}{\mathcal{M}}

\newcommand{\odd}{\mathsf{Odd}}
\newcommand{\even}{\mathsf{Even}}
\newcommand{\MGS}{\operatorname{\mathsf{MGS}}}
\newcommand{\MGSm}{\underset{{\rm max}}{\operatorname{\mathsf{MGS}}}}

\newcommand{\ttilt}{\operatorname{\mathsf{\tau -tilt}}}
\newcommand{\sttilt}{\operatorname{\mathsf{s\tau -tilt}}}

\newcommand{\ftors}{\operatorname{\mathsf{f-tors}}}

\newcommand{\trigid}{\operatorname{\mathsf{\tau\text{-}rigid}}}

\newcommand{\supp}{\operatorname{Supp}}

\newcommand{\Z}{{\mathbb{Z}}}

\newcommand{\R}{\mathbb{R}}

\newcommand{\Tor}{\operatorname{Tor}}
\newcommand{\Tr}{\operatorname{Tr}}

\newcommand{\id}{\operatorname{id}}
\newcommand{\Ker}{\operatorname{Ker}}

\newcommand{\rad}{\operatorname{rad}}

\newcommand{\fac}{\operatorname{\mathsf{Fac}}}
\newcommand{\ind}{\operatorname{ind}}

\newcommand{\dimvec}{\operatorname{\underline{\mathsf{dim}}}}

\newcommand{\dip}{\mathsf{dpre}}
\newcommand{\dis}{\mathsf{dsuc}}
\newcommand{\pre}{\mathsf{pre}}
\newcommand{\suc}{\mathsf{suc}}

\newcommand{\ORA}{\overrightarrow}

\newcommand{\ORAD}{\overrightarrow{\Delta}}
\newcommand{\surj}{\twoheadrightarrow}
\newcommand{\inj}{\hookrightarrow}
\newcommand{\Hasse}{\ORA{\mathcal{H}}}
 
\newcommand{\num}{\#}

\begin{document}
\title[Lengths of maximal green sequences for tame path algebras] 
{Lengths of maximal green sequences for tame path algebras}

\author{Ryoichi Kase}
\address{R.~Kase: Faculty of Informatics, Okayama University of Sciemce, 1-1 Ridaicho, Kita-ku, Okayama-shi 700-0005, Japan}
\email{r-kase@mis.ous.ac.jp}
\thanks{R.~Kase is supported by JSPS KAKENHI Grant-in-Aid for Young Scientists (B) 17K14169.}
\author{Ken Nakashima}
\address{K.~Nakashima: Center for Advanced Intelligence Project, RIKEN, Nihonbashi 1-chome Mitsui Building, 15th floor, 1-4-1 Nihonbashi, Chuo-ku, Tokyo 103-0027, Japan}
\email{ken.nakashima@riken.jp}


\maketitle

\begin{abstract}
	In this paper, we study the maximal length of maximal green sequences  for quivers of type $\widetilde{\mathbf{D}}$ and $\widetilde{\mathbf{E}}$ by using the theory of tilting mutation. 
	We show that the maximal length does not depend on the choice of the orientation, and determine it explicitly.    
	Moreover, we give a program which counts all maximal green sequences by length for a given Dynkin/extended Dynkin quiver.
\end{abstract}

\section{Introduction}
Maximal green sequences were introduced by B. Keller to obtain quantum dilogarithm identities and refined Donaldson-Thomas invariants \cite{Ke}.
These are maximal sequences of ``green quiver mutations". 
If we consider a finite acyclic quiver and its path algebra, then maximal green sequences induce maximal chains of
torsion classes in the module category \cite{DK,Ke2,Q} and maximal paths in the Hasse quiver of
support ($\tau$-)tilting poset \cite{BDP, BST}. By using these connections, T.~Br$\ddot{\mathrm{u}}$stle, G.~Dupont and M.~P$\acute{\mathrm{e}}$rotin
showed the finiteness of maximal green sequences for (simply laced) Dynkin/extended Dynkin quivers \cite{BDP}.
Moreover, they presented a conjecture for lengths of maximal green sequences so-called ``no gap conjecture".
\subsection{Lengths of maximal green sequences and no gap conjecture}
Let $Q$ be an acyclic quiver and $A=KQ$ its path algebra over an algebraically closed field $K$.
We denote by $\sttilt A$ the poset of support $\tau$-tilting modules for $A$ and by $\ORA{\mathcal{H}}(\sttilt A)$
its Hasse quiver (see Section\;2). Then, in terms of ($\tau$-)tilting theory, the no gap conjecture is presented as follows.
\begin{conjecture}[\cite{BDP}]
	For each finite dimensional path algebra $A$, possible lengths of maximal paths in $\ORA{\mathcal{H}}(\sttilt A)$
	form an interval in $\Z$.
\end{conjecture}
We remark that no gap conjecture does not hold if we consider arbitrary finite dimensional algebras.
It was shown by Garver-McConville that no gap conjecture holds for cluster tilted algebras of type $\mathbf{A}$ \cite{GM}. 
Hermez-Igusa extended this result to cluster tilted algebras of finite type and path algebras of tame type \cite{HI}. 
Therefore, for each Dynkin or extended Dynkin quiver $Q$, there are integers $\ell(Q)$ and $\ell'(Q)$ such that  
\[
\ell\MGS(A)=[\ell'(Q),\ell(Q)]:=\{\ell\in \Z\mid \ell'(Q)\le \ell \le \ell (Q)\},
\]  
where $\ell\MGS(A)$ denotes the set of possible lengths of  maximal paths in $\ORA{\mathcal{H}}(\sttilt A)$.
If we regard $\ell'(Q)$ as the minimal number in $\ell\MGS(A)$ for an arbitrary acyclic quiver $Q$,
then it is well-known that $\ell'(Q)$ is equal to the number of vertices of $Q$. In particular, we have the following equations. 
\[
\begin{array}{lll}
\ell'(Q)&=&\num Q_0\\
        &=&\num \{\text{simple modules}\}/\simeq\\
        &=&\num \{\text{indecomposable projective modules}\}/\simeq\\
        &=&\num \{\text{indecomposable injective modules}\}/\simeq\\
\end{array}
\]
Moreover, the minimal length $\ell'(Q)$ was given in \cite{CDRS,GMS} for each quiver $Q$ mutation equivalent to type $\mathbf{A}$, $\mathbf{D}$, or $\widetilde{\mathbf{A}}$.

For a quiver $Q$ of type $\mathbf{A}_n$, $\mathbf{D}_n$, $\mathbf{E}_{6,7,8}$,
$\widetilde{\mathbf{A}}_n$, the maximal length $\ell(Q)$
is also calculated.
\begin{theorem}
	\label{ell(Q)}
	\begin{enumerate}
		\item If $Q$ is a (simply-laced) Dynkin quiver, then we have
		\[
		\ell(Q)=\num\{\text{indecomposable modules}\}/\simeq.
		\]
		\item If $Q$ is a quiver of type $\widetilde{A}_{a,b}$, then we have
		\[
		\ell(Q)=\dfrac{n(n+1)}{2}+ab.
		\] 
		In particular, $\ell(Q)$ is an invariant for sink or source mutations \cite{ApIg}.

	\end{enumerate}
\end{theorem}
Further, Br$\ddot{\mathrm{u}}$stle-Dupont-P$\acute{\mathrm{e}}$rotin calculated $\ell(Q)$ for certain quivers
of type $\widetilde{\mathbf{D}}_{4,5,6,7}, \widetilde{\mathbf{E}}_{6,7}$. They also checked that
$\ell(Q)$ (hence $\ell\MGS(KQ)$) does not depend on the choice of the orientation of a quiver $Q$ of type $\widetilde{\mathbf{D}}_{4}$ (\cite{BDP} and its arXiv version [arXiv:1205.2050v1]).

\subsection{Aim of this paper}
Theorem\;\ref{ell(Q)} and the Br$\ddot{\mathrm{u}}$stle-Dupont-P$\acute{\mathrm{e}}$rotin's calculation give us the following question.
\begin{question}
	Let $Q$ be a finite acyclic quiver and $A=KQ$.
	\begin{enumerate}[{\rm (1)}]
		\item Let $\mu_i Q$ be the quiver given by mutating $Q$ at $i$. 
		Does the equation $\ell(Q)=\ell(\mu_i Q)$ hold for each sink or source vertex $i$ of $Q$?  
		\item What is $\ell(Q)$? (If $Q$ is of type $\mathbf{A},\mathbf{D}, \mathbf{E}$, or $\widetilde{\mathbf{A}}$, then we know $\ell(Q)$ explicitly. Moreover, it is the number of (isomorphism classes of) indecomposable modules for each Dynkin quiver $Q$.)  
	\end{enumerate}	
\end{question}
In this paper, we consider the above questions in the case of $\widetilde{\mathbf{D}}$ and $\widetilde{\mathbf{E}}$. 
\begin{maintheorem*}
	\label{intro:main}
Let $Q$ be a quiver of type $\widetilde{\mathbf{D}}$ or $\widetilde{\mathbf{E}}$ and $A=KQ$.
\begin{enumerate}[{\rm (1)}]
	\item $\ell(Q)$ does not depend on the choice of orientation.
	\item $\ell(Q)$ is given by the following table.
	 \begin{center}
	 	\begin{tabular}{|c||c|c|c|c|}
	 		\hline
	 		$Q$	&  $\widetilde{\mathbf{D}}_n$& $\widetilde{\mathbf{E}}_6$  & $\widetilde{\mathbf{E}}_7$  & $\widetilde{\mathbf{E}}_8$\\
	 		\hline
	 		$\ell(Q)$	& $2n^2-2n-2$ & $78$ & $159$ & $390$ \\
	 		\hline
	 	\end{tabular} 
 	\end{center}
 In addition, we have created a program which counts all maximal green sequences of the path algebra $KQ$ by length for a given tame quiver $Q$.
 The program is available at the following URL and can be freely used by anyone for research purposes.
 \begin{center}
 {\rm	https://hfipy3.github.io/MGS/en.html }
 \end{center} 
\end{enumerate}
\end{maintheorem*}
\begin{remark}
	In \cite{ApIg}, Apruzzese-Igusa also discuss the connection to the stability conditions. 
	It is interesting whether the similar result holds for the cases $\widetilde{\mathbf{D}}$ and $\widetilde{\mathbf{E}}$.
\end{remark}

\subsection*{Notations} Throughout this paper, we use the following notations.
\begin{itemize}
	\item[1.] Algebras are finite dimensional over an algebraically closed field $K$.
	\item[2.] For an algebra $A$, we denote by $\mod A$ (resp. $\proj A$) the category of finite dimensional right $A$-modules (resp. finite dimensional projective right $A$-modules), and denote by $\tau=\tau_A$ the Auslander-Reiten translation of $\mod A$ (refer to \cite{ASS,ARS} for definition and properties). 
	\item[3.] For a poset $(\mathbb{P},\le)$, we denote by $\Hasse(\P)$ its Hasse quiver.
	We also define $[p,q]:=\{x\in \P\mid p\le x \le q\}$ for each $p,q \in \P$.
	\item[4.] For a quiver $\Hasse$, we denote by $\Hasse_0$ (resp. $\Hasse_1$) the set of all vertices (resp. arrows) of $\Hasse$. 
	\item[5.] For an acyclic quiver $\Hasse$ and a vertex $a$ of $\Hasse$, 
	we define $\suc(a)$, $\dis(a)$, $\pre(a)$, $\dip(a)$ as follows.
	\[
	\begin{array}{lll}
	\suc(a)&:=&\{b\in \Hasse_0 \mid \text{there is a (nontrivial) path from $a$ to $b$ in $\Hasse$}\}\\
	\dis(a)&:=&\{b\in \Hasse_0 \mid \text{there is an arrow from $a$ to $b$ in $\Hasse$}\}\\
	\pre(a)&:=&\{b\in \Hasse_0 \mid \text{there is a (nontrivial) path from $b$ to $a$ in $\Hasse$}\}\\
	\dip(a)&:=&\{b\in \Hasse_0 \mid \text{there is an arrow from $b$ to $a$ in $\Hasse$}\}\\
	\end{array}	
	\]
	\item[6.] For a $(n\times m)$-matrix $M$ with coefficients in $K$, we denote by $f_M$
	the $K$-linear map given by
	\[
	K^n\ni{}^{t}(x_1,\dots, x_n)\mapsto {}^{t}\left( (x_1\dots, x_n)\cdot M\right)\in \R^m. 
	\]
	Then we often write
	$
		K^n\stackrel{M}{\longrightarrow} K^m
	$	instead of $K^n\stackrel{f_M}{\longrightarrow} K^m$.
\end{itemize}

\section{Preliminary}
\subsection{Support $\tau$-tilting modules}
In this subsection, we recall the definitions and basic properties of support $\tau$-tilting modules.
Let $A$ be a basic algebra, and let $Q$ be its Gabriel quiver.
\subsubsection{Definition of support $\tau$-tilting modules}
For a module $M$, we denote by $|M|$ the number of non-isomorphic indecomposable direct summands of $M$
and by $\supp(M):=\{i\in Q_0\mid Me_i\neq 0\}$ the support of $M$, where $e_i$ is a primitive idempotent corresponding 
to a vertex $i\in Q_0$. We put $e_M:=\sum_{i\in \supp(M)} e_i$.

A module $M\in\mod A$ is said to be {\bf $\tau$-rigid} if it satisfies
$\Hom_{A}(M, \tau M)=0$.  
If $\tau$-rigid module $T$ satisfies 
$|T|=\num\supp(T)$ (resp. $|T|=n$), then we call $T$ a {\bf support $\tau$-tilting module}
(resp. {\bf $\tau$-tilting module}).
We denote by $\sttilt A$ (resp. $\ttilt A$, $\trigid A$)
the set of (isomorphism classes of) basic support $\tau$-tilting modules 
(resp. $\tau$-tilting modules, $\tau$-rigid modules)
of $A$. 

We call a pair $(M,P)\in \mod A\times \proj A$ a {\bf $\tau$-rigid pair} (resp. {\bf support $\tau$-tilting pair})
if $M$ is $\tau$-rigid (resp. support $\tau$-tilting) and $\add P\subset \add (1-e_M) A$ (resp. $\add P=\add (1-e_M)A$).
\begin{remark}
If $A$ is hereditary, then it follows from the Auslander--Reiten duality, the notion of support $\tau$-tilting modules 
coincides with that of support tilting modules introduced by Ingalls-Thomas  in \cite{IT}.    
\end{remark}

Let $(M,P)$ be a $\tau$-rigid pair.
We say that $(M,P)$ is {\bf basic} if so are $M$ and $P$.
A direct summand $(N, R)$ of $(M,P)$ is a pair of a module $N$ and a projective module $R$
which are direct summands of $M$ and $P$, respectively. Further, we say
$(N,R)$ is isomorphic to $(M,P)$ if $M\simeq N$ and $P\simeq R$. 
From now on, we put 
\[M\oplus P^-:=(M,P) \text{ and } |M\oplus P^-|:=|M|+|P|.\]
\begin{remark}
	If $M$ is $\tau$-rigid, then we have $|M|\leq \num \supp(M)$ (see \cite[Proposition 1.3]{AIR}).
	In particular, a $\tau$-rigid pair $M\oplus P^-$ is support $\tau$-tilting if and only if $|M\oplus P^-|=|A|$.
\end{remark}
\begin{remark}
In this paper, we often identify $\sttilt A$ with the set of (isomorphism classes of)
basic support $\tau$-tilting pairs.
\end{remark}
The following proposition gives us a connection between $\tau$-rigid modules of $A$
and that of a factor algebra of $A$.
\begin{proposition}[{\cite[Lemma 2.1]{AIR}}]
	\label{basicfact} 
	Let $M$ and $N$ be $A/J$-modules, where $J$ is a two-sided ideal of $A$.
	If $\Hom_{A}(M,\tau N)=0$, then $\Hom_{A/J}(M,\tau_{A/J} N)=0$.
	Moreover, if $J=(e)$ is a two-sided ideal generated by an idempotent $e$, then the converse holds.
\end{proposition}
By the above proposition, we regard $\sttilt A/(e)$ as a subset of $\sttilt A$. More precisely, we have
\[
\sttilt A/(e)=\{T\in \sttilt A\mid Te=0\}.
\]
\subsubsection{Torsion classes induced by support $\tau$-tilting modules}
A full subcategory $\mathcal{T}$ of $\mod A$ which is closed under taking factor modules and extensions is called a {\bf torsion class} of $\mod A$. 
We say that $\mathcal{T}$ is {\bf functorially finite} if for any $M\in \mod \Lambda$, there are $f\in \Hom_{\Lambda}(X,M)$ and $g\in \Hom_{\Lambda}(M,Y)$ with 
$X,Y\in \mathcal{T}$ such that $\Hom_{A}(N,f):\Hom_{A}(N,X)\to \Hom_{A}(N,M)$
and $\Hom_{A}(g,N):\Hom_{A}(Y,N)\to \Hom_{A}(M,N)$ are surjective for all $N\in \mathcal{T}$.
We denote by $\ftors A$ the set of functorially finite torsion classes of $\mod A$.

Then the following theorem implies the connection between support $\tau$-tilting modules and functorially finite torsion classes.
\begin{theorem}\cite[Theorem\;2.7]{AIR}
	\label{ftors}
An assignment	$M\mapsto \fac M$ implies a bijection
\[
\sttilt A \to \ftors A,
\]
where $\fac M$ is the category of factor modules of finite direct sums of copies of M.
\end{theorem}
\subsection{Partial order on $\sttilt A$}
By using the connection in Theorem\;\ref{ftors}, we can define a partial order on $\sttilt A$.
\begin{definitiontheorem}[{\cite[Lemma 2.25]{AIR}}]
	\label{posetstr}
	For support $\tau$-tilting modules $M$ and $M'$, we write $M\geq M' $ if $\fac M\supseteq \fac M'$.
	Then the following are equivalent.
	\begin{enumerate}
		\item $M\geq M'$.
		\item  $\Hom_A(M',\tau M)=0$ and $\supp(M)\supseteq \supp(M')$.
	\end{enumerate}
	Moreover, $\geq$ gives a partial order on $\sttilt A$.
\end{definitiontheorem} 
\subsubsection{The mutation and the Hasse quiver of $\sttilt A$}
A $\tau$-rigid pair $M$ is said to be 
{\bf almost complete support $\tau$-tilting} provided it satisfies $|M|=|A|-1$.
Then the mutation of support $\tau$-tilting modules is formulated by the following theorem.
\begin{theorem}\label{basicprop}
\begin{enumerate}[{\rm (1)}]
	\item \cite[Theorem\;2.18]{AIR} 
	Let $M$ be a basic almost complete support $\tau$-tilting pair. Then there are exactly two basic support $\tau$-tilting modules
	$T$ and $T'$ such that $M$ is a direct summand of $T\oplus \left((1-e_T) A\right)^-$ and $T'\oplus \left((1-e_{T'}) A\right)^-$. 
	\item \cite[Corollary\;2.34]{AIR} Let $T$ and $T'$ be basic support $\tau$-tilting modules.
	Then $T$ and $T'$ are connected by an arrow of $\Hasse(\sttilt A)$ if and only if $T\oplus\left((1-e_T) A\right)^-$ and $T'\oplus \left((1-e_{T'}) A\right)^-$
	have a common basic almost complete $\tau$-tilting pair as a direct summand.
	In particular, $\sttilt A$ is $|A|$-regular. 
	\item\cite[Theorem\;2.35]{AIR} Let $T,T'\in\sttilt A$. If $T<T'$, 
	then there is a direct predecessor $U$ of $T$ (resp. a direct successor $U'$ of $T'$) in $\Hasse(\sttilt A)$ 
	such that $U\leq T'$ (resp. $T\leq U'$). 
	\item \cite[Corollary\;2.38]{AIR} If $\Hasse(\sttilt A)$ has a finite connected component $\mathcal{C}$, then
	$\mathcal{C}=\Hasse(\sttilt A)$. 
\end{enumerate}   
\end{theorem}
\subsubsection{An anti-isomorphism between $\sttilt A$ and $\sttilt A^{{\rm op}}$}
The following proposition gives a relationship between the support $\tau$-tilting poset of $A$ and that of $A^{\mathrm{op}}$.
\begin{proposition}[{\cite[Theorem\;2.14,\;Proposition\;2.27]{AIR}}]
	\label{revers}
	Let $M\oplus P^-=M_{\mathrm{np}}\oplus M_{\mathrm{pr}}\oplus P^-$ be a $\tau$-rigid pair with
	$M_{\mathrm{pr}}$ being a maximal projective direct summand of $M$. 
	We put $(M\oplus P^-)^{\dagger}:=\Tr M_{\mathrm{np}}\oplus P^*\oplus (M_{\mathrm{pr}}^*)^-$,
	where $(-)^*=\Hom_{A}(-,A):\proj A\to \proj A^{\mathrm{op}}$.
	Then $(M\oplus P^-)^{\dagger}$ is a $\tau$-rigid pair. Moreover, $(-)^{\dagger}$
	gives a poset anti-isomorphism from $\sttilt A$ to $\sttilt A^{\mathrm{op}}$ with $\left((-)^{\dagger}\right)^{\dagger}=\id$.
\end{proposition} 
\subsubsection{$\tau$-rigid pairs are determined by their $g$-vectors}
Let $X:=N\oplus U^-$ be a $\tau$-rigid pair and $P'\to P\to N\to 0$ be a minimal projective presentation of $N$.
Then the {\bf $g$-vector} $g^X:=(g_i^X)_{i\in Q_0}\in \Z^{Q_0}$ of $X$ is defined by the following equation in the Grothendiec group $K_0(\proj A)$ of $\proj A$. 
\[
[P]-[P']-[U]=\underset{i\in Q_0}{\sum} g_i^X[P_i]
\]
Then we have the following theorem.
\begin{theorem}[{\cite[Theorem\;5.5]{AIR}}]
	\label{gvector}
The assignment $X\mapsto g^X$ induces an injection from the set of isomorphism classes of $\tau$-rigid pairs of $A$
to $\Z^{Q_0}$.
\end{theorem}

\subsection{Maximal green sequences}
In this subsection, we recall the definition of maximal green sequences.
\begin{definition}Let $Q$ be a finite connected quiver.
	\begin{itemize}
		\item[(1)] $Q$ is a cluster quiver if $Q$ does not admit loops or 2-oriented cycles.
		\item[(2)] An ice quiver is a pair $(Q,F)$ where $Q$ is a cluster quiver and $F$ is a subset of $Q_{0}$
		such that there is no arrow between two vertices of $F$. For an ice quiver $(Q,F)$, we call a vertex in $F$
		a frozen vertex. 
	\end{itemize}
\end{definition}
\begin{definition}
	Let $(Q,F)$ be an ice quiver and $k\in Q_{0}\setminus F$. We define a new ice quiver $(\mu_{k}Q,F)$ from $(Q,F)$ by applying
	following $4$-steps.
	\begin{description}
		\item[$(\mathrm{Step}\;1)$] For any pair of arrows $i\stackrel{\alpha}{\to} k\stackrel{\beta}{\to} j$, add an arrow $i\stackrel{[\alpha\beta]}{\to}j$.
		\item[$(\mathrm{Step}\;2)$] Replace any arrow $i\stackrel{\alpha}{\to} k$ by an arrow $i\stackrel{\alpha^{\ast}}{\leftarrow} k$. 
		\item[$(\mathrm{Step}\;3)$] Replace any arrow $k\stackrel{\beta}{\to} j$ by an arrow $k\stackrel{\beta^{\ast}}{\leftarrow} j$.
		\item[$(\mathrm{Step}\;4)$] Remove a maximal collection of $2\text{-}cycles$ and all arrows between frozen vertices.
	\end{description}
	We call $(\mu_{k}Q,F)$ the mutation of $(Q,F)$ at a non-frozen vertex $k$.
\end{definition}
\begin{definition}Let $(Q,F),\;(Q',F)$ be ice quivers with $Q_{0}=Q'_{0}$.
	\begin{enumerate}[{\rm (1)}]
		\item $(Q,F)$ and $(Q',F)$ are {\bf mutation-equivalent} if there is a finite sequence
		$(k_{1},\dots,k_{l})$ of non-frozen vertices such that 
		\[(Q',F)=(\mu_{k_{l}}\mu_{k_{l-1}}\cdots\mu_{k_{1}}Q,F).\] 
		Then we denote by $\mathsf{Mut}(Q,F)$ the mutation-equivalence class of $(Q,F)$.
		\item $(Q,F)$ is isomorphic to $(Q',F)$ as ice quivers if there is an isomorphism
		$\varphi:Q\rightarrow Q'$ of quivers fixing any frozen vertex. In this case, we denote
		$(Q,F)\simeq (Q',F)$ and denote by $[(Q,F)]$ the isomorphism class of $(Q,F)$.
	\end{enumerate}
\end{definition}
From now on, we assume that $Q$ is a cluster quiver and $Q'_{0}:=\{c(i)\mid i\in Q_{0}\}$ is a copy of $Q_{0}$.
\begin{definition}
	The framed quiver associated with $Q$ is the quiver $\hat{Q}$ defined as follows:
	\begin{itemize}
		\item $\hat{Q}_{0}:=Q_{0}\sqcup Q'_{0}$.
		\item $\hat{Q}_{1}:=Q_{1}\sqcup \{i\to c(i)\mid i\in Q_{0}\}$.
	\end{itemize}
	The coframed quiver associated with $Q$ is the quiver $\check{Q}$ defined as follows:
	\begin{itemize}
		\item $\check{Q}_{0}:=Q_{0}\sqcup Q'_{0}$.
		\item $\check{Q}_{1}:=Q_{1}\sqcup \{c(i)\to i\mid i\in Q_{0}\}$.
	\end{itemize}
	Note that $(\hat{Q},Q'_{0})$ and $(\check{Q},Q'_{0})$ are ice quivers. We denote by $\mathsf{Mut}(\hat{Q})$ the
	mutation-equivalence class $\mathsf{Mut}(\hat{Q},Q'_{0})$ of $(\hat{Q},Q'_{0})$.  
\end{definition}
\begin{definition}
	Let $(R, Q'_0)\in \mathsf{Mut}(\hat{Q})$ and let $i$ be a non-frozen vertex.
	\begin{itemize}
		\item[(1)] $i$ is said to be a {\bf green vertex} if $\{\alpha\in R_{1}\mid s(\alpha)\in Q'_{0},\;t(\alpha)=i\}=\emptyset$.
		\item[(2)] $i$ is said to be a {\bf red vertex} if $\{\alpha\in R_{1}\mid  s(\alpha)=i,\;t(\alpha)\in Q'_{0}\}=\emptyset$.
	\end{itemize}  
\end{definition}
It was shown in \cite{BDP} that  each non-frozen vertex in $(R, Q'_0)\in \mathsf{Mut}(\hat{Q})$ is either "green" or "red". Moreover, 
$(R, Q'_0)\in \mathsf{Mut}(\hat{Q})$ has no green vertices if and only if $(R, Q'_0)$ is isomorphic to $(\check{Q},Q'_{0})$.
Then a maximal green sequence is defined as follows.
\begin{definition}
	A {\bf green sequence} for a cluster quiver $Q$ is a sequence ${\bf i}=(i_{1},i_{2},\dots,i_{\ell})$ of $Q_{0}$
	such that $i_{1}$ is green in $\hat{Q}$ and for any $2\leq k \leq l$, $i_{k}$ is green in $\mu_{i_{k-1}}\cdots\mu_{1}\hat{Q}$.
	In this case, $\ell$ is called the length of ${\bf i}$. A green sequence ${\bf i}=(i_{1},i_{2},\dots,i_{\ell})$ is said to be maximal if
	$\mu_{i_{\ell}}\cdots\mu_{i_{1}}\hat{Q}$ has no green vertex. 
\end{definition}
In the case that $Q$ is acyclic and $A=KQ$, there exists a length-preserving bijection between the set of maximal green sequences for $Q$ 
and the set of paths from $\mod A$ to $\{0\}$ in $\Hasse(\ftors A)$ (\cite[Section\;6]{BDP}, \cite[Theorem\;5.2]{DK}).
Therefore, it follows from Definition-Theorem\;\ref{posetstr} that we may identify maximal green sequences
with paths from $A$ to $0$ in $\Hasse(\sttilt A)$ via the above bijection. 
%
Then we use the following notations.
\[
\begin{array}{rll}
\MGS(A)&:=&\text{the set of paths from $A$ to $0$ in $\Hasse(\sttilt A)$}\\
\ell(\omega)&:=&\text{the length of }\omega\in \MGS(A)\\
\ell(Q)&:=&\max\{\ell(\omega) \mid \omega\in \MGS(A) \}\\
\MGSm(A)&:=&\{\omega\in \MGS(A)\mid \ell(\omega)=\ell(Q)\}.\\
\end{array}
\]
\begin{remark}
In the rest of this paper,	we often use the following fact.
	 \[
	 \ell(Q)=\ell(Q^{\rm op})
	 \]
\end{remark}

 \section{Key tools}
 In this section, we prepare useful  tools to show our main theorem. From here on, we use the following notations.
 \begin{itemize}
 	\item $A_i:=A/(e_i)$ for each $i\in Q_0$.
 	\item When we treat another basic algebra $B$ with Gabriel quiver $Q'$, we set $P_j^B=e_j B$, $S_j^B=P_j^B/\rad P_j^B$, $I_j^B=D(B e_j)$ for each $j\in Q'_0$. 
 	\item For a support tilting module $T$, we denote by $\MGS(A,T)$ the set of all maximal green sequences which factors through $T$ and
 	\[\MGSm(A,T):=\MGSm(A)\cap \MGS(A,T).\]
  \end{itemize}
\begin{remark}
	If $A=KQ$, then we often regard $A_i$ as a path algebra $K(Q\setminus\{i\})$ and $P_j^{A_i}$ as an $A$-module. 
\end{remark}
 \subsection{Permutation on $\sttilt KQ$ induced by Auslander-Reiten translation}
In this subsection, we assume $Q$ is acyclic and $A=KQ$. We define a correspondence $\mod A\times \proj A\to \mod A \times \proj A$ by
 \[\left\{\begin{array}{llll}
 X &\mapsto & 
 \tau X & (X\in \ind A\setminus \proj A)\\
 P & \mapsto &P_i^- & (P\simeq P_i)\\
 P^- & \mapsto& I_i& (P\simeq P_i)\\
 \end{array}
 \right.
 \]  
As with the Auslander--Reiten translation, we write this correspondence by $\tau=\tau_A$.
We also define a correspondence $\mod A\times \proj A\to \mod A \times \proj A$ by
 \[\left\{\begin{array}{llll}
X &\mapsto & 
\tau^- X & (X\in \ind A\setminus \inje A)\\
I & \mapsto &P_i^- & (I\simeq I_i)\\
P^- & \mapsto& P_i& (P\simeq P_i)\\
\end{array}
\right.
\]
and denote it by $\tau^{-1}$.
\begin{lemma}
	\label{permutation}
	For $X\in \mod A\times \proj A$, the following statements are equivalent.
	\begin{enumerate}[{\rm (a)}]
		\item $X$ is $\tau$-rigid.
		\item $\tau X$ is $\tau$-rigid.
		\item $\tau^- X$ is $\tau$-rigid.
	\end{enumerate}
In particular, $\tau, \tau^{-1}:\mod A\times \proj A\to \mod A \times \proj A$ 
	 induce permutations on $\sttilt A$.
\end{lemma}	
\begin{proof}
		Since $\tau^{-1}\tau X\simeq X\simeq \tau\tau^{-1}X $, it is sufficient to check
	\[
	X\text{ is $\tau$-rigid}\Leftrightarrow \tau X \text{ is $\tau$-rigid}.
	\]
Let $X=M\oplus P^-$ with $M=N \oplus U\oplus \tau^{-1}U'$ such that $U,U'\in \proj A$ and $\add N\cap \add (A\oplus \tau^{-1}A)=\{0\}$,  
then we have 
\[
\tau(X)=\tau (M\oplus P^-)=\nu P\oplus \tau N\oplus U'\oplus U^-.
\]
Hence we obtain
\[
\begin{array}{ll}
&X\text{ is $\tau$-rigid}\\
\Leftrightarrow& \Hom_A(N\oplus \tau^{-1}U' \oplus U, \tau N\oplus U')=0, \Hom_A(P,N\oplus \tau^{-1}U'\oplus U)=0\\
\Leftrightarrow& \Hom_A(\tau N\oplus U', \tau(\tau N))=0, \Hom_A(U, \tau N\oplus U')=0, \Hom_A(N\oplus \tau^{-1}U'\oplus U,\nu P)=0\\
\Leftrightarrow& \Hom_A(\tau N\oplus U', \tau(\tau N))=0, \Hom_A(U, \tau N\oplus U'\oplus \nu P)=0, \Hom_A(\tau N\oplus U',\tau\nu P)=0\\
\Leftrightarrow& \Hom_A(\tau N\oplus U', \tau(\tau N\oplus \nu P))=0, \Hom_A(U, \tau N\oplus U'\oplus \nu P)=0\\
\underset{\ast}{\Leftrightarrow}& \Hom_A(\tau N\oplus U', \tau(\tau N\oplus \nu P))=0, \Hom_A(U, \tau N\oplus U'\oplus \nu P)=0, \Hom_A(\nu P, \tau (\tau N \oplus U'))=0\\
\Leftrightarrow & \tau X\text{ is $\tau$-rigid},\\
\end{array}
\]
where $(\ast)$ follows from $\nu P$ is injective and $\tau (\tau N \oplus U')$ is not injective.
Therefore, we have the assertion.
\end{proof}
%
%

\subsection{Rotation property via Ladkani's result}
	In this subsection, we assume $Q$ is acyclic, $i$ is a sink vertex of $Q$, $Q':=\mu_i Q$ the quiver mutation of $Q$ at $i$,  $A=KQ$, and $B=KQ'$.
	We denote by $\mu_i A$ the APR-tilting module corresponding to $i$, i.e., 
	\[ \mu_i A:=(\underset{k\ne i}{\bigoplus}P_k)\oplus \tau^{-1}P_i.\]
	In this setting, we have an isomorphism
	\[B\cong \End_A(\mu_i A)\]
	with $e_k B\simeq \Hom_A(\mu_i A, P_k)$ ($k\ne i$) and $e_i B\simeq \Hom_A(\mu_i A,\tau^{-1} P_i)$.
	We denote by $F_i^+$ (resp. $F^-_i$) the BGP reflection functor $\Hom_A(\mu_i A):\mod A\to \mod B$ (resp. $-\underset{B}{\otimes} {}_BT_A:\mod B\to \mod A$).
	It is well-known that $F_i^+$ and $F_i^-$ induce the following equivalence (see  \cite[Chap VI\hspace{-1pt}I,Theorem5.2]{ASS} for example).
	\[
    \xymatrix{
	\mathcal{A}_i:=\{X\in \mod A\mid P_i\not \in \add X\}\ar@<0.5ex>[r]^-{F^+_i} \ar@{=}[d]& \{Y\in \mod B\mid I_i^B\not \in \add Y\}=:\mathcal{B}_i \ar@{=}[d]\ar@<0.5ex>[l]^-{F^-_i}\\
	                                           \{X\in \mod A\mid \Ext_A^1(T, X)=0\} & \{Y\in \mod B\mid \Tor_1^B(Y,{}_BT)=0\}\\
}
	\]
	Furthermore, for $M\in \mod A$, $F_i^+(M)$ is isomorphic to $N\in \mod B$ given by
	 \[
Ne_k=\left\{\begin{array}{ll}
	 Me_k&(k\ne i)\\
	 \Ker\left(\underset{\alpha:\bullet\to i}{\bigoplus}M^{(\alpha)}\stackrel{\oplus(-\cdot \alpha)}{\longrightarrow} Me_i \right)& (k=i)\\ 
	 \end{array}
	 \right.\]
	 \[\begin{array}{ccccclll}
	 \left(-\cdot \beta \right)& :&Ne_{s(\beta)}&\to& Ne_{t(\beta)} & =&  (-\cdot \beta) :Me_{s(\beta)}\to Me_{t(\beta)} \\ 
	 \left(-\cdot \alpha^*\right) &:&Ne_i&\to& Ne_j  & = & (Ne_i\inj \underset{\alpha:\bullet\to i}{\bigoplus}M^{(\alpha)}\surj M^{(\alpha)}=Ne_j), \\ 
	 \end{array}
	 \]    
	 where $M^{(\alpha)}:=Me_{s(\alpha)}$, $\beta\in Q_1\cap Q'_1$ and $Q_1\ni \alpha:j\to i $ (see  \cite[Chap VI\hspace{-1pt}I, Propostion5.6]{ASS} for example). 
	 \begin{remark}
	 	\label{remark:bgp_dimvec} If $k\ne i$, then 
	 \[\dim _K\Hom_A(P_k,M)=\dim_K\Hom_B(P_k^B, F_i^+ (M)).\]	
	 \end{remark}
\begin{proposition}[{\cite[Propoistion\;4.3]{L}}]
	\label{ladkani}
	Let $Q$ be an acyclic quiver, $i$ be a sink vertex of $Q$, $Q':=\mu_i Q$ the quiver mutation of $Q$ at $i$,  $A=KQ$ and $B=KQ'$.
	 Then an assignment $F$ given by
\[
	P_i=S_i\mapsto (P_i^B)^-,\ \ind A\setminus \add P_i \ni X\mapsto F_i^+(X),\ P_i^-\mapsto I_i^B=S_i^B,\ P_k^-\mapsto (P_k^B)^- (k\ne i) 
\]
 induces a poset isomorphism
\[
\psi:(\sttilt A)_{\le \mu_i A} \stackrel{\sim}{\to} (\sttilt B)_{\ge S^B_i}.
\]
In particular, for an integer $\ell$, the following two statements are equivalent.
\begin{itemize}
	\item  There exists $\omega\in \MGS(A,S_i)$ with $\ell(\omega)=\ell$.
	\item  There exists $\omega'\in \MGS(B,\mu_iB)$ with $\ell(\omega')=\ell$.
\end{itemize}
\end{proposition}
For the reader's convenience, we give a proof.
\begin{proof}
	
 Let $X,Y\in \ind A\setminus \add P_i$ and $Q_0\ni k\ne i$. 
 Since $\mathcal{A}_i$ and $\mathcal{B}_i$ are closed under extensions, we obtain
 \[
 \Ext_A^1(X,Y)=0\Leftrightarrow \Ext_B^1(F_i^+(X),F_i^+(Y))=0.
 \]
 
 Let $0\to F^+_i(X)\to \underset{k\in Q_0}{\bigoplus}(I^B_{k})^{n_k}\to \underset{k\in Q_0}{\bigoplus}(I^B_k)^{m_k}\to 0$ be a minimal injective copresentation of $F^{+}(X)$.
 Note that we have an exact sequence
 \[
 0\to \underset{k\in Q_0}{\bigoplus}(P^B_{k})^{n_k}\to \underset{k\in Q_0}{\bigoplus}(P^B_k)^{m_k}\to \tau_B^{-1} F^+(X)\to  0.
 \]
 Since $I_i^B$ is simple injective, we have
 \[
 \dim_K  F^+_i(X)e_j=\left\{\begin{array}{ll}
  \underset{i\ne k\in Q_0}{\sum}n_k \dim_K I^B_k e_j-\underset{i\ne k\in Q_0}{\sum}m_k \dim_K I^B_k e_j & (j\ne i)\\
 \underset{k\in Q_0}{\sum}n_k \dim_K I^B_k e_i-\underset{k\in Q_0}{\sum}m_k \dim_K I^B_k e_i & (j= i)\\   
 \end{array}\right.
 \]
Therefore, it follows from  $P_i\not \in \add X$ and $i$ is source in $Q'$ that
 \[
 \begin{array}{lll}
 \Hom_A(P_i, X)=0&\Leftrightarrow& \dim_K F_i^+(X)e_i= \underset{Q_1\ni \alpha:\bullet\to i}{\sum}\dim_K F_i^+(X)e_{s(\alpha)}\\
                                &\Leftrightarrow& \dim_K F_i^+(X)e_i= \underset{Q_1\ni \alpha:\bullet\to i}{\sum}\left(\underset{i\ne k\in Q_0}{\sum}n_k \dim_K I^B_k e_{s(\alpha)}-\underset{i\ne k\in Q_0}{\sum}m_k \dim_K I^B_k e_{s(\alpha)}\right)\\
                                &\Leftrightarrow& \dim_K F_i^+(X)e_i= \underset{i\ne k\in Q_0}{\sum}n_k \left(\underset{Q_1\ni \alpha:\bullet\to i}{\sum} \dim_K I^B_k e_{s(\alpha)}\right)
                                -\underset{i\ne k\in Q_0}{\sum}m_k \left(\underset{Q_1\ni \alpha:\bullet\to i}{\sum} \dim_K I^B_k e_{s(\alpha)}\right)\\
                                &\Leftrightarrow& \dim_K F_i^+(X)e_i= \underset{i\ne k\in Q_0}{\sum}n_k \dim_K I^B_k e_i
                                -\underset{i\ne k\in Q_0}{\sum}m_k \dim_K I^B_k e_i\\
                                &\Leftrightarrow&\underset{k\in Q_0}{\sum}n_k \dim_K I^B_k e_i-\underset{k\in Q_0}{\sum}m_k \dim_K I^B_k e_i= \underset{i\ne k\in Q_0}{\sum}n_k \dim_K I^B_k e_i
                                -\underset{i\ne k\in Q_0}{\sum}m_k \dim_K I^B_k e_i\\
                                &\Leftrightarrow& n_i=m_i\\\\
                                &\Leftrightarrow& \tau_B^{-1} F^+(X) e_i =0\\\\
                                &\Leftrightarrow& \Hom_B(F^+(X),\tau_B I^B_i)=0=\Hom_B(I^B_i, \tau_B F^+(X))=0.
 \end{array}
 \]
We also have 
\[\Hom_A(P_k, X)=0\Leftrightarrow \Hom_{B}(P^B_k,F_i^+X)=0\]	
by Remark\;\ref{remark:bgp_dimvec}.
Then we can easily check that the assignment induces a poset isomorphism
\[\sttilt A\setminus \sttilt_{P_i} A \stackrel{\sim}{\rightarrow} \sttilt B\setminus\sttilt_{(P^B_i)^-} B.\]

It remains to show 
\[\sttilt A\setminus \sttilt_{P_i} A=(\sttilt A)_{\le \mu_i A},\ \sttilt B\setminus\sttilt_{(P^B_i)^-} B=(\sttilt B)_{\ge S_i^B}.\]
The first equation follows from
\[\Hom_A(T, \tau\mu_i A)=0\Leftrightarrow \Hom_A(T, P_i)=0\Leftrightarrow P_i=S_i\not\in \add T.\]
The second equation follows from
\[S_i^B=I_i^B\le T\Leftrightarrow \supp(S_i^B)\subset \supp (T)\Leftrightarrow (P^B_i)^-\not\in \add T.\]
This finishes the proof.
\end{proof}
\subsection{Jasso's reduction theorem and elementary polygonal deformations}
\subsubsection{Jasso's reduction theorem}
Here, we recall the $\tau$-tilting reduction theorem by Jasso \cite{J}.
Let $R=U\oplus (eA)^-$ be a basic $\tau$-rigid pair of $A$ and $T\oplus (eA)^-=X\oplus U\oplus (eA)^- $ the Bongartz completion of $R$.
We set $B=\End_{A}(T)=\End_{A/(e)}(T)$ and $C=B/(e_U)$, where $e_U$ be an idempotent of $B$ corresponding to a projective module $\Hom_{A/(e)}(T,U)$.
\begin{theorem}[{\cite[Theorem\;3.13 and Corollary\;3.18]{J}}]
	\label{Jasso'sreduction} In the above setting, we set the functor $F:=\Hom_{A/(e)}(T,-):\mod A\to \mod B$ and $U^{\perp}=\{N\in \mod A\mid \Hom_A(U, N)=0\}$. 
\begin{enumerate}[{\rm (1)}]
	\item The assignment $\varphi :T'\mapsto F(\fac(T')\cap U^{\perp} \cap \mod A/(e))$ induces a poset isomorphism
	\[
	\sttilt_R A\stackrel{\sim}{\to} \ftors C.
	\]
	\item  If $A$ is hereditary, then $C$ is also hereditary.
\end{enumerate}
\end{theorem}
\begin{remark}
	\label{remark for Jasso}
Let $M=U\oplus (eA)^-$ be a basic $\tau$-rigid pair such that its Bongartz completion can be written by $M\oplus X \oplus P_i$.
If $i$ is a sink vertex, then $P_i\in U^{\perp }$.
Therefore, for any $T\in \sttilt_M A$ such that $P_i\in \add T$, we have 
\[\begin{array}{lll}
e_{P_i}C\simeq e_{P_i}B/(e_U)&\simeq& \Hom_{A/(e)}(U\oplus X\oplus P_i, P_i)/[U]\\
&=&\Hom_{A/(e)}(U\oplus X\oplus P_i, P_i)\\
&=&F(P_i)\in F(\fac (T) \cap U^\perp \cap \mod A/(e))=\varphi(T). 
\end{array}
\] 
\end{remark}
\subsubsection{Elementary polygonal deformations}
We recall the notion of elementary polygonal deformations defined in \cite{HI}.

Let $M$ be a basic $\tau$-rigid pair.
Assume that $|M|=|A|-2$ and $\sttilt_M A=[T_M, T^M]$ is a finite poset.
In this case, the Jasso's reduction theorem (Theorem\;\ref{Jasso'sreduction}) implies that there exists precisely two maximal paths $w_M$ and $w'_M$ in $\sttilt_M A$.
Then, for a path $\omega: T \to \cdots \to \underbrace{T^M\to \cdots \to T_M}_{w_M}\to \cdots \to S$ in $\Hasse(\sttilt A)$,
we say that $\mathsf{d}(\omega): T \to \cdots \to \underbrace{T^M\to \cdots \to T_M}_{w'_M}\to \cdots \to S$ in $\Hasse(\sttilt A)$ the {\bf elementary polygonal deformation}
of $\omega$ by $M$.
\begin{example} Let $A=1\to 2 \to 3$ and $M=S_2$.
	Then $\sttilt_{S_2} A$ is given by
	
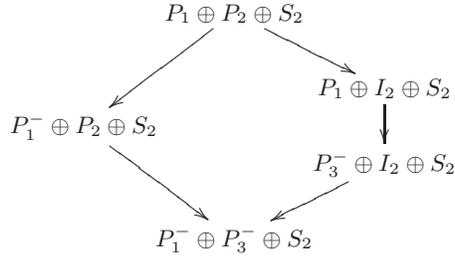
\[\begin{xy}
	(0,0)*[o]+{P_1\oplus P_2 \oplus S_2}="A",
	(-20,-15)*[o]+{P_1^-\oplus P_2 \oplus S_2}="B",
	(0,-30)*[o]+{P_1^-\oplus P_3^- \oplus S_2}="C",
	(20,-10)*[o]+{P_1\oplus I_2 \oplus S_2}="D",
	(20,-20)*[o]+{P_3^-\oplus I_2 \oplus S_2}="E",
	\ar "A";"B"
	\ar "B";"C"
	\ar "A";"D"
	\ar "D";"E"
	\ar "E";"C" 
	\end{xy}
	\]
Therefore, we have the following. 
 \[\begin{xy}
 (0,0)*[o]+{\left[A\to P_1\oplus P_2\oplus S_2\to P_1^-\oplus P_2\oplus S_2 \to P_1^-\oplus P_3^-\oplus S_2\to P_1^-\oplus P_2^-\oplus P_3^-\right]}="A",
 (0,-15)*[o]+{\left[A\to P_1\oplus P_2\oplus S_2\to P_1\oplus I_2\oplus S_2 \to P_3^-\oplus I_2\oplus S_2 \to P_1^-\oplus P_3^-\oplus S_2\to P_1^-\oplus P_2^-\oplus P_3^-\right]}="B",
 \ar@<3pt> "A";"B"^{\mathsf{d}}
 \ar@<3pt> "B";"A"^{\mathsf{d}}
 \end{xy}
 \]
\end{example}
Then, by applying Jasso's reduction theorem (Theorem\;\ref{Jasso'sreduction}), we have the following statement.
 \begin{lemma}
	\label{lemma_epd} Let $A=KQ$ be a finite dimensional path algebra and $i$ be a sink vertex of $Q$.
	Consider a path 
	\[T_1\to T_2 \to T_3\]
	in $\overrightarrow{\mathcal{H}}(\sttilt A)$ such that $P_i\in \add T_2\setminus \add T_3$.
	Assume that 
	$M\in \add T_1\cap \add T_2 \cap  \add T_3$ with $|M|=|A|-2$.
	Then $\sttilt_M A $ has one of the following forms.	
	
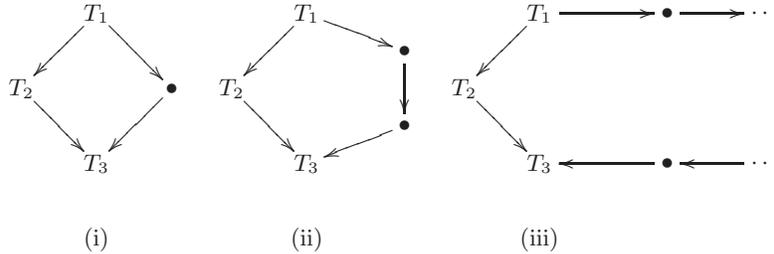
\[\begin{xy}
	(0,-30)*[o]+{{\rm (i)}}, (0,0)*[o]+{T_1}="A", (-10,-10)*[o]+{T_2}="B", (10,-10)*[o]+{\bullet}="C", (0,-20)*[o]+{T_3}="D",
	\ar "A";"B"
	\ar "A";"C"
	\ar "B";"D"  
	\ar "C";"D"  
	\end{xy}\hspace{10pt}
	\begin{xy}
	(0,-30)*[o]+{{\rm (ii)}}, (0,0)*[o]+{T_1}="A", (-10,-10)*[o]+{T_2}="B", (13,-5)*[o]+{\bullet}="C", (0,-20)*[o]+{T_3}="D", (13,-15)*[o]+{\bullet}="E",
	\ar "A";"B"
	\ar "A";"C"
	\ar "B";"D"  
	\ar "E";"D"  
	\ar "C";"E"  
	\end{xy}\hspace{10pt}
	\begin{xy}
	(0,-30)*[o]+{{\rm (iii)}}, (0,0)*[o]+{T_1}="A", (-10,-10)*[o]+{T_2}="B", (17,0)*[o]+{\bullet}="C", (0,-20)*[o]+{T_3}="D", (30,0)*[o]+{\cdots}="E",
	(17,-20)*[o]+{\bullet}="F", (30,-20)*[o]+{\cdots}="G",
	\ar "A";"B"
	\ar "A";"C"
	\ar "B";"D" 
	\ar "C";"E"  
	\ar "G";"F" 
	\ar "F";"D"   
	\end{xy}
	\]	

In particular, if $|\sttilt_M A|<\infty$, then 
we have
\[\ell(\omega) \le \ell(\mathsf{d}_M \omega)\] for each $\omega$ which contains $T_1\to T_2 \to T_3$ as a suppath.

\end{lemma}
\begin{proof}
Let $T'\in \dis(T_1)$ such that $P_i\not\in \add T'$.
Then  $T_1, T_2, T_3$ and $T'$ are in $\sttilt_M A$, $T_1$ is the Bongartz completion of $M$. 

We write $M=N\oplus (eA)^-$ and $T_1=T\oplus (eA)^-=M\oplus X\oplus P_i$ with $N, X, T\in \mod A$.
Then by Jasso's reduction theorem (Theorem\;\ref{Jasso'sreduction}),
$C:=\End_{A/(e)}(T)/(e_{N})\cong K\Delta$ and $\sttilt_M A\simeq \sttilt C$ for some acyclic quiver $\Delta$ with two vertices  $u$ and $v$.

Let $e_{P_i},e_X$ be idempotents of $B$ such that $e_{P_i} B\simeq \Hom_{A/(e)}(T, P_i)$ and $e_X B\simeq \Hom_{A/(e)}(T, X)$.
Then $C$ has the following two indecomposable projective modules.
	\[
	    U=\dfrac{e_{P_i} B}{e_{P_i} B e_N B},\ V=\dfrac{e_{X} B}{e_{X} B e_N B}.
	\]
	Let $e_u$ (resp. $e_v$) be a primitive idempotent of $C\cong K\Delta$ corresponding to $u\in \Delta_0$ (resp. $v\in \Delta_0$). 
	 Then we may assume that $e_u C\simeq U$ and
	$e_v C\simeq V$.
	
	Let $f\in \Hom_C(V,U)=\Hom_B(V,U)$. Since $e_{X} B$ is a projective $B$-module, there exists $g\in \Hom_B(e_XB,e_{P_i}B)$ such that
	the following diagram commutes.
	\[
	    \xymatrix{
	     V \ar[r]^-f  & U \\
	     e_{X}B\ar[r]_-g \ar@{->>}[u] & e_{P_i}B \ar@{->>}[u]
	    }	
	 \]
	 Since $P_i$ is a simple projective $A$-module, we have
	 \[
	 \Hom_B(e_X B, e_{P_i}B)\simeq \Hom_B\left(\Hom_{A/(e)}(T, X), \Hom_{A/(e)}(T, P_i)\right)\simeq \Hom_{A/(e)}(X, P_i)=0.
	 \]
	 Therefore,  $g=0$ and $f=0$. In particular, we have
	 \[
	 e_u C e_v\simeq \Hom_C(V,U)=0.
	 \]	 
Therefore, there is no arrow from $u$ to $v$ in $\Delta$ and $\sttilt C$ has one of the following forms.	
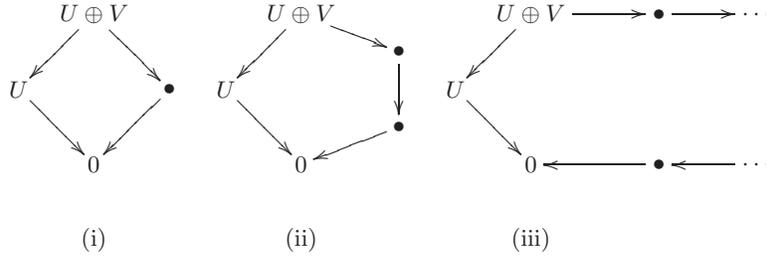
\[\begin{xy}
(0,-30)*[o]+{{\rm (i)}}, (0,0)*[o]+{U\oplus V}="A", (-10,-10)*[o]+{U}="B", (10,-10)*[o]+{\bullet}="C", (0,-20)*[o]+{0}="D",
\ar "A";"B"
\ar "A";"C"
\ar "B";"D"  
\ar "C";"D"  
\end{xy}\hspace{10pt}
\begin{xy}
(0,-30)*[o]+{{\rm (ii)}}, (0,0)*[o]+{U\oplus V}="A", (-10,-10)*[o]+{U}="B", (13,-5)*[o]+{\bullet}="C", (0,-20)*[o]+{0}="D", (13,-15)*[o]+{\bullet}="E",
\ar "A";"B"
\ar "A";"C"
\ar "B";"D"  
\ar "E";"D"  
\ar "C";"E"  
\end{xy}\hspace{10pt}
\begin{xy}
(0,-30)*[o]+{{\rm (iii)}}, (0,0)*[o]+{U\oplus V}="A", (-10,-10)*[o]+{U}="B", (17,0)*[o]+{\bullet}="C", (0,-20)*[o]+{0}="D", (30,0)*[o]+{\cdots}="E",
(17,-20)*[o]+{\bullet}="F", (30,-20)*[o]+{\cdots}="G",
\ar "A";"B"
\ar "A";"C"
\ar "B";"D" 
\ar "C";"E"  
\ar "G";"F" 
\ar "F";"D"   
\end{xy}
\]
Since $P_i\in \add T_2$, we also have that an indecomposable projective $C$-module $U\simeq e_{P_i}C$
is in $\varphi(T_2)$ by Remark \ref{remark for Jasso}.
Hence, by a poset isomorphism 
\[\sttilt_M A\underset{\varphi}{\stackrel{\sim}{\longrightarrow}} \ftors C \underset{\fac}{\stackrel{\sim}{\longleftarrow}}\sttilt C \]
the path $T_1\to T_2\to T_3$ is corresponding to the path $U\oplus V\to U\to 0$ in $\Hasse(\sttilt C)$.
\end{proof}

\subsection{Technical propositions for a proof of Main Theorem (1)}
Here we prove two propositions (Proposition\;\ref{mgs:keyprop} and Proposition\;\ref{mgs:keyprop2})
which are useful to show that $\ell(Q)$ is an invariant for the source (resp. sink) mutation. 

For a path algebra $A=KQ$, we denote by $\mathcal{P}/\mathcal{R}/\mathcal{I}$, the set of (isomorphism classes of) indecomposable
preprojective/regular/preinjective modules, respectively. In addition, we use the following notations.
\[
\begin{array}{rll}
\widetilde{\mathcal{I}}&:=&\mathcal{I}\cup\{P_i^-\mid i\in Q_0\}\\
\add \mathcal{M}&:=&\{X\in \mod A\times \proj A \mid \text{ each indecomposable direct summand of $X$ is in $\mathcal{M}$} \}\ (\mathcal{M}=\mathcal{P},\ \mathcal{R},\ \mathcal{I},\text{ or } \widetilde{\mathcal{I}}).
\end{array}
\]

Let $Q$ be a quiver of type $\widetilde{\mathbf{D}}$ or $\widetilde{\mathbf{E}}$, $A=KQ$.
For $\MGSm(A)\ni \omega:T_0\to \cdots \to T_{\ell}$, we define 
\[
\begin{array}{lllllll}
s_{\omega}^{(i')}&:=&\max\{r\mid P_{i'}\in\add T_{r+1}\} \\
t_{\omega}^{(i)}&:=&\min\{r\mid P_i^-\in \add T_{r-1}\}\\
t^{(i)}&:=&\max \{t_{\omega}^{(i)}\mid \omega\in \MGSm(A) \}\\
\end{array}
\]
\subsubsection{Two classes in $\mod A$}
To prove Main Theorem (1), we only check the following statement.
\begin{center}
	$\ell(Q)\le \ell(\mu_i Q)$ holds for each source vertex $i$.
\end{center}

Then the following two classes in $\mod A$ play an important role. 
	 \[\begin{array}{lllll}
	 \mathcal{X}_i&=&\mathcal{X}_i(A)&:=&\{X\in \ind A\mid \Hom_A(P_i, X)=0,\ \dim_K\Hom_A(\tau^{-1}P_i, X)\ge 2\}\\
	 \mathcal{X}_i'&=&\mathcal{X}'_i(A)&:=&\{X'\in \ind A\mid \Hom_A(P_i, \tau X')=0,\ \dim_K\Hom_A(P_i, X')\ge 2\}
	 \end{array}
	 \]
	 We note that 
	 \[
	 \mathcal{X}_i\cap \proj A=\emptyset=\mathcal{X}'_i\cap \proj A
	 \]
	  and the transpose $\Tr$ induces a bijection 
	 \[
\begin{xy}
(0,0)*[o]+{\mathcal{X}_i(A)}="x",(30,0)*[o]+{\mathcal{X}'_i(A^{{\rm op}})}="x'o",
\ar@{<->} "x";"x'o"^{\Tr}
\end{xy}	 
	 \]
	 between $\mathcal{X}_i(A)$ and $\mathcal{X}'_i(A^{\rm{op}})$.

To find $X\in \mathcal{X}_i$, the following lemma is useful.
\begin{lemma}
	\label{findXi }
Let $Q$ be a quiver of type $\widetilde{\mathbf{D}}$ or $\widetilde{\mathbf{E}}$, and $i$ a source vertex of $Q$. 
For $X\in \ind A_i$, we define a quadruple 
\[
\left(\ORA{C},\ORA{C}_+, j, \ORAD\right)=\left(\ORA{C}_{(i,X)},\ORA{C}_{(i,X,+)}, j_{(i,X)}, \ORAD_{(i,X)}\right)
\]
as follows.
\begin{itemize}
	\item $\ORA{C}$ is the unique connected component of $Q\setminus\{i\}$  containing $\supp X$.
	\item $\ORA{C}_+$ is the full subquiver of $Q$ having $\ORA{C}_0\cup \{i\}$ as the vertex set. 
	\item $j$ is the unique neighbor of $i$ which is in $\ORA{C}$.
	\item $\ORA{\Delta}$ is the full subquiver of $Q$ having the following vertex set.
	\[
	\mathsf{suc}(j)\cup \left(\underset{j'\in \mathsf{suc}(j)}{\bigcup} \mathsf{dpre}(j') \right)
	\]
\end{itemize}
\begin{enumerate}[{\rm (1)}]
	\item If $X\in \mathcal{X}_i$, then $X$ is not projective and $\tau X\in \ind K\ORA{C}_+$ such that $\dim_K\Hom_A(P_i, \tau X)\ge 2$.
	\item If $X\in \mathcal{X}_i$, then $\ORA{C}_+$ is neither of type $\mathbf{A}$ nor $\mathbf{D}$, and $\ORA{C}$ is not of type $\mathbf{A}$.
	\item If $\deg i\ne 1$, then $\mathcal{X}_i=\emptyset$. 
\end{enumerate}
\end{lemma}
\begin{proof}
	(1). Let $0\to \underset{k=1}{\overset{s}{\oplus}} P_{j_k}\to \underset{k=1}{\overset{t}{\oplus}} P_{i_k}\to X \to 0$
	be a minimal projective presentation of $X$. Since $\supp X\subset \ORA{\mathcal{C}}_0$ and 
	$i$ is a source vertex, $i_1,\dots, i_t$ and $j_1,\dots,j_k$ are in $\ORA{\mathcal{C}}_0$. 
	Since we have an exact sequence
	\[
	0\to \tau X \to \underset{k=1}{\overset{s}{\oplus}} I_{j_k}\to \underset{k=1}{\overset{t}{\oplus}} I_{i_k}\to 0,
	\]
	we have 
		\[
	\dim_K \Hom_A(P_{j'},\tau X)=0
	\]	for each $j'\in Q_0\setminus (\ORA{C}_0\cup \{i\})\cdots(\ast)$.
	In particular, $\tau X\in \ind K\ORA{C}_+$ and
	\[
	\dim_K\Hom_A(P_i, \tau X)=\dim_K(\tau^{-1}P_i, X)\ge 2.
	\] 
	
(2). Since the degree of $i$ in $\ORA{C}_+$ is equal to $1$, $\ORA{C}_+$ is neither of type $\mathbf{A}$ nor $\mathbf{D}$
by (1) and the classification of indecomposable modules for path algebras of type $\mathbf{A}$, $\mathbf{D}$.
Hence we check that $\ORA{C}$ is not of type $\mathbf{A}$.
Suppose that $\ORA{C}$ is of type $\mathbf{A}$. Since $\ORA{C}_+$ is not of type $\mathbf{A}$, $\ORA{C}_+$ has the following forms.
\[
  	\xymatrix@=13pt
{
	&   & & & & i\ar@{->}[d] \\ 
	\bullet &\ar@{<->}[l] \cdots &\ar@{<->}[l]    t'  &\ar@{<-}[l] t   &\ar@{->}[l] \cdots  & \ar@{->}[l] j & \ar@{<-}[l]\cdots & \ar@{<-}[l] s & \ar@{->}[l] s' & \ar@{<->}[l] \cdots & \ar@{<->}[l] \bullet \\ 
	&&& &\rm{(i)} &&&
}
\]
Here, we admit the case that $t'$ ($s'$) does not exist.
If either $t$ or $s$ is $j$, then $\ORA{C}_+$ is of type $\mathbf{A}$ or $\mathbf{D}$. Therefore, we may assume $t\ne j \ne s$.
In this case, we obtain 
\[
\begin{array}{lll}
	&&\dim_K\Hom_A(\tau^{-1}P_i, X)\\
	&=&\dim_K\Hom_A(\tau^{-1}P_j, X)-\dim_K\Hom_A(P_i, X)\\
	&=&\dim_K\Hom_A(\tau^{-1}P_{t'''}, X)+\dim_K\Hom_A(\tau^{-1}P_{s'''}, X)-\dim_K\Hom_A(P_j, X)\\
	&=&\left(\dim_K\Hom_A(\tau^{-1}P_{t'''}, X)-\dim_K\Hom_A(P_j, X)\right)+\left(\dim_K\Hom_A(\tau^{-1}P_{s'''}, X)-\dim_K\Hom_A(P_j, X)\right)+\dim_K\Hom_A(P_j, X)\\
	                             &=&\dim_K\Hom_A(\tau^{-1}P_t, X)-\dim_K\Hom_A(P_{t"}, X)+\dim_K\Hom_A(\tau^{-1}P_s, X)-\dim_K\Hom_A(P_{s"}, X)+\dim_K\Hom_A(P_j, X),\
                             \end{array}
\]
where $t"\in \dip(t)\setminus \{t'\}$, $s"\in \dip(s)\setminus \{s'\}$, $t'''\in (\pre (t)\cup\{t\})\cap \dis(j)$, and $s'''\in (\pre (s)\cup\{s\})\cap \dis(j)$. Hence we obtain that $\dim_K\Hom_A(\tau^{-1}P_i, X)$ is given by
\[
	\dim_K\Hom_A(P_{t'}, X)-\dim_K\Hom_A(P_{t}, X)+\dim_K\Hom_A(P_{s'}, X)-\dim_K\Hom_A(P_{s}, X)+\dim_K\Hom_A(P_j, X). 
\]
(Here, we assume $\dim_K\Hom_A(P_{t'}, X)=0$ (resp. $\dim_K\Hom_A(P_{s'}, X)=0$) if $t'$ (resp. $s'$) does not exists.)
Then the cassification of indecomposable modules for path algebras of type $\mathbf{A}$ implies  
\[
\dim_K\Hom_A(P_i, \tau X)=\dim_K\Hom_A(\tau^{-1}P_i, X)\le 1.
\]
This is a contradiction.

(3). Suppose $\mathcal{X}_i\ne \emptyset$. We take $X\in \mathcal{X}_i$ and let $x_v:=\dim_K\Hom_A(P_v,X)$.

If $Q$ is of type $\widetilde{\mathbf{D}}$ and $\deg(i)\ge 2$, then
  	$\ORA{C}_+$ is of type $\mathbf{A}$, $\mathbf{D}$. This is a contradiction. Hence, it follows from (2) that
  	we may assume $Q$ is of type $\widetilde{\mathbf{E}}_6$, $\widetilde{\mathbf{E}}_7$, or $\widetilde{\mathbf{E}}_8$.
  	
  	Assume that $Q$ is of type $\widetilde{\mathbf{E}}_6$. By (2), $\ORA{C}_+$ has one of the following forms.
  	
  		 $\xymatrix@=12pt{
  		&&i\ar[d]&&&&\\
  		p\ar@{<->}[r]&q\ar[r]&j &\ar[l] r &\ar@{<->}[l]s\\
  		&&\rm{(i)}&&
  	}$
   $\xymatrix@=12pt{
  	&&i\ar[d]&&&&\\
  	p\ar@{->}[r]&q\ar@{<-}[r]&j &\ar[l] r &\ar@{<->}[l]s\\
  	&&\rm{(ii)}&&
  }$
  $\xymatrix@=12pt{
	&&i\ar[d]&&&&\\
	p\ar@{<-}[r]&q\ar@{<-}[r]&j &\ar[l] r &\ar@{<->}[l]s\\
	&&\rm{(iii)}&&
}$

$\xymatrix@=12pt{
	&&i\ar[d]&&&&\\
	p\ar@{->}[r]&q\ar@{<-}[r]&j &r\ar@{<-}[l]  &\ar@{->}[l]s\\
	&&\rm{(vi)}&&
}$
$\xymatrix@=12pt{
	&&i\ar[d]&&&&\\
	p\ar@{<-}[r]&q\ar@{<-}[r]&j &r\ar@{<-}[l] &\ar@{->}[l]s\\
	&&\rm{(v)}&&
}$
$\xymatrix@=12pt{
	&&i\ar[d]&&&&\\
	p\ar@{<-}[r]&q\ar@{<-}[r]&j &r\ar@{<-}[l] &\ar@{<-}[l]s\\
	&&\rm{(vi)}&&
}$

Then we have
\[\dim_K\Hom_A(\tau^{-1}P_6,X_i)= \left\{\begin{array}{cl}
x_q+x_r-x_j & {\rm (i)} \\
x_r-(x_q-x_p) & {\rm (ii)} \\
x_r-x_p &{\rm (iii)\ }\\
x_j-(x_q-x_p)-(x_r-x_s)& {\rm (iv)}\\
x_j-x_p-(x_r-x_s)& {\rm (v)}\\
x_j-x_p-x_s&  {\rm (vi)}\\
\end{array}\right.\]  
Therefore, it follows from $X\in \ind K\ORA{C}$ and $\ORA{C}$ is of type $\mathbf{A}_5$, we have
\[
\dim_K\Hom_A(\tau^{-1}P_i,X)\le 1.
\]
This is a contradiction.

Assume that $Q$ is of type $\widetilde{\mathbf{E}}_7$. By (2), we may assume $\ORA{C}_+$ has the following forms.

\[\xymatrix@=13pt{
	&&5\ar@{<->}[d]^{\epsilon_5}&&&&\\
	i\ar[r]&j:=0\ar@{<->}[r]^-{\epsilon_1}&1 \ar@{<->}[r]^{\epsilon_2} &2 \ar@{<->}[r]^{\epsilon_3} &3\ar@{<->}[r]^{\epsilon_4}& 4\\
}\]
where $\epsilon_1,\dots,\epsilon_5\in \{\pm\}$ and, for a pair $(a<b)$, $a \stackrel{+}\leftrightarrow b$ 
(resp. $a \stackrel{-}\leftrightarrow b$) means 
$a\rightarrow b$ (resp. $a\leftarrow b$).

Then we have
\[\dim_K\Hom_A(\tau^{-1}P_6,X)= \left\{\begin{array}{cl}
x_1-x_0 & (\epsilon_1=-) \\
x^{\epsilon_5}_5-(x_1-x_2) & (\epsilon_1=+,\epsilon_2=-) \\
x^{\epsilon_5}_5-(x_2-x_3) &(\epsilon_1=\epsilon_2=+,\epsilon_3=-) \\
x^{\epsilon_5}_5-(x_3-x_4) & (\epsilon_1=\epsilon_2=\epsilon_3=+,\epsilon_4=-) \\
x^{\epsilon_5}_5-x_4 & (\epsilon_1=\epsilon_2=\epsilon_3=\epsilon_4=+), \\
\end{array}\right.\]
  where $x^{\pm}_5$ is given by the following.
  \[
  \begin{array}{lllllll}
  x^+_5&:=&\dim_K\Hom_A(\tau^{-1}P_5,X_i)&=&x_1-x_5.\\ 
  x^-_5&:=&\dim_K\Hom_A(P_5,X_i)&=&x_5\\
  \end{array}
  \] 
Therefore, it follows from $X_i\in \ind K\ORA{C}$ and $\ORA{C}$ is of type $\mathbf{D}_6$, we have
\[
\dim_K\Hom_A(\tau^{-1}P_i,X)\le 1.
\]
  This is a contradiction.
  
  Assume that $Q$ is of type $\widetilde{\mathbf{E}}_8$. By (2), we may assume $\ORA{C}_+$ has the following forms.
  	  	
  	  	$
  	\xymatrix@=13pt
  	{
  		&   & & &6\ar@{<->}[d]^{\epsilon_6} &\\ 
  	i &\ar@{<-}[l] j=0 &\ar@{<->}[l]_-{\epsilon_1}   1  &\ar@{<->}[l]_{\epsilon_2}  2  &\ar@{<->}[l]_{\epsilon_2}  3  &\ar@{<->}[l]_{\epsilon_4} 4 & \ar@{<->}[l]_{\epsilon_4} 5\\ 
  	&&& &\rm{(i)} &&&
  	}
$
  	  	$
\xymatrix@=13pt
{
	   & & &5\ar@{<->}[d]^{\epsilon_5} &\\ 
	i &\ar@{<-}[l] j=0 &\ar@{<->}[l]_-{\epsilon_1}   1  &\ar@{<->}[l]_{\epsilon_2}  2  &\ar@{<->}[l]_{\epsilon_2}  3  &\ar@{<->}[l]_{\epsilon_4} 4 \\ 
	&&& \rm{(ii)} &&&
}
$
$
\xymatrix@=13pt
{
	& & 4\ar@{<->}[d]^{\epsilon_4} &\\ 
	i &\ar@{<-}[l] j=0 &\ar@{<->}[l]_-{\epsilon_1}   1  &\ar@{<->}[l]_{\epsilon_2}  2  &\ar@{<->}[l]_{\epsilon_2}  3  \\ 
	&&\rm{(iii)}&  &&&
}
$
 
Suppose that $\ORA{C}_+$ has the form (ii) or (iii) and $Y\in \ind K\ORA{C}_+\subset\ind A$.
Then, by the classification of indecomposable modules for path algebras of type $\mathbf{E}_7$, $\mathbf{E}_6$,
we have
\[
\dim_K\Hom_A(P_i, Y)\le 1.
\] 
This contradicts (1) and we may assume $\ORA{C}_+$ has the form (i).
  	Then we obtain
  	\[
  	\dim_K\Hom_A(\tau^{-1}P_i, X)=\left\{
  	\begin{array}{cl}
  	x_1-x_0 & (\epsilon_1=-)\\
  	x_2-x_1 & (\epsilon_1=+,\epsilon_2=-)\\
  	x_3-x_2 & (\epsilon_1=\epsilon_2=+,\epsilon_3=-)\\
  	x^{\epsilon_6}_6-(x_3-x_4) & (\epsilon_1=\epsilon_2=\epsilon_3=+,\epsilon_4=-)\\
  	x^{\epsilon_6}_6-(x_4-x_5) & (\epsilon_1=\epsilon_2=\epsilon_3=\epsilon_4=+,\epsilon_5=-)\\
  	x^{\epsilon_6}_6-x_5 & (\epsilon_1=\epsilon_2=\epsilon_3=\epsilon_4=\epsilon_5=+)\\
  	\end{array}
  	\right.
  	\]
  	where, $x^{\pm}_6$ is given by the following.
  \[
  \begin{array}{lllllll}
	x^+_6&:=&\dim_K\Hom_A(\tau^{-1}P_6,X)&=&x_3-x_6\\ 
	x^-_6&:=&\dim_K\Hom_A(P_6,X)&=&x_6.\\
\end{array}
\]
  	Since $X\in \ind K\ORA{C}\subset\ind A$, it follows from the classification of the indecomposable modules of $K\ORA{C}$ that
  	\[
  	\dim_K\Hom_A(\tau^{-1}P_i, X)\le 1.
  	\]
This is a contradiction. Therefore, we have the assertion.	
\end{proof}		 
\subsubsection{Technical propositions}

First we show the following lemma.
\begin{lemma}
	\label{mgs:keylemma}
	Let $Q$ be a quiver of type $\widetilde{\mathbf{D}}$ or $\widetilde{\mathbf{E}}$, $i$  a source vertex of $Q$ and $i'$ a sink vertex of $Q$, then the following statements
	hold.
	\begin{enumerate}[{\rm (1)}] \item 
		Let $\omega=(T_0\to T_1 \to \cdots \to T_{\ell(Q)})\in \MGSm(A) $ with $t=t_{\omega}^{(i)}<\ell(Q)+1$ (or equivalently, $T_{\ell(Q)-1}\not\simeq P_i\simeq S_i$).
		If we write 
		\[T_t=M\oplus X_i\oplus P_i^-,\ T_{t-1}=M\oplus \overline{X_i} \oplus P_i^-,
		\] then one of the following statements holds.
		\begin{enumerate}[{\rm (a)}]
			\item $X_i\in \mathcal{X}_i$ and $M\in \add \mathcal{R}$.
			\item There exists a maximal green sequence 
			\[T_0\to \cdots \to T_{t-2}\to T'_{t-1}\to T_t \to \cdots \to T_{\ell(Q)} \]
			with $P_i^-\not\in \add T'_{t-1}$.
		\end{enumerate}
		\item  
		Let $\omega=(T_0\to T_1 \to \cdots \to T_{\ell(Q)})\in \MGSm(A) $ with $s=s_{\omega}^{(i')}>-1$.
		If we write 
		\[T_s=M'\oplus X'_{i'}\oplus P_{i'},\ T_{s+1}=M'\oplus \overline{X'_{i'}} \oplus P_{i'},
		\] then one of the following statements holds.
		\begin{enumerate}[{\rm (a)}]
			\item $X'_{i'}\in \mathcal{X}'_{i'}$ and $M'\in \add \mathcal{R}$.
			\item There exists a maximal green sequence 
			\[T_0\to \cdots \to T_s\to T'_{s+1}\to T_{s+2} \to \cdots \to T_{\ell(Q)} \]
			with $P_{i'}\not\in \add T'_{s+1}$.
		\end{enumerate}
		\item $\MGSm(A, S_i)\neq\emptyset$ $($resp. $\MGSm(A, \mu_{i'}A)\neq \emptyset)$ implies
		\[\ell(Q)\le \ell(\mu_i Q)\ (\text{resp}.\  \ell (Q)\le \ell (\mu_{i'}Q)).\]
			\end{enumerate}
		\end{lemma}
	\begin{proof}
		Let $Q':=Q^{{\rm op}}$, $A':=A^{{\rm op}}=KQ'$, $B=K(\mu_i Q)$.

		(1). Let $T'_{t-1}\in \dip(T_t)$ such that $P_{i}^-\not\in T'_{t-1}$.
		Then  $T_{t}, T_{t-1}, T_{t-2}$ and $T'_{t-1}$ are in $\sttilt_M A$. 
		By applying anti-poset isomorphism  $\dagger:\sttilt A\to \sttilt A'$, we have 
		
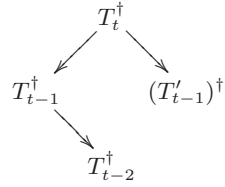
\[\begin{xy}
		(0,0)*[o]+{T_t^{\dagger}}="A", (-10,-10)*[o]+{T_{t-1}^{\dagger}}="B", (10,-10)*[o]+{(T'_{t-1})^{\dagger}}="C", (0,-20)*[o]+{T_{t-2}^{\dagger}}="D",
		\ar "A";"B"
		\ar "A";"C"
		\ar "B";"D"  
		\end{xy}
		\]  in $\sttilt_{M^{\dagger}} A'$.
		In particular, $T_t^{\dagger}$ is the Bongartz completion of $M^{\dagger}$.		
		Since $\omega\in \MGSm(A)$, $\sttilt_{M^{\dagger}} A'$ has the form (i) or (iii) in Lemma\;\ref{lemma_epd}. 
		Note that if $\sttilt_M A$ is finite (or equivalently, $\sttilt_{M^\dagger} A'$ is finite), then
		$\sttilt_{M^\dagger} A'$ has the form (i) of Lemma\;\ref{lemma_epd}. This implies
		 $\sttilt_M A$ has the following form and we obtain (b).
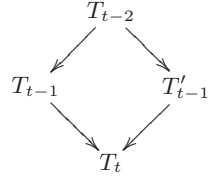
		 \[\begin{xy}
		 (0,0)*[o]+{T_{t-2}}="A", (-10,-10)*[o]+{T_{t-1}}="B", (10,-10)*[o]+{T'_{t-1}}="C", 
		 (0,-20)*[o]+{T_t}="D",
		 \ar "A";"B"
		 \ar "A";"C"
		 \ar "B";"D"  
		 \ar "C";"D"
		 \end{xy}
		 \] 
		 
		Now we assume that either $X_i\not\in \mathcal{X}_i$ or $M\not\in\add \mathcal{R}$ holds.
         
         If $M\not\in \add \mathcal{R}$, then there exists a pair $(r, a)\in \Z\times Q_0$ such that
         \[\tau^r P_a^-\in \add M.\]	
         Then we have
         \[\begin{xy}
         (0,0)*[o]+{\sttilt_M A\subset \sttilt_{\tau^r P_a^-} A},
         (25,0)*[0]+{}="A",
         (45,0)*[o]+{}="B",
         (70,0)*[o]+{\sttilt_{P_a^-} A =\sttilt A/(e_a).},
         \ar @<3pt>"A";"B"^{\tau^{-r}}
         \ar @<3pt>"B";"A"^{\tau^r}.	
         \end{xy}\]
         Since $A$ is a tame hereditary algebra, $A/(e_a)$ is representation-finite. In particular, $\sttilt_M A$ is finite   
         and we obtain (b).
         
		Hence we may assume $X_i\not\in \mathcal{X}_i$ and $M\in \add \mathcal{R}$. Then one of the following three cases occurs.
		\begin{itemize}
			\item $X_i\in \add A^-$.
			\item $X_i$ is projective or equivalently,  $X_i^{\dagger}\in \add (A')^-$.
			\item $X_i^{\dagger}\in \ind A'\setminus \mathcal{X}_i'(A')$ such that 
			$\Hom_{A'}(P_i', \tau_{A'}X_i^{\dagger})=0.$ 
		\end{itemize}
	    Since $(P_{\bullet}^-)^{\dagger}\in \proj A'$, the first case is included in the third case.  
		The second case implies that $X_i\oplus T_{t-1}$ should be $\tau$-rigid. This is a contradiction.

        Therefore, we may assume the third case occurs.
        Since $i$ is a sink vertex of $Q'$ and $X_i^{\dagger}\not\simeq P_i'$, we have
		\[\Hom_{A'}(X_i^{\dagger}, P_i')=0,\ \Hom_{A'}(P_i', X_i^{\dagger})\le 1.\]
		By Jesso's reduction theorem(Theorem\;\ref{Jasso'sreduction}), we can take an acyclic quiver $\Delta$ with two vertices $u,v$
		such that  
		\[
		K\Delta\cong \dfrac{\End_{A'}(X_i^{\dagger}\oplus P_i'\oplus M^{\dagger})}{[M^{\dagger}]}\cong \dfrac{\End_{A'}(X_i^{\dagger}\oplus P_i')}{[M^{\dagger}]},\ \sttilt K\Delta\simeq \sttilt_{M^{\dagger}} A'.
		\]
		 Thus, $\dim e_u(K\Delta) e_v,\;\dim e_v(K\Delta)e_u\le 1$ and $\num\Delta_1\le 1$. This shows that
		$K\Delta$ is a representation-finite algebra. In particular, $\sttilt_{M^{\dagger}}A'$
		is a finite poset and we obtain (b).

		
		Therefore, we have the assertion (1).
		
	(2). Let $\omega=(T_0\to T_1 \to \cdots \to T_{\ell})\in \MGSm(A)$ with $s=s_{\omega}^{(i')}>1$. Then applying anti-poset isomorphism 
	$\dagger:\sttilt A \to \sttilt A'$ to $\omega$, we obtain $\omega'=\omega^{\dagger}\in \MGSm(A')$
	satisfying $s=t_{\omega'}^{(i')}<\ell(Q)+1=\ell(Q')+1$.
	Since $Q'$ is a quiver of type $\widetilde{\mathbf{D}}$ or $\widetilde{\mathbf{E}}$ and $i'$ is a source vertex of $Q'$,
	we can apply (1) to $\omega'$, i.e. one of the following statements holds. 
	\begin{enumerate}[{\rm (a')}]
		\item $X^{\dagger}_i\in \mathcal{X}_i(A')$ and $M^{\dagger}\in \add \mathcal{R}(A')$.
		\item There exists a maximal green sequence 
		\[T^{\dagger}_0\to \cdots \to T^{\dagger}_{s-2}\to (T'_{s-1})^{\dagger}\to T^{\dagger}_s \to \cdots \to T^{\dagger}_{\ell(Q)}\]
		with $(P^{A'}_{i'})^-\not\in \add (T'_{s-1})^{^{\dagger}}$.
	\end{enumerate}	
    Since $\Tr$ induces the bijection between $\mathcal{X}_{i'}(A')$ and $\mathcal{X}'_{i'}(A)$,
    the assertion (2) follows from the definition of $\dagger$.
		
	(3). Consider $\MGSm(A,S_i)\ni \omega:T_0\to \cdots \to T_{\ell(Q)}$.
	Since $T_{\ell(Q)-1}=S_i$, by applying a poset isomprphism
	 $\psi^{-1}:(\sttilt A)_{\ge S_i}\simeq (\sttilt B)_{\le \mu_i B}$ in Theorem\;\ref{ladkani}, we obtain
	a maximal green sequence 
	\[\omega':B\to \mu_{i}B=\psi^{-1}(T_0)\to \cdots \to \psi^{-1}(T_{\ell(Q)-2})\to \psi^{-1}(T_{\ell(Q)-1})=0\]	 
	of $B$ with $\ell(\omega')=\ell(Q)$.
	This gives the assertion (3).
			\end{proof}

	\begin{proposition}
		\label{mgs:keyprop}
Let $Q$, $i$ and $i'$ be as in Lemma\;\ref{mgs:keylemma}.
\begin{enumerate}[{\rm (1)}]
\item If $i$ is a source $($resp. sink$)$ vertex and $\mathcal{X}_i=\emptyset$ $($resp. $\mathcal{X}'_{i'}=\emptyset$$)$, then we have
	\[\ell (Q)\le \ell (\mu_i Q).\]
	In particular, if $\deg i\ge 2$, then we have 
	\[
	\ell(Q)=\ell(\mu_i Q).
	\]
\item If there exists a quiver automorphism $\sigma$ of $Q$ satisfying $\sigma(i)\neq i$ (resp. $\sigma(i')\ne i'$), then
\[
\MGSm(A,S_{i})\ne \emptyset\ (\text{resp. }\MGSm(A,\mu_{i'}A)\ne \emptyset).
\]
In particular,
	\[
	\ell (Q)\le \ell (\mu_{i} Q)\ (\text{resp. }\ell (Q)\le \ell (\mu_{i'} Q)).
	\]
\item $\MGSm(A,S_i)=\emptyset$ (resp. $\MGSm(A,\mu_{i'}A)=\emptyset$) if and only if
$t^{(i)}\le \ell(Q)$ (resp. $s^{(i')}\ge 0$).
In particular, if $t^{(i)}=\ell(Q)+1$ (resp. $s^{(i')}=-1$), then we have
\[
\ell (Q)\le \ell (\mu_{i} Q)\ (\text{resp. }\ell (Q)\le \ell (\mu_{i'} Q)).
\]
\end{enumerate}
	\end{proposition}
\begin{proof}
	(1). The assertion (1) follows from Lemma\;\ref{findXi } and Lemma\;\ref{mgs:keylemma}.
	
	(2). Take $\omega=(T_0\to \cdots \to T_{\ell(Q)})\in \MGSm(A)$ satisfying the following two conditions.
	\begin{itemize}
		\item $t:=t_{\omega}^{(i)}=t^{(i)}$
		\item $t':=t_{\omega}^{(\sigma(i))}=\max\{t_{\omega'}^{(\sigma(i))}\mid \omega'\in \MGSm(A)\text{ with }t_{\omega'}^{(i)}=t^{(i)}\}$
	\end{itemize}
	  If $t=\ell(Q)+1$, then we have nothing to show.
	Therefore, we may assume $t\ge \ell (Q)$.
	We write
	\[T_t=M\oplus X \oplus P_i^-,\ T_{t-1}=M\oplus \overline{X}\oplus P_i^-,\ T_{t'}=M'\oplus X' \oplus P_{\sigma(i)}^-, T_{t'-1}=M'\oplus \overline{X'}\oplus P_{\sigma(i)}^-.\]
	
	If $t>t'$, then $P_{\sigma(i)}^-\in \add M\oplus X$. This implies that
	either $M\not\in \add\mathcal{R}$ or $X\not\in \mathcal{X}_i$ holds. Then it follows from Lemma\;\ref{mgs:keylemma}(1) that
	there exists $\omega'\in \MGSm(A)$ such that $t_{\omega'}^{(i)}>t$. This contradicts the maximality of $t$.
	
	If $\ell(Q)\le t'<t$, then $P_{i}^-\in \add M'\oplus X'$. This implies that
	either $M'\not\in \add\mathcal{R}$ or $X'\not\in \mathcal{X}_{\sigma(i)}$ holds. Then it follows from Lemma\;\ref{mgs:keylemma}(1) that
	there exists $\omega'\in \MGSm(A)$ such that $t_{\omega'}^{(\sigma(i))}>t'$ and $t_{\omega'}^{(i)}=t$. 
	This contradicts the maximality of $t'$.
	
	Hence, we have $\ell(Q)-1=t'<t$. In particular, 
	\[
	\MGSm(A,S_{\sigma(i)})\ne \emptyset.
	\]
	Then isomorphism $A\stackrel{\cong}{\to} A$ given by $\sigma^{-1}$ implies
	\[
	\MGSm(A,S_{i})\ne \emptyset.
	\] 
	By considering poset anti-isomorphism $\dagger:\sttilt A\to \sttilt A^{{\rm op}}$,
	we also have
	\[
	\MGSm(A,\mu_{i'}A)\stackrel{1:1}{\longleftrightarrow} \MGSm(A^{{\rm op}},S_{i'}^{A^{{\rm op}}})\ne \emptyset.
	\]
	This finishes the proof.
	
	(3). The assertion (3) follows from Lemma\;:\ref{mgs:keylemma}.
	\end{proof}

Let $(i,\mathbf{i}:=(i_1,\dots,i_m),L)$ be a triple satisfying the following conditions.
\begin{itemize}
	\item $i$ is a source vertex of $Q$.
	\item $i_p\ne i_q$ if $p\ne q$.
	\item $Q=Q^{(0)}\stackrel{\mu_{i_1}}{\to} Q^{(1)}\stackrel{\mu_{i_2}}{\to}\cdots  \stackrel{\mu_{i_m}}{\to} Q^{(m)}$ is a sink mutation sequence, i.e. $i_k$ is sink in $Q^{(k-1)}$ for each $k\in \{1,\dots,m\}$.
	\item $i\not\in\{i_1,\dots,i_m\}$.
	\item $L\in \ind A$.
\end{itemize}	
For such a triple $(i,\mathbf{i},L)$, we consider
 the following assumption.
 \begin{assumption}
 	\label{assumption}
Let $(i,\mathbf{i},L)$ be as above, $A=A^{(0)}=KQ$, $A^{(p)}=KQ^{(p)}$ $(0\le p\le m-1)$.
 \begin{description}
 	\item[$\mathsf{A1}$] $\mathcal{X}_i(A)$ contains a unique module $X_i$ up to isomorphism.
 	\item[$\mathsf{A2}$] If $M\oplus X_i \oplus P_i^-\in \sttilt A$ with $M\in \add \mathcal{R}$, then $L\in \add M.$
 	\item[$\mathsf{A3}$] $\Hom_{A^{(p-1)}}\left(P_{i_p}^{A^{(p-1)}},\tau_{A^{(p-1)}}(F_{i_{p-1}}^+\circ\cdots\circ F_{i_1}^+(L))\right)\ne 0$ for each $p\in\{1,\dots, m\}$.
 	\item[$\mathsf{A4}$] For each $p\in \{1,\dots, m\}$, one of the following statements holds.
 	\begin{enumerate}[{\rm (i)}]
 		\item $\mathcal{X}'_{i_p}(A^{(p-1)})=\emptyset$.
 		\item $\mathcal{X}'_{i_p}(A^{(p-1)})$ contains a unique module $X'_{i_p}\in \mod A^{(p-1)}$ up to isomorphism. In addition, 
 		if $M'\oplus X_{i_p}' \oplus P^{A^{(p-1)}}_{i_p}\in \sttilt A^{(p-1)}$ with $M'\in \add \mathcal{R}(A^{(p-1)})$, then 
 		$F_{i_{p-1}}^+\circ\cdots\circ F_{i_1}^+(L)\in \add \tau_{A^{(p-1)}} M'$. 
 	\end{enumerate}
 \item[$\mathsf{A5}$] $\ell(\mu_{i_{p}}\cdots \mu_{i_1}Q)\le \ell(Q) $ holds for each $p\in\{1,\dots, m-1\}$.
\end{description}
   \end{assumption}
\begin{proposition}
	\label{mgs:keyprop2}
	Let $Q$ be a quiver of type $\widetilde{\mathbf{D}}$ or $\widetilde{\mathbf{E}}$ and $B=A^{(m-1)}$. 
	Assume that $(i,\mathbf{i},L)$ satisfies Assumption\;\ref{assumption} and 
	one of the following conditions holds. 
\begin{enumerate}[{\rm (i)}]
	\item $\ell(Q)\not\le \ell(\mu_i Q)$.
	\item There exists a maximal green sequence
	\[
	\omega: T_0\to \cdots \to T_{\ell(Q)} 
	\]	
	in $\MGSm(A)$
	such that $T_{t_{\omega}^{(i)}}=M\oplus X_i\oplus P^-_i$, $T_{t_{\omega}^{(i)}-1}=M\oplus \overline{X_i}\oplus P^-_i$ with $M\in \add\mathcal{R}$.
\end{enumerate}	
Then we have the following statements.
\begin{enumerate}[{\rm (1)}]
	\item There exists a maximal green sequence 
	\[
	\omega^B :B=T^B_0\to T^B_1\to \cdots \to T_{\ell(Q^{(m-1)})}=0. 
	\]
	in $\MGSm(B)$ such that $P_{i_m}^B\not \in \add T_1^B $.

\item We have
\[
\ell(Q)= \ell(\mu_{i_1}Q)=\cdots=\ell(\mu_{i_{m-1}}\cdots \mu_{i_1}Q)\le \ell (\mu_{\mathbf{i}}Q).
\]
\item If $(i,\mathbf{i}'=(i_1,\dots,i_{m-1},i'),L')$ also satisfies Assumption\;\ref{assumption} and $\ell(\mu_{\mathbf{i}}Q)\le \ell(Q)$ holds, then
we have 
\[
\ell(Q)\le \ell(\mu_{i'}\mu_{\mathbf{i}}Q).
\]
\end{enumerate}
\end{proposition}
\begin{proof}
Let $\ell_p:=\ell(\mu_{i_{p}}\cdots \mu_{i_1}Q)$, $P_{\bullet}^{(p)}:=P_{\bullet}^{A^{(p)}}$, and
$F_{p}^+:=F_{i_{p}}^+\circ\cdots\circ F_{i_1}^+$ for each $p\in\{0,\dots,m-1\}$.

We first show the following claim.
\begin{claim}
	\label{claim:keyprop2}
	 For any $p\in\{0,1,\dots,m-1\}$ and for any $M\in \add\mathcal{R}$ with $|M|=n-2$, we define
	 $\boldsymbol{\Omega}_{p}(M)\subset \MGS(A^{(p)})$ as follows: $\omega^{(p)}=(T^{(p)}_0\to \cdots \to T^{(p)}_{\ell'})\in \boldsymbol{\Omega}_{p}(M)$
	 if and only if we can write
	 \[
	 T^{(p)}_{t}=F_{p}^+(M)\oplus F_{p}^+(X_i)\oplus (P_i^{(p)})^-,\ T^{(p)}_{t-1}=F_{p}^+(M)\oplus \overline{F_{p}^+(X_i)}\oplus (P_i^{(p)})^-
	  \]
	 where $t:=t^{(i)}_{\omega^{(p)}}$. We also define $\boldsymbol{\Omega}^{\max}_p(M)\subset \MGSm(A^{(p)})$ by
	 \[
	 \boldsymbol{\Omega}^{\max}_p(M):=\boldsymbol{\Omega}_{p}(M)\cap\MGSm(A^{(p)})
	 \]
%
Then, for each $p\in \{0,\dots, m\}$, we have $\mathsf{c}_p$ and $\mathsf{c}'_p$. 
\begin{description}
	\item[$\mathsf{c}_p$] $\ell_{p-1}\le \ell_p$. $($Here, we put $\ell_{-1}:=\ell_0$ and $\ell_m:=\ell_{m-1}$.$)$
	\item[$\mathsf{c}'_p$ ($0\le p \le m-1$)] 
	There exists $M\in \add\mathcal{R}$ such that $M\oplus X_i\oplus P_i^-\in \sttilt A$ and $\Omega_p^{\max}(M)\ne\emptyset$. 
	\item[$\mathsf{c}'_m$] There are $M\in \add \mathcal{R}$ and a maximal green sequence 
	\[
	\omega^B :B=T^B_0\to T^B_1\to \cdots \to T_{\ell(Q^{(m-1)})}=0, 
	\]
	such that $\omega^B$ in $\Omega_{m-1}^{\max}(M)$ and $P_{i_m}^B\not \in \add T_1^B $.
\end{description}
  \end{claim}	
\begin{pfclaim}
	We show the assertion by using an induction on $p$. Let $\ell:=\ell_0=\ell(Q)$.
	
	$(p=0):$ If the condition (ii) holds, then we have nothing to show. Thus, we may assume $\ell(Q)\not\le \ell(\mu_i Q)$. 
	Take
	\[\MGSm(A)\ni \omega: T_0\to  \cdots \to T_{\ell}=0\]
	such that
	 $t:=t_{\omega}^{(i)}=\max\{t_{\omega'}^{(i)}\mid \omega'\in \MGSm(A) \}.$
If $t=\ell+1$, then it follows from Lemma\;\ref{mgs:keylemma} (3) that
\[
\ell(Q)\le  \ell(\mu_i Q).
\]
This is a contradiction.
Hence, we have $t\le \ell$. Then Lemma\;\ref{mgs:keylemma} (1) and the maximality of $t$ imply 
\[
T_t=M\oplus X_i\oplus P_i^-,\ T_{t-1}=M\oplus \overline{X_i}\oplus P_i^-
\] 
for some $M\in \add\mathcal{R}$.

$(p>0):$ Note that $\mathsf{c}_0,\dots, \mathsf{c}_{p-1}$, and $\mathsf{A}_{5}$ imply
\[
\ell=\ell_{p-1}.
\]  
Then take $\omega_{p-1}=
\left(
T_0^{(p-1)}\to \cdots\to T^{(p-1)}_\ell
\right)
\in \boldsymbol{\Omega}^{\max}_{p-1}(M)$ such that
\[
s:=s_{\omega_{p-1}}^{(i_p)}=\min\{s_{\omega'}^{(i_{p-1})}\mid \omega'\in \boldsymbol{\Omega}^{\max}_{p-1}(M)\},
\]
and put $t:=t_{\omega_{p-1}}^{(i)}$.

Suppose that $t-1\le s+2$. Then we have $T^{(p-1)}_{t-2}\ge T^{(p-1)}_{s+1}$. By definition, we can write
\[
T^{(p-1)}_{t-2}=F_{p-1}^+(M)\oplus \overline{F_{p-1}^+(X_i)}\oplus \overline{(P_i^{(p-1)})^-}.
\]
Therefore, we obtain $F_{p-1}^+(M)\in \add T_{t-2}^{(p-1)}$ and 
\[
\Hom_{A^{(p-1)}}\left(P^{(p-1)}_{i_p}, \tau_{A^{(p-1)}}\left(F_{p-1}^+(M)\right) \right)=0.
\]
On the other hand, it follows from $\mathsf{A}_2$ and $\mathsf{A}_3$ that 
\[
F_{p-1}^+(L)\in \add \left(F_{p-1}^+(M)\right) \text{ and } \Hom_{A^{(p-1)}}\left(P^{(p-1)}_{i_p}, \tau_{A^{(p-1)}}\left(F_{p-1}^+(L)\right) \right)\ne 0.
\] 
This is a contradiction. Therefore, we obtain $s+2 \le t-2$.

Now we write 
\[T^{(p-1)}_s=M'\oplus X'\oplus P^{(p-1)}_{i_p},\ T_{s+1}=M'\oplus \overline{X'} \oplus P^{(p-1)}_{i_p}.
\] 
Suppose that $s\ge 0$. Then, by applying Lemma\;\ref{mgs:keylemma} (2), one of the following conditions holds.
\begin{enumerate}[{\rm (a)}]
	\item $X'\in \mathcal{X}'_{i_p}(A^{(p-1)})$ and $M'\in \add \mathcal{R}(A^{(p-1)})$.
	\item There exists a maximal green sequence 
	\[T^{(p-1)}_0\to \cdots \to T^{(p-1)}_s\to T'_{s+1}\to T^{(p-1)}_{s+2} \to \cdots \to T_{t-2}^{(p-2)}\to T_{t-1}^{(p-1)}\to T_t^{(p-1)}\to \dots \to T_{\ell} \]
	with $P^{(p-1)}_{i_p}\not\in \add T'_{s+1}$.
\end{enumerate} 
The condition (b) contradicts the minimality of $s$, and if the condition $\mathsf{A4}$(i) in Assumption\;\ref{assumption} holds for $p$, then 
the condition (a) does not occur. Therefore, the condition (a) and the condition $\mathsf{A4}$(ii) in Assumption\;\ref{assumption} hold for $p$.
In particular, we may assume $X'=X'_{i_p}$, $M'\in \add\mathcal{R}(A^{(p-1)})$, and obtain
\[
0\ne F_{p-1}^+(L)\in \add F_{p-1}^+(M)\cap\add \tau_{A^{(p-1)}} M'.
\]  
This contradicts $T_t\le T_s$. Hence, we have $s=-1$. 
Note that this already showed an implication 
\[
\mathsf{c}_{m-1}\ \&\  \mathsf{c}'_{m-1}\Rightarrow \mathsf{c}'_m.
\]
If $p\le m-1$, then, by applying $\psi_{i_p}:\sttilt_{\le \mu_{i_p} A^{(p-1)}}A^{(p-1)} \simeq \sttilt_{\ge S^{(p)}_{i_p}} A^{(p)}$ to $\omega_{p-1}$, we obtain 
a maximal green sequence of length $\ell=\ell_{p-1}$ and it is in 
$\boldsymbol{\Omega}_p(M)$. In particular, we have an implication
 \[
 \mathsf{c}_{p-1}\ \&\  \mathsf{c}'_{p-1}\Rightarrow \mathsf{c}_p \text{ for each $p\in \{1,\dots, m\}$}.
 \]
 Furthermore,  $\mathsf{A}_5$ implies $\ell_p=\ell$. This shows an implication
 \[
 \mathsf{c}_{p-1}\ \&\  \mathsf{c}'_{p-1}\Rightarrow \mathsf{c}'_p \text{ for each $p\in\{1,\dots, m-1\} $}.
 \]
 \end{pfclaim}
The assertion (1) follows from Claim\;\ref{claim:keyprop2}. Then the assertion (2) follows from (1), Claim\;\ref{claim:keyprop2}, and Lemma\;\ref{mgs:keylemma} (3).

We prove (3). We set
\[
\boldsymbol{\Omega}_{m-1}^{\max}(M,\mu_{i_m}B):=\boldsymbol{\Omega}_{m-1}^{\max}(M)\cap \MGSm(B,\mu_{i_m}B).
\]
By Claim\;\ref{claim:keyprop2}, we can
take $\omega'_{B}=
\left(
T'_0\to \cdots\to T'_\ell
\right)
\in \boldsymbol{\Omega}_{m-1}^{\max}(M,\mu_{i_m}B)$ such that
\[
s:=s_{\omega'_{B}}^{(i')}=\min\{s_{\omega'}^{(i')}\mid \omega'\in \boldsymbol{\Omega}^{\max}_{m-1}(M,\mu_{i_m}B)\}.
\]
Suppose that $s\ge 1$. Then the same argument used in the proof of Claim\;\ref{claim:keyprop2} $(p>0)$ implies a contradiction. Therefore, we have
\[
s=0.
\]
Since $\ell(\mu_{\mathbf{i}}(Q))\le \ell(Q)$, it follows from (2) that
\[
\ell(Q)=\ell(\mu_{\mathbf{i}}Q).
\]
 Then the statement (3) follows from Proposition\;\ref{ladkani}.
\end{proof}

\section{A computational approach}
\label{computational approach}
In this section, we give a program which counts all maximal green sequences of the path algebra $KQ$ by length for a given tame quiver $Q$.
We define 
\[
\begin{array}{rll}
\Hasse_{\mathrm{fin}}(\sttilt A)&:=&\Hasse(\suc(A)\cap \pre(0))\\
\end{array}
\]
Note that $\Hasse_{\mathrm{fin}}(\sttilt A)$ is the full subquiver of $\Hasse(\sttilt A)$ given by the support $\tau$-tilting modules which appear in $\MGS(A)$.
\subsection{Indecomposable modules which may appear in MGS}
For a maximal green sequence 
\[
\omega:(T_0\to \cdots \to T_\ell)\in \MGS(A),
\] we put
\[
\ind\omega:=\underset{k=0}{\overset{\ell}{\cup}}\add T_k 
\]
and define
\[
\ind \MGS(A):=\underset{\omega\in \MGS(A)}{\cup} \ind \omega.
\]
In this subsection, we give a finite subset $\Lambda=\Lambda_A$ of $\ind A \cup \{P_i^-\mid i\in Q_0\}$ satisfying
\[
\ind \MGS(A)\subset \Lambda.
\]
The following proposition is the key to this purpose.
\begin{proposition}[{\cite[Sect.\;4.2(8)]{R}}]
	\label{noregulartilt}
	Let $A$ be a tame hereditary algebra with $|A|=n$. If $R\in \add\mathcal{R}$ is $(\tau$-$)$rigid, then $|R|\le n-2$.  
	In particular, each $T\in \sttilt A$ has at least two indecomposable direct summands in $\mathcal{P}\cup\widetilde{\mathcal{I}}$.
\end{proposition}

\newcommand{\dimv}[1]{\underline{\dim}\,{#1}}
\newcommand{\koho}[1]{\mathfrak{#1}} 
First, consider the candidates for regular modules which may appear in MGS.
\begin{lemma}\label{rinmgs}
	We define $\koho{R}:=\{X\in \mathcal{R}\mid \text{$\tau^{r}X$ is nonsincere for some $r\in\mathbb{Z}$}\}$, then each indecomposable $\tau$-rigid regular module
	is in $\koho{R}$. In particular, the following inclusion relation holds.
	$$\mathcal{R}\cap\ind\MGS(A) \subseteq \koho{R}$$
	\begin{proof}
		Let $X$ be in $\mathcal{R}\cap\ind\MGS(A)$.
		By Proposition \ref{noregulartilt}, there exists $Y\in\mathcal{P}\cup\widetilde{\mathcal{I}}$ such that $X\oplus Y$ is $\tau$-rigid.
		In the case $Y\in\mathcal{P}$, we have $Y\cong \tau^{-r}P_i$ for some $r\in\mathbb{Z}_{\geq 0}$ and $i\in Q_0$. Then $\Hom_A(\tau^{-r}P_i, \tau X)=0$, that is, $\Hom_A(P_i, \tau^{r+1}X)=0$ holds, so $\tau^{r+1}X$ is nonsincere.
		In the case $Y\in\widetilde{\mathcal{I}}$, we similarly obtain $\Hom_A(\tau^{-(r+1)}X, I_i)=0$ for some $r\in\mathbb{Z}_{\geq -1}$ and $i\in Q_0$, namely, $\tau^{-(r+1)}X$ is nonsincere. Either way, we have $X\in\koho{R}$.
	\end{proof}
\end{lemma}

For each $i\in Q_0$, we define 
$p_i:=\min \{r\in\mathbb{Z}_{\geq 0} \mid \text{$\tau^{-s}P_i$ is sincere for all $s\geq r$ } \}$ and 
$q_i:=\min \{r\in\mathbb{Z}_{\geq 0} \mid \text{$\tau^{s}I_i$ is sincere for all $s\geq r$ } \}$. Further, we define $m:=\max\{p_i, q_i\mid i\in Q_0\}.$
\begin{lemma}\label{ppinmgs}
	We define $\koho{P}:=\bigcup_{i\in Q_0}\{\tau^{-r}P_i\mid 0\leq r<m+p_i\}$, then the following inclusion relation holds.
	$$\mathcal{P}\cap\ind\MGS(A) \subseteq \koho{P}$$
	\begin{proof}
		For some $\tau^{-r}P_i\in \mathcal{P}\cap\ind\MGS(A)$, we suppose $m+p_i\leq r$.
		Then, there exists a maximal green sequence $\omega:A=T_0\rightarrow T_1\rightarrow\cdots\rightarrow T_{\ell}=0$ such that $\tau^{-r}P_i\in\ind\omega$. Now, let $\tau^{-r}P_i$ be in $\add T_t\,(0\leq t\leq\ell)$, then the following claim holds.
		\begin{claim*}
			Let $j\in Q_0,\,r'\in \mathbb{Z}_{\geq 0}$ and $t'\in \mathbb{Z}_{\geq t}$.
			Then, $\tau^{-r'}P_j\in \add T_{t'}$ implies $m\leq r'$.
		\end{claim*}
		Once we suppose $r'\leq m-1$.
		By $t\leq t', \Hom_A(T_{t'}, \tau T_t)=0$ holds.
		Thus, $\Hom_A(\tau^{-r'}P_j, \tau(\tau^{-r}P_i))=0$, that is, 
		$\Hom_A(P_j, \tau^{-(r-r'-1)}P_i)=0$ holds.
		Now, $r'\leq m-1$ and $m+p_i\leq r$ hold, so $p_i\leq r-r'-1$ holds.
		However this means that $\tau^{-(r-r'-1)}P_i$ is sincere, that is contradiction. Hence, the above claim stands.
		
		Let $\tau^{-s}P_k$ be the last preprojective direct summand to disappear in the sequence $\omega$.
		Then $m\leq s$ holds by the above claim.
		By Proposition \ref{noregulartilt}, there are integers $s'\geq -1$ and $k' \in Q_0$ such that $\tau^{-s}P_k\oplus\tau^{s'}I_{k'}$ is $\tau$-rigid.
		If $s'=-1$, then $\left(\tau^{-s}P_k\right)e_{k'}=0$. In particular, $\tau^{-s}P_k$ is not sincere. This contradicts $s\ge m$. Therefore, $s\ge 0$ and 
		$\Hom_A(\tau^{-s}P_k, \tau(\tau^{s'}I_{k'}))=0$ holds.
		Then we have $\Hom_A(P_k, \tau^{s+s'+1}I_{k'})=0$. Since $s+s'+1\ge s \ge m$, we have $\tau^{s+s'+1}I_{k'}$ is sincere.
		 This is a contradiction.
	\end{proof}
\end{lemma}
As a dual of Lemma \ref{ppinmgs}, the following lemma also stands. 
\begin{lemma}\label{piinmgs}
	We define $\koho{I}:=\bigcup_{i\in Q_0}\{\tau^{r}I_i\mid 0\leq r<m+q_i\}$, then the following inclusion relation holds.
	$$\mathcal{I}\cap\ind\MGS(A) \subseteq \koho{I}$$
\end{lemma}
	Then we define $\Lambda$ by
	\[
	\koho{P}\cup\koho{R}\cup\koho{I}\cup\{P^{-}_i\mid i\in Q_0\}
	\]
	and put
	\[
	\begin{array}{lll}
	\sttilt A\mid_\Lambda&:=&\{T\in \sttilt A\mid \add T\subset \Lambda\}\\
	\suc(A)\mid_\Lambda &:=&\text{the subposet of $\sttilt A$ given by $\suc(A)\cap \sttilt A\mid_\Lambda$}.
	\end{array}
	\] 
\subsection{Pseudocode}
In this subsection, we give a pseudocode of our program.

Let $A=KQ$ be a finite dimensional path algebra. We assume that $A$ is either a rep-finite algebra or an indecomposable tame algebra. 
\begin{definition}
	\label{prec}
We define a relation $\prec=\underset{A}{\prec}$ on $\ind A\cup \{P_i^-\mid i\in Q_0\}$ as follows:	
\begin{description}
	\item[rep-finite case] Assume that $A$ is rep-finite. In this case, each indecomposable module is preprojective.
	Therefore, we can write
	\[
	\ind A=\mathcal{P}=\underset{i\in Q_0}{\cup}\{\tau^{-r}P_i\mid 0\le r \le k_i\}.
	\]
	Then we define $\prec$ as follows:
	\begin{itemize}
		\item 
		For $i,j\in Q_0$, $r\in\{0,\dots, k_i\}$, and $s\in \{0,\dots, k_j\}$,
		\[
		 \tau^{-s}P_j \prec \tau^{-r}P_i:\Leftrightarrow [r\le s] \text{ or } [r>s \text{ and } \left(\tau^{-(r-s-1)}P_i\right)e_j=0].
		\]
		\item 
		$
		P_i^-\prec X
		$ for each $i\in Q_0$ and $X\in \ind A\cup\{P_j^-\mid j\in Q_0\}$.
		\item If $X\in \ind A$, then
		\[
		X\prec P_i^-:\Leftrightarrow Xe_i=0.
		\]
	\end{itemize}
	\item[tame case] Assume that $A$ is an indecomposable tame algebra. In this case, $A_i:=A_i=K(Q\setminus\{i\})$ is rep-finite for each vertex $i$ of $Q$.  Then the condition for $Y\prec X$ is given as follows.
	
	\begin{center}
		\begin{tabular}{|c|c|c|c|}
			\hline
			$Y$\textbackslash $X$ & $ \tau^{-r}P_i$ & $\mathcal{R}$ & $\tau^{r}P_i^-$ \\
			\hline
			$\tau^{-s}P_j$ & (i) & $(\tau^{s+1}X)e_j=0$ & (ii) \\
			\hline
			$\mathcal{R}$ & ALWAYS & (iii) & $(\tau^{-r}Y)e_i=0$ \\
			\hline
			$\tau^{s}P_j^-$ & ALWAYS & ALWAYS & (iv) \\
			\hline
		\end{tabular}
	\end{center}
	
	\begin{enumerate}[{\rm (i)}]
		\item Either ($r\le s$) or ($r>s$ and $(\tau^{-(r-s-1)}P_i)e_j=0$) holds. 
		\item One of the following equivalent conditions holds.
		\begin{itemize} 
			\item $(\tau^{r+s}I_i)e_j=0$
			\item $(\tau^{-(r+s)}P_j)e_i=0$
		\end{itemize}	
		\item There is $(i,r)\in Q_0\times \Z_{\ge 0}$ such that $\tau^r (X\oplus Y)\in \ind A_i$ and 
		$\tau^r X \underset{A_i}{\prec} \tau^r Y$.
		\item Either ($r\ge s$) or ($r<s$ and $(\tau^{s-r-1}I_i)e_j=0$) holds.				
	\end{enumerate}
	\end{description}
\end{definition}
\begin{lemma}
	\label{mgs:program:relation}
	Let $X,Y \in \ind A\cup\{P_i^-\mid i\in Q_0\}$.
	\begin{enumerate}[{\rm (1)}]
		\item If $Y\in \ind A$ and $(X,Y)\not\in \mathcal{R}\times \mathcal{R}$, then we have
		\[
		Y\prec X \Leftrightarrow 
		\left\{
		\begin{array}{lll}
			\Hom_A(Y,\tau X)=0 & (Y\in \ind A)\\
			Ye_i=0 & (X=P_i^-)\\
		\end{array}	
	\right.
		\]
		In particular, if $A$ is rep-finite, then $X\oplus Y$ is $\tau$-rigid if and only if
		both $X\prec Y$ and $Y\prec X$ hold.
		\item If $X,Y\in \mathcal{R}$, then we have
		\[
		Y\prec X \Rightarrow \Hom_A(Y,\tau X)=0.
		\]
		\item $X\oplus Y$ is $\tau$-rigid if and only if $X\prec Y$ and $Y\prec X$.
		\item Let $T=M\oplus X$ and $T'=M\oplus Y$ be in $\sttilt A$ with $|M|=|A|-1$.
		Then there is an arrow $T\to T'$ in $\Hasse(\sttilt A)$ if and only if
		\[
		Y\prec X. 
		\]
	\end{enumerate}		
\end{lemma}
\begin{proof}
	(1). 
	We first assume $A$ is rep-finite and $Y=\tau^{-s}P_j$ with $0\le s \le k_j$.
	If $X\in \ind A$, then we can write $X=\tau^{-r}P_i$ with $0\le r \le k_i$.
	Then we have
	\[
	\left\{
	\begin{array}{ll}
		\Hom_A(Y,\tau X)=0 \text{ always holds} & (r\le s)\\
		\Hom_A(Y,\tau X)=0\Leftrightarrow \Hom_A(P_j,\tau^{-(r-s-1)} P_i)=0 & (r> s)\\
	\end{array}
	\right.
	\]
	If $X=P_i^-$, then we have
	\[
	Y\prec X \Leftrightarrow Y e_i=0
	\]
	by definition.
	Therefore, the assertion holds for tha case that $A$ is representation-finite.
	
	We assume $A$ is an indecomposable (representation-infinite) tame algebra.
	Assume $Y\in \mathcal{P}$ and put
	\[
	Y=\tau^{-s}P_j. 
	\]
    Then we have
    \[
	\left\{
	\begin{array}{ll}
	\Hom_A(Y,\tau X)=0 \text{ always holds} & (X\simeq \tau^{-r}P_i\text{ with } 0\le r\le s)\\
	\Hom_A(Y,\tau X)=0\Leftrightarrow \Hom_A(P_j,\tau^{s+1} X)=0 & (\text{otherwise})\\
	\end{array}
	\right.
	\]
	Hence, if $Y$ is preprojective, then we have
	\[
	\Hom_A(Y,\tau X)=0\Leftrightarrow Y\prec X.
	\] 
	Similarly, if $X$ is in $\mathcal{I}$, then we can check
	\[
	\Hom_A(Y,\tau X)=0\Leftrightarrow Y\prec X.
	\] 
	If $X=P_i^-$, then we have
	\[
	Y\prec X\Leftrightarrow Ye_i=0
	\]
	by definition.
		
	Hence we may assume $(X,Y)\in \left(\mathcal{P}\times (\mathcal{R}\cup \mathcal{I})\right)\cup 
	\left((\mathcal{P}\cup \mathcal{R})\times \mathcal{I}\right)$.
	In this case, $\Hom_A(Y,\tau X)=0$ always holds. In particular, we have
	\[
	\Hom_A(Y,\tau X)=0\Leftrightarrow Y\prec X.
	\] 
	
	(2). Assume $X,Y\in \mathcal{R}$ and $Y\prec X$. Then $ \tau^r Y\underset{A_i}{\prec} \tau^r X$ holds for some $(i,r)\in Q_0\times \Z$.
	 Since $A_i$ is representation finite hereditary algebra, it follows from (1) that $\Hom_{A_i}(\tau^r Y, \tau_{A_i}(\tau^r X))=0$.
	 By Theorem\;\ref{basicfact}, we have
	 \[
	 \Hom_{A}(\tau^r Y, \tau(\tau^r X))=0.
	 \]
	 This shows 
	 \[
	  \Hom_{A}(Y, \tau X)=0.
	 \]
	 	
	(3). By (1), we may assume either $(X,Y)\in \mathcal{R}\times \mathcal{R}$ or $X=P_i^-$.
	If $X=P_i^-$, then it is easy to check the assertion.
	Therefore, we assume $(X,Y)\in \mathcal{R}\times \mathcal{R}$. Note that $M$ is $\tau$-rigid if and only if
	$\sttilt_{M}A\ne \emptyset$. Then it follows from Proposition\;\ref{noregulartilt} that
	\[
	\begin{array}{lll}
	X\oplus Y \text{ is $\tau$-rigid } &\Leftrightarrow& X\oplus Y\oplus \tau^{-r}P_i^- \text{ is $\tau$-rigid for some }
	(i,r)\in Q_0\times \Z\\
	&\Leftrightarrow & \tau^{r}(X\oplus Y)\oplus P_i^- \text{ is $\tau$-rigid for some }
	(i,r)\in Q_0\times \Z\\
	&\Leftrightarrow & \tau^{r}(X\oplus Y)\text{ is $\tau$-rigid in $\mod A_i$ for some } (i,r)\in Q_0\times \Z\\
	\end{array}	
	\]
	Then, by applying (1) to a rep-finite algebra $A_i$, we have the assertion.
	
	(4). First assume there is an arrow from $T\to T'$. Then we have $X\not\in\{P_i^-\mid i\in Q_0\}$.
	If $Y \in\{P_i^-\mid i\in Q_0\}$, then we have $Y\prec X$. Furthermore, if $(X,Y)\not\in \mathcal{R}\times \mathcal{R}$,
	then $Y\prec X $ follows from (1). Hence we may assume $(X,Y)\in \mathcal{R}\times \mathcal{R}$.
	In this case, it follows from Proposition\;\ref{noregulartilt} that  $\tau^{-r}P_i^-\in \add M$ for some $(i, r)\in Q_0\times \Z$. 
	Then  we have
	\[
	\begin{array}{lll}
	\Hom_A(Y,\tau X)=0 &\Rightarrow& \Hom_A(\tau^r Y,\tau (\tau^r X))=0\\
	&\Rightarrow& \Hom_{A_i}(\tau^r Y,\tau_{A_i} (\tau^r X))=0\ (\text{Theorem\;\ref{basicfact}})\\
	&\Rightarrow& Y\prec X.\\
	\end{array}
	\]  
	
	Next we assume $Y\prec X$. Suppose $X\in \{P_i^-\mid i\in Q_0\}$. Then $X\prec Y$ holds. In particular, $M\oplus X \oplus Y$ is $\tau$-rigid with
	$|M\oplus X \oplus Y|=n+1$. This is a contradiction. If $Y\in \{P_i^-\mid i\in Q_0\}$, then there is an arrow $T\to T'$.
	Therefore, we may assume $X,Y\in \ind A$. Then it follows from (1) and (2) that there is an arrow $T\to T'$ in $\Hasse(\sttilt A)$. 
	\end{proof}

We are going to present a procedure for counting up the elements of $\MGS(A)$ by length with some pseudocodes.
From here on, let $A$ be an indecomposable tame algebra $KQ$ with $Q_0=\{0,\cdots, n\}$,
$C_A:=\begin{pmatrix} \dimvec P_0& \dimvec P_1 &\cdots&\dimvec P_n \end{pmatrix}$  be the Cartan matrix of $A$,   
 and $\Phi_A=-{}^t C_AC_A^{-1}$ be the Coxeter matrix of $A$. Then,  for each $X\in \ind A\setminus \mathcal{P}$, we have the following equation (see \cite[Corollary\;I\hspace{-1pt}V--2.9]{AIR} for example).
 \[
\dimvec (\tau X)=\Phi_A \dimvec(X) 
 \] 
In addition, we use the following notation.
\begin{definition}
	For each vertex $i\in Q_0$, let $\overline{Q}_{i}$ be the quiver made by erasing all arrows directly connected to $i$, and
	we define the Cartan matrix $\overline{C}_{A_i}$ of $A_i$ on $A$ as follows.
	\[
	(\overline{C}_{A_i})_{j,k}:=
	\left\{ \begin{array}{ll}
	1 & \text{if there exists a path from $k$ to $j$ on $\overline{Q}_{i}$ } \\
	0 & \text{(otherwise)}
	\end{array} \right.
	\text{($j, k\in \{0, \cdots, n\}$)}
	\]
	This gives us the Coxeter matrix $\overline{\Phi}_{A_i}:=-{}^t\overline{C}_{A_i}\overline{C}_{A_i}^{-1}$ of $A_i$ on $A$.
\end{definition}

\begin{remark}
	For each $i\in Q_0$ and $i\neq k\in Q_0$, 
	the $k$-th column of $\overline{C}_{A_i}$ is equal to the dimension vector of $P^{A_i}_{k}$ as an $A$-module.
\end{remark}

Given a tame quiver $Q$ as an input, we can easily compute matrices $\overline{C}_{A}$, $\overline{\Phi}_{A}$, $\overline{C}_{A_i}$, $\overline{\Phi}_{A_i}$, and their inverse matrices for each $i\in Q_0$. 
Further, we can also compute dimension vectors $\dimv{P_i}$ and $\dimv{I_i}$ for each $i\in Q_0$. 
Therefore, in programs, we treat these as global variables.

The following pseudocode is a procedure for constructing the set $\mathcal{S}$ which consists of all the dimension vectors of the non-sincere indecomposable $A$-modules.
\begin{algorithm}[H]
	\caption{construction of the set $\mathcal{S}$}
	\label{alg:nonsincere}
	\begin{algorithmic}[1]
		\Function{getNonsincereModules}{}
		\State $\mathcal{S}\gets\text{an empty set (not a list)}$
		\For{$i\in Q_0$} 
		\For{$i\neq j\in Q_0$}
		\State $v\gets$ the $j$-th column of $\overline{C}_{A_i}$ 
		\While{$v$ has no negative component}
		\State $\text{Add the vector $v$ to the set $\mathcal{S}$}$
		\State $v\gets\overline{\Phi}^{-1}_{A_i}v$
		\EndWhile
		\EndFor
		\EndFor
		\State \Return $\mathcal{S}$
		\EndFunction
	\end{algorithmic}
\end{algorithm}
\begin{remark}
	By an easy combinatorial argument, the number of elements of  $\mathcal{S}$ can be found as follows. This helps to quickly check whether an output of the program is correct.
	\begin{table}[htb]
		\begin{tabular}{|c||c|c|c|c|c|} \hline
			$Q$ & $\widetilde{\mathbf{A}}_n$ & $\widetilde{\mathbf{D}}_n$ & $\widetilde{\mathbf{E}}_6$ & $\widetilde{\mathbf{E}}_7$ & $\widetilde{\mathbf{E}}_8$ \\ \hline
			$\#S$ & $n(n+1)$ & $\frac{1}{2}n(3n+1)-3$ & $60$ & $91$ & $135$ \\ \hline
		\end{tabular}
	\end{table}
\end{remark}

\subsubsection{Determine $\Lambda$}

Using the set $S$ obtained by Algorithm \ref{alg:nonsincere}, we construct a finite set $\Lambda$, which contains $\ind\MGS(A)$ as a subset.
First, we are going to explain how we treat support $\tau$-tilting $A$-modules and indecomposable $A$-modules as data.

In programs, an indecomposable $A$-module is denoted by a triple of non-negative integers $(a, b, c)$ as follows.
\begin{itemize}
	\item $(0, b, c)$ denotes the preprojective $A$-module $\tau^{-c}P_b$.
	\item $(1, b, c)$ denotes the regular $A$-module $R$ 
	which satisfies $\tau^{b}(R)\cong R$ and has numbering $c$.
	\item $(2, b, c)$ denotes the preinjective $A$-module $\tau^{c}I_b$.
\end{itemize}
In addition, the object $P^{-}_b$ corresponding to the $\tau$-rigid pair $(0, P_b)$ is denoted by $(2, b, -1)$ and treated as if it were an $A$-module.
Therefore, a support $\tau$-tilting $A$-module can be represented as a sequence of $n+1$ triples, for example, the support $\tau$-tilting $A$-module $A=\bigoplus_{i=0}^{n}P_i$ is represented as the sequence $((0,0,0),(0,1,0),\cdots,(0,n,0))$.
For each support $\tau$-tilting $A$-module, this sequence is uniquely determined by setting a linear order (e.g. a lexicographic order) of the triples in advance.
That is, we can judge whether two support $\tau$-tilting $A$-modules are isomorphic by judging whether they are equal as a sequence.

For each $X\in \koho{P}\cup\koho{R}\cup\koho{I}$, we set $\koho{M}(X):=\dimv{X}$.
The following pseudocode shows a procedure for constructing the set $\koho{R}$ and the restriction $\koho{M}\!\restriction_{\koho{R}}\colon\koho{R}\longrightarrow\mathbb{Z}^{n+1}$ where the two maps $\mathrm{isPrj}\colon\mathbb{Z}^{n+1}\longrightarrow\mathbb{Z}$ and  $\mathrm{isInj}\colon\mathbb{Z}^{n+1}\longrightarrow\mathbb{Z}$ are defined by
\[
\mathrm{isPrj}(u):=
\left\{ \begin{array}{ll}
k & \text{if $u=\dimv{P_k}$} \\
-1 & \text{(otherwise)}
\end{array} \right.
\text{ and }
\ \ \mathrm{isInj}(u):=
\left\{ \begin{array}{ll}
k & \text{if $u=\dimv{I_k}$} \\
-1 & \text{(otherwise)}
\end{array} \right.
\]
respectively.
\begin{algorithm}[H]
	\caption{construction of $\koho{R}$ and $\koho{M}\!\restriction_{\koho{R}}$}
	\label{alg:constR}
	\begin{algorithmic}[1]
		\State $\koho{R}\gets$ an empty set, $\koho{M}\gets$ an empty map 
		\State $p_i\gets 0,\ q_i\gets 0$ for each $i\in\{0,1,\cdots, n\}$
		\State $\mathcal{S}\gets\Call{getNonsincereModules}$
		\State $\mathcal{L}\gets$ an empty set
		\State $c\gets 1$
		\For{$u \in \mathcal{S}$}
		\If{$u\in \mathcal{L}$} \State ${\bf continue}$ \EndIf
		\State $\mathcal{T}\gets$ an empty list
		\State $v\gets u, w\gets u$
		\For{$b = 0,1,2,\cdots $}
		\State $j\gets\mathrm{isPrj}(v)$, $k\gets\mathrm{isInj}(w)$
		\If{$j\geq 0$}
		\State $p_j \gets \max(p_j, b+1)$
		\State ${\bf break}$
		\EndIf
		\If{$k\geq 0$}
		\State $q_k \gets \max(q_k, b+1)$
		\State ${\bf break}$
		\EndIf
		\If{$b\geq1$ and $u=v$}
		\For{$i=0,\cdots,b-1$}
		\State Append the triple $(1, b, c)$ to the list $\koho{R}$.
		\State Define $\koho{M}((1, b, c)):=$ the $i$-th element of list $\mathcal{T}.$
		\State Append the vector $\koho{M}((1, b, c))$ to the set $\mathcal{L}$.
		\State $c\gets c+1$
		\EndFor
		\State ${\bf break}$
		\EndIf
		\State Append the vector $v$ to the end of list $\mathcal{T}$
		\State $v\gets\Phi_{A}v,\ w\gets\Phi_{A}^{-1}w$
		\EndFor
		\EndFor
		\State $m\gets\max(p_0,\cdots,p_n,q_0,\cdots,q_n)$
	\end{algorithmic}
\end{algorithm}
In Algorithm \ref{alg:constR}, the $\{p_i\}_{0\leq i\leq n}$, $\{q_i\}_{0\leq i\leq n}$ and $m$ which appear in Lemma \ref{ppinmgs} and Lemma \ref{piinmgs} are also obtained at the same time as $\koho{R}$ and $\koho{M}\!\restriction_{\koho{R}}$.
Thus, we can further construct $\koho{P, I}$ and $\koho{M}$ by performing the following procedure.
\begin{algorithm}[H]
	\caption{construction of $\koho{P, I}$ and $\koho{M}$}
	\label{alg:constPandI}
	\begin{algorithmic}[1]
		\For{$i=0,\cdots,n$}
		\State Append the triple $(0, i, 0)$ to the list $\koho{P}$.
		\State Define $\koho{M}((0, i, 0)):=\dimv{P_i}$
		\For{$j=1,\cdots,m+p_i-1$}
		\State Append the triple $(0, i, j)$ to the list $\koho{P}$.
		\State Define $\koho{M}((0, i, j)):=\Phi_{A}^{-1}\koho{M}((0, i, j-1))$
		\EndFor
		\State Append the triple $(2, i, 0)$ to the list $\koho{I}$.
		\State Define $\koho{M}((2, i, 0)):=\dimv{I_i}$
		\For{$j=1,\cdots,m+q_i-1$}
		\State Append the triple $(2, i, j)$ to the list $\koho{I}$.
		\State Define $\koho{M}((2, i, j)):=\Phi_{A}\koho{M}((2, i, j-1))$
		\EndFor
		\State Define $\koho{M}((2,i,-1)):=-\dimv{P_i}$.
		\EndFor
		\State $\Lambda\gets\koho{P}\cup\koho{R}\cup\koho{I}\cup\{(2,0,-1),\cdots,(2,n,-1)\}$
	\end{algorithmic}
\end{algorithm}

\subsubsection{Check whether $X\prec Y$}

In this subsection, we are going to present an algorithm to determine whether $X\prec Y$ for all $X=(X_0, X_1, X_2)$ and $Y=(Y_0, Y_1, Y_2) \in \Lambda$. 
We first discuss the case where $X$ and $Y$ are both regular.
In this case, the following statement is useful.
\begin{lemma}[{\cite[Sect. 3.6 (5)]{R}}]
	\label{tublar_family_is_standard}
Let $X,Y\in \mathcal{R}$. If $\num \{\tau^{r}X\mid r\in \Z\}\ne \num \{\tau^{r}Y\mid r\in \Z\}$, then
we have
\[
\Hom_A(X,Y)=0=\Hom_A(Y,X).
\]	
\end{lemma} 

For each $i\in Q_0$, we define a map $\mathrm{isPrj}_i\colon\mathbb{Z}^{n+1}\longrightarrow\mathbb{Z}$ as follows.
\[
\mathrm{isPrj}_i(u):=
\left\{ \begin{array}{ll}
k & \text{if $u=\dimv{\iota_i(P^{A_i}_k)}$} \\
-1 & \text{(otherwise)}
\end{array} \right.
\]
\begin{algorithm}[H]
	\caption{check whether $X\prec Y$}
	\label{alg:regreg}
	\begin{algorithmic}[1]
		\Function{caseOfRegularPair}{$X,Y$}
		\If{$X_1\neq Y_1$} \State \Return True \EndIf 
		\For{$i\in Q_0$}
		\For{$r=0,\dots, X_1-1$} 
		\State $u\gets\Phi_{A}^{r}\koho{M}(X)$, $v\gets\Phi_{A}^{r}\koho{M}(Y)$
		\If{$u_{i}\neq 0$ or $v_{i}\neq 0$} \State ${\bf continue}$ \EndIf
		\State $j\gets \mathrm{isPrj}_i(u)$, $k\gets \mathrm{isPrj}_i(v)$
		\While{$j=-1$ and $k=-1$}
		\State $u\gets\overline{\Phi}_{A_i}u$, $v\gets\overline{\Phi}_{A_i}v$
		\State $j\gets \mathrm{isPrj}_i(u)$, $k\gets \mathrm{isPrj}_i(v)$
		\EndWhile
		\If{$k\geq 0$ or $(\overline{\Phi}_{A_i}v)_{j}=0$}
		\State \Return True
		\EndIf
		\EndFor
		\EndFor
		\State \Return False
		\EndFunction
	\end{algorithmic}
\end{algorithm}

The above has solved the most troublesome case.
The other cases can also be judged as per Definition\;\ref{prec} by the following procedure. For a non-negative integer $a$ and a positive integer $b$, the remainder of $a$ divided by $b$ is denoted by $a\%b$.
\begin{algorithm}[H]
	\caption{check whether $X\prec Y$}
	\label{alg:check}
	\begin{algorithmic}[1]
		\Function{comparison}{$X,Y$}
		\If{$X_0=0$ and $Y_0=0$}
		\If{$X_2\geq Y_2$ or $\koho{M}((0,Y_1,Y_2-X_2-1))_{X_1}=0$}
		\State \Return True
		\Else
		\State \Return False
		\EndIf
		\EndIf
		\If{$X_0=0$ and $Y_0=1$}
		\If{$(\Phi^{(X_2+1)\%Y_1}_A\koho{M}(Y))_{X_1}=0$}
		\State \Return True
		\Else
		\State \Return False
		\EndIf
		\EndIf
		\If{$X_0=0$ and $Y_0=2$}
		\If{$Y_2+X_2+1<q_{Y_1}$ and $\koho{M}((2,Y_1,Y_2+X_2+1))_{X_1}=0$}
		\State \Return True
		\Else
		\State \Return False
		\EndIf
		\EndIf
		\If{$X_0=1$ and $Y_0=1$}
		\State \Return \Call{caseOfRegularPair}{$X,Y$}  
		\EndIf
		\If{$X_0=1$ and $Y_0=2$}
		\If{$(\Phi^{-((Y_2+1)\%X_1)}_A\koho{M}(X))_{Y_1}=0$}
		\State \Return True
		\Else
		\State \Return False
		\EndIf
		\EndIf
		\If{$X_0=2$ and $Y_0=2$}
		\If{$X_2\leq Y_2$ or $\koho{M}((2,X_1,X_2-Y_2-1))_{Y_1}=0$}
		\State \Return True
		\Else
		\State \Return False
		\EndIf
		\EndIf
		\State \Return True
		\EndFunction
	\end{algorithmic}
\end{algorithm}

Precalculating $\Call{comparison}{X, Y}$ for all pairs $(X, Y)\in\Lambda\times\Lambda$ can be done in a short enough time because the number $\#\Lambda$ is only about $750$, even for a case $Q=\widetilde{\mathbf{E}}_8$.
This makes it possible to determine whether $X\prec Y$ with a computational complexity of $O(1)$.

\subsubsection{Construct the Hasse quiver of $\suc(A)\cap \pre(0)$}

In $\Hasse(\sttilt A)$, the two vertices $T$ and $T'$ connected by an arrow are mutations of each other. 
In order to treat the mutations of support $\tau$-tilting modules, it is useful to compute in advance the set
\[
\mathrm{COEXIST}(X):=\{Y\in \Lambda \mid X\neq Y\text{ and }X\oplus Y \colon \tau\text{-rigid} \}
\]
for all $X\in \Lambda$.
With Lemma\;\ref{mgs:program:relation}(2) and the precalculations performed in the previous subsection, this task can be completed in a sufficiently short time.
The following proposition is clear by Theorem \ref{basicprop}(\cite[Theorem\;2.18]{AIR}). 
\begin{proposition}
	Let $T=\bigoplus_{i=0}^{n}T_i$ be a support $\tau$-tilting module satisfying $T_i\in \Lambda$ for all $i\in\{0,\cdots,n\}$.
	Now take a number $j\in\{0,\cdots,n\}$ and construct the subset $\Mu(T,j)$ of $\Lambda$ as follows.
	\[
	\Mu(T,j):=\displaystyle\bigcap_{j\neq i\in\{0,\cdots,n\}}\mathrm{COEXIST}(T_i)
	\]
	Then $\#\Mu(T,j)$ is equal to $1$ or $2$.
\end{proposition}
It is obvious from the way $\Mu(T,j)$ is constructed that $T_j$ always belongs to the set $\Mu(T,j)$. 
In a case $\#\Mu(T,j)=2$, let $T'_j$ be the one of the two elements of $\Mu(T,j)$ that is not $T_j$, and $T'$ be the support $\tau$-tilting module obtained by replacing the direct summand $T_j$ of $T$ with $T'_j$, then 
\[
\text{there exists an arrow from $T$ to $T'$ in $\Hasse(\sttilt A)$ if and only if $T'_j\prec T_j$ holds}
\]
by Lemma\;\ref{mgs:program:relation}(3).

\begin{algorithm}[H]
	\caption{mutation of a support $\tau$-tilting module  $T=\bigoplus_{i=0}^{n}T_i$ on $T_j$}
	\label{alg:mutation}
	\begin{algorithmic}[1]
		\Function{mutation}{$T,j$}
		\State $\Mu\gets\Mu(T, j)$
		\If{$\#\Mu\neq 2$} \State \Return False \EndIf
		\State $T'_j\gets$ the one of the two elements of $\Mu$ that is not $T_j$
		\If{$T'_j\prec T_j$}
		\State $T'\gets \bigoplus_{i=0}^{n}T'_i$ where $T'_i:=T_i$ for all $i\neq j$
		\State \Return $\mathrm{SORT}(T')$
		\Else
		\State \Return False
		\EndIf
		\EndFunction
	\end{algorithmic}
\end{algorithm}

This means that given a vertex $u$, it is possible to enumerate all vertices $v$ that are connected by an arrow from $u$ in $\Hasse(\sttilt A\!\restriction_{\Lambda})$.
Therefore, we can construct $\Hasse(\suc(A)\!\restriction_{\Lambda})$ 
by the following breadth-first search algorithm.
\begin{algorithm}[H]
	\caption{construction of Hasse quiver of $\sttilt A\mid_{\Lambda}$}
	\label{alg:constructHasseQuiver}
	\begin{algorithmic}[1]
		\State $\mathcal{L}\gets$ the list $\{A\ (=\bigoplus_{i=0}^{n}P_i=((0,0,0),\cdots,(0,n,0)))\}$
		\State $\mathcal{V}\gets$ the list $\{0\}$, $\mathcal{E}\gets$ an empty list. 
		\State Define $\mathrm{SEEN}(A):=0$ and  $\mathrm{SEEN}(T):=-1$ for any support $\tau$-tilting module $T\neq A$.
		\For {$j=0 \, \dots \, \#\mathcal{L}-1$}
		\State $T\gets$ the $j$-th element of list $\mathcal{L}$
		\For {$i=0 \, \dots \, n$}
		\State $T'\gets\Call{mutation}{T,\ i}$
		\If{$T'=$ False} \State ${\bf continue}$ \EndIf
		\State $k\gets\mathrm{SEEN}(T')$
		\If{$k<0$}
		\State Append a new vertex $\#\mathcal{L}$ to the end of list $\mathcal{V}$.
		\State Append a new edge $(j,\ \#\mathcal{L})$ to the end of list $\mathcal{E}$.
		\State Redefine $\mathrm{SEEN}(T'):=\#\mathcal{L}$.
		\State Append the new support $\tau$-tilting module $T'$ to the end of list $\mathcal{L}$.
		\Else
		\State Append a new edge $(j,\ k)$ to the end of list $\mathcal{E}$.
		\EndIf
		\EndFor
		\EndFor
	\end{algorithmic}
\end{algorithm}
The map $\mathrm{SEEN}$ in the above algorithm can be interpreted as the following map.
In short, $\mathrm{SEEN}$ is a way to remember which vertex corresponds to each support $\tau$-tilting module once appeared.
\[
\mathrm{SEEN}(T)=
\left\{ \begin{array}{ll}
j & \text{if $T$ is isomorphic to the $j$-th element of the list $\mathcal{L}$ } \\
-1 & \text{(otherwise)}
\end{array} \right.
\]

Then the quiver $H:=(\mathcal{V}, \mathcal{E})$ obtained by Algorithm \ref{alg:constructHasseQuiver} is none other than $\Hasse(\suc(A)\!\restriction_{\Lambda})$. 
However, $H$ has also some vertices which have no path of finite length to the vertex corresponding to $0=\bigoplus_{i=0}^{n}P^{-}_i$.
These "useless" vertices can be removed by a simple task with a computational complexity of $O(\#\mathcal{E})$.
By putting this finishing touch to $H$, we obtain 
\[
\Hasse_{\mathrm{fin}}(\sttilt A).
\]

\subsubsection{Counting maximal green sequences.}

Finally, we explain how to count all elements of $\MGS(A)$ by length.
Hasse quivers are Directed Acyclic Graph (DAG).
In general, if a quiver $H:=(\mathcal{V}, \mathcal{E})$ is a DAG then $H$ has a topological ordering, that is, an ordering of the vertices such that for every arrow $(u, v)\in \mathcal{E}$, the vertex $u$ comes before $v$ in the ordering.
It is known that this sorting can be completed by Kahn's algorithm\cite{Kahn} with a computational complexity of $O(\#\mathcal{V}+\#\mathcal{E})$.

The following proposition is obvious.
\begin{proposition}
	Let $H=(\mathcal{V}, \mathcal{E})$ be a Hasse quiver and $s$ be a vertex of $H$. For each vertex $v\in \mathcal{V}$ and each number $\ell\in\{0,\cdots,\#\mathcal{V}-1\}$, we define
	\[
	\mathrm{PATH}_s(v,\ell):= \text{the number of paths of length $\ell$ from $s$ to $v$ in $H$}.
	\]
	Then, the following equation holds.
	\[
	\mathrm{PATH}_s(v,\ell+1)=\displaystyle\sum_{(u,v)\in\mathcal{E}}\mathrm{PATH}_s(u,\ell)
	\]
\end{proposition}
Therefore, our goal in this section is accomplished by the following procedure.
\begin{algorithm}[H]
	\caption{Counting maximal green sequences}
	\label{alg:countMGS}
	\begin{algorithmic}[1]
		\State $(\mathcal{V}, \mathcal{E})\gets\Hasse_{\mathrm{fin}}(\sttilt A)$
		\State $L[u] \gets$ the list $\{0\}$ for each $u\in\mathcal{V}$
		\State $\mathrm{NEXT}(u) \gets$ an empty list for each $u\in\mathcal{V}$
		\For{$(u_1, u_2)\in\mathcal{E}$}
		\State $\text{Append the vertex $u_2$ to the end of list $\mathrm{NEXT}(u_1)$.}$
		\EndFor  
		\State $\text{Sort the elements of $\mathcal{V}$ in topological order.}$
		\State $\text{$s \gets$ the first vertex of $\mathcal{V}$, $t \gets$ the last vertex of $\mathcal{V}$}$
		\State $L[s] \gets$ the list $\{1\}$
		\For{$i=0 \, \dots \, \#\mathcal{V}-1$}
		\State $\text{$v \gets$ the $i$-th vertex of $\mathcal{V}$}$
		\For{$w \in \mathrm{NEXT}(v)$}
		\While{$\#L[w] \leq \#L[v]$}
		\State $\text{Append the number $0$ to the end of list $L[w]$.}$
		\EndWhile
		\For{$j=0 \, \dots \, \#L[v]-1$}
		\State $\text{Add the $j$-th number of list $L[v]$ to the $(j+1)$-th number of list $L[w]$.}$
		\EndFor
		\EndFor
		\EndFor
	\end{algorithmic}
\end{algorithm}
In line 7, the vertices are topologically sorted, so in line 8, $s$ is the vertex corresponding to $A=\bigoplus_{i=0}^{n}P_i$ and $t$ is the vertex corresponding to $0=\bigoplus_{i=0}^{n}P^{-}_i$.
That is, the goal of Algorithm \ref{alg:countMGS} is to calculate $\mathrm{PATH}_s(t,\ell)$ for each $\ell$. Eventually, $\mathrm{PATH}_s(t,\ell)$ can be found as the $\ell$-th number of list $L[t]$.

\subsection{Demonstration and examples.}
The following is a demonstration of the methods described so far in this section and runs in a web browser.
\[ \text{https://hfipy3.github.io/MGS/en.html} \] 
Here are some examples of the results by this tool.
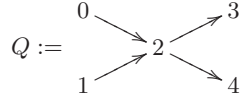
\begin{example}
	\[
	Q:=
	\begin{xy}
	(-5,5)*[o]+{0}="0",(-5,-5)*[o]+{1}="1",(5,0)*[o]+{2}="2", (15,5)*[o]+{3}="3",(15,-5)*[o]+{4}="4",
	\ar "0";"2" \ar "1";"2" \ar "2";"3" \ar "2";"4"
	\end{xy}
	\]
	In this case, the quiver $\Hasse_{\mathrm{fin}}(\sttilt A)$ consists of $314$ vertices and $743$ arrows. The total numbers of MGS by length are below.
	\begin{table}[htb]
		\begin{tabular}{|c||c|c|c|c|c|c|c|c|c|c|} \hline
			$\ell$ & 5 & 6 & 7 & 8 & 9 & 10 & 11 & 12 & 13 & \\ \hline
			$\mathrm{PATH}_s(t,\ell)$ & 4 & 24 & 40 & 168 & 144 & 272 & 400 & 1144 & 1720 & \\ \hline \hline
			& 14 & 15 & 16 & 17 & 18 & 19 & 20 & 21 & 22 & total \\ \hline
			& 1792 & 2912 & 4928 & 8192 & 9984 & 12672 & 31104 & 72576 & 62208 & 210284 \\ \hline
		\end{tabular}
	\end{table}
\end{example}

The results in the above table are consistent with those presented in the paper \cite[Page 36 of Version 1]{BDP} by Bristol et al. Please try the other examples presented in their paper.

\begin{example}
	As a larger case, we present an example of type $\widetilde{\mathbf{E}}_8$.
	\[
	Q:=
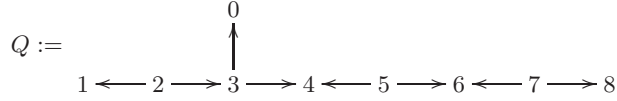
	\begin{xy}
	(20,5)*[o]+{0}="0",
	(0,-5)*[o]+{1}="1",(10,-5)*[o]+{2}="2",(20,-5)*[o]+{3}="3",(30,-5)*[o]+{4}="4",
	(40,-5)*[o]+{5}="5",(50,-5)*[o]+{6}="6",(60,-5)*[o]+{7}="7",(70,-5)*[o]+{8}="8",
	\ar "3";"0" \ar "2";"1" \ar "2";"3" \ar "3";"4" \ar "5";"4" \ar "5";"6" \ar "7";"6" \ar "7";"8"
	\end{xy}
	\]
	In this case, the quiver $\Hasse_{\mathrm{fin}}(\sttilt A)$ consists of $528510$ vertices and $2353207$ arrows. Since the number $\#\MGS(A)$ is very huge, we present some excerpts.
	\begin{table}[htb]
		\begin{tabular}{|c||c|c|c|c|c|c|c|c|} \hline
			$\ell$ & 9 & 10 & 11 & $\cdots$ & 389 & 390 & total \\ \hline
			$\mathrm{PATH}_s(t,\ell)$ & 4224 & 36884 & 191819 &  & $3.543\cdots\times10^{188}$ & $3.758\cdots\times10^{187}$ & $2.546\cdots\times10^{192}$ \\ \hline
		\end{tabular}
	\end{table}
\end{example}
The calculation of this example takes less than 4 minutes using a standard home computer.
It is possible in less than a day that we check whether the No Gap Conjecture is true for all cases of type $\widetilde{\mathbf{E}}$.
Hence, this demonstration give us an "elephant proof" of the No Gap Conjecture for type $\widetilde{\mathbf{E}}$.

\section{A proof of Maintheorem (1): the case $\widetilde{\mathbf{D}}$}
The purpose of this section is to show the following statement.
\begin{theorem}
	\label{D:main} Let $Q$ be a quiver of type $\widetilde{\mathbf{D}}_n$. Then $\ell(Q)$ does not depend on the choice of the orientation of $Q$.
\end{theorem}
We have already known that the statement in Theorem\;\ref{D:main} holds for the case $n=4$\cite{BDP}. 
Hence, we assume $n\ge 5$ in the rest of this section.
\subsection{Settings}
In this section, for $k\in\{2,3,\dots, n-2\}$ and $\underline{\epsilon}=(\epsilon_0,\epsilon_1,\epsilon_2,\epsilon_3 )
\in \{\pm\}\times\{\pm\}\times \{\pm \}\times \{\pm \}$, $\mathcal{Q}(k,\underline{\epsilon})$ denotes the set of quivers having the following form: 
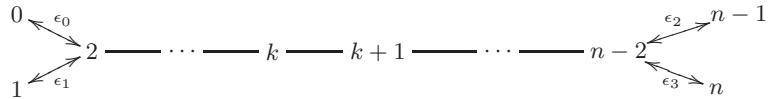
\[\begin{xy}
(0,5)*[o]+{0}="0",(0,-5)*[o]+{1}="1",(10,0)*[o]+{2}="2",(22,0)*[o]+{\cdots}="dots",(34,0)*[o]+{k}="k"
,(48,0)*[o]+{k+1}="k+1",(64,0)*[o]+{\cdots}="dots2",(80,0)*[o]+{n-2}="n-2", (96,5)*[o]+{n-1}="n-1",(93,-5)*[o]+{n}="n",
\ar @{<->}"0";"2"^{\epsilon_0}
\ar@{<->} "2";"1"^{\epsilon_1}
\ar @{-}"dots";"2"
\ar @{-}"k";"dots"
\ar @{-}"k";"k+1"
\ar @{-}"k+1";"dots2"
\ar @{-}"dots2";"n-2"
\ar @{<->} "n-2";"n-1"^{\epsilon_2}
\ar @{<->}"n-2";"n"_{\ \ \ \ \epsilon_3}
\end{xy}\] 
with $\num\{t\in\{2,3,\dots, n-3\}\mid t\leftarrow t+1 \}=k-2$.
 We define a quiver
$Q(k,\underline{\epsilon})\in \mathcal{Q}(k,\underline{\epsilon})$ as follows:
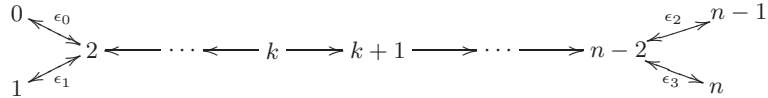
\[
\begin{xy}
		 (0,5)*[o]+{0}="0",(0,-5)*[o]+{1}="1",(10,0)*[o]+{2}="2",(22,0)*[o]+{\cdots}="dots",(34,0)*[o]+{k}="k"
		 ,(48,0)*[o]+{k+1}="k+1",(64,0)*[o]+{\cdots}="dots2",(80,0)*[o]+{n-2}="n-2", (96,5)*[o]+{n-1}="n-1",(93,-5)*[o]+{n}="n",
		\ar @{<->}"0";"2"^{\epsilon_0}
		\ar@{<->} "2";"1"^{\epsilon_1}
		\ar "dots";"2"
		\ar "k";"dots"
		\ar "k";"k+1"
		\ar "k+1";"dots2"
		\ar "dots2";"n-2"
		\ar @{<->} "n-2";"n-1"^{\epsilon_2}
		\ar @{<->}"n-2";"n"_{\ \ \ \ \epsilon_3}
			\end{xy}
\]
where, for $p<q$, 
\[p\stackrel{\epsilon}{\longleftrightarrow}q=\left\{\begin{array}{ll}
p\rightarrow q&\epsilon=+\\
p\leftarrow q&\epsilon=-\\
\end{array}
\right. \] 

We say that $Q$ is of type $(k,\underline{\epsilon})$ if $Q$ is isomorphic to the quiver in $\mathcal{Q}(k,\underline{\epsilon})$.

\begin{lemma}
\label{D:representative_mutation}	
Let $Q\in \mathcal{Q}(k,\underline{\epsilon})$. Then there exists a sink mutation sequence
\[Q\to \mu_{k_1}Q\to\cdots\to \mu_{k_t}\cdots\mu_{k_1}Q=Q(k, \underline{\epsilon})\]
with $k_1,\dots,k_t\in \{3,4,\dots,n-3\}$.
\end{lemma}

\begin{proof}
	If $\{k\in\{3,4,\dots,n-3\}\mid k-1\rightarrow k\leftarrow k+1 \}=\emptyset$, then $Q$ is $Q(k, \underline{\epsilon})$
and there is nothing to show.
	Assume that $\{k\in\{3,4,\dots,n-3\}\mid k-1\rightarrow k\leftarrow k+1 \}\neq \emptyset$. 
	We take $k_1=\min\{ k\in\{3,4,\dots,n-3\}\mid k-1\rightarrow k\leftarrow k+1\}$ and consider the new quiver $\mu_{k_1} Q$.

	By repeating same arguments, we obtain a desired sink mutation sequence
	\[Q\to \mu_{k_1}Q\to\cdots\to \mu_{k_t}\cdots\mu_{k_1}Q=Q(k, \underline{\epsilon}).\]
	This finishes the proof.
\end{proof}
\begin{lemma}
	\label{D:	 foundation}
 Assume that $Q\in \mathcal{Q}(k,\underline{\epsilon})$, $i$ a source vertex of $Q$ and $i'$ a sink vertex of $Q$.
		\begin{enumerate}[{\rm (1)}]
\item If the degree of $i$ is equal to one, then 
\[\ell(Q)= \ell (Q(k,\underline{\epsilon})),\ \ell(\mu_i Q)= \ell(\mu_i Q(k,\underline{\epsilon})). \]
\item If the degree of $i'$ is equal to one, then 
\[\ell(Q)= \ell(Q(k,\underline{\epsilon})),\ \ell(\mu_{i'} Q)=\ell(\mu_{i'} Q(k,\underline{\epsilon})). \]
 \end{enumerate}
		\end{lemma}
\begin{proof}
The assertion follows from Proposition\;\ref{mgs:keyprop} and Lemma\;\ref{D:representative_mutation}.
\end{proof}

\subsection{Case: $Q=Q(k,+,-,-,+)$ }
\label{subsect:D+--+}
In this subsection, we treat the case that $Q=Q(k,+,-,-,+)$. Let $L$ be an indecomposable module given by
\[
\begin{xy}
(-13,5)*[o]+{0}="0",(-13,-5)*[o]+{0}="1",(-3,0)*[o]+{K}="2",(10,0)*[o]+{K}="3",(22,0)*[o]+{\cdots}="dots",(34,0)*[o]+{K}="k"
,(48,0)*[o]+{K}="k+1",(64,0)*[o]+{\cdots}="dots2",(80,0)*[o]+{K}="n-2", (93,5)*[o]+{0}="n-1",(93,-5)*[o]+{0}="n",
\ar "0";"2"
\ar "2";"1"
\ar "3";"2"
\ar "dots";"3"
\ar "k";"dots"
\ar "k";"k+1"
\ar "k+1";"dots2"
\ar "dots2";"n-2"
\ar"n-1";"n-2"
\ar"n-2";"n"
\end{xy}
\] 
We will show the following proposition.
\begin{proposition}
	\label{prop:D+--+}
	Let $Q$ and $L$ be as above.
	\begin{enumerate}[\rm (1)]
		\item $(0,n,L)$ and $(0,1,L)$ satisfy Assumption\;\ref{assumption}.
		\item $\ell(Q)=\ell( \mu_0 Q)=\ell( \mu_1 Q)=\ell(\mu_{n-1}Q)=\ell(\mu_n Q)$.
	\end{enumerate}
\end{proposition} 
\subsubsection{$\mathcal{X}_0$ and $\mathcal{X}'_n$}
Here, we determine $\mathcal{X}_0$ and $\mathcal{X}_n'$.

Let $X$ be an indecomposable non-projevtive module with $Xe_0=0$. 
We set
\[x_t:=\dim_K\Hom_A(P_t, X),\ x'_t:=\dim_K\Hom_A(\tau^{-1}P_t, X),\]
for any $t\in \{0,1,\dots, n\}$.

If $k=2$, then we have
 \[\begin{array}{lll}
 x'_0&=&x'_2\\
 &=& x'_1+x'_3-x_2 \\
 &=& (x_2-x_1)+(x'_4-x_3) \\
 &\vdots& \\
 &=& (x_2-x_1)+(x'_{n-2}-x_{n-3})\\
 &=& (x_2-x_1)+(x_{n-1}+x'_{n}-x_{n-2})\\
 &=& (x_2-x_1)+(x_{n-1}-x_n)\\
 \end{array}\]
 If $k>2$, then we have
\[
\begin{array}{lll}
x'_0&=&x'_2\\
&=& x'_1+x_3-x_2 \\
&=& x_3-x_1 \\
\end{array}
\]
Therefore, by the classification of indecomposable modules of $A_0$, 
$\mathcal{X}_0$ has a unique module $X_0$ (up to isomorphism) and it is given by the following quiver representation.
\[\left\{\begin{array}{cl}
\begin{xy}
(-20,5)*[o]+{0}="0",(-20,-5)*[o]+{0}="1",(-10,0)*[o]+{K}="2",(3,0)*[o]+{K}="3",(17,0)*[o]+{\cdots}="dots",(31,0)*[o]+{K}="n-2"
, (43,5)*[o]+{K}="n-1",(43,-5)*[o]+{0}="n",
\ar "0";"2"
\ar "2";"1"
\ar "2";"3"^{\id}
\ar "3";"dots"^{\id}
\ar "dots";"n-2"^{\id}
\ar "n-1";"n-2"_{\id}
\ar "n-2";"n"
\end{xy} & (k=2)\\\\
\begin{xy}
(-31,5)*[o]+{0}="0",(-31,-5)*[o]+{0}="1",(-21,0)*[o]+{K}="2",(-8,0)*[o]+{K^2}="3",(6,0)*[o]+{\cdots}="dots",(20,0)*[o]+{K^2}="k-1",(34,0)*[o]+{K^2}="k"
,(48,0)*[o]+{K^2}="k+1",(64,0)*[o]+{\cdots}="dots2",(80,0)*[o]+{K^2}="n-3",(93,0)*[o]+{K^2}="n-2", (106,5)*[o]+{K}="n-1",(106,-5)*[o]+{K}="n",
\ar "0";"2"
\ar "2";"1"
\ar "3";"2"_{\left(\begin{smallmatrix}1\\ 0\end{smallmatrix}\right)}
\ar "dots";"3"_{\id}
\ar "k-1";"dots"_{\id}
\ar "k";"k-1"_{\id}
\ar "k";"k+1"^{\id}
\ar "k+1";"dots2"^{\id}
\ar "dots2";"n-3"^{\id}
\ar "n-3";"n-2"^{\id}
\ar "n-1";"n-2"_{(\begin{smallmatrix}1& 1\end{smallmatrix})}
\ar "n-2";"n"_{\left(\begin{smallmatrix}0\\ 1\end{smallmatrix}\right)}
\end{xy}& (k>2)
\end{array}\right.
\]

Similarly, let $X'$ be an indecomposable non-projective module with $(\Tr X')e_n=0$. 
We set $y_t:=\dim_K\Hom_{A'}(P^{A'}_t, \Tr X')$ 
and $y'_t:=\dim_K\Hom_{A'}(\tau_{A'}^{-1}P^{A'}_t, \Tr X')$.

If $k=2$, then we have
\[
\begin{array}{lcl}
y'_n &=&y'_{n-2}-y_n\\ 
&=&y'_{n-1}+y'_{n-3}-y_{n-2}\\ 
&=&(y_{n-2}-y_{n-1})+(y'_{n-3}-y_{n-2})\\
&=&(y_{n-2}-y_{n-1})+(y'_{n-4}-y_{n-3})\\
&\vdots&\\
&=&(y_{n-2}-y_{n-1})+(y'_{2}-y_{3})\\
&=&(y_{n-2}-y_{n-1})+(y'_0+y_1-y_2).\\
&=&(y_{n-2}-y_{n-1})+(y_2-y_0+y_1-y_2).\\
&=&(y_{n-2}-y_{n-1})+(y_1-y_0).\\

\end{array}
\]

If $2<k<n-2$, then we have
	\[
	\begin{array}{lcl}
	 y'_n &=&y'_{n-2}-y_n\\ 
	&=&y'_{n-1}+y'_{n-3}-y_{n-2}\\ 
	&=&(y_{n-2}-y_{n-1})+(y'_{n-3}-y_{n-2})\\
	&=&(y_{n-2}-y_{n-1})+(y'_{n-4}-y_{n-3})\\
	&\vdots&\\
	&=&(y_{n-2}-y_{n-1})+(y'_{k}-y_{k+1})\\
	&=&(y_{n-2}-y_{n-1})+(y_{k-1}-y_{k}).\\
	\end{array}
	\]
	
If $k=n-2$, then we have
\[
\begin{array}{lcl}
y'_n &=&y'_{n-2}-y_n\\ 
&=&y'_{n-1}+y_{n-3}-y_{n-2}\\ 
&=&(y_{n-2}-y_{n-1})+y_{n-3}-y_{n-2}\\
&=&y_{n-3}-y_{n-1}\\
\end{array}
\]
 
Hence, we obtain that $\mathcal{X}'_n=\Tr(\mathcal{X}_n(A'))$ has a unique module
 $X'_n$ (up to isomorphisms) and $Y_n:=\tau X'_n=D(\Tr X'_n)$ is given by the following quiver representation. 
		\[\left\{
\begin{array}{cl}
\begin{xy}
(-20,5)*[o]+{0}="0",(-20,-5)*[o]+{K}="1",(-10,0)*[o]+{K}="2",(3,0)*[o]+{K}="3",(17,0)*[o]+{\cdots}="dots",(31,0)*[o]+{K}="n-2"
, (43,5)*[o]+{0}="n-1",(43,-5)*[o]+{0}="n",
\ar "0";"2"
\ar "2";"1"_{\id}
\ar "2";"3"^{\id}
\ar "3";"dots"^{\id}
\ar "dots";"n-2"^{\id}
\ar "n-1";"n-2"
\ar "n-2";"n"
\end{xy} & (k=2)\\\\		
\begin{xy}
(-31,5)*[o]+{K}="0",(-31,-5)*[o]+{K}="1",(-21,0)*[o]+{K^2}="2",(-8,0)*[o]+{K^2}="3",(6,0)*[o]+{\cdots}="dots",(20,0)*[o]+{K^2}="k-1",(34,0)*[o]+{K}="k"
,(48,0)*[o]+{K}="k+1",(64,0)*[o]+{\cdots}="dots2",(80,0)*[o]+{K}="n-3",(93,0)*[o]+{K}="n-2", (106,5)*[o]+{0}="n-1",(106,-5)*[o]+{0}="n",
\ar "0";"2"^{{(\begin{smallmatrix}0& 1\end{smallmatrix})}}
\ar "2";"1"^{(\begin{smallmatrix}1\\ 1\end{smallmatrix})}
\ar "3";"2"_{\id}
\ar "dots";"3"_{\id}
\ar "k-1";"dots"_{\id}
\ar "k";"k-1"_{(\begin{smallmatrix}1& 0\end{smallmatrix})}
\ar "k";"k+1"^{\id}
\ar "k+1";"dots2"^{\id}
\ar "dots2";"n-3"^{\id}
\ar "n-3";"n-2"^{\id}
\ar "n-1";"n-2"
\ar "n-2";"n"
\end{xy}& (k>2)
\end{array}
\right.
\]


\begin{lemma}
	\label{MGS:D:+--+}
	Let $X_0$ and $Y_n$ be as above. Then we have the following statements.
	\begin{enumerate}[{\rm (1)}]
				\item $\tau_{A_0}^{-(k-1)}X_0= P_{n-2}^{-}$.
		\item 
		$\tau_{A_n}^{k-1}Y_n=P_2^-$.
		\item $\tau_{A_0}^{k-1}P_{n-1}^-$ and 
		$\tau_{A_0}^{k-1}P_n$ are not in $\mathcal{R}$.
		\item $\tau_{A_n}^{-(k-1)}I_0$ and
		 $\tau_{A_n}^{-(k-1)}P_1^-$ are not in $\mathcal{R}$. 
		\item If $M\oplus X_0\oplus P_0^-\in \sttilt A$ with $M\in \add\mathcal{R}$ $($resp. $M'\oplus X'_n\oplus P_n\in \sttilt A$ with $M'\in \add\mathcal{R})$, then
\[L\in \add M\ (\text{resp. }L\in \add \tau M').\]
\item Both $(0,n,L)$ and $(0,1,L)$ satisfy $\mathsf{A1}$, $\mathsf{A2}$, and $\mathsf{A4}$ in Assumption\;\ref{assumption}.
\end{enumerate}
\end{lemma}
\begin{proof}
	(1) and (2). If $k=2$, then the assertion is obvious. Therefore we assume $k>2$ and 	
	for $t\in \{2,3,\dots,k-1 \}$, we set $X(t)\in \ind A_0\subset \ind A$ as follows.
\[\begin{xy}
(-41,0)*[o]+{X(t):}="x",(-31,5)*[o]+{0}="0",(-31,-5)*[o]+{0}="1",(-21,0)*[o]+{0}="2",(-8,0)*[o]+{\cdots}="3",
(6,0)*[o]+{0}="t-1",(6,-5)*[o]+{t-1},(20,0)*[o]+{K}="t",(20,-5)*[o]+{t},(34,0)*[o]+{K^2}="t+1",(34,-5)*[o]+{t+1}
,(48,0)*[o]+{\cdots}="dots",(62,0)*[o]+{K^2}="k",(77,0)*[o]+{\cdots}="dots2",(91,0)*[o]+{K^2}="n-2", (103,5)*[o]+{K}="n-1",(103,-5)*[o]+{K}="n",
\ar "0";"2"
\ar "2";"1"
\ar "3";"2"
\ar "t-1";"3"
\ar "t";"t-1"
\ar "t+1";"t"_{(\begin{smallmatrix}1\\ 0\end{smallmatrix})}
\ar "dots";"t+1"_{\id}
\ar "k";"dots"_{\id}
\ar "k";"dots2"^{\id}
\ar "dots2";"n-2"^{\id}
\ar "n-2";"n"_{(\begin{smallmatrix}0& 1\end{smallmatrix})}
\ar "n-1";"n-2"_{(\begin{smallmatrix}1& 1\end{smallmatrix})}
\end{xy}
\]	
We also define $Y(t)\in \ind A_n\subset \ind A$ as follows.
\[\left\{
\begin{array}{lc}
Y_n & (t=k-1)\\\\
\begin{xy}
(-31,5)*[o]+{K}="0",(-31,-5)*[o]+{K}="1",(-21,0)*[o]+{K^2}="2",(-8,0)*[o]+{\cdots}="3",
(6,0)*[o]+{K^2}="t-1",(6,-5)*[o]+{t},(20,0)*[o]+{K}="t",(20,-5)*[o]+{t+1},(34,0)*[o]+{0}="t+1",(34,-5)*[o]+{t+2}
,(48,0)*[o]+{\cdots}="dots",(62,0)*[o]+{0}="k",(77,0)*[o]+{\cdots}="dots2",(91,0)*[o]+{0}="n-2", (103,5)*[o]+{0}="n-1",(103,-5)*[o]+{0}="n",
\ar "0";"2"^{(\begin{smallmatrix}1& 0\end{smallmatrix})}
\ar "2";"1"^{(\begin{smallmatrix}1\\ 1\end{smallmatrix})}
\ar "3";"2"_{\id}
\ar "t-1";"3"_{\id}
\ar "t";"t-1"_{(\begin{smallmatrix}0 & 1\end{smallmatrix})}
\ar "t+1";"t"
\ar "dots";"t+1"
\ar "k";"dots"
\ar "k";"dots2"
\ar "dots2";"n-2"
\ar "n-2";"n"
\ar "n-1";"n-2"
\end{xy}& (t<k-1)\\
\end{array}
\right.
\]

If $k<n-2$, then we have an exact sequence 
	\[
	\left\{
	\begin{array}{lll}
 0\to X(t)\to  I_t^{A_0}\oplus I_{t+1}^{A_0}\oplus I_{n-2}^{A_0}\oplus I_n^{A_0}\to I_k^{A_0}\oplus I_k^{A_0}\oplus I_{n-1}^{A_0}
 \to 0 & (t<k-1<n-3)\\\\
 0\to X(t)\to  I_t^{A_0}\oplus I_{t+1}^{A_0}\oplus I_n^{A_0}\to I_{n-2}^{A_0}\oplus I_{n-1}^{A_0}
 \to 0 & (t<k-1=n-3)\\\\
0\to X(k-1)\to  I_{k-1}^{A_0}\oplus I_{n-2}^{A_0}\oplus I_n^{A_0}\to I_k^{A_0}\oplus I_{n-1}^{A_0}\to 0	 & (t=k-1<n-3)\\\\
0\to X(k-1)\to  I_{n-3}^{A_0}\oplus I_{n}^{A_0}\to I_{n-1}^{A_0}\to 0	 & (t=k-1=n-3)\\
	\end{array}
	\right.
	\]	
	
Note that the above exact sequence gives a minimal injective copresentation of $X(t)$ in $\mod A_0$.	
Since $A_0$ is hereditary, we have an exact sequence
\[\left\{
\begin{array}{lll}
0\to P_t^{A_0}\oplus P_{t+1}^{A_0}\oplus P_{n-2}^{A_0}\oplus P_n^{A_0}\to P_k^{A_0}\oplus P_k^{A_0}\oplus P_{n-1}^{A_0}\to \tau^{-1}_{A_0} X(t)\to 0 & (t<k-1<n-3)\\\\
0\to P_t^{A_0}\oplus P_{t+1}^{A_0}\oplus P_n^{A_0}\to P_{n-2}^{A_0}\oplus P_{n-1}^{A_0}\to \tau^{-1}_{A_0} X(t)\to 0 & (t<k-1=n-3)\\\\
0\to P_{k-1}^{A_0}\oplus P_{n-2}^{A_0}\oplus P_n^{A_0}\to   
P^{A_0}_k\oplus P_{n-1}^{A_0}\to \tau^{-1}_{A_0} X(k-1)\to 0 & (t=k-1<n-3)\\\\
0\to P_{n-3}^{A_0}\oplus P_{n}^{A_0}\to   
P_{n-1}^{A_0}\to \tau^{-1}_{A_0} X(k-1)\to 0 & (t=k-1=n-3)\\
\end{array}
\right.\]	

Then, by comparing dimension vectors, we obtain
\[\tau^{-1}_{A_0}X(t)=\left\{
\begin{array}{lll}
X(t+1)& (t<k-1)\\
 I_{n-2}^{A_0}& (t=k-1)\\
\end{array}\right.
\]	
Similarly, we obtain 
\[\tau_{A_n}Y(t)=\left\{
\begin{array}{lll}
Y(t-1)& (2<t)\\
P_{2}^{A_n}& (t=2)\\
\end{array}\right.
\]

(3).
We can easily check that 
$\tau_{A_0}^t P_{n-1}^- =\tau_{A_0}^{t-1}I^{A_0}_{n-1}$ is equal to $\tau^{t-1} I_{n-1}$ 
for any $1\le t \le k-1$  and given by
\[\left\{
\begin{array}{ll}
I^{A_0}_{n-1}=I_{n-1}=S_{n-1} & (t=1)\\\\
\begin{xy}
(-31,5)*[o]+{0}="0",(-31,-5)*[o]+{0}="1",(-21,0)*[o]+{0}="2",(-3,0)*[o]+{\cdots}="3",
(16,0)*[o]+{0}="t-1",(16,-5)*[o]+{k-t+1},(34,0)*[o]+{K}="t",(34,-5)*[o]+{k-t+2},
,(48,0)*[o]+{\cdots}="dots",(62,0)*[o]+{K}="k",(62,-5)*[o]+{k},(77,0)*[o]+{\cdots}="dots2",(91,0)*[o]+{K}="n-2", (103,5)*[o]+{V_t}="n-1",(103,-5)*[o]+{V_t}="n",
\ar "0";"2"
\ar "2";"1"
\ar "3";"2"
\ar "t-1";"3"
\ar "t";"t-1"
\ar "dots";"t"
\ar "k";"dots"
\ar "k";"dots2"
\ar "dots2";"n-2"
\ar "n-2";"n"
\ar "n-1";"n-2"
\end{xy}
 & (t\ge 2)\\
\end{array}
\right.\]
 where $V_t=\left\{
\begin{array}{cl}
K& \text{if $t$ is odd}\\ 0& \text{if $t$ is even}.\\\end{array}\right.$

We can also check that 
$\tau_{A_0}^t P_n=\tau_{A_0}^{t-1}P^-_n$ is equal to $\tau^{t-1} P^-_n$ for any $1\le t\le k-1$ and given by
\[\left\{
\begin{array}{ll}
P_n^- & (t=1)\\\\
\begin{xy}
(-31,5)*[o]+{0}="0",(-31,-5)*[o]+{0}="1",(-21,0)*[o]+{0}="2",(-3,0)*[o]+{\cdots}="3",
(16,0)*[o]+{0}="t-1",(16,-5)*[o]+{k-t+1},(34,0)*[o]+{K}="t",(34,-5)*[o]+{k-t+2},
,(48,0)*[o]+{\cdots}="dots",(62,0)*[o]+{K}="k",(62,-5)*[o]+{k},(77,0)*[o]+{\cdots}="dots2",(91,0)*[o]+{K}="n-2", (103,5)*[o]+{V_{t-1}}="n-1",(103,-5)*[o]+{V_{t-1}}="n",
\ar "0";"2"
\ar "2";"1"
\ar "3";"2"
\ar "t-1";"3"
\ar "t";"t-1"
\ar "dots";"t"
\ar "k";"dots"
\ar "k";"dots2"
\ar "dots2";"n-2"
\ar "n-2";"n"
\ar "n-1";"n-2"
\end{xy}
& (t\ge 2)\\
\end{array}
\right.\]
where $V_t=\left\{
\begin{array}{cl}
K& \text{if $t$ is odd}\\ 0& \text{if $t$ is even}.\\\end{array}\right.$

(4). 
We can easily check that 
$\tau_{A_n}^{-t} P_1^- =\tau_{A_n}^{-(t-1)}P^{A_n}_{1}$ is equal to $\tau^{-(t-1)} P_{1}$ 
for any $1\le t \le k-1$  and given by
\[\left\{
\begin{array}{ll}
P^{A_n}_{1}=P_{1}=S_{1} & (t=1)\\\\
\begin{xy}
(-31,5)*[o]+{V_{t}}="0",(-31,-5)*[o]+{V_{t}}="1",(-21,0)*[o]+{K}="2",(-3,0)*[o]+{\cdots}="3",
(16,0)*[o]+{K}="t-1",(16,-5)*[o]+{t},(34,0)*[o]+{0}="t",(34,-5)*[o]+{t+1},
,(48,0)*[o]+{\cdots}="dots",(62,0)*[o]+{0}="k",(62,-5)*[o]+{k},(77,0)*[o]+{\cdots}="dots2",(91,0)*[o]+{0}="n-2", (103,5)*[o]+{0}="n-1",(103,-5)*[o]+{0}="n",
\ar "0";"2"
\ar "2";"1"
\ar "3";"2"
\ar "t-1";"3"
\ar "t";"t-1"
\ar "dots";"t"
\ar "k";"dots"
\ar "k";"dots2"
\ar "dots2";"n-2"
\ar "n-2";"n"
\ar "n-1";"n-2"
\end{xy}
& (t\ge 2)\\
\end{array}
\right.\]
where $V_t=\left\{
\begin{array}{cl}
K& \text{if $t$ is odd}\\ 0& \text{if $t$ is even}.\\\end{array}\right.$

We can also check that 
$\tau_{A_n}^{-t} I_0=\tau_{A_n}^{-(t-1)}P^-_0$ is equal to $\tau^{-(t-1)} P^-_0$ for any $1\le t\le k-1$ and given by
\[\left\{
\begin{array}{ll}
P_0^- & (t=1)\\\\
\begin{xy}
(-31,5)*[o]+{V_{t-1}}="0",(-31,-5)*[o]+{V_{t-1}}="1",(-21,0)*[o]+{K}="2",(-3,0)*[o]+{\cdots}="3",
(16,0)*[o]+{K}="t-1",(16,-5)*[o]+{t},(34,0)*[o]+{0}="t",(34,-5)*[o]+{t+1},
,(48,0)*[o]+{\cdots}="dots",(62,0)*[o]+{0}="k",(62,-5)*[o]+{k},(77,0)*[o]+{\cdots}="dots2",(91,0)*[o]+{0}="n-2", (103,5)*[o]+{0}="n-1",(103,-5)*[o]+{0}="n",
\ar "0";"2"
\ar "2";"1"
\ar "3";"2"
\ar "t-1";"3"
\ar "t";"t-1"
\ar "dots";"t"
\ar "k";"dots"
\ar "k";"dots2"
\ar "dots2";"n-2"
\ar "n-2";"n"
\ar "n-1";"n-2"
\end{xy}

& (t\ge 2)\\
\end{array}
\right.\]
where $V_t=\left\{
\begin{array}{cl}
K& \text{if $t$ is odd}\\ 0& \text{if $t$ is even}.\\\end{array}\right.$

(5). 
Assume that $M\oplus X_0\oplus P_0^{-}\in \sttilt A$ with $M\in \add \mathcal{R}$.
Since $M\oplus X_0\in \sttilt A_0$, we have
\[
\tau_{A_0}^{-(k-1)}(M\oplus X_0)\stackrel{\text{(1)}}{=}\tau_{A_0}^{-(k-1)}M\oplus P_{n-2}^-\in \sttilt A_0.
\]
In particular, either $P_{n-1}^-$ or $S_{n-1}$ is in $\add \tau_{A_0}^{-(k-1)}M$ and 
either $P_n^-$ or $S_n$ is in $\add \tau^{-(k-1)}M$.
Suppose $P_{n-1}^-\in \add \tau_{A_0}^{-(k-1)}M$. Then (3) implies 
\[
\mathcal{R}\not\ni \tau_{A_0}^{k-1}P_{n-1}^-\in \add M.
\]
This contradicts $M\in \add\mathcal{R}$. Similarly, if $S_n\in \add \tau^{-(k-1)}M$, then we reach a
contradiction. In particular, we have
\[
\tau_{A_0}^{k-1}S_{n-1}\oplus 
\tau_{A_0}^{k-1}P^-_n \in \add M.
\]

Next we assume $M'\oplus X'_n\oplus P_n\in \sttilt A$ with $M'\in \add \mathcal{R}$.
Let $N=\tau M'$. Since $\tau(M'\oplus X_0)=N\oplus Y_n\in \sttilt A_n$, we have
\[
\tau_{A_n}^{k-1}(N\oplus Y_n)\stackrel{\text{(2)}}{=}\tau_{A_n}^{k-1}N\oplus P_2^-\in \sttilt A_n.
\]
In particular, either $P_0^-$ or $S_0$ is in $\add \tau_{A_n}^{k-1}N$ and 
either $P_1^-$ or $S_1$ is in $\add \tau^{k-1}N$.
Suppose $S_0=I_0^{A_n}\in \add \tau_{A_n}^{k-1}N$. Then (4) implies
\[
\mathcal{R}\not\ni \tau_{A_n}^{-(k-1)}S_0\in \add N=\add \tau M'\subset \mathcal{R}.
\]
This contradicts $M\in \add \mathcal{R}$. Similarly, if $P_1^-\in \add \tau^{k-1}N$, we reach a
contradiction. In particular, we have
\[\tau_{A_n}^{-(k-1)}P_0^-\oplus 
\tau_{A_n}^{-(k-1)}S_1 \in \add N.\]

Therefore, it is sufficient to show 
\[
L\in \add (\tau_{A_0}^{k-1}S_{n-1}\oplus 
\tau_{A_0}^{k-1}P^-_n)\cap \add (\tau_{A_n}^{-(k-1)}P_0^-\oplus 
\tau_{A_n}^{-(k-1)}S_1) .
\]
We show
\[L\in \add (\tau_{A_0}^{k-1}S_{n-1}\oplus 
\tau_{A_0}^{k-1}P^-_n).\]
As we already checked, either $\tau_{A_0}^{k-1}P_{n-1}^-=\tau_{A_0}^{k-2}S_{n-1}$ or $\tau_{A_0}^{k-1}S_n=\tau_{A_0}^{k-2}P_n^-$
has the following form.
\[
\begin{xy}
(-31,5)*[o]+{0}="0",(-31,-5)*[o]+{0}="1",(-21,0)*[o]+{0}="2",(-8,0)*[o]+{K}="3",(6,0)*[o]+{\cdots}="dots",(20,0)*[o]+{K}="k-1",(34,0)*[o]+{K}="k"
,(48,0)*[o]+{K}="k+1",(64,0)*[o]+{\cdots}="dots2",(80,0)*[o]+{K}="n-3",(93,0)*[o]+{K}="n-2", (106,5)*[o]+{K}="n-1",(106,-5)*[o]+{K}="n",
\ar "0";"2"
\ar "2";"1"
\ar "3";"2"
\ar "dots";"3"
\ar "k-1";"dots"
\ar "k";"k-1"
\ar "k";"k+1"
\ar "k+1";"dots2"
\ar "dots2";"n-3"
\ar "n-3";"n-2"
\ar "n-1";"n-2"
\ar "n-2";"n"
\end{xy}
\] 
In particular, either $\tau_{A_0}^{k-1}S_{n-1}$ or $\tau_{A_0}^{k-1}P_n^-$ is isomorphic to $L$.
Similarly, we can check 
\[L\in \add (\tau_{A_n}^{-(k-1)}P_0^-\oplus 
\tau_{A_n}^{-(k-1)}S_1).\]

(6). We already checked $\mathcal{X}_0$ has a unique module $X_0$ (up to isomorphism). Then it follows from (5) that
both $(0,n,L)$ and $(0,1,L)$ satisfy the assumptions $\mathsf{A}_1$ and $\mathsf{A}_2$.
(5) also implies that $(0,n,L)$ satisfies $\mathsf{A}_4$. Moreover,
by considering the correspondence 
\[\begin{array}{llll}
	0 &\leftrightarrow & n-1\\
	1 &\leftrightarrow & n\\
	v &\leftrightarrow & n-v & (v\in \{2,\dots,n-2\})\\
	k &\leftrightarrow & n-k,
\end{array}\]
we have that $(n-1, 1, L)$ satisfies Assumption\;\ref{assumption} with $\mathcal{X}'_1\ne \emptyset$. 
In particular, the following statements hold.
\begin{itemize}
	\item $\mathcal{X}'_1$ has a unique module $X'_{1}$ up to isomorphism.
	\item If $M'\oplus X'_1\oplus P_1\in \sttilt A$ with $M'\in \add\mathcal{R}$, then
	\[L\in \add \tau M'.\]
\end{itemize} 
This shows that $(0,1,L)$ satisfies $\mathsf{A}_4$.
\end{proof}
\subsubsection{A proof of Proposition\;\ref{prop:D+--+}}
Here we prove Proposition\;\ref{prop:D+--+}.
\begin{proof}
(1). By Lemma\;\ref{MGS:D:+--+}(6), $(0,n,L)$ and $(0,1,L)$ satisfy $\mathsf{A1}$, $\mathsf{A2}$, and $\mathsf{A4}$.
Moreover, $\mathsf{A5}$ follows from Proposition\;\ref{mgs:keyprop}(2).
Therefore, it is sufficient to check $\mathsf{A3}$.

By applying $\Hom_A(-,L)$ to the almost split sequence
\[
0\to P_n\to P_{n-2}\to \tau^{-1}P_n\to 0,
\]
we obtain an exact sequence
\[
0\to \Hom_A(\tau^{-1}P_n, L)\to \Hom_A(P_{n-2}, L)\to \Hom_A(P_n, L)\to \Ext^1_A(\tau^{-1}P_n, L).
\]
Since $L$ is not projective, we have $\Ext_A^1(\tau^{-1}P_n, L)=0$.
This shows 
\[
\dim_K\Hom_A(P_n,\tau L)=\dim_K\Hom_A(\tau^{-1}P_n, L)=\dim_K\Hom_A(P_{n-2}, L)-\dim_K\Hom_A(P_n, L)=1.
\]
Hence, $(0,n,L)$ satisfies $\mathsf{A3}$.

Similarly, we have an exact sequence
\[
0\to \Hom_A(\tau^{-1}P_1, L)\to \Hom_A(P_2, L)\to \Hom_A(P_1, L)\to \Ext^1_A(\tau^{-1}P_n, L).
\]
Since $L$ is not projective, we have $\Ext_A^1(\tau^{-1}P_n, L)=0$.
This shows 
\[
\dim_K\Hom_A(P_1,\tau L)=\dim_K\Hom_A(\tau^{-1}P_1, L)=\dim_K\Hom_A(P_{2}, L)-\dim_K\Hom_A(P_1, L)=1.
\]
Hence, $(0,1,L)$ satisfies $\mathsf{A3}$.

(2). By (1), $(0,n,L)$ satisfies Assumption\;\ref{assumption}.
Suppose that $\ell(Q)\not\le \ell(\mu_0 Q)$. Then it follows from Proposition\;\ref{mgs:keyprop2} (2) that
\[\ell(Q)\le \ell(\mu_n Q).\]
Since $(\mu_n Q)^{{\rm op}}$ is isomorphic to the quiver
\[\begin{xy}
(-5,5)*[o]+{0}="0",(-5,-5)*[o]+{1}="1",(5,0)*[o]+{2}="2",(17,0)*[o]+{\cdots}="dots",(34,0)*[o]+{n-k}="n-k"
,(51,0)*[o]+{\cdots}="dots2",(70,0)*[o]+{n-2}="n-2", (88,5)*[o]+{n-1}="n-1",(85,-5)*[o]+{n}="n",
\ar "2";"0"
\ar "2";"1"
\ar "2";"dots"
\ar "dots";"n-k"
\ar "dots2";"n-k"
\ar "n-2";"dots2"
\ar "n-1";"n-2"
\ar "n-2";"n"
\end{xy}\]
in $\mathcal{Q}(k,-,-,-,+)$. Hence, we have
\[\ell(Q)\le \ell(\mu_nQ)=\ell((\mu_n Q)^{{\rm op}})\stackrel{{\rm Lemma\;\ref{D: foundation}}}{=} \ell(Q(k,-,-,-,+))=\ell(\mu_0 Q).\]
This is a contradiction. Therefore, we have 
\[
\ell(Q)\le\ell(\mu_0 Q)\stackrel{{\rm Prop.\;\ref{mgs:keyprop}(2)}}{\le } \ell(Q).
\]

By the correspondence 
\[\begin{array}{llll}
	0 &\leftrightarrow & n\\
	1 &\leftrightarrow & n-1\\
	v &\leftrightarrow & n-v & (v\in \{2,\dots,n-2\})\\
\end{array}\]
we get a quiver $Q'\in \mathcal{Q}(k,+--+)$ such that
\[
Q^{{\rm op}}\simeq Q',\ \mu_n Q^{{\rm op}}\simeq \mu_0Q'.
\]
Therefore, it follows from Lemma\;\ref{D:	 foundation} and $\ell(Q)=\ell(\mu_0Q)$ that
\[
\ell(Q)=\ell(\mu_0Q)=\ell(\mu_0 Q')=\ell(\mu_n Q^{{\rm op}})=\ell((\mu_n Q)^{\rm op})=\ell(\mu_n Q).
\]


Then, by considering the correspondence 
\[\begin{array}{llll}
0 &\leftrightarrow & n-1\\
1 &\leftrightarrow & n\\
v &\leftrightarrow & n-v & (v\in \{2,\dots,n-2\})\\
k &\leftrightarrow & n-k,
\end{array}\]
we also obtain 
\[
\ell(Q)=\ell(\mu_1Q)=\ell(\mu_{n-1}Q).
\]
This shows the assertion.	
\end{proof}
\subsection{Case: $Q=Q(k,+,-,+,+)$ }
In this subsection, we treat the case that $Q=Q(k,+,-,+,+)$ and show the following proposition.
\label{subsect:D+-++}
\begin{proposition}
	\label{MGS:D:+-++}
	We have $\ell(Q)= \ell(\mu_0 Q).$
\end{proposition}
\begin{proof}
	Let $B=K(\mu_{n-1}Q)$ and suppose that $\ell:=\ell(Q)\not\le \ell(\mu_0 Q)$.

Note that $\mu_{n-1}Q=Q(k,+,-,-,+)$. We have $\ell(Q)=\ell(\mu_{n-1}Q)$ by Proposition\;\ref{prop:D+--+}. Then Proposition\;\ref{ladkani} implies that
\[\MGSm(B,S^B_{n-1})\neq \emptyset\Leftrightarrow \MGSm(A,\mu_{n-1}A)\neq \emptyset.\]
Since $\MGSm(A,\mu_{n-1}A)\neq \emptyset$ by Proposition\;\ref{mgs:keyprop} (2), we obtain
\[\MGSm(B,S^B_{n-1})\neq \emptyset.\]
We take $\omega=(T_0\to \cdots \to T_\ell)\in \MGSm(B,S^B_{n-1})$ such that
 \[t:=t^{(0)}_\omega=\max\{t^{(0)}_{\omega'}\mid \omega'\in \MGSm(B,S^B_{n-1})\}.\]
Then we show
\[T_{t}=M\oplus X_0\oplus P_0^-,\ T_{t-1}=M\oplus \overline{X_0}\oplus P_0^- \]
with $X_0\in \mathcal{X}^B_0$ and $M\in \add\mathcal{R}(B)$. Otherwise, Lemma\;\ref{mgs:keylemma} implies that
one of the following two cases occurs.
\begin{enumerate}[{\rm (i)}]
	\item $t=\ell $.
	\item There exists $\omega'\in \MGSm(B,S^B_{n-1})$ with $t^{(0)}_{\omega'}>t$.
\end{enumerate}
The second case contradicts the maximality of $t$. Thus $t=\ell$ holds.
Therefore, we obtain
\[T_{\ell-2}=S^B_0\oplus S^B_{n-1}.\] 
Then, by applying Proposition\;\ref{ladkani} to $n-1$, we get the following maximal green sequence with length $\ell$.
\[A=T'_0\to T'_1=\mu_{n-1} A\to  \cdots \to T'_{\ell-1}=S_0\to T'_{\ell}=0\] 
In particular, $\MGSm(A,S_0)\neq \emptyset$ and it follows from Lemma\;\ref{mgs:keylemma}(3) that
$\ell(Q)\le \ell(\mu_0 Q)$. This is a contradiction.
Therefore, we have
\[X_0\in \mathcal{X}_0(B),\  M\in \add\mathcal{R}(B).\]
Now let $L\in \mod B$ as in Lemma\;\ref{MGS:D:+--+} (note that $B=K(Q(k,+,-,-,+))$). 
Then it follows from Proposition\;\ref{prop:D+--+}(1) that we can apply Proposition\;\ref{mgs:keyprop2} (3) to
$\left(\mu_{n-1}Q,(0,n,L), (0,1,L)\right)$ and obtain
\[
\ell(Q)=\ell(\mu_{n-1}Q)\le \ell(\mu_1\mu_{n}\mu_{n-1}Q).
\]

Since $Q':=(\mu_1\mu_{n}\mu_{n-1}Q)^{{\rm op}}$ is the following quiver.
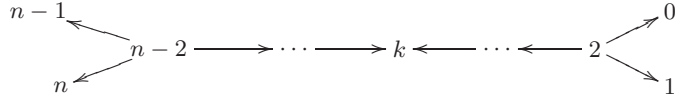
\[\begin{xy}
(-11,5)*[o]+{n-1}="0",(-8,-5)*[o]+{n}="1",(5,0)*[o]+{n-2}="2",(23,0)*[o]+{\cdots}="dots",(37,0)*[o]+{k}="n-k"
,(50,0)*[o]+{\cdots}="dots2",(63,0)*[o]+{2}="n-2", (73,5)*[o]+{0}="n-1",(73,-5)*[o]+{1}="n",
\ar "2";"0"
\ar "2";"1"
\ar "2";"dots"
\ar "dots";"n-k"
\ar "dots2";"n-k"
\ar "n-2";"dots2"
\ar "n-2";"n"
\ar "n-2";"n-1"
\end{xy}\]
In particular, $Q'$ is of type $(k,-,-,+,+)$. Then it follows from Lemma\;\ref{D: foundation} that
 $\ell(Q')=\ell(Q(k,-,-,+,+))=\ell(\mu_0 Q)$. This shows
 \[
 \ell(Q)=\ell(\mu_{n-1}Q)\le \ell(\mu_1\mu_{n}\mu_{n-1}Q)=\ell(\mu_0 Q).
 \]
 This is a contradiction. Therefore, we have $\ell(Q)\le \ell(\mu_0 Q)$. Then 
 it follows from Proposition\;\ref{mgs:keyprop}(2) that $\ell(Q)\ge \ell(\mu_0 Q)$. Therefore, we have the assertion.
\end{proof}

\subsection{Case: $Q=Q(k,+,-,-,-)$}
\label{subsect:D+---}
In this subsection, we treat the case that $Q=Q(k,+,-,-,-)$ and show the following equality.
\[
\ell(Q)=\ell(\mu_0 Q).
\]
If $k\neq n-2$, then we have
\[\begin{array}{llll}
\ell(Q)& = &\ell(\mu_{n-2} Q) &  {\rm (Proposition\;\ref{mgs:keyprop})}\\ 
&= & \ell(\mu_0\mu_{n-2} Q) &  {\rm (Lemma\;\ref{D: foundation}\ and\ 
	Proposition\;\ref {MGS:D:+-++} ) }\\
&= & \ell(\mu_{n-2} \mu_0\mu_{n-2} Q) &  {\rm (Proposition\;\ref{mgs:keyprop}) }\\
&=& \ell(\mu_0 Q)\\
\end{array}
\]
Hence we may assume $k=n-2$.

We put $Q=Q(n-2,+,-,-,-)$, $\mathbf{i}=(1,2,\dots,n-2,n-1)$, $\mathbf{i}'=(1,2,\dots,n-2,n)$, and
  define two indecomposable modules $L_{n-1}$ and $L_n$ as follows:  
 \[\begin{array}{cll}
 L_{n-1}&:&
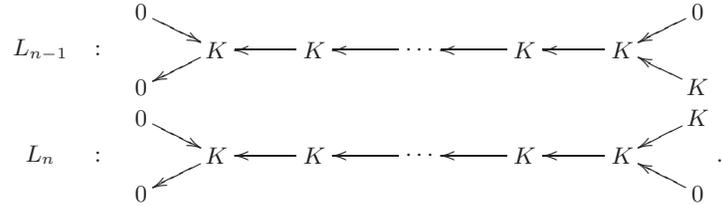
\begin{xy}
 (-31,5)*[o]+{0}="0",(-31,-5)*[o]+{0}="1",(-21,0)*[o]+{K}="2",
 (-8,0)*[o]+{K}="3",(6,0)*[o]+{\cdots}="dots",
 (20,0)*[o]+{K}="n-3",(33,0)*[o]+{K}="n-2", 
 (43,5)*[o]+{0}="n-1",(43,-5)*[o]+{K}="n",
 \ar "0";"2"
 \ar "2";"1"
 \ar "3";"2"
 \ar "dots";"3"
 \ar "n-3";"dots"
 \ar "n-2";"n-3"
 \ar "n-1";"n-2"
 \ar "n";"n-2"
 \end{xy}\\
 L_{n}&:&\begin{xy}
 (-31,5)*[o]+{0}="0",(-31,-5)*[o]+{0}="1",(-21,0)*[o]+{K}="2",
 (-8,0)*[o]+{K}="3",(6,0)*[o]+{\cdots}="dots",
 (20,0)*[o]+{K}="n-3",(33,0)*[o]+{K}="n-2", 
 (43,5)*[o]+{K}="n-1",(43,-5)*[o]+{0}="n",
 \ar "0";"2"
 \ar "2";"1"
 \ar "3";"2"
 \ar "dots";"3"
 \ar "n-3";"dots"
 \ar "n-2";"n-3"
 \ar "n-1";"n-2"
 \ar "n";"n-2"
 \end{xy}.
 \end{array}
 \]
 We will show the following proposition.
 \begin{proposition}
 	\label{D:statement:+---}
Let $Q$, $\mathbf{i}$, $\mathbf{i}'$, $L_{n-1}$, and $L_n$ be as above. 
\begin{enumerate}[{\rm (1)}]
	\item $(0,\mathbf{i},L_{n-1})$ and $(0,\mathbf{i}',L_{n})$ satisfy Assumption\;\ref{assumption}.
	\item We have $\ell(Q)=\ell(\mu_0 Q)$.
\end{enumerate}
\end{proposition}
\subsubsection{First step}
Here we check that $(0,(1,2,\dots,n-2),L_{n-1})$ and $(0,(1,2,\dots,n-2),L_n)$ satisfy Assumption\;\ref{assumption}.

Let $Q=Q(n-2,+,-,-,-)$ and $A=KQ$.
Assume $X\in \ind A_0\subset \ind A$ and we set $x_i:=\dim_K\Hom_A(P_i, X)$, $x'_i:=\dim_K\Hom_A(\tau^{-1}P_i, X)$.
Then we have
\[x'_0=x'_2-x_0=x'_1+x_3-x_2=x_3-x_1.\]
Therefore, it follows from the classification of indecomposable modules of $A_0$ that 
$\mathcal{X}_0$ has a unique module $X_0$ (up to isomorphism) and it has the following form. 
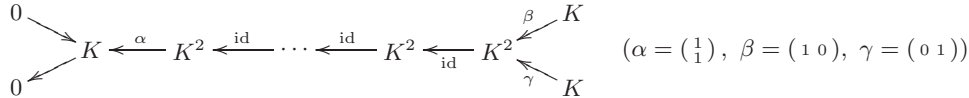
\[
\begin{xy}
(-31,5)*[o]+{0}="0",(-31,-5)*[o]+{0}="1",(-21,0)*[o]+{K}="2",
(-8,0)*[o]+{K^2}="3",(6,0)*[o]+{\cdots}="dots",
(20,0)*[o]+{K^2}="n-3",(33,0)*[o]+{K^2}="n-2", 
(43,5)*[o]+{K}="n-1",(43,-5)*[o]+{K}="n",
\ar "0";"2"
\ar "2";"1"
\ar "3";"2"_{\alpha}
\ar "dots";"3"_{\id}
\ar "n-3";"dots"_{\id}
\ar "n-2";"n-3"^{\id}
\ar "n-1";"n-2"_{\beta}
\ar "n";"n-2"^{\gamma}
\end{xy}\hspace{10pt}
\left(
\alpha=\left(\begin{smallmatrix}1\\ 1\end{smallmatrix}\right),\ 
\beta=(\begin{smallmatrix}1& 0\end{smallmatrix}),\ 
\gamma=(\begin{smallmatrix}0& 1\end{smallmatrix})
\right)
\]

Let $A'=A^{\rm op}$, $Y'\ind A'/(e_1)\subset \ind A'$, $y_i:=\dim_K\Hom_{A'}(P_i^{A'}, Y')$
and $y'_i:=\dim_K\Hom_{A'}(\tau_{A'}^{-1}P_i^{A'}, Y')$.
Then we have
\[\begin{array}{lcl}
y'_1&=&y'_2-y_1\\ 
&=&y'_0+y'_3-y_2\\ 
&=&y_2-y_0+y'_3-y_2\\
&=&y_2-y_0+y'_4-y_3\\
&\vdots&\\
&=&y_2-y_0+y'_{n-2}-y_{n-3}\\
&=&y_2-y_0+y'_{n-1}+y'_n-y_{n-2}\\
&=&y_2-y_0+y_{n-2}-y_{n-1}-y_n\\
\end{array}\]
Since $y_1=0$ and $A'/(e_1)$ is isomorphic to a path algebra of type $D_n$, the classification of indecomposable modules of $A'/(e_1)$ implies 
that $\mathcal{X}_1(A')$ has a unique module $Y'_1$ (up to isomorphism) and it has the following form.
\[
\begin{xy}
	(-31,5)*[o]+{0}="0",(-31,-5)*[o]+{0}="1",(-21,0)*[o]+{K}="2",
	(-8,0)*[o]+{K}="3",(6,0)*[o]+{\cdots}="dots",
	(20,0)*[o]+{K}="n-3",(33,0)*[o]+{K}="n-2", 
	(43,5)*[o]+{0}="n-1",(43,-5)*[o]+{0}="n",
	\ar "2";"0"
	\ar "1";"2"
	\ar "2";"3"
	\ar "3";"dots"
	\ar "dots";"n-3"
	\ar "n-3";"n-2"
	\ar "n-2";"n-1"
	\ar "n-2";"n"
\end{xy}
\]
Since $\Tr$ induces a bijection between $\mathcal{X}_1(A')$ and $\mathcal{X}'_1(A)$, $\mathcal{X}'_1(A)$ has a 
unique module $X'_1=\Tr Y_1$. Moreover, we have
\[
\tau X'_1=D\Tr (\Tr Y_1)\simeq D Y'_1\simeq P_{n-2}^{A_1}.
\]

\begin{lemma}
	\label{MGS:D:+---}
	Let $Q$, $X_0$, $X'_1$ be as above. 
\begin{enumerate}[{\rm (1)}]
				\item $\tau_{A_0}^{3}X_0= P_{n-2}^{-}$ and $\tau X_1'\simeq P_{n-2}^{A_1}$.
				\item $\tau_{A_0}^{-3}P_{n-1}^-$ and $\tau_{A_0}^{-3}P_n^-$ are preinjective. 
				\item Let $L_{n-1}$ and $L_n$ be indecomposable modules having the following forms:  
				\[\begin{array}{cll}
				L_{n-1}&:&\begin{xy}
				(-31,5)*[o]+{0}="0",(-31,-5)*[o]+{0}="1",(-21,0)*[o]+{K}="2",
				(-8,0)*[o]+{K}="3",(6,0)*[o]+{\cdots}="dots",
				(20,0)*[o]+{K}="n-3",(33,0)*[o]+{K}="n-2", 
				(43,5)*[o]+{0}="n-1",(43,-5)*[o]+{K}="n",
				\ar "0";"2"
				\ar "2";"1"
				\ar "3";"2"
				\ar "dots";"3"
				\ar "n-3";"dots"
				\ar "n-2";"n-3"
				\ar "n-1";"n-2"
				\ar "n";"n-2"
				\end{xy}\\
				L_{n}&:&\begin{xy}
				(-31,5)*[o]+{0}="0",(-31,-5)*[o]+{0}="1",(-21,0)*[o]+{K}="2",
				(-8,0)*[o]+{K}="3",(6,0)*[o]+{\cdots}="dots",
				(20,0)*[o]+{K}="n-3",(33,0)*[o]+{K}="n-2", 
				(43,5)*[o]+{K}="n-1",(43,-5)*[o]+{0}="n",
				\ar "0";"2"
				\ar "2";"1"
				\ar "3";"2"
				\ar "dots";"3"
				\ar "n-3";"dots"
				\ar "n-2";"n-3"
				\ar "n-1";"n-2"
				\ar "n";"n-2"
				\end{xy}.
				\end{array}
				\]
				Then $\tau_{A_0}^{-3}S_{n-1}=\tau_{A_1}^{-1}P_{n}^-=L_{n-1}$ and  $\tau_{A_0}^{-3}S_n=\tau_{A_1}^{-1}P_{n-1}^-=L_n$.  
		\item If $T=M\oplus X_0\oplus P_0^-\in \sttilt A$ $($resp. $T'=M'\oplus X'_1\oplus P_1\in \sttilt A)$ with
		 $M\in \add\mathcal{R}$ $($resp. $M'\in \add \mathcal{R})$. Then we have
		 \[L_{n-1}\oplus L_n\in \add M\ (\text{resp. $L_{n-1}\oplus L_n\in \add \tau M'$})\]
		\item Both $(0, (1,2,\dots,n-2),L_{n-1})$ and $(0, (1,2,\dots,n-2),L_{n})$ satisfy Assumption\;\ref{assumption}.
	\end{enumerate}
\end{lemma}
\begin{proof}
	(1). We already checked $\tau X_1'\simeq P_{n-2}^{A_1}$. Thus, we only check
	\[
	\tau_{A_0}^{3}X_0= P_{n-2}^{-}.
	\]
	A minimal projective presentation of $X_0$ in $\mod A_0\subset \mod A$ is given by
	\[0\to P_1^{A_0}\oplus P_2^{A_0}\to P_{n-1}^{A_0}\oplus P_n^{A_0}\to X_0\to 0.\]
	Since $A_0$ is hereditary, we have the following exact sequence. 
\[0\to \tau_{A_0} X_0\to I_1^{A_0}\oplus I_2^{A_0}\to I_{n-1}^{A_0}\oplus I_n^{A_0}\to 0.\]
Therefore, $\tau_{A_0}X_0\in \ind A_0\subset \ind A$ is given by the following quiver representation.
\[\begin{xy}
(-31,5)*[o]+{0}="0",(-31,-5)*[o]+{K}="1",(-21,0)*[o]+{K^2}="2",
(-8,0)*[o]+{K^2}="3",(6,0)*[o]+{\cdots}="dots",
(20,0)*[o]+{K^2}="n-3",(33,0)*[o]+{K^2}="n-2", 
(43,5)*[o]+{K}="n-1",(43,-5)*[o]+{K}="n",
\ar "0";"2"
\ar "2";"1"_{\alpha}
\ar "3";"2"_{\id}
\ar "dots";"3"_{\id}
\ar "n-3";"dots"_{\id}
\ar "n-2";"n-3"^{\id}
\ar "n-1";"n-2"_{\beta}
\ar "n";"n-2"^{\gamma}
\end{xy}\]	 
Then a minimal projective presentation of $\tau_{A_0}X_0$ in $\mod A_0\subset \mod A$ is given by
\[0\to P_1^{A_0}\to P_{n-1}^{A_0}\oplus P_n^{A_0}\to X_0\to 0,\]
and the same argument as above implies that $\tau_{A_0}^2 X_0$ is given by the following quiver representation.
\[\begin{xy}
(-31,5)*[o]+{0}="0",(-31,-5)*[o]+{K}="1",(-21,0)*[o]+{K}="2",
(-8,0)*[o]+{K}="3",(6,0)*[o]+{\cdots}="dots",
(20,0)*[o]+{K}="n-3",(33,0)*[o]+{K}="n-2", 
(43,5)*[o]+{0}="n-1",(43,-5)*[o]+{0}="n",
\ar "0";"2"
\ar "2";"1"
\ar "3";"2"
\ar "dots";"3"
\ar "n-3";"dots"
\ar "n-2";"n-3"
\ar "n-1";"n-2"
\ar "n";"n-2"
\end{xy}\]
In particular, we have
 \[\tau_{A_0}^{2}X_0\simeq P_{n-2}^{A_0},\ \tau_{A_0}^{3}X_0= P_{n-2}^-.\]

(2) and (3).
Since $S_n=I_n=I_n^{A_0}$, we have $\tau^{-1}_{A_0} S_n=P_n^-$ and $\tau_{A_0}^{-1}P_n^-=P_n^{A_0}$.
Then a minimal injective co-presentation of $P_n^{A_0}$ in $\mod A_0\subset \mod A$ is given by
\[0\to P_n^{A_0}\to I_1^{A_0}\to I_{n-1}^{A_0}\to 0.\]
Hence we have an exact sequence
\[0\to P_1^{A_0}\to P_{n-1}^{A_0}\to \tau^{-1}_{A_0}P_n^{A_0}\to 0.\]
This gives us that  $\tau^{-1}_{A_0}P_n^{A_0}=\tau^{-3}S_n$ is given by the following quiver representation.
\[\begin{xy}
(-31,5)*[o]+{0}="0",(-31,-5)*[o]+{0}="1",(-21,0)*[o]+{K}="2",
(-8,0)*[o]+{K}="3",(6,0)*[o]+{\cdots}="dots",
(20,0)*[o]+{K}="n-3",(33,0)*[o]+{K}="n-2", 
(43,5)*[o]+{K}="n-1",(43,-5)*[o]+{0}="n",
\ar "0";"2"
\ar "2";"1"
\ar "3";"2"
\ar "dots";"3"
\ar "n-3";"dots"
\ar "n-2";"n-3"
\ar "n-1";"n-2"
\ar "n";"n-2"
\end{xy}\]
In particular, $\tau_{A_0}^{-3}S_n=L_n$.

Similarly, we can check that
$M_n:=\tau_{A_0}^{-3}P_n^-=\tau^{-1}_{A_0} L_n$  is given by 
\[\begin{xy}
(-31,5)*[o]+{0}="0",(-31,-5)*[o]+{0}="1",(-21,0)*[o]+{0}="2",
(-8,0)*[o]+{K}="3",(6,0)*[o]+{\cdots}="dots",
(20,0)*[o]+{K}="n-3",(33,0)*[o]+{K}="n-2", 
(43,5)*[o]+{0}="n-1",(43,-5)*[o]+{K}="n",
\ar "0";"2"
\ar "2";"1"
\ar "3";"2"
\ar "dots";"3"
\ar "n-3";"dots"
\ar "n-2";"n-3"
\ar "n-1";"n-2"
\ar "n";"n-2"
\end{xy}\]
and $\tau^{-t} M_n$ is given by
\[\left\{
\begin{array}{ll}
\begin{xy}
(-18,5)*[o]+{0}="0",(-18,-5)*[o]+{0}="1",(-8,0)*[o]+{0}="3",(6,0)*[o]+{\cdots}="dots",(20,0)*[o]+{0}="k-1",(34,0)*[o]+{K}="k"
,(48,0)*[o]+{K}="k+1",(64,0)*[o]+{\cdots}="dots2",(80,0)*[o]+{K}="n-3", (93,5)*[o]+{V_{t+1}}="n-1",(93,-5)*[o]+{V_t}="n",
(34,-7)*[o]+{t+3},
\ar "0";"3"
\ar "3";"1"
\ar "dots";"3"
\ar "k-1";"dots"
\ar "k";"k-1"
\ar "k+1";"k"
\ar "dots2";"k+1"
\ar "n-3";"dots2"
\ar "n-1";"n-3"
\ar "n";"n-3"
\end{xy} & (t\le n-5)\\\\
S_{n-1}=I_{n-1}&
  (t=n-4,\ n:{\rm odd})\\\\
S_{n}=I_{n}&  (t=n-4,\ n:{\rm even})
\end{array}
\right.\]
 where $V_t=\left\{
\begin{array}{cl}
0& \text{if $t$ is odd}\\ K& \text{if $t$ is even}.\\\end{array}\right.$

Hence we obtain that $\tau_{A_0}^{-3} P_n^-$ is preinjective and $\tau_{A_0}^{-3}S_n=L_n$.
Similarly, we can check that
$\tau_{A_0}^{-3} P_{n-1}^-$ is preinjective and $\tau_{A_0}^{-3}S_{n-1}=L_{n-1}$.
Then the remaining assertions $\tau_{A_1}^{-1}P_n^-=L_{n-1}$ and $\tau_{A_1}^{-1}P_{n-1}^-=L_{n}$ follow form
$L_{n-1}\simeq P_n^{A_1}$ and $L_{n}\simeq P_{n-1}^{A_1}$.

(4).
Assume that $M\oplus X_0\oplus P_0^-\in \sttilt A$ and $M\in \add \mathcal{R}$.
Since $M\oplus X_0\in \sttilt A_0$ , $\tau_{A_0}^3M\oplus \tau_{A_0}^3 X_0$ is also in $\sttilt A_0$.
 Then it follows from (1) that
\[\tau_{A_0}^3M\oplus P_{n-2}^-\in \sttilt A_0.\]
 In particular, either $S_{n-1}$ or $P_{n-1}^-$ is in $\add \tau_{A_0}^3 M$.
If $P_{n-1}^-\in \add \tau_{A_0}^3 M$, then $\tau_{A_0}^{-3}P_{n-1}\in \add M$.
Therefore, (2) implies that $M$ has a non-regular indecomposable direct summand. This is a contradiction.
Hence, we obtain $S_{n-1}\in \add \tau_{A_0}^3 M$. Then $L_{n-1}\in \add M$ follows from (3). 
Similarly, we can check $L_n\in \add M$.

Next we assume that $M'\oplus X'_1\oplus P_1\in \sttilt A$ and $M'\in \add \mathcal{R}$. Let $N:=\tau M$. 
Since $N\oplus \tau X'_1\in \sttilt A_1$, it follows from (1) that
\[\tau_{A_1}N \oplus P_{n-2}^-\in \sttilt A_1.\]
 In particular, either $S_{n-1}$ or $P_{n-1}^-$ is in $\add \tau_{A_1} N$.
If $S_{n-1}=I_{n-1}=I^{A_1}_{n-1}\in \add \tau_{A_1} N$, then $\tau_{A_1}^{-1}I^{A_1}_{n-1}=P_{n-1}^-\in \add N=\add \tau M'$.
This contradicts $M'\in \add \mathcal{R}$.
Hence, we obtain $P_{n-1}^-\in \add \tau_{A_1} N$. 
Therefore, $L_n=P^{A_1}_{n-1}=\tau_{A_1}^{-1}P_{n-1}^-\in \add N=\add \tau M'$ by (3). 
Similarly, we can check $L_{n-1}\in \add \tau M'$.

(5). As we already checked, $X_0$ is a unique module in $\mathcal{X}_1$ and $X_1'$ is a unique module in $\mathcal{X}'_1$.
Then $\mathsf{A1}$, $\mathsf{A2}$, and $\mathsf{A}_4$ follow from (4) and Lemma\;\ref{findXi }.
We can also check $\mathsf{A5}$. In fact, it follows from Proposition\;\ref{mgs:keyprop} that
\[
\ell(Q)=\ell(\mu_1\mu_1 Q)\ge \ell(\mu_1 Q)=\ell(\mu_2\mu_1 Q)=\cdots=\ell(\mu_{n-3}\cdots\mu_{1}Q).
\]
Therefore, it is sufficient to check $\mathsf{A3}$.

For each $p\in \{2,\dots, n-1\}$, we set 
 $L_{n-1}^{(p-1)}:=F_{p-1}^+\circ\cdots \circ F_1^+(L_{n-1})$ and $L_{n}^{(p-1)}:=F_{p-1}^+\circ\cdots \circ F_1^+(L_{n})$.
 Then it is easy to check the following equations.
\[(\ast)\left\{
\begin{array}{cll}
\dimvec\left(L_{n-1}^{(p-1)}\right)&=&{}^{{\rm t}}(011\cdots 101),\\\\
\dimvec\left(L_{n}^{(p-1)}\right)&=&{}^{{\rm t}}(011\cdots 110).\\
\end{array}\right.
\]

By applying $\Hom_A(-,L_{n-1})$ to the almost split sequence
\[
0\to P_1\to P_2\to \tau^{-1}P_1\to 0,
\]
we obtain an exact sequence
\[
0\to \Hom_A(\tau^{-1}P_1, L_{n-1})\to \Hom_A(P_{2}, L_{n-1})\to \Hom_A(P_1, L_{n-1})\to \Ext^1_A(\tau^{-1}P_{1}, L_{n-1}).
\]
Since $L_{n-1}$ is not projective, we have $\Ext_A^1(\tau^{-1}P_1, L_{n-1})=0$.
In particular, we obtain
\[
\begin{array}{rll}
\dim_K\Hom_A(P_1,\tau L_{n-1})&=&\dim_K\Hom_A(\tau^{-1}P_1, L_{n-1})\\
                                                     &=&\dim_K\Hom_A(P_2, L_{n-1})-\dim_K\Hom_A(P_1, L_{n-1})\\
                                                     &=&1.\\\\
\dim_K\Hom_A(P_1,\tau L_{n})&=&\dim_K\Hom_A(\tau^{-1}P_1, L_{n})\\
                                                     &=&\dim_K\Hom_A(P_2, L_{n})-\dim_K\Hom_A(P_1, L_{n})\\
                                                     &=&1.\\
\end{array}
\]

Similarly, we can verify
\[
\begin{array}{rll}
\dim_K\Hom_{A^{(p-1)}}\left(P^{A^{(p-1)}}_{p},\tau_{A^{(p-1)}} (L^{(p-1)}_{n-1})\right)&=&1.\\\\
\dim_K\Hom_{A^{(p-1)}}\left(P^{A^{(p-1)}}_{p},\tau_{A^{(p-1)}} (L^{(p-1)}_{n})\right)&=&1.\\
\end{array}
\]
by using the equations $(\ast)$ .
Therefore, we have the assertion.
 \end{proof}
\subsubsection{Second step}
Let $B=K(\mu_{n-2}\cdots\mu_1 Q)$, $B_{n-1}=B/(e_{n-1})$, $B_n=B/(e_n)$.
$B':=B^{\rm op}$, $Y'\in \ind B'/(e_{n-1})\subset \ind B'$.
If we set $y_i:=\dim_K\Hom_{B'}(P_i^{B'}, Y')$
and $y'_i:=\dim_K\Hom_{B'}(\tau_{A'}^{-1}P_i^{B'}, Y')$,
 then we have
\[
\begin{array}{lcl}
y'_{n-1}&=&y'_{n-2}-y_{n-1}\\ 
&=&y_{n-3}+y_n-y_{n-2}\\
\end{array}
\]
Since $B'/(e_{n-1})$ is isomorphic to a path algebra of type $D_n$, the classification of indecomposable modules of $B'/(e_{n-1})$ implies 
that $\mathcal{X}_{n-1}(B')$ has a unique module $Y'_{n-1}$ and it has the following form.
\[
\begin{xy}
	(-13,5)*[o]+{K}="0",(-13,-5)*[o]+{K}="1",(-3,0)*[o]+{K^2}="2",
	(10,0)*[o]+{\cdots}="dots",(22,0)*[o]+{K^2}="n-3",(34,0)*[o]+{K}="n-2",
	(43,5)*[o]+{0}="n-1",(43,-5)*[o]+{K}="n",
	\ar "0";"2"
	\ar "1";"2"
	\ar "2";"dots"
	\ar "dots";"n-3"
	\ar "n-3";"n-2"
	\ar "n-1";"n-2"
	\ar "n";"n-2"
\end{xy}
\]	
In particular, $X'_{n-1}:=\Tr Y'_{n-1}$ is a unique module in $\mathcal{X}'_{n-1}(B)$ (up to isomorphism).  
Moreover, we have that $Y_{n-1}:=\tau_B X'_{n-1}=D\Tr X'_{n-1}=D Y'_{n-1}$ has the following form.
\[
\begin{xy}
(-13,5)*[o]+{K}="0",(-13,-5)*[o]+{K}="1",(-3,0)*[o]+{K^2}="2",
(10,0)*[o]+{\cdots}="dots",(22,0)*[o]+{K^2}="n-3",(34,0)*[o]+{K}="n-2",
(43,5)*[o]+{0}="n-1",(43,-5)*[o]+{K}="n",
\ar "2";"0"
\ar "2";"1"
\ar "dots";"2"
\ar "n-3";"dots"
\ar "n-2";"n-3"
\ar "n-2";"n-1"
\ar "n-2";"n"
\end{xy}
\]	
The same argument implies that $\mathcal{X}'_{n}(B)$ has the unique module $X'_n$ (up to isomorphism) and $Y_n:=\tau_B X'_n$ has 
the following form.
\[
\begin{xy}
(-13,5)*[o]+{K}="0",(-13,-5)*[o]+{K}="1",(-3,0)*[o]+{K^2}="2",
(10,0)*[o]+{\cdots}="dots",(22,0)*[o]+{K^2}="n-3",(34,0)*[o]+{K}="n-2",
(43,5)*[o]+{K}="n-1",(43,-5)*[o]+{0}="n",
\ar "2";"0"
\ar "2";"1"
\ar "dots";"2"
\ar "n-3";"dots"
\ar "n-2";"n-3"
\ar "n-2";"n-1"
\ar "n-2";"n"
\end{xy}
\]
\begin{lemma}
\label{MGS:lemma+---2}
Let $Y_{n-1}$ and $Y_n$ be as above.
\begin{enumerate}[{\rm (1)}]
\item $\tau^{-3}_{B_{n-1}} Y_{n-1}=(P_2^B)^-$ and 
$\tau^{-3}_{B_n} Y_n=(P_2^B)^-$.
\item If $M_{n-1}\oplus X'_{n-1} \oplus (P_{n-1}^B)\in \sttilt B$ with
$M_{n-1}\in \add \mathcal{R}(B)$, then 
\[F_{n-2}^+\circ \cdots\circ F_1^+ (L_{n-1})\in \add \tau_B M_{n-1}. \]
\item If $M_n\oplus X'_n \oplus P_n^{B}\in \sttilt B$ with
$M_n\in \add \mathcal{R}(B)$, then 
\[F_{n-2}^+\circ \cdots\circ F_1^+ (L_n)\in \add \tau_B M_n. \]
\item $\ell(Q)=\ell(\mu_1Q)=\cdots=\ell(\mu_{n-2}\cdots\mu_{1}Q)$.
\item Both $(0,\mathbf{i},L_{n-1})$ and $(0,\mathbf{i}',L_n )$ satisfy Assumption\;\ref{assumption}.
\end{enumerate}
\end{lemma}
\begin{proof}
	(1). A minimal projective presentation 
	of $I_2^{B_{n-1}}$ in $\mod B_{n-1}$ is given by
	\[
	 0 \to  P_0^{B_{n-1}}\oplus P_1^{B_{n-1}}\oplus P_n^{B_{n-1}}  \to P_{n-2}^{B_{n-1}} \to I_2^{B_{n-1}} \to 0. 
	\] 
	Therefore, we have an exact sequence
	\[
	0 \to \tau_{B_{n-1}}I_2^{B_{n-1}}\to I_0^{B_{n-1}}\oplus I_1^{B_{n-1}}\oplus I_n^{B_{n-1}}  \to I_{n-2}^{B_{n-1}} \to 0.
	\] 
    In particular, $\tau_{B_{n-1}}I_2^{B_{n-1}}$ has the following form.
    \[
    \begin{xy}
    	(-13,5)*[o]+{K}="0",(-13,-5)*[o]+{K}="1",(-3,0)*[o]+{K^2}="2",
    	(10,0)*[o]+{\cdots}="dots",(22,0)*[o]+{K^2}="n-3",(34,0)*[o]+{K^2}="n-2",
    	(43,5)*[o]+{0}="n-1",(43,-5)*[o]+{K}="n",
    	\ar "2";"0"
    	\ar "2";"1"
    	\ar "dots";"2"
    	\ar "n-3";"dots"
    	\ar "n-2";"n-3"
    	\ar "n-2";"n-1"
    	\ar "n-2";"n"
    \end{xy}
    \]
    Then a minimal projective presentation of $\tau_{B_{n-1}}I_2^{B_{n-1}}$ in $\mod B_{n-1}$ is given by
    \[
    0 \to  P_0^{B_{n-1}}\oplus P_1^{B_{n-1}}\oplus P_n^{B_{n-1}}  \to \left(P_{n-2}^{B_{n-1}}\right)^{\oplus 2} \to \tau_{B_{n-2}}I_2^{B_{n-1}} \to 0. 
    \] 
    Therefore, we have an exact sequence
    \[
    0 \to \tau_{B_{n-1}}^2 I_2^{B_{n-1}}\to I_0^{B_{n-1}}\oplus I_1^{B_{n-1}}\oplus I_n^{B_{n-1}}  \to \left(I_{n-2}^{B_{n-1}}\right)^{\oplus 2} \to 0.
    \] 
    In particular, $\tau_{B_{n-1}}^2I_2^{B_{n-1}}$ has the following form.
    \[
    \begin{xy}
    	(-13,5)*[o]+{K}="0",(-13,-5)*[o]+{K}="1",(-3,0)*[o]+{K^2}="2",
    	(10,0)*[o]+{\cdots}="dots",(22,0)*[o]+{K^2}="n-3",(34,0)*[o]+{K}="n-2",
    	(43,5)*[o]+{0}="n-1",(43,-5)*[o]+{K}="n",
    	\ar "2";"0"
    	\ar "2";"1"
    	\ar "dots";"2"
    	\ar "n-3";"dots"
    	\ar "n-2";"n-3"
    	\ar "n-2";"n-1"
    	\ar "n-2";"n"
    \end{xy}
    \] 
    Hence, we have
	\[
	Y_{n-1}\simeq \tau_{B_{n-1}}^{2}I_2^{B_{n-1}}.
	\]
	The same argument implies 
	\[
	Y_{n}\simeq \tau_{B_{n}}^{2}I_2^{B_{n}}.
	\]
	In particular, we obtain the assertions (1).
	
	(2) and (3). Let $N_{n-1}=\tau_B M_{n-1}$, 
	$L'_{n-1}=F_{n-2}^+\circ \cdots\circ F_1^+ (L_{n-1})$.
	
	 We have
	$
	N_{n-1}\oplus Y_{n-1} \oplus (P_{n-1}^B)^-=\tau_B(M_{n-1}\oplus X'_{n-1}\oplus P_{n-1}^B)\in \sttilt B.
	$
	In particular, \[N_{n-1}\oplus Y_{n-1}\in \sttilt B_{n-1}.\]
	Then (1) implies 
	\[
	\tau_{B_{n-1}}^{-3}N_{n-1}\in \sttilt B_{n-1}/(e_2).
	\]
	This shows that either $S_0$ or $(P_0^{B})^-$ is in $\add \tau_{B_{n-1}}^{-3}N_{n-1}$.
	Hence, either $\tau_{B_{n-1}}^3 S_0$ or $\tau_{B_{n-1}}^3 (P_0^{B})^-=\tau_{B_{n-1}}^4 S_0$ is in $\add N_{n-1}$.
	We can easily verify
	\[
	\begin{array}{lcl}
	S_0^B\stackrel{\tau_{B_{n-1}}}{\longrightarrow} (P_0^B)^-\stackrel{\tau_{B_{n-1}}}{\longrightarrow}
	I_0^{B_{n-1}}\stackrel{\tau_{B_{n-1}}}{\longrightarrow} L'_{n-1}&=&
	\begin{xy}
	(-13,5)*[o]+{0}="0",(-13,-5)*[o]+{K}="1",(-3,0)*[o]+{K}="2",
	(10,0)*[o]+{\cdots}="dots",(22,0)*[o]+{K}="n-3",(34,0)*[o]+{K}="n-2",
	(43,5)*[o]+{0}="n-1",(43,-5)*[o]+{K}="n",
	\ar "2";"0"
	\ar "2";"1"
	\ar "dots";"2"
	\ar "n-3";"dots"
	\ar "n-2";"n-3"
	\ar "n-2";"n-1"
	\ar "n-2";"n"
	\end{xy}\\\\
	&\stackrel{\tau_{B_{n-1}}}{\longrightarrow}& \begin{xy}
	(-13,5)*[o]+{K}="0",(-13,-5)*[o]+{0}="1",(-3,0)*[o]+{K}="2",
	(10,0)*[o]+{\cdots}="dots",(22,0)*[o]+{K}="n-3",(34,0)*[o]+{0}="n-2",
	(43,5)*[o]+{0}="n-1",(43,-5)*[o]+{0}="n",
	\ar "2";"0"
	\ar "2";"1"
	\ar "dots";"2"
	\ar "n-3";"dots"
	\ar "n-2";"n-3"
	\ar "n-2";"n-1"
	\ar "n-2";"n"
	\end{xy}.\\
	\end{array}
	\]  
	By using induction, we can also check that
	\[
	\tau_B^{q}\left(
	\tau_{B_{n-1}}^4 S_0
	\right)=\left\{
	\begin{array}{lll}
	\begin{xy}
	(-13,5)*[o]+{K}="0",(-13,-5)*[o]+{0}="1",(-3,0)*[o]+{K}="2",(10,0)*[o]+{K}="3",(24,0)*[o]+{\cdots}="dots",(38,0)*[o]+{K}="k"
	,(50,0)*[o]+{0}="k+1",(64,0)*[o]+{\cdots}="dots2",(80,0)*[o]+{0}="n-2", (93,5)*[o]+{0}="n-1",(93,-5)*[o]+{0}="n",
	(50,-10)*[o]+{n-2-q},
	\ar "2";"0"
	\ar "2";"1"
	\ar "3";"2"
	\ar "dots";"3"
	\ar "k";"dots"
	\ar "k+1";"k"
	\ar "dots2";"k+1"
	\ar "n-2";"dots2"
	\ar"n-2";"n-1"
	\ar"n-2";"n"
	\end{xy} & (q:\text{even})\\\\
	\begin{xy}
	(-13,5)*[o]+{0}="0",(-13,-5)*[o]+{K}="1",(-3,0)*[o]+{K}="2",(10,0)*[o]+{K}="3",(24,0)*[o]+{\cdots}="dots",(38,0)*[o]+{K}="k"
	,(50,0)*[o]+{0}="k+1",(64,0)*[o]+{\cdots}="dots2",(80,0)*[o]+{0}="n-2", (93,5)*[o]+{0}="n-1",(93,-5)*[o]+{0}="n",
	(50,-10)*[o]+{n-2-q},
	\ar "2";"0"
	\ar "2";"1"
	\ar "3";"2"
	\ar "dots";"3"
	\ar "k";"dots"
	\ar "k+1";"k"
	\ar "dots2";"k+1"
	\ar "n-2";"dots2"
	\ar"n-2";"n-1"
	\ar"n-2";"n"
	\end{xy} & (q:\text{odd})\\
	\end{array}
	\right.
	\]
	for each $q\in \{1,2,\dots, n-4\}$. In particular, $\tau_B^{n-4}\left(\tau_{B_{n-1}}^4 S_0\right)$ is either $P^B_0$ or $P^B_1$.
	Thus, $\tau_{B_{n-1}}^4 S_0$ is preprojective and not in $\add N_{n-1}=\add \tau_B M_{n-1}\subset \add \mathcal{R}(B)$. 	
	Therefore, we have
	\[
	\tau_{B_{n-1}}^3 S_0=L'_{n-1}\in \add \tau_B M_{n-1}.
	\]
	This shows the assertion (2). Then the same argument implies the assertion (3).
	
	(4). By Proposition\;\ref{mgs:keyprop2} and Lemma\;\ref{MGS:D:+---}, we have
	\[
	\ell(Q)=\ell(\mu_1Q)=\cdots=\ell(\mu_{n-3}\cdots\mu_{1}Q)\le \ell(\mu_{n-2}\mu_{n-3}\cdots\mu_1 Q).
	\]
	Then it follows from Proposition\;\ref{mgs:keyprop} (2) that $\ell(\mu_{n-3}\cdots\mu_{1}Q)=\ell(\mu_{n-2}\mu_{n-3}\cdots\mu_1 Q)$.
	Therefore, we have the assertion (4).
	
	(5). By applying $\Hom_B(-,L'_{n-1})$ to the almost split sequence
\[0\to P^B_{n-1}\to P_{n-2}^B\to \tau_B^{-1}P_{n-1}\to 0,\]
we obtain an exact sequence
\[0\to \Hom_B(\tau_B^{-1}P^B_{n-1}, L'_{n-1})\to \Hom_B(P^B_{n-2}, L'_{n-1})\to \Hom_B(P^B_{n-1}, L'_{n-1})\to \Ext^1_B(\tau_B^{-1}P^B_{n-1}, L'_{n-1}).\]
Since $L'_{n-1}$ is not projective, we have $\Ext_B^1(\tau_B^{-1}P^B_{n-1}, L'_{n-1})=0$.
In particular, we obtain
\[
\begin{array}{rll}
\dim_K\Hom_B(P^B_{n-1},\tau_B L'_{n-1})&=&\dim_K\Hom_B(\tau_B^{-1}P^B_{n-1}, L'_{n-1})\\
&=&\dim_K\Hom_B(P^B_{n-2}, L'_{n-1})-\dim_K\Hom_B(P^B_{n-1}, L'_{n-1})\\
&=&1.\\
\end{array}
\]	
Then it follows from (2), (4), and Lemma\;\ref{MGS:D:+---} (5) that $(0,\mathbf{i},L_{n-1})$ satisfies $\mathsf{A}_1$.
Similarly, $(0,\mathbf{i}',L_{n})$ also satisfies Assumption\;\ref{assumption}.
\end{proof}
\subsubsection{A proof of Proposition\;\ref{D:statement:+---}}
Proposition\;\ref{D:statement:+---} (1) was proved in Lemma\;\ref{MGS:lemma+---2}. 
Thus it is sufficient to show the following equation. 
\[
\ell(Q)= \ell(\mu_0 Q).
\]
Suppose $\ell(Q)\not\le \ell(\mu_0 Q)$.
Since $\mu_{\mathbf{i}}Q\simeq Q(2,+-++)$ and $n-1$ is a unique source vertex of $\mu_{\mathbf{i}}Q$, it follows from Proposition\;\ref{MGS:D:+-++}
and Lemma\;\ref{MGS:lemma+---2} (4) that
 \[
 \ell(Q)=\ell(\mu_{n-2}\cdots\mu_1 Q)=\ell(\mu_{n-1}\mu_{\mathbf{i}}Q)\ge \ell(\mu_{\mathbf{i}}Q).
 \]
Then, by Lemma\;\ref{MGS:D:+---} (5), we can apply Proposition\;\ref{mgs:keyprop2} (3) and obtain
\[
\ell(Q)\le \ell(\mu_{n}\mu_{n-1}\cdots\mu_1 Q)=\ell(\mu_0 Q).
\]
This is a contradiction. Hence we have
\[
\ell(Q)\le \ell(\mu_0 Q).
\]
On the other hand, we have $\ell(Q)\ge \ell(\mu_0 Q)$ by Proposition\;\ref{mgs:keyprop} (2). This shows
\[
\ell(Q)=\ell(\mu_0 Q).
\]
Therefore, we have Proposition\;\ref{D:statement:+---}.
\subsection{A proof of Theorem\;\ref{D:main}}
	 We prove Theorem\;\ref{D:main}. If $Q$ and $Q'$ are quivers of type $\widetilde{\mathbf{D}}_n$, then there exists a
	 source mutation sequence
	 \[
	 Q\to \mu_{i_1}Q\to \cdots\to \mu_{i_m}\cdots\mu_{i_1}Q\simeq Q'. 
	 \]
	 Hence, it is sufficient to show that if $Q$ is a quiver of type $\widetilde{\mathbf{D}}_n$ and $i$ is a source vertex of $Q$, then
	 \[
	 \ell(Q)\le \ell(\mu_i Q).
	 \]
	 Assume that $Q\in \mathcal{Q}(k,\underline{\epsilon})$.
	 If $i\ne 0,1,n-1,n$, then it follows from Proposition\;\ref{mgs:keyprop}(1) that $\ell(Q)= \ell(\mu_i Q)$.
	 Therefore, by Proposition\;\ref{mgs:keyprop} and Lemma\;\ref{D:	 foundation}, we may assume $i=0$ and $Q$ is either
	 $Q(k,+,-,+,+)$, $Q(k,+,-,-,+)$, or $Q(k,+,-,+,-)$. 
	 Then the assertion follows from Proposition\;\ref{prop:D+--+}, Proposition\;\ref{MGS:D:+-++}, and Proposition\;\ref{D:statement:+---}.
 \section{A proof of Maintheorem (1): the case $\widetilde{\mathbf{E}}_6$}
The main result of this section is as follows.
\begin{theorem}\label{main:E6}
	Let $Q$ be a quiver of type $\widetilde{\mathbf{E}}_6$ and $i$ be a source of $Q$. Then we have $\ell(Q)\le \ell(\mu_i Q)$.
	In particular, $\ell (Q)$ does not depend on the choice of the orientation of $Q$.
\end{theorem}
Theorem\;\ref{main:E6} follows from the following two lemmas.
\begin{lemma}
	\label{mgs:lemmae6}
	Let $Q$ be a quiver of type $\widetilde{\mathbf{E}}_6$, $i$ be a source of $Q$ and $i'$ a sink of $Q$. 
	\begin{enumerate}[{\rm (1)}]
		\item If $\deg i\ne 1$, then $\ell(Q)= \ell(\mu_i Q)$.
		\item If $\deg i'\ne 1$, then $\ell(Q)= \ell(\mu_{i'} Q)$.
	\end{enumerate}	
\end{lemma}	
\begin{proof}
	This follows from Proposition\;\ref{mgs:keyprop}.
\end{proof}
\begin{lemma}
	\label{mgs:lemmae6_2}
	Let $Q$ be a quiver of type $\widetilde{\mathbf{E}}_6$, $i$ be a source vertex of $Q$.
	If the degree of $i$ is equal to $1$, then we have
	\[
	\ell(Q)\le \ell(\mu_i Q).
	\] 
\end{lemma} 
In the rest of this section, we prove Lemma\;\ref{mgs:lemmae6_2}.
\subsection{A proof of Lemma\;\ref{mgs:lemmae6_2}}
Here we give a proof of Lemma\;\ref{mgs:lemmae6_2}. Assume that $i$ is a source vertex of $Q$ with $\deg i=1$.
We rewrite $Q$ as follows.
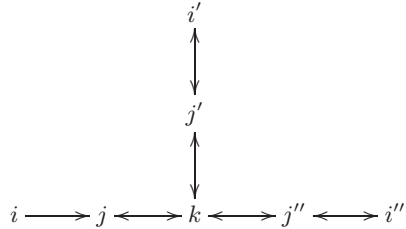
\[
\xymatrix{
& & i'\ar@{<->}[d] & & \\
& & j'\ar@{<->}[d] & & \\
i\ar[r]&j\ar@{<->}[r] & k &j''\ar@{<->}[l] &i''\ar@{<->}[l] \\
}
\]	
We set \[C_A:=\begin{pmatrix}
\dimvec P_{i} & \dimvec P_j & \dimvec P_{i'} & \dimvec P_{j'} &\dimvec P_{i''} & \dimvec P_{j''} &\dimvec P_k
\end{pmatrix}
\] the Cartan matrix of $A$.
\subsubsection{{\bf Case: $j\to k$} }
Here, we consider the the case $j\to k$.
\begin{lemma}
	\label{mgse6:jtok}
	Assume that there is an arrow from $j$ to $k$. Then we have  
	\[\ell(Q)\le \ell(\mu_i Q)\ (\text{resp. }\ell(Q)\le \ell(\mu_{i'}Q)).\]
\end{lemma}
If either $i'\to j'\to k$ or $i''\to j''\to k$ holds, then it follows from Proposition\;\ref{mgs:keyprop} (2) that
\[\ell(Q)\le \ell(\mu_i Q).\]
Thus we may assume $Q$ is one of the following quivers.
  
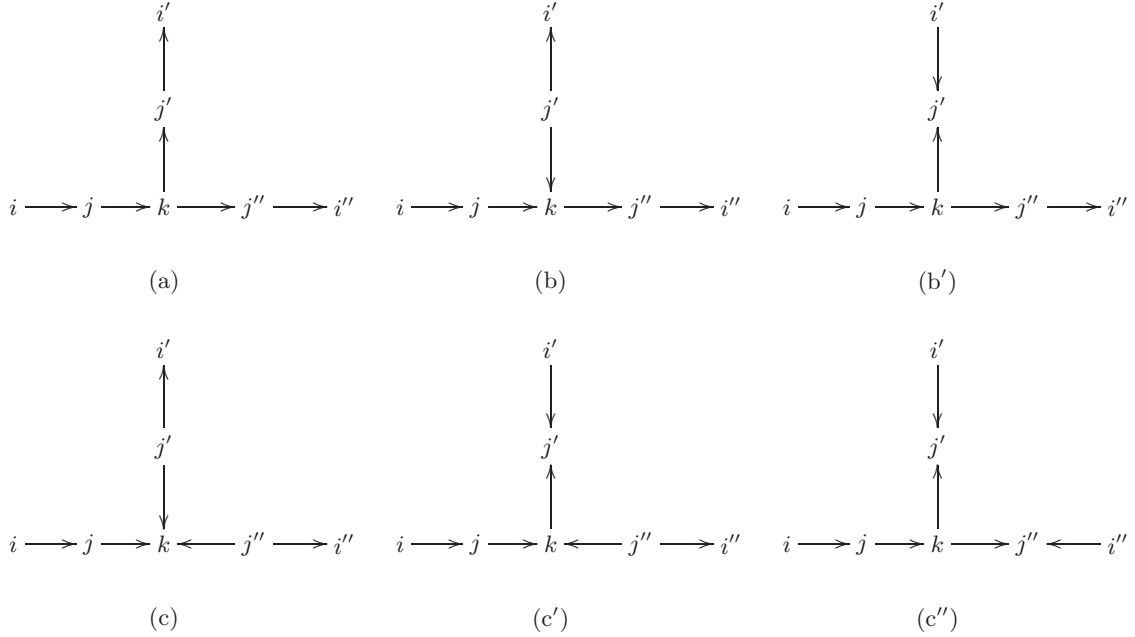
\[
  \begin{xy}
  (20,-10)*[o]+{{\rm(a)}}="a", (0,0)*[o]+{i}="i",(10,0)*[o]+{j}="j",
  (20,13)*[o]+{j'}="j'",(20,26)*[o]+{i'}="i'",
  (32,0)*[o]+{j''}="j''",(44,0)*[o]+{i''}="i''",
  (20,0)*[o]+{k}="k",
  \ar @{<-} "i''";"j''"
  \ar @{<-} "i'";"j'"
  \ar @{<-} "j''";"k"
  \ar @{<-} "j'";"k"
  \ar  "j";"k"
  \ar "i";"j"
  \end{xy}\hspace{10pt}
  \begin{xy}
 (20,-10)*[o]+{{\rm(b)}}="a", (0,0)*[o]+{i}="i",(10,0)*[o]+{j}="j",
 (20,13)*[o]+{j'}="j'",(20,26)*[o]+{i'}="i'",
 (32,0)*[o]+{j''}="j''",(44,0)*[o]+{i''}="i''",
 (20,0)*[o]+{k}="k",
 \ar @{<-} "i''";"j''"
 \ar @{<-} "i'";"j'"
 \ar @{<-} "j''";"k"
 \ar @{->} "j'";"k"
 \ar  "j";"k"
 \ar "i";"j"
 \end{xy}\hspace{10pt}
 \begin{xy}
 (20,-10)*[o]+{{\rm(b')}}="a", (0,0)*[o]+{i}="i",(10,0)*[o]+{j}="j",
 (20,13)*[o]+{j'}="j'",(20,26)*[o]+{i'}="i'",
 (32,0)*[o]+{j''}="j''",(44,0)*[o]+{i''}="i''",
 (20,0)*[o]+{k}="k",
 \ar @{<-} "i''";"j''"
 \ar @{->} "i'";"j'"
 \ar @{<-} "j''";"k"
 \ar @{<-} "j'";"k"
 \ar "j";"k"
 \ar "i";"j"
 \end{xy}
 \]
  \[
\begin{xy}
(20,-10)*[o]+{{\rm(c)}}="a", (0,0)*[o]+{i}="i",(10,0)*[o]+{j}="j",
(20,13)*[o]+{j'}="j'",(20,26)*[o]+{i'}="i'",
(32,0)*[o]+{j''}="j''",(44,0)*[o]+{i''}="i''",
(20,0)*[o]+{k}="k",
\ar @{<-} "i''";"j''"
\ar @{<-} "i'";"j'"
\ar @{->} "j''";"k"
\ar @{->} "j'";"k"
\ar "j";"k"
\ar "i";"j"
\end{xy}\hspace{10pt}
\begin{xy}
(20,-10)*[o]+{{\rm(c')}}="a", (0,0)*[o]+{i}="i",(10,0)*[o]+{j}="j",
(20,13)*[o]+{j'}="j'",(20,26)*[o]+{i'}="i'",
(32,0)*[o]+{j''}="j''",(44,0)*[o]+{i''}="i''",
(20,0)*[o]+{k}="k",
\ar @{<-} "i''";"j''"
\ar @{->} "i'";"j'"
\ar @{->} "j''";"k"
\ar @{<-} "j'";"k"
\ar  "j";"k"
\ar "i";"j"
\end{xy}\hspace{10pt}
\begin{xy}
(20,-10)*[o]+{{\rm(c'')}}="a", (0,0)*[o]+{i}="i",(10,0)*[o]+{j}="j",
(20,13)*[o]+{j'}="j'",(20,26)*[o]+{i'}="i'",
(32,0)*[o]+{j''}="j''",(44,0)*[o]+{i''}="i''",
(20,0)*[o]+{k}="k",
\ar @{->} "i''";"j''"
\ar @{->} "i'";"j'"
\ar @{<-} "j''";"k"
\ar @{<-} "j'";"k"
\ar "j";"k"
\ar "i";"j"
\end{xy}
\]  
\begin{lemma}
	\label{E6:jk:b'}
	Assume that $Q$ is a quiver {\rm (b')}. Then we have $\ell(Q)=\ell(\mu_i Q)=\ell(\mu_{i''}Q)$.
\end{lemma}
\begin{proof}
We can check this assertion by a computational approach.
\end{proof}
\begin{lemma}
	\label{E6:jk:b}
	Assume that $Q$ is a quiver {\rm (b)}. Then we have $\ell(Q)= \ell(\mu_i Q)$.
\end{lemma}
\begin{proof}
	Note that $\mu_{j'}Q$ is the quiver (b'). Therefore, 
by Proposition\;\ref{mgs:keyprop}(1) and Lemma\;\ref{E6:jk:b'}, we obtain
	\[\ell(Q)= \ell(\mu_{j'}Q)=\ell( \mu_i\mu_{j'}Q)=\ell(\mu_{j'}\mu_i\mu_{j'}Q).\]
Then the assertion follows from $\mu_{j'}\mu_i\mu_{j'}Q=\mu_i Q$.
\end{proof}
\begin{lemma}
	\label{E6:jk:a}
Assume that $Q$ is a quiver {\rm (a)}.	
Then we have $\ell(Q)\le \ell(\mu_i Q)$.
\end{lemma}
\begin{proof}
Note that $\mu_{i'}Q$ is the quiver (b'). Thus,  Proposition\;\ref{mgs:keyprop} and Lemma\;\ref{E6:jk:b'} imply
\[\ell(Q)\le \ell(\mu_{i'}Q)= \ell(\mu_{i''}\mu_{i'}Q)= \ell(\mu_{j'}\mu_{i''}\mu_{i'}Q)
=\ell(\mu_{j''}\mu_{j'}\mu_{i''}\mu_{i'}Q)=\ell(\mu_k\mu_{j''}\mu_{j'}\mu_{i''}\mu_{i'}Q)
= \ell(\mu_j\mu_k\mu_{j''}\mu_{j'}\mu_{i''}\mu_{i'}Q).\]
Then the assertion follows from $\mu_j\mu_k\mu_{j''}\mu_{j'}\mu_{i''}\mu_{i'}Q=\mu_i Q$.	
 \end{proof}
\begin{lemma}
	\label{E6:jk:c}
	Assume that $Q$ is a quiver {\rm (c)}. Then we have $\ell(Q)=\ell(\mu_i Q)$.
	
\end{lemma}
\begin{proof}
We can check this assertion by a computational approach.
 \end{proof}

\begin{lemma}
	\label{E6:jk:c'c''}
	Assume that $Q$ is either a quiver {\rm (c')} or a quiver {\rm (c'')}.
	Then we have
	$\ell Q\le \ell(\mu_{i}Q)$. 
\end{lemma}
\begin{proof}
If $Q$ is a quiver (c'), then $\mu_{j'} Q$ is a quiver (c).
Hence, Proposition\;\ref{mgs:keyprop}(1)  and Lemma\;\ref{E6:jk:c} imply
\[\ell(Q)= \ell(\mu_{j'} Q)=\ell( \mu_i \mu_{j'}Q)= \ell(\mu_{j'}\mu_i \mu_{j'}Q)=\ell(\mu_i Q).\]

If $Q$ is a quiver (c''), then $\mu_{j''} Q$ is a quiver (c').
Hence, Proposition\;\ref{mgs:keyprop}(1) and Lemma\;\ref{mgs:lemmae6} imply
\[\ell(Q)= \ell(\mu_{j''} Q)=\ell( \mu_i \mu_{j''}Q )=\ell(\mu_{j''}\mu_i \mu_{j''}Q)=\ell(\mu_i Q).\]
Thus, we obtain the assertion.
\end{proof}

\subsubsection{{\bf case: $j\leftarrow k$}}
Next we show the following proposition. 
\begin{proposition}
	\label{E6:kj}
Assume that there is an arrow from $k$ to $j$, then we have $\ell(Q)\le \ell(\mu_i Q).$	
\end{proposition}
If either $j'$ or $j''$ is sink, then it follows from Proposition\;\ref{mgs:keyprop} (2) that
\[\ell(Q)\le  \ell(\mu_i Q).\]
If $j'$ (resp. $j''$) is source, then Proposition\;\ref{mgs:keyprop} implies
\[\ell(Q) = \ell(\mu_{j'}Q) \le \ell(\mu_i \mu_{j'}Q)=\ell(\mu_{j'}\mu_i \mu_{j'}Q)=\ell(\mu_i Q)\ 
(\text{resp. } \ell(Q) = \ell(\mu_{j''}Q) \le \ell(\mu_i \mu_{j''}Q)= \ell(\mu_{j''}\mu_i \mu_{j''}Q)=\ell(\mu_i Q).\]
If $k$ is source, then Proposition\;\ref{mgs:keyprop}(1) and Proposition\;\ref{mgse6:jtok} imply
\[\ell(Q)= \ell(\mu_k Q)\le \ell(\mu_i \mu_k Q)= \ell(\mu_k \mu_i \mu_k Q)=\ell(\mu_i Q). \]  
Thus we may assume $Q$ is one of the following quivers.
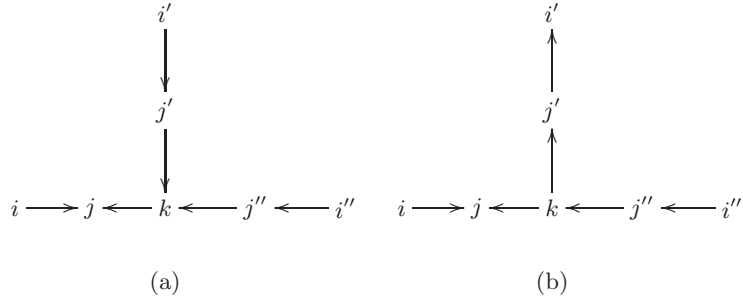
\[
\begin{xy}
(20,-10)*[o]+{{\rm(a)}}="a", (0,0)*[o]+{i}="i",(10,0)*[o]+{j}="j",
(20,13)*[o]+{j'}="j'",(20,26)*[o]+{i'}="i'",
(32,0)*[o]+{j''}="j''",(44,0)*[o]+{i''}="i''",
(20,0)*[o]+{k}="k",
\ar @{->} "i''";"j''"
\ar @{->} "i'";"j'"
\ar @{->} "j''";"k"
\ar @{->} "j'";"k"
\ar  "k";"j"
\ar "i";"j"
\end{xy}\hspace{10pt}
\begin{xy}
(20,-10)*[o]+{{\rm(b)}}="a", (0,0)*[o]+{i}="i",(10,0)*[o]+{j}="j",
(20,13)*[o]+{j'}="j'",(20,26)*[o]+{i'}="i'",
(32,0)*[o]+{j''}="j''",(44,0)*[o]+{i''}="i''",
(20,0)*[o]+{k}="k",
\ar @{->} "i''";"j''"
\ar @{<-} "i'";"j'"
\ar @{->} "j''";"k"
\ar @{<-} "j'";"k"
\ar  "k";"j"
\ar "i";"j"
\end{xy}
\]

Then one can check the assertion by a computational approach.

\section{A proof of Maintheorem (1): the case $\widetilde{\mathbf{E}}_7$}
In this section, we prove the following statement.
\begin{theorem}
	\label{mgs:thm:e7}
	Let $Q$ be a quiver of type $\widetilde{\mathbf{E}}_7$ and $i$ a source vertex of $Q$. Then
	we have
	\[
	\ell(Q)\le \ell(\mu_i Q).
	\]
	In particular, $\ell(Q)$ does not depend on the orientation. 
\end{theorem}
Throughout this section, we assume that the underlying graph of $Q$ has the following form. 
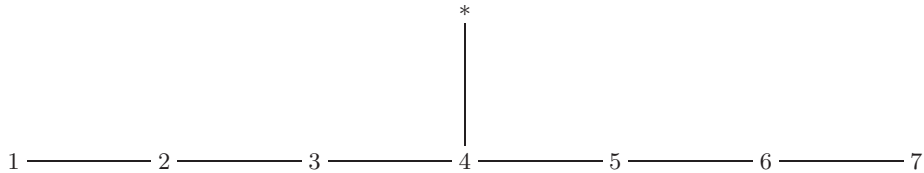
\[\begin{xy}
(0,0)*[o]+{1}="1",	(20,0)*[o]+{2}="2",(40,0)*[o]+{3}="3",(60,0)*[o]+{4}="4",(60,20)*[o]+{\ast}="a",(80,0)*[o]+{5}="5",
(100,0)*[o]+{6}="6",	(120,0)*[o]+{7}="7",
\ar@{-} "1";"2" 
\ar@{-} "2";"3" 
\ar@{-} "3";"4" 
\ar@{-} "4";"a" 
\ar@{-} "4";"5" 
\ar@{-} "5";"6" 
\ar@{-} "6";"7" 
\end{xy}\]

Assume that $i=\ast$.
Then it follows from Lemma\;\ref{findXi }(2) and Proposition\;\ref{mgs:keyprop}(1) that 
\begin{center}
	$\ell(\mu_\ast Q)\le \ell(Q)$ and 
	$\ell((\mu_\ast Q)^{{\rm op}})\le \ell(\mu_\ast(\mu_\ast Q)^{{\rm op}})=\ell(Q^{{\rm op}})=\ell(Q)$.
\end{center}
In particular, we have the following lemma.
\begin{lemma}
	\label{ast}
	We have $\ell(Q)=\ell(\mu_\ast Q)$.
\end{lemma}
  
Then by Proposition\;\ref{mgs:keyprop}(1) and Lemma\;\ref{ast}, we may assume $i=1$ and Theorem\;\ref{mgs:thm:e7} follows from the following two lemmas.
\begin{lemma}
	\label{key:e7}
	Assume that $1$ is source and $7$ is sink. Then we have
	$\ell(Q)=  \ell(\mu_1 Q)$.	
\end{lemma}
\begin{lemma}
\label{key:e7:sourcesource}
Assume that both $1$ and $7$ are source. Then we have
$\ell(Q)= \ell(\mu_1 Q).$
\end{lemma}

\subsection{A proof of Lemma\;\ref{key:e7}}
\label{sourcesink}
In this subsection, we show Lemma\;\ref{key:e7}. 
Let $\mathcal{A}$ be the set of quivers having the following form.
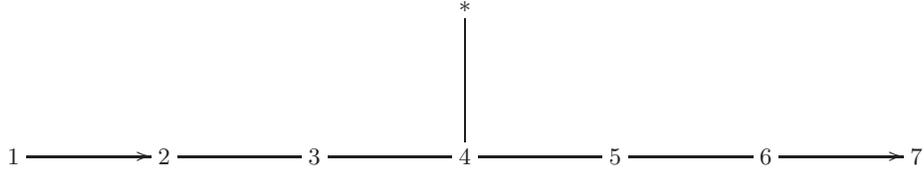
\[\begin{xy}
(0,0)*[o]+{1}="1",	(20,0)*[o]+{2}="2",(40,0)*[o]+{3}="3",(60,0)*[o]+{4}="4",(60,20)*[o]+{\ast}="a",(80,0)*[o]+{5}="5",
(100,0)*[o]+{6}="6",	(120,0)*[o]+{7}="7",
\ar "1";"2" 
\ar@{-} "2";"3" 
\ar@{-} "3";"4" 
\ar@{-} "4";"a" 
\ar@{-} "4";"5" 
\ar@{-} "5";"6" 
\ar "6";"7" 
\end{xy}\]
To prove Lemma\;\ref{key:e7}, we divide $\mathcal{A}$ into the following five classes.
 \begin{itemize}
 	\item[(Class\;1)] 
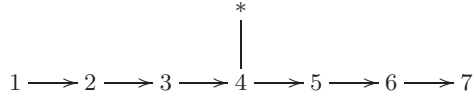
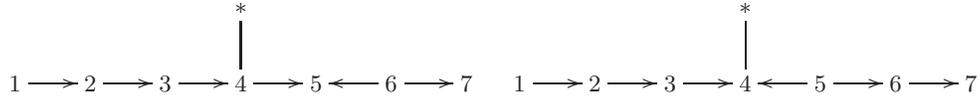
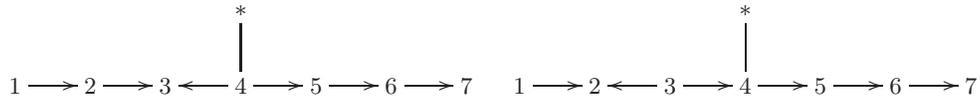
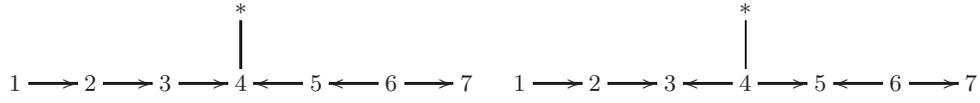
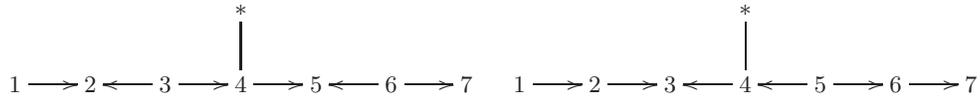
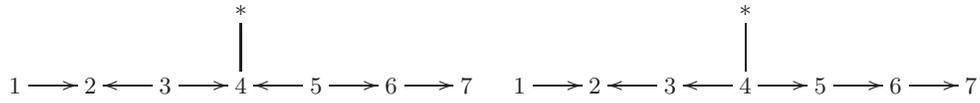
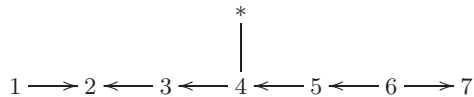
\[\begin{xy}
 	(0,0)*[o]+{1}="1",	(10,0)*[o]+{2}="2",(20,0)*[o]+{3}="3",(30,0)*[o]+{4}="4",(30,10)*[o]+{\ast}="a",(40,0)*[o]+{5}="5",
 	(50,0)*[o]+{6}="6",	(60,0)*[o]+{7}="7",
 	\ar "1";"2" 
 	\ar "2";"3" 
 	\ar "3";"4" 
 	\ar@{-} "4";"a" 
 	\ar "4";"5" 
 	\ar "5";"6" 
 	\ar "6";"7" 
 	\end{xy}\]
 	\item[(Class\;2)] \[\begin{xy}
 	(0,0)*[o]+{1}="1",	(10,0)*[o]+{2}="2",(20,0)*[o]+{3}="3",(30,0)*[o]+{4}="4",(30,10)*[o]+{\ast}="a",(40,0)*[o]+{5}="5",
 	(50,0)*[o]+{6}="6",	(60,0)*[o]+{7}="7",
 	\ar "1";"2" 
 	\ar "2";"3" 
 	\ar "3";"4" 
 	\ar@{-} "4";"a" 
 	\ar "4";"5" 
 	\ar "6";"5" 
 	\ar "6";"7" 
 	\end{xy}\hspace{10pt}
 	\begin{xy}
 	(0,0)*[o]+{1}="1",	(10,0)*[o]+{2}="2",(20,0)*[o]+{3}="3",(30,0)*[o]+{4}="4",(30,10)*[o]+{\ast}="a",(40,0)*[o]+{5}="5",
 	(50,0)*[o]+{6}="6",	(60,0)*[o]+{7}="7",
 	\ar "1";"2" 
 	\ar "2";"3" 
 	\ar "3";"4" 
 	\ar@{-} "4";"a" 
 	\ar "5";"4" 
 	\ar "5";"6" 
 	\ar "6";"7" 
 	\end{xy}
 	\]
 	\[\begin{xy}
 	(0,0)*[o]+{1}="1",	(10,0)*[o]+{2}="2",(20,0)*[o]+{3}="3",(30,0)*[o]+{4}="4",(30,10)*[o]+{\ast}="a",(40,0)*[o]+{5}="5",
 	(50,0)*[o]+{6}="6",	(60,0)*[o]+{7}="7",
 	\ar "1";"2" 
 	\ar "2";"3" 
 	\ar "4";"3" 
 	\ar@{-} "4";"a" 
 	\ar "4";"5" 
 	\ar "5";"6" 
 	\ar "6";"7" 
 	\end{xy}\hspace{10pt}
 	\begin{xy}
 	(0,0)*[o]+{1}="1",	(10,0)*[o]+{2}="2",(20,0)*[o]+{3}="3",(30,0)*[o]+{4}="4",(30,10)*[o]+{\ast}="a",(40,0)*[o]+{5}="5",
 	(50,0)*[o]+{6}="6",	(60,0)*[o]+{7}="7",
 	\ar "1";"2" 
 	\ar "3";"2" 
 	\ar "3";"4" 
 	\ar@{-} "4";"a" 
 	\ar "4";"5" 
 	\ar "5";"6" 
 	\ar "6";"7" 
 	\end{xy}
 	\]
 	 	 	 	\item[(Class\;3)] \[\begin{xy}
 	(0,0)*[o]+{1}="1",	(10,0)*[o]+{2}="2",(20,0)*[o]+{3}="3",(30,0)*[o]+{4}="4",(30,10)*[o]+{\ast}="a",(40,0)*[o]+{5}="5",
 	(50,0)*[o]+{6}="6",	(60,0)*[o]+{7}="7",
 	\ar "1";"2" 
 	\ar "2";"3" 
 	\ar "3";"4" 
 	\ar@{-} "4";"a" 
 	\ar "5";"4" 
 	\ar "6";"5" 
 	\ar "6";"7" 
 	\end{xy}\hspace{10pt}
 	\begin{xy}
 	(0,0)*[o]+{1}="1",	(10,0)*[o]+{2}="2",(20,0)*[o]+{3}="3",(30,0)*[o]+{4}="4",(30,10)*[o]+{\ast}="a",(40,0)*[o]+{5}="5",
 	(50,0)*[o]+{6}="6",	(60,0)*[o]+{7}="7",
 	\ar "1";"2" 
 	\ar "2";"3" 
 	\ar "4";"3" 
 	\ar@{-} "4";"a" 
 	\ar "4";"5" 
 	\ar "6";"5" 
 	\ar "6";"7" 
 	\end{xy}
 	\]
 	\[\begin{xy}
 	(0,0)*[o]+{1}="1",	(10,0)*[o]+{2}="2",(20,0)*[o]+{3}="3",(30,0)*[o]+{4}="4",(30,10)*[o]+{\ast}="a",(40,0)*[o]+{5}="5",
 	(50,0)*[o]+{6}="6",	(60,0)*[o]+{7}="7",
 	\ar "1";"2" 
 	\ar "3";"2" 
 	\ar "3";"4" 
 	\ar@{-} "4";"a" 
 	\ar "4";"5" 
 	\ar "6";"5" 
 	\ar "6";"7" 
 	\end{xy}\hspace{10pt}
 	\begin{xy}
 	(0,0)*[o]+{1}="1",	(10,0)*[o]+{2}="2",(20,0)*[o]+{3}="3",(30,0)*[o]+{4}="4",(30,10)*[o]+{\ast}="a",(40,0)*[o]+{5}="5",
 	(50,0)*[o]+{6}="6",	(60,0)*[o]+{7}="7",
 	\ar "1";"2" 
 	\ar "2";"3" 
 	\ar "4";"3" 
 	\ar@{-} "4";"a" 
 	\ar "5";"4" 
 	\ar "5";"6" 
 	\ar "6";"7" 
 	\end{xy}\]
 \[	\begin{xy}
 		(0,0)*[o]+{1}="1",	(10,0)*[o]+{2}="2",(20,0)*[o]+{3}="3",(30,0)*[o]+{4}="4",(30,10)*[o]+{\ast}="a",(40,0)*[o]+{5}="5",
 		(50,0)*[o]+{6}="6",	(60,0)*[o]+{7}="7",
 		\ar "1";"2" 
 		\ar "3";"2" 
 		\ar "3";"4" 
 		\ar@{-} "4";"a" 
 		\ar "5";"4" 
 		\ar "5";"6" 
 		\ar "6";"7" 
 	\end{xy}\hspace{10pt}\begin{xy}
 	(0,0)*[o]+{1}="1",	(10,0)*[o]+{2}="2",(20,0)*[o]+{3}="3",(30,0)*[o]+{4}="4",(30,10)*[o]+{\ast}="a",(40,0)*[o]+{5}="5",
 	(50,0)*[o]+{6}="6",	(60,0)*[o]+{7}="7",
 	\ar "1";"2" 
 	\ar "3";"2" 
 	\ar "4";"3" 
 	\ar@{-} "4";"a" 
 	\ar "4";"5" 
 	\ar "5";"6" 
 	\ar "6";"7" 
 	\end{xy}\]
 	\item[(Class\;4)] \[\begin{xy}
 	 	(0,0)*[o]+{1}="1",	(10,0)*[o]+{2}="2",(20,0)*[o]+{3}="3",(30,0)*[o]+{4}="4",(30,10)*[o]+{\ast}="a",(40,0)*[o]+{5}="5",
 	 	(50,0)*[o]+{6}="6",	(60,0)*[o]+{7}="7",
 	 	\ar "1";"2" 
 	 	\ar "2";"3" 
 	 	\ar "4";"3" 
 	 	\ar@{-} "4";"a" 
 	 	\ar "5";"4" 
 	 	\ar "6";"5" 
 	 	\ar "6";"7" 
 	 	\end{xy}\hspace{10pt}
 	 	\begin{xy}
 	 	(0,0)*[o]+{1}="1",	(10,0)*[o]+{2}="2",(20,0)*[o]+{3}="3",(30,0)*[o]+{4}="4",(30,10)*[o]+{\ast}="a",(40,0)*[o]+{5}="5",
 	 	(50,0)*[o]+{6}="6",	(60,0)*[o]+{7}="7",
 	 	\ar "1";"2" 
 	 	\ar "3";"2" 
 	 	\ar "3";"4" 
 	 	\ar@{-} "4";"a" 
 	 	\ar "5";"4" 
 	 	\ar "6";"5" 
 	 	\ar "6";"7" 
 	 	\end{xy}
 	 	\]
 	 	\[\begin{xy}
 	 	(0,0)*[o]+{1}="1",	(10,0)*[o]+{2}="2",(20,0)*[o]+{3}="3",(30,0)*[o]+{4}="4",(30,10)*[o]+{\ast}="a",(40,0)*[o]+{5}="5",
 	 	(50,0)*[o]+{6}="6",	(60,0)*[o]+{7}="7",
 	 	\ar "1";"2" 
 	 	\ar "3";"2" 
 	 	\ar "4";"3" 
 	 	\ar@{-} "4";"a" 
 	 	\ar "4";"5" 
 	 	\ar "6";"5" 
 	 	\ar "6";"7" 
 	 	\end{xy}\hspace{10pt}
 	 	\begin{xy}
 	 	(0,0)*[o]+{1}="1",	(10,0)*[o]+{2}="2",(20,0)*[o]+{3}="3",(30,0)*[o]+{4}="4",(30,10)*[o]+{\ast}="a",(40,0)*[o]+{5}="5",
 	 	(50,0)*[o]+{6}="6",	(60,0)*[o]+{7}="7",
 	 	\ar "1";"2" 
 	 	\ar "3";"2" 
 	 	\ar "4";"3" 
 	 	\ar@{-} "4";"a" 
 	 	\ar "5";"4" 
 	 	\ar "5";"6" 
 	 	\ar "6";"7" 
 	 	\end{xy}\]
 	 	 	\item[(Class\;5)] \[\begin{xy}
 	 	(0,0)*[o]+{1}="1",	(10,0)*[o]+{2}="2",(20,0)*[o]+{3}="3",(30,0)*[o]+{4}="4",(30,10)*[o]+{\ast}="a",(40,0)*[o]+{5}="5",
 	 	(50,0)*[o]+{6}="6",	(60,0)*[o]+{7}="7",
 	 	\ar "1";"2" 
 	 	\ar "3";"2" 
 	 	\ar "4";"3" 
 	 	\ar@{-} "4";"a" 
 	 	\ar "5";"4" 
 	 	\ar "6";"5" 
 	 	\ar "6";"7" 
 	 	\end{xy}\]
 \end{itemize}
We note that if $Q$ and $Q'$ belong the same class, then there exists sink or source mutation sequence
\[Q\to \mu_{i_1}Q\to \cdots\to \mu_{i_k}\cdots\mu_{i_1}Q=Q'\]
with $i_1,\dots,i_k\in\{3,4,5,\ast \}$. Therefore it follows from Lemma\;\ref{mgs:keyprop}(1) and Lemma\;\ref{ast} that
\[\ell(Q)=\ell(Q'),\ \ell(\mu_1 Q)=\ell(\mu_1 Q').\]

Now put $Q(1), Q(2), Q(3), Q(4), Q(5)$ as follows.
\[Q(1)=\begin{xy}
(0,0)*[o]+{1}="1",	(10,0)*[o]+{2}="2",(20,0)*[o]+{3}="3",(30,0)*[o]+{4}="4",(30,10)*[o]+{\ast}="a",(40,0)*[o]+{5}="5",
(50,0)*[o]+{6}="6",	(60,0)*[o]+{7}="7",
\ar "1";"2" 
\ar "2";"3" 
\ar "3";"4" 
\ar "a";"4" 
\ar "4";"5" 
\ar "5";"6" 
\ar "6";"7" 
\end{xy},\  Q(2)=\begin{xy}
(0,0)*[o]+{1}="1",	(10,0)*[o]+{2}="2",(20,0)*[o]+{3}="3",(30,0)*[o]+{4}="4",(30,10)*[o]+{\ast}="a",(40,0)*[o]+{5}="5",
(50,0)*[o]+{6}="6",	(60,0)*[o]+{7}="7",
\ar "1";"2" 
\ar "2";"3" 
\ar "3";"4" 
\ar "a";"4" 
\ar "4";"5" 
\ar "6";"5" 
\ar "6";"7" 
\end{xy}\]
\[Q(3)=\begin{xy}
(0,0)*[o]+{1}="1",	(10,0)*[o]+{2}="2",(20,0)*[o]+{3}="3",(30,0)*[o]+{4}="4",(30,10)*[o]+{\ast}="a",(40,0)*[o]+{5}="5",
(50,0)*[o]+{6}="6",	(60,0)*[o]+{7}="7",
\ar "1";"2" 
\ar "2";"3" 
\ar "3";"4" 
\ar "a";"4" 
\ar "5";"4" 
\ar "6";"5" 
\ar "6";"7" 
\end{xy},\  Q(4)=\begin{xy}
(0,0)*[o]+{1}="1",	(10,0)*[o]+{2}="2",(20,0)*[o]+{3}="3",(30,0)*[o]+{4}="4",(30,10)*[o]+{\ast}="a",(40,0)*[o]+{5}="5",
(50,0)*[o]+{6}="6",	(60,0)*[o]+{7}="7",
\ar "1";"2" 
\ar "2";"3" 
\ar "4";"3" 
\ar "a";"4" 
\ar "5";"4" 
\ar "6";"5" 
\ar "6";"7" 
\end{xy}\]
\[ Q(5)=\begin{xy}
	(0,0)*[o]+{1}="1",	(10,0)*[o]+{2}="2",(20,0)*[o]+{3}="3",(30,0)*[o]+{4}="4",(30,10)*[o]+{\ast}="a",(40,0)*[o]+{5}="5",
	(50,0)*[o]+{6}="6",	(60,0)*[o]+{7}="7",
	\ar "1";"2" 
	\ar "3";"2" 
	\ar "4";"3" 
	\ar "a";"4" 
	\ar "5";"4" 
	\ar "6";"5" 
	\ar "6";"7" 
\end{xy}\]
Then the following lemma implies the Lemma\;\ref{key:e7}.
\begin{lemma}
	\label{lemma:proof of key:e7}
Let $Q=Q(z)$ $(z=1,2,3,4,5)$. Then we have $\ell(Q)=\ell(\mu_1 Q)$.
\end{lemma}
\begin{proof}
One can check the assertion by a computational approach.
\end{proof}

\subsection{A proof of Lemma\;\ref{key:e7:sourcesource}}
\label{sourcesource}
In this subsection, we prove Lemma\;\ref{key:e7:sourcesource}.
Assume that $1$ and $7$ are sources and put $Q'=\mu_1 Q$.
The we have
\[
\ell(Q')=\ell((Q')^{{\rm op}})=\ell(\mu_1(Q')^{{\rm op}})=\ell(Q^{\rm{op}})=\ell(Q).
\]
by Lemma\;\ref{key:e7}. 
This finishes the proof.
\section{A proof of Maintheorem (1): the case $\widetilde{\mathbf{E}}_8$}

In this section, we prove the following statement.
\begin{theorem}
	\label{mgs:thm:e8}
	Let $Q$ be a quiver of type $\widetilde{\mathbf{E}}_8$ and $i$ be a source vertex of $Q$. Then
	we have
	\[
	\ell(Q)\le \ell(\mu_i Q).
	\]
	In particular, $\ell(Q)$ does not depend on the orientation. 
\end{theorem}
Throughout this section, we assume that the underlying graph of $Q$ has the following form. 
\[
\xymatrix
{
	&   & 9\ar@{-}[d] & && & &\\ 
	1 &\ar@{-}[l] 2 &\ar@{-}[l]   3  &\ar@{-}[l]  4  &\ar@{-}[l]  5  &\ar@{-}[l] 6&\ar@{-}[l] 7&\ar@{-}[l] 8 \\ 
}
\]

Then the following lemma implies Theorem\;\ref{mgs:thm:e8}.
\begin{lemma}
	\label{mgs:lemmae8}
	Let $Q$ be a quiver of type $\widetilde{\mathbf{E}}_8$, $i$ be a source of $Q$. If $i\ne 8$, then $\mathcal{X}_i=\emptyset$.
	\end{lemma}	
In fact, if $i\ne 8$, then the assertion follows from Proposition\;\ref{mgs:keyprop}  and Lemma\;\ref{mgs:lemmae8}.
Thus we may assume $i=8$. In this case, there exists a source mutation sequence
\[
Q^{\rm{op}}\to \mu_{i_1 }Q^{\rm{op}} \to \cdots \to \mu_{i_8}\mu_{i_7}\mu_{i_6}\mu_{i_5}\mu_{i_4}\mu_{i_3}\mu_{i_2}\mu_{i_1} Q=\mu_8 Q^{\rm{op}}
\]
with $\{i_1,\dots,i_8\}=Q_0\setminus\{8\}$.
Then it also follows from Proposition\;\ref{mgs:keyprop} and Lemma\;\ref{mgs:lemmae8} that
\[
\ell(Q)=\ell(Q^{\rm{op}})\le \ell(\mu_{i_8}\mu_{i_7}\mu_{i_6}\mu_{i_5}\mu_{i_4}\mu_{i_3}\mu_{i_2}\mu_{i_1} Q^{\rm{op}})=\ell( \mu_8 Q^{\rm{op}})=\ell(\mu_8 Q).
\]
Therefore, it is sufficient to show Lemma\;\ref{mgs:lemmae8}.
\subsection{A proof of Lemma\;\ref{mgs:lemmae8}}
	By Lemma\;\ref{findXi }, it is sufficient to check that the following implication. 
	\begin{center}
		$1$ is a source vertex $\Rightarrow \mathcal{X}_1=\emptyset$.
	\end{center} 
Assume $i=1$ is source and $X\in \mathcal{X}_1$. Then $\ORA{C}_+=Q$ is given by the following.
\[
\xymatrix@=13pt
{
	&   & 9\ar@{<->}[d]^{\epsilon_9} &\\ 
	1 &\ar@{<-}[l] 2  &\ar@{<->}[l]_-{\epsilon_3}   3  &\ar@{<->}[l]_{\epsilon_4}  4 
	 &\ar@{<->}[l]_{\epsilon_5}  5  &\ar@{<->}[l]_{\epsilon_6} 6 
	 & \ar@{<->}[l]_{\epsilon_7} 7 
	 &\ar@{<->}[l]_{\epsilon_8} 8
}
\]
 Then we obtain
 \[
 \dim_K\Hom_A(\tau^{-1}P_1, X_1)=\left\{
 \begin{array}{cl}
 x_3-x_2 & (\epsilon_3=-)\\
x^{\epsilon_9}_9-(x_3-x_4) & (\epsilon_3=+,\epsilon_4=-)\\
  x^{\epsilon_9}_9-(x_4-x_5) & (\epsilon_3=\epsilon_4=+,\epsilon_5=-)\\
 x^{\epsilon_9}_9-(x_5-x_6) & (\epsilon_3=\epsilon_4=\epsilon_5=+,\epsilon_6=-)\\
 x^{\epsilon_9}_9-(x_6-x_7) & (\epsilon_3=\epsilon_4=\epsilon_5=\epsilon_6=+,\epsilon_7=-)\\
  x^{\epsilon_9}_9-(x_7-x_8) & (\epsilon_3=\epsilon_4=\epsilon_5=\epsilon_6=\epsilon_7=+,\epsilon_8=-)\\
 x^{\epsilon_9}_9-x_8 & (\epsilon_3=\epsilon_4=\epsilon_5=\epsilon_6=\epsilon_7=\epsilon_8=+)\\
 \end{array}
 \right.
 \]
 where, $x_k:=\dim_K\Hom_A(P_k, X)$ and 
 \[
\begin{array}{lllllll}
 x^+_9&:=&\dim_K\Hom_A(\tau^{-1}P_9,X_1)&=&x_3-x_9\\ 
 x^-_9&:=&\dim_K\Hom_A(P_9,X_i)&=&x_9.\\
 \end{array}
 \]
 Since $X\in \ind K\ORA{C}\subset\ind A$, it follows from the classification of the indecomposable modules of $K\ORA{C}$ that
 \[
 \dim_K\Hom_A(\tau^{-1}P_1, X_1)\le 1.
 \]
 This is a contradiction.  Therefore, we have the assertion.	
\section{Length of maximal green sequences for quivers of type $\widetilde{\mathbf{D}}$ and $\widetilde{\mathbf{E}}$}
In this section, we show Main theorem (2).
\begin{theorem}\label{length of mgs for tilde D}
Let $Q$ be a quiver of type $\widetilde{\mathbf{D}}$ or $\widetilde{\mathbf{E}}$. Then 
	$\ell(Q)$ is given by the following table.
	\begin{center}
		\begin{tabular}{|c||c|c|c|c|}
			\hline
			Type of $Q$	&  $\widetilde{\mathbf{D}}_n$& $\widetilde{\mathbf{E}}_6$  & $\widetilde{\mathbf{E}}_7$  & $\widetilde{\mathbf{E}}_8$\\
			\hline
			$\ell(Q)$	& $2n^2-2n-2$ & $78$ & $159$ & $390$ \\
			\hline
		\end{tabular} 
	\end{center}
\end{theorem}

For the case that $Q$ is of $\widetilde{\mathbf{D}}_4$ or type $\widetilde{\mathbf{E}}_{6,7,8}$, we can check the assertion 
by using the computer program (see Section\;\ref{computational approach}). 
Therefore, we assume $Q$ is of type $\widetilde{\mathbf{D}}_n$ with $n\ge 5$.
  
For a proof, we consider the following quiver $Q$ and its path algebra $A$.
\[\begin{xy}
(0,10)*[o]+{0}="0", 
(0,-10)*[o]+{1}="1",
(15,0)*[o]+{2}="2",  
(30,0)*[o]+{3}="3", 
(50,0)*[o]+{\cdots}="cdots",
(75,0)*[o]+{n-2}="n-2",
(93,10)*[o]+{n-1}="n-1", 
(90,-10)*[o]+{n}="n",
\ar "0";"2",
\ar "1";"2",
\ar "3";"2",
\ar "cdots";"3",
\ar "n-2";"cdots"
\ar "n-1";"n-2",
\ar "n";"n-2",
 \end{xy}
\]

\subsection{Nonsincere indecomposable preprojective/preinjective modules and indecomposable regular $\tau$-rigid modules}
In this subsection, we classify nonsincere indecomposable preprojective/preinjective modules  and indecomposable regular $\tau$-rigid module of $A$.
Note that if $M$ is a nonsincere indecomposable module of $A$, then there is $v\in \{0,1,n-1,n\}$ such that
\[M\in \ind A_v\subset \ind A\ (\text{recall that, for each $i\in Q_0$, we set $A_i:=A/(e_i)$}).\]
Note that $A_v$ is isomorphic to a path algebra of type $\mathbf{D}_n$. In particular, each indecomposable module
is $\tau$-rigid by Proposition\;\ref{basicfact}. Furthermore, each indecomposable module whose $\tau$-orbit contains a
non-sincere module is $\tau$-rigid by the Auslander--Reiten duality.
We also note that indecomposable $\tau$-rigid modules are determined by their $g$-vectors (see Theorem\;\ref{gvector}).

To the end of this subsection, we put $\bm{e}_i=(\delta_{ij})_{j\in Q_0}\in \Z^{Q_0}$, where $\delta$ is the Kronecker delta.
We also denote by $\odd$ (resp. $\even$) the set of odd (resp. even) integers.  
\begin{lemma}
\label{nonsincere_p}
\begin{enumerate}[{\rm (1)}]
	\item For each $t\in \{2,\dots, n-1\}$, $\tau^{-t+2}P_{\epsilon}$ $(\epsilon\in \{0,1\})$ has the following form.
	\[\left\{
	\begin{array}{cl}
		\begin{xy}
	(0,5)*[o]+{0}="0", 
	(0,-5)*[o]+{K}="1",
	(10,0)*[o]+{K}="2",  
	(25,0)*[o]+{\cdots}="cdots", 
	(40,0)*[o]+{K}="t",
	(40,-5)*[o]+{t}="t'",
	(50,0)*[o]+{0}="t+1",
	(65,0)*[o]+{\cdots}="cdots2", 
	(80,0)*[o]+{0}="n-2",
	(90,5)*[o]+{0}="n-1", 
	(90,-5)*[o]+{0}="n",
	\ar "0";"2",
	\ar "1";"2",
	\ar "cdots";"2",
	\ar "t";"cdots",
	\ar "t+1";"t"
	\ar "cdots2";"t+1",
	\ar "n-2";"cdots2",
	\ar "n-1";"n-2"
	\ar "n";"n-2"
	\end{xy}
	& (t\le n-2,\ \epsilon+t\in \odd)\\
	\begin{xy}
	(0,5)*[o]+{K}="0", 
	(0,-5)*[o]+{0}="1",
	(10,0)*[o]+{K}="2",  
	(25,0)*[o]+{\cdots}="cdots", 
	(40,0)*[o]+{K}="t",
	(40,-5)*[o]+{t}="t'",
	(50,0)*[o]+{0}="t+1",
	(65,0)*[o]+{\cdots}="cdots2", 
	(80,0)*[o]+{0}="n-2",
	(90,5)*[o]+{0}="n-1", 
	(90,-5)*[o]+{0}="n",
	\ar "0";"2",
	\ar "1";"2",
	\ar "cdots";"2",
	\ar "t";"cdots",
	\ar "t+1";"t"
	\ar "cdots2";"t+1",
	\ar "n-2";"cdots2",
	\ar "n-1";"n-2"
	\ar "n";"n-2"
	\end{xy}
	& (t\le n-2,\ \epsilon+t\in \even)\\
	\begin{xy}
	(0,5)*[o]+{0}="0", 
	(0,-5)*[o]+{K}="1",
	(10,0)*[o]+{K^2}="2",  
	(25,0)*[o]+{\cdots}="cdots", 
	(40,0)*[o]+{K^2}="n-2",
	(50,5)*[o]+{K}="n-1", 
	(50,-5)*[o]+{K}="n",
	\ar "0";"2",
	\ar "1";"2",
	\ar "cdots";"2",
	\ar "n-2";"cdots",
	\ar "n-1";"n-2"
	\ar "n";"n-2"
	\end{xy}
	& (t=n-1,\ \epsilon+n-1\in \odd)\\
	\begin{xy}
	(0,5)*[o]+{K}="0", 
	(0,-5)*[o]+{0}="1",
	(10,0)*[o]+{K^2}="2",  
	(25,0)*[o]+{\cdots}="cdots", 
	(40,0)*[o]+{K^2}="n-2",
	(50,5)*[o]+{K}="n-1", 
	(50,-5)*[o]+{K}="n",
	\ar "0";"2",
	\ar "1";"2",
	\ar "cdots";"2",
	\ar "n-2";"cdots",
	\ar "n-1";"n-2"
	\ar "n";"n-2"
	\end{xy}
	& (t=n-1,\ \epsilon+n-1\in \even)\\
		\end{array}\right.
	\]
\item For each $(s,t)$ satisfying $1\le s<t\le n-2$, $\tau^{-s+1}P_{t-s+1}$  has the following form.
\[\left\{
\begin{array}{cl}
\begin{xy}
(0,5)*[o]+{0}="0", 
(0,-5)*[o]+{0}="1",
(10,0)*[o]+{K}="2",  
(25,0)*[o]+{\cdots}="cdots", 
(40,0)*[o]+{K}="t",
(40,-5)*[o]+{t}="t'",
(50,0)*[o]+{0}="t+1",
(65,0)*[o]+{\cdots}="cdots2", 
(80,0)*[o]+{0}="n-2",
(90,5)*[o]+{0}="n-1", 
(90,-5)*[o]+{0}="n",
\ar "0";"2",
\ar "1";"2",
\ar "cdots";"2",
\ar "t";"cdots",
\ar "t+1";"t"
\ar "cdots2";"t+1",
\ar "n-2";"cdots2",
\ar "n-1";"n-2"
\ar "n";"n-2"
\end{xy}
& (\text{$s=1$})\\
\begin{xy}
(0,5)*[o]+{K}="0", 
(0,-5)*[o]+{K}="1",
(10,0)*[o]+{K^2}="2",  
(25,0)*[o]+{\cdots}="cdots", 
(40,0)*[o]+{K^2}="s",
(40,-5)*[o]+{s}="s'",
(50,0)*[o]+{K}="s+1",
(65,0)*[o]+{\cdots}="cdots2",
(80,0)*[o]+{K}="t",
(80,-5)*[o]+{t}="t'",
(90,0)*[o]+{0}="t+1",
(105,0)*[o]+{\cdots}="cdots3",
(120,0)*[o]+{0}="n-2",
(130,5)*[o]+{0}="n-1", 
(130,-5)*[o]+{0}="n",
\ar "0";"2",
\ar "1";"2",
\ar "cdots";"2",
\ar "s";"cdots",
\ar "s+1";"s"
\ar "cdots2";"s+1",
\ar "t";"cdots2",
\ar "t+1";"t"
\ar "cdots3";"t+1",
\ar "n-2";"cdots3",
\ar "n-1";"n-2"
\ar "n";"n-2"
\end{xy}
& (\text{$s\ge 2$})\\
\end{array}\right.
\]
\item 	For each $t\in \{1,\dots, n-2\}$, $\tau^{-t+1}P_{n-\epsilon}$ $(\epsilon\in \{0,1\})$ has the following form.
\[\left\{
\begin{array}{cl}
\begin{xy}
(0,5)*[o]+{0}="0", 
(0,-5)*[o]+{0}="1",
(10,0)*[o]+{K}="2",  
(25,0)*[o]+{\cdots}="cdots", 
(40,0)*[o]+{K}="n-2",
(50,5)*[o]+{K}="n-1", 
(50,-5)*[o]+{0}="n",
\ar "0";"2",
\ar "1";"2",
\ar "cdots";"2",
\ar "n-2";"cdots",
\ar "n-1";"n-2"
\ar "n";"n-2"
\end{xy}
& (t=1,\ \epsilon=1)\\
\begin{xy}
(0,5)*[o]+{0}="0", 
(0,-5)*[o]+{0}="1",
(10,0)*[o]+{K}="2",  
(25,0)*[o]+{\cdots}="cdots", 
(40,0)*[o]+{K}="n-2",
(50,5)*[o]+{0}="n-1", 
(50,-5)*[o]+{K}="n",
\ar "0";"2",
\ar "1";"2",
\ar "cdots";"2",
\ar "n-2";"cdots",
\ar "n-1";"n-2"
\ar "n";"n-2"
\end{xy}
& (t=1,\ \epsilon =0)\\
\begin{xy}
(0,5)*[o]+{K}="0", 
(0,-5)*[o]+{K}="1",
(10,0)*[o]+{K^2}="2",  
(25,0)*[o]+{\cdots}="cdots", 
(40,0)*[o]+{K^2}="t",
(40,-5)*[o]+{t}="t'",
(50,0)*[o]+{K}="t+1",
(65,0)*[o]+{\cdots}="cdots2", 
(80,0)*[o]+{K}="n-2",
(90,5)*[o]+{K}="n-1", 
(90,-5)*[o]+{0}="n",
\ar "0";"2",
\ar "1";"2",
\ar "cdots";"2",
\ar "t";"cdots",
\ar "t+1";"t"
\ar "cdots2";"t+1",
\ar "n-2";"cdots2",
\ar "n-1";"n-2"
\ar "n";"n-2"
\end{xy}
& (t\ge 2,\ t+\epsilon\in \even)\\
\begin{xy}
(0,5)*[o]+{K}="0", 
(0,-5)*[o]+{K}="1",
(10,0)*[o]+{K^2}="2",  
(25,0)*[o]+{\cdots}="cdots", 
(40,0)*[o]+{K^2}="t",
(40,-5)*[o]+{t}="t'",
(50,0)*[o]+{K}="t+1",
(65,0)*[o]+{\cdots}="cdots2", 
(80,0)*[o]+{K}="n-2",
(90,5)*[o]+{0}="n-1", 
(90,-5)*[o]+{K}="n",
\ar "0";"2",
\ar "1";"2",
\ar "cdots";"2",
\ar "t";"cdots",
\ar "t+1";"t"
\ar "cdots2";"t+1",
\ar "n-2";"cdots2",
\ar "n-1";"n-2"
\ar "n";"n-2"
\end{xy}
& (t\ge 2,\ \epsilon+t\in \odd)\\
\end{array}\right.
\]
\end{enumerate}	
\end{lemma}
\begin{proof}
We denote by $\alpha$, $\beta$, $\gamma$ the $K$-linear maps from $K$ to $K^2$ given by $(1\ 0)$, $(0\ 1)$, and $(1\ 1)$, respectively.  	
	
(1). For each $t\in \{2,\dots, n-1\}$ and $\epsilon\in \{0,1\}$, we define $M_\epsilon(t)$ as follows.
\[\left\{
\begin{array}{cl}
\begin{xy}
(0,5)*[o]+{K}="0", 
(0,-5)*[o]+{0}="1",
(10,0)*[o]+{K}="2",  
(25,0)*[o]+{\cdots}="cdots", 
(40,0)*[o]+{K}="t",
(40,-5)*[o]+{t}="t'",
(50,0)*[o]+{0}="t+1",
(65,0)*[o]+{\cdots}="cdots2", 
(80,0)*[o]+{0}="n-2",
(90,5)*[o]+{0}="n-1", 
(90,-5)*[o]+{0}="n",
\ar "0";"2"^{\id}
\ar "1";"2"
\ar "cdots";"2"_{\id}
\ar "t";"cdots"_{\id}
\ar "t+1";"t"
\ar "cdots2";"t+1"
\ar "n-2";"cdots2"
\ar "n-1";"n-2"
\ar "n";"n-2"
\end{xy}
& (t\le n-2,\ \epsilon=0) \\
\begin{xy}
(0,5)*[o]+{0}="0", 
(0,-5)*[o]+{K}="1",
(10,0)*[o]+{K}="2",  
(25,0)*[o]+{\cdots}="cdots", 
(40,0)*[o]+{K}="t",
(40,-5)*[o]+{t}="t'",
(50,0)*[o]+{0}="t+1",
(65,0)*[o]+{\cdots}="cdots2", 
(80,0)*[o]+{0}="n-2",
(90,5)*[o]+{0}="n-1", 
(90,-5)*[o]+{0}="n",
\ar "0";"2"
\ar "1";"2"_{\id}
\ar "cdots";"2"_{\id}
\ar "t";"cdots"_{\id}
\ar "t+1";"t"
\ar "cdots2";"t+1"
\ar "n-2";"cdots2"
\ar "n-1";"n-2"
\ar "n";"n-2"
\end{xy}
& (t\le n-2,\ \epsilon=1)\\
\begin{xy}
(0,5)*[o]+{K}="0", 
(0,-5)*[o]+{0}="1",
(10,0)*[o]+{K^2}="2",  
(25,0)*[o]+{\cdots}="cdots", 
(40,0)*[o]+{K^2}="n-2",
(50,5)*[o]+{K}="n-1", 
(50,-5)*[o]+{K}="n",
\ar "0";"2"^{\gamma}
\ar "1";"2"
\ar "cdots";"2"_{\id}
\ar "n-2";"cdots"_{\id}
\ar "n-1";"n-2"_{\alpha}
\ar "n";"n-2"^{\beta}
\end{xy}
& (t= n-1,\ \epsilon=0)\\
\begin{xy}
(0,5)*[o]+{0}="0", 
(0,-5)*[o]+{K}="1",
(10,0)*[o]+{K^2}="2",  
(25,0)*[o]+{\cdots}="cdots", 
(40,0)*[o]+{K^2}="n-2",
(50,5)*[o]+{K}="n-1", 
(50,-5)*[o]+{K}="n",
\ar "0";"2",
\ar "1";"2"_{\gamma}
\ar "cdots";"2"_{\id}
\ar "n-2";"cdots"_{\id}
\ar "n-1";"n-2"_{\alpha}
\ar "n";"n-2"^{\beta}
\end{xy}
& (t= n-1,\ \epsilon=1)\\
\end{array}\right.
\]
It is easy to check that $M_0(t)$ and $M_1(t)$ are indecomposable for any $t\in \{2,\dots, n-1\}$.

If $2\le t \le n-3$, then the minimal injective copresentation of $M_0(t)$ (resp. $M_1(t)$)
is given by
\[0\to M_0(t)\to I_2 \to I_1\oplus I_{t+1} \to 0\ (\text{resp. }0\to M_1(t)\to I_2 \to I_0\oplus I_{t+1} \to 0).\]
Therefore, we have
\[
g^{\tau^{-1}M_0(t)}=g^{M_1(t+1)}=\bm{e}_1+\bm{e}_{t+1}-\bm{e}_2.
\]
In particular, we obtain
\[
\tau^{-1} M_0(t)\simeq M_1(t+1).
\]  
Similarly, we obtain
\[
\tau^{-1} M_1(t)\simeq M_0(t+1).
\]

If $t=n-2$, then the minimal injective copresentation of $M_0(n-2)$ (resp. $M_1(n-2)$)
is given by
\[0\to M_0(n-2)\to I_2 \to I_1\oplus I_{n-1}\oplus I_n \to 0\ (\text{resp. }0\to M_1(n-2)\to I_2 \to I_0\oplus I_{n-1}\oplus I_n \to 0).\]
Therefore, we have
\[
g^{\tau^{-1}M_0(n-2)}=g^{M_1(n-1)}=\bm{e}_1+\bm{e}_{n-1}+\bm{e}_n-\bm{e}_2.
\]
In particular, we obtain
\[
\tau^{-1} M_0(n-2)\simeq M_1(n-1).
\]  
Similarly, we obtain
\[
\tau^{-1} M_1(n-2)\simeq M_0(n-1).
\]

Then the assertion follows from $M_0(2)\simeq P_0$ and $M_1(2)\simeq P_1$.

(2). For each $(s,t)$ satisfying $1\le s<t \le n-2$, we define $M(s,t)$ as follows.
\[
\left\{
\begin{array}{cl}
\begin{xy}
(0,5)*[o]+{0}="0", 
(0,-5)*[o]+{0}="1",
(10,0)*[o]+{K}="2",  
(25,0)*[o]+{\cdots}="cdots", 
(40,0)*[o]+{K}="t",
(40,-5)*[o]+{t}="t'",
(50,0)*[o]+{0}="t+1",
(65,0)*[o]+{\cdots}="cdots2", 
(80,0)*[o]+{0}="n-2",
(90,5)*[o]+{0}="n-1", 
(90,-5)*[o]+{0}="n",
\ar "0";"2"
\ar "1";"2"
\ar "cdots";"2"_{\id}
\ar "t";"cdots"_{\id}
\ar "t+1";"t"
\ar "cdots2";"t+1"
\ar "n-2";"cdots2"
\ar "n-1";"n-2"
\ar "n";"n-2"
\end{xy}
& (s=1)\\
\begin{xy}
(0,5)*[o]+{K}="0", 
(0,-5)*[o]+{K}="1",
(10,0)*[o]+{K^2}="2",  
(25,0)*[o]+{\cdots}="cdots", 
(40,0)*[o]+{K^2}="s",
(40,-5)*[o]+{s}="s'",
(53,0)*[o]+{K}="s+1",
(65,0)*[o]+{\cdots}="cdots2",
(78,0)*[o]+{K}="t",
(78,-5)*[o]+{t}="t'",
(90,0)*[o]+{0}="t+1",
(105,0)*[o]+{\cdots}="cdots3",
(120,0)*[o]+{0}="n-2",
(130,5)*[o]+{0}="n-1", 
(130,-5)*[o]+{0}="n",
\ar "0";"2"^{\alpha}
\ar "1";"2"_{\beta}
\ar "cdots";"2"_{\id}
\ar "s";"cdots"_{\id}
\ar "s+1";"s"_{\gamma}
\ar "cdots2";"s+1"_{\id}
\ar "t";"cdots2"_{\id}
\ar "t+1";"t"
\ar "cdots3";"t+1"
\ar "n-2";"cdots3"
\ar "n-1";"n-2"
\ar "n";"n-2"
\end{xy}
& (s\ge 2)\\
\end{array}\right.
\] 
It is easy to check that $M(s,t)$ is indecomposable for each $(s,t)$ satisfying $1\le s<t \le n-2$.

If $1=s< t\le n-3$, then the minimal injective copresentation of $M(1,t)$ has the following form.
\[
0\to  M(1,t)\to I_2\to I_0\oplus I_1\oplus I_{t+1}\to 0.
\]
Therefore, we have
\[
g^{\tau^{-1}M(1,t)}=g^{M(2,t+1)}=\bm{e}_0+\bm{e}_1+\bm{e}_{t+1}-\bm{e}_2.
\]
In particular, we obtain
\[\tau^{-1} M(1,t)\simeq M(2,t+1).\]  

If $2\le s < t\le n-3$, then the minimal injective copresentation of $M(s,t)$ is given by
\[0\to M(s,t)\to I_2^{\oplus 2}\to I_0\oplus I_1\oplus I_{s+1}\oplus I_{t+1}\to 0.\]
Therefore, we have
\[
g^{\tau^{-1}M(s,t)}=g^{M(s+1,t+1)}=\bm{e}_0+\bm{e}_1+\bm{e}_{s+1}+\bm{e}_{t+1}-2\bm{e}_2.
\]
In particular, we obtain
\[\tau^{-1} M(s,t)\simeq M(s+1,t+1).\]  


Then the assertion follows from $M(1,t)\simeq P_t$.

(3). For each $t\in \{1,\dots,n-2\}$ and $\epsilon\in\{0,1\}$, we define $M_{n-\epsilon}(t)$ as follows.
 \[\left\{
\begin{array}{cl}
\begin{xy}
(0,5)*[o]+{0}="0", 
(0,-5)*[o]+{0}="1",
(10,0)*[o]+{K}="2",  
(25,0)*[o]+{\cdots}="cdots", 
(40,0)*[o]+{K}="n-2",
(50,5)*[o]+{0}="n-1", 
(50,-5)*[o]+{K}="n",
\ar "0";"2"
\ar "1";"2"
\ar "cdots";"2"_{\id}
\ar "n-2";"cdots"_{\id}
\ar "n-1";"n-2"
\ar "n";"n-2"^{\id}
\end{xy}
& (t=1,\ \epsilon=0)\\
\begin{xy}
(0,5)*[o]+{0}="0", 
(0,-5)*[o]+{0}="1",
(10,0)*[o]+{K}="2",  
(25,0)*[o]+{\cdots}="cdots", 
(40,0)*[o]+{K}="n-2",
(50,5)*[o]+{K}="n-1", 
(50,-5)*[o]+{0}="n",
\ar "0";"2"
\ar "1";"2"
\ar "cdots";"2"_{\id}
\ar "n-2";"cdots"_{\id}
\ar "n-1";"n-2"_{\id}
\ar "n";"n-2"
\end{xy}
& (t=1,\ \epsilon =1)\\
\begin{xy}
(0,5)*[o]+{K}="0", 
(0,-5)*[o]+{K}="1",
(10,0)*[o]+{K^2}="2",  
(25,0)*[o]+{\cdots}="cdots", 
(40,0)*[o]+{K^2}="t",
(40,-5)*[o]+{t}="t'",
(53,0)*[o]+{K}="t+1",
(65,0)*[o]+{\cdots}="cdots2", 
(80,0)*[o]+{K}="n-2",
(90,5)*[o]+{0}="n-1", 
(90,-5)*[o]+{K}="n",
\ar "0";"2"^{\alpha}
\ar "1";"2"_{\beta}
\ar "cdots";"2"_{\id}
\ar "t";"cdots"_{\id}
\ar "t+1";"t"_{\gamma}
\ar "cdots2";"t+1"_{\id}
\ar "n-2";"cdots2"_{\id}
\ar "n-1";"n-2"
\ar "n";"n-2"^{\id}
\end{xy}
&  (2\le t \le n-3,\ \epsilon=0)\\
\begin{xy}
(0,5)*[o]+{K}="0", 
(0,-5)*[o]+{K}="1",
(10,0)*[o]+{K^2}="2",  
(25,0)*[o]+{\cdots}="cdots", 
(40,0)*[o]+{K^2}="t",
(40,-5)*[o]+{t}="t'",
(53,0)*[o]+{K}="t+1",
(65,0)*[o]+{\cdots}="cdots2", 
(80,0)*[o]+{K}="n-2",
(90,5)*[o]+{K}="n-1", 
(90,-5)*[o]+{0}="n",
\ar "0";"2"^{\alpha}
\ar "1";"2"_{\beta}
\ar "cdots";"2"_{\id}
\ar "t";"cdots"_{\id}
\ar "t+1";"t"_{\gamma}
\ar "cdots2";"t+1"_{\id}
\ar "n-2";"cdots2"_{\id}
\ar "n-1";"n-2"_{\id}
\ar "n";"n-2"
\end{xy}
& (2\le t\le n-3,\ \epsilon=1)\\
\begin{xy}
	(0,5)*[o]+{K}="0", 
	(0,-5)*[o]+{K}="1",
	(10,0)*[o]+{K^2}="2",  
	(25,0)*[o]+{\cdots}="cdots", 
	(40,0)*[o]+{K^2}="n-2",
	(50,5)*[o]+{0}="n-1", 
	(50,-5)*[o]+{K}="n",
	\ar "0";"2"^{\alpha}
	\ar "1";"2"_{\beta}
	\ar "cdots";"2"_{\id}
	\ar "n-2";"cdots"_{\id}
	\ar "n-1";"n-2"
	\ar "n";"n-2"^{\gamma}
\end{xy}
&  (t=n-2,\ \epsilon=0)\\
\begin{xy}
	(0,5)*[o]+{K}="0", 
	(0,-5)*[o]+{K}="1",
	(10,0)*[o]+{K^2}="2",  
	(25,0)*[o]+{\cdots}="cdots", 
	(40,0)*[o]+{K^2}="n-2",
	(50,5)*[o]+{K}="n-1", 
	(50,-5)*[o]+{0}="n",
	\ar "0";"2"^{\alpha}
	\ar "1";"2"_{\beta}
	\ar "cdots";"2"_{\id}
	\ar "n-2";"cdots"_{\id}
	\ar "n-1";"n-2"_{\gamma}
	\ar "n";"n-2"
\end{xy}
& (t=n-2,\ \epsilon=1)\\
\end{array}\right.
\]
It is easy to check that $M_{n-\epsilon}(t)$ is indecomposable for each $(\epsilon,t)\in \{0,1\}\times \{1,\dots,n-2\}$.

If $t=1$, then then the minimal injective copresentation of $M_{n}(1)$ (resp. $M_{n-1}(1)$) is given by
\[0\to  M_{n}(1)\to I_2\to I_0\oplus I_1 \oplus I_{n-1}\to 0\;(\text{resp. } 0\to M_{n-1}(1)\to I_2\to I_0\oplus I_1\oplus I_n\to 0).\]
Therefore, we have
\[
g^{\tau^{-1}M_n(1)}=g^{M_{n-1}(2)}=\bm{e}_0+\bm{e}_1+\bm{e}_{n-1}-\bm{e}_2.
\]
In particular, we obtain
\[\tau^{-1} M_n(1)\simeq M_{n-1}(2).\]  
Similarly, we obtain 
\[\tau^{-1} M_{n-1}(1)\simeq M_{n}(2).\]  

If $2\le t\le n-3$, then the minimal injective copresentation of $M_{n}(t)$ (resp. $M_{n-1}(t)$) is given by
\[
0\to M_{n}(t)\to I_2^{\oplus 2}\to I_0\oplus I_1 \oplus I_{t+1}\oplus I_{n-1}\to 0\;(\text{resp. } 0\to \tau M_{n-1}(t)\to I_2^{\oplus 2}\to I_0\oplus I_1 \oplus I_{t+1}\oplus I_n\to 0).
\]
Therefore, we have
\[
g^{\tau^{-1}M_n(t)}=g^{M_{n-1}(t+1)}=\bm{e}_0+\bm{e}_1+\bm{e}_{t+1}+\bm{e}_{n-1}-2\bm{e}_2.
\]
In particular, we obtain
\[\tau^{-1} M_n(t)\simeq M_{n-1}(t+1).\]  
Similarly, we obtain 
\[\tau^{-1} M_{n-1}(t)\simeq M_{n}(t+1).\]  

Then the assertion follows from $M_{n-1}(1)\simeq P_{n-1}$ and $M_n(1)\simeq P_n$. 
\end{proof}
\begin{lemma}
	\label{nonsincere_i}
\begin{enumerate}[{\rm (1)}]
	\item 	For each $t\in \{2,\dots, n-1\}$, $\tau^{n-1-t}I_{\epsilon}$ $(\epsilon\in \{0,1\})$ has the following form.
	\[\left\{
	\begin{array}{cl}
	
	\begin{xy}
	(0,5)*[o]+{K}="0", 
	(0,-5)*[o]+{0}="1",
	(10,0)*[o]+{K}="2",  
	(25,0)*[o]+{\cdots}="cdots", 
	(40,0)*[o]+{K}="t",
	(40,-5)*[o]+{t}="t'",
	(50,0)*[o]+{K^2}="t+1",
	(65,0)*[o]+{\cdots}="cdots2", 
	(80,0)*[o]+{K^2}="n-2",
	(90,5)*[o]+{K}="n-1", 
	(90,-5)*[o]+{K}="n",
	\ar "0";"2",
	\ar "1";"2",
	\ar "cdots";"2",
	\ar "t";"cdots",
	\ar "t+1";"t"
	\ar "cdots2";"t+1",
	\ar "n-2";"cdots2",
	\ar "n-1";"n-2"
	\ar "n";"n-2"
	\end{xy}
	& (t\le n-2,\ n-1-t+\epsilon\in \even)\\
	\begin{xy}
	(0,5)*[o]+{0}="0", 
	(0,-5)*[o]+{K}="1",
	(10,0)*[o]+{K}="2",  
	(25,0)*[o]+{\cdots}="cdots", 
	(40,0)*[o]+{K}="t",
	(40,-5)*[o]+{t}="t'",
	(50,0)*[o]+{K^2}="t+1",
	(65,0)*[o]+{\cdots}="cdots2", 
	(80,0)*[o]+{K^2}="n-2",
	(90,5)*[o]+{K}="n-1", 
	(90,-5)*[o]+{K}="n",
	\ar "0";"2",
	\ar "1";"2",
	\ar "cdots";"2",
	\ar "t";"cdots",
	\ar "t+1";"t"
	\ar "cdots2";"t+1",
	\ar "n-2";"cdots2",
	\ar "n-1";"n-2"
	\ar "n";"n-2"
	\end{xy}
	& (t\le n-2,\ n-1-t+\epsilon\in \odd)\\
	\begin{xy}
	(0,5)*[o]+{K}="0", 
	(0,-5)*[o]+{0}="1",
	(10,0)*[o]+{0}="2",  
	(25,0)*[o]+{\cdots}="cdots", 
	(40,0)*[o]+{0}="n-2",
	(50,5)*[o]+{0}="n-1", 
	(50,-5)*[o]+{0}="n",
	\ar "0";"2",
	\ar "1";"2",
	\ar "cdots";"2",
	\ar "n-2";"cdots",
	\ar "n-1";"n-2"
	\ar "n";"n-2"
	\end{xy}
	& (t=n-1,\ \epsilon=0)\\
	\begin{xy}
	(0,5)*[o]+{0}="0", 
	(0,-5)*[o]+{K}="1",
	(10,0)*[o]+{0}="2",  
	(25,0)*[o]+{\cdots}="cdots", 
	(40,0)*[o]+{0}="n-2",
	(50,5)*[o]+{0}="n-1", 
	(50,-5)*[o]+{0}="n",
	\ar "0";"2",
	\ar "1";"2",
	\ar "cdots";"2",
	\ar "n-2";"cdots",
	\ar "n-1";"n-2"
	\ar "n";"n-2"
	\end{xy}
	& (t=n-1,\ \epsilon =1)\\
	\end{array}\right.
	\]
\item For each $(s,t)$ satisfying $2\le s<t\le n-2$, $\tau^{n-2-t}I_{n-1-t+s}$  has the following form.
\[
\begin{xy}
(0,5)*[o]+{0}="0", 
(0,-5)*[o]+{0}="1",
(10,0)*[o]+{0}="2",  
(25,0)*[o]+{\cdots}="cdots", 
(40,0)*[o]+{0}="s",
(40,-5)*[o]+{s}="s'",
(50,0)*[o]+{K}="s+1",
(65,0)*[o]+{\cdots}="cdots2",
(80,0)*[o]+{K}="t",
(80,-5)*[o]+{t}="t'",
(90,0)*[o]+{K^2}="t+1",
(105,0)*[o]+{\cdots}="cdots3",
(120,0)*[o]+{K^2}="n-2",
(130,5)*[o]+{K}="n-1", 
(130,-5)*[o]+{K}="n",
\ar "0";"2",
\ar "1";"2",
\ar "cdots";"2",
\ar "s";"cdots",
\ar "s+1";"s"
\ar "cdots2";"s+1",
\ar "t";"cdots2",
\ar "t+1";"t"
\ar "cdots3";"t+1",
\ar "n-2";"cdots3",
\ar "n-1";"n-2"
\ar "n";"n-2"
\end{xy}
\]	
\item  	For each $t\in \{1,\dots, n-2\}$, $\tau^{n-2-t}I_{n-\epsilon}$ $(\epsilon\in \{0,1\})$ has the following form.
\[\left\{
\begin{array}{cl}
\begin{xy}
(0,5)*[o]+{K}="0", 
(0,-5)*[o]+{K}="1",
(10,0)*[o]+{K}="2",  
(25,0)*[o]+{\cdots}="cdots", 
(40,0)*[o]+{K}="n-2",
(50,5)*[o]+{K}="n-1", 
(50,-5)*[o]+{0}="n",
\ar "0";"2",
\ar "1";"2",
\ar "cdots";"2",
\ar "n-2";"cdots",
\ar "n-1";"n-2"
\ar "n";"n-2"
\end{xy}
& (t=1,\ n-3+\epsilon\in \odd)\\
\begin{xy}
(0,5)*[o]+{K}="0", 
(0,-5)*[o]+{K}="1",
(10,0)*[o]+{K}="2",  
(25,0)*[o]+{\cdots}="cdots", 
(40,0)*[o]+{K}="n-2",
(50,5)*[o]+{0}="n-1", 
(50,-5)*[o]+{K}="n",
\ar "0";"2",
\ar "1";"2",
\ar "cdots";"2",
\ar "n-2";"cdots",
\ar "n-1";"n-2"
\ar "n";"n-2"
\end{xy}
& (t=1,\ n-3+\epsilon\in \even)\\
\begin{xy}
(0,5)*[o]+{0}="0", 
(0,-5)*[o]+{0}="1",
(10,0)*[o]+{0}="2",  
(25,0)*[o]+{\cdots}="cdots", 
(40,0)*[o]+{0}="t",
(40,-5)*[o]+{t}="t'",
(50,0)*[o]+{K}="t+1",
(65,0)*[o]+{\cdots}="cdots2", 
(80,0)*[o]+{K}="n-2",
(90,5)*[o]+{K}="n-1", 
(90,-5)*[o]+{0}="n",
\ar "0";"2",
\ar "1";"2",
\ar "cdots";"2",
\ar "t";"cdots",
\ar "t+1";"t"
\ar "cdots2";"t+1",
\ar "n-2";"cdots2",
\ar "n-1";"n-2"
\ar "n";"n-2"
\end{xy}
& (t\ge 2,\ n-2-t+\epsilon\in \odd)\\
\begin{xy}
(0,5)*[o]+{0}="0", 
(0,-5)*[o]+{0}="1",
(10,0)*[o]+{0}="2",  
(25,0)*[o]+{\cdots}="cdots", 
(40,0)*[o]+{0}="t",
(40,-5)*[o]+{t}="t'",
(50,0)*[o]+{K}="t+1",
(65,0)*[o]+{\cdots}="cdots2", 
(80,0)*[o]+{K}="n-2",
(90,5)*[o]+{0}="n-1", 
(90,-5)*[o]+{K}="n",
\ar "0";"2",
\ar "1";"2",
\ar "cdots";"2",
\ar "t";"cdots",
\ar "t+1";"t"
\ar "cdots2";"t+1",
\ar "n-2";"cdots2",
\ar "n-1";"n-2"
\ar "n";"n-2"
\end{xy}
& (t\ge 2,\ n-2-t+\epsilon\in \even)\\
\end{array}\right.
\]
\end{enumerate}

\end{lemma}
\begin{proof}
We denote by $\alpha$, $\beta$ the $K$-linear maps from $K$ to $K^2$ given by $(1\ 0)$ and $(0\ 1)$, respectively, and by $\delta$
the $K$-linear map from $K^2$ to $K$ given by $\left(\begin{smallmatrix}1\\1\end{smallmatrix}\right)$.  	
	
(1). 
For each $t\in \{1,\dots,n-2\}$ and $\epsilon\in\{0,1\}$, we define $N_{\epsilon}(t)$ as follows.
\[\left\{
\begin{array}{cl}
\begin{xy}
(0,5)*[o]+{K}="0", 
(0,-5)*[o]+{0}="1",
(10,0)*[o]+{K}="2",  
(25,0)*[o]+{\cdots}="cdots", 
(40,0)*[o]+{K}="t",
(40,-5)*[o]+{t}="t'",
(53,0)*[o]+{K^2}="t+1",
(65,0)*[o]+{\cdots}="cdots2", 
(80,0)*[o]+{K^2}="n-2",
(90,5)*[o]+{K}="n-1", 
(90,-5)*[o]+{K}="n",
\ar "0";"2"^{\id}
\ar "1";"2"
\ar "cdots";"2"_{\id}
\ar "t";"cdots"_{\id}
\ar "t+1";"t"_{\delta}
\ar "cdots2";"t+1"_{\id}
\ar "n-2";"cdots2"_{\id}
\ar "n-1";"n-2"_{\alpha}
\ar "n";"n-2"^{\beta}
\end{xy}
& (t\le n-2,\ \epsilon=0)\\
\begin{xy}
(0,5)*[o]+{0}="0", 
(0,-5)*[o]+{K}="1",
(10,0)*[o]+{K}="2",  
(25,0)*[o]+{\cdots}="cdots", 
(40,0)*[o]+{K}="t",
(40,-5)*[o]+{t}="t'",
(53,0)*[o]+{K^2}="t+1",
(65,0)*[o]+{\cdots}="cdots2", 
(80,0)*[o]+{K^2}="n-2",
(90,5)*[o]+{K}="n-1", 
(90,-5)*[o]+{K}="n",
\ar "0";"2"
\ar "1";"2"_{\id}
\ar "cdots";"2"_{\id}
\ar "t";"cdots"_{\id}
\ar "t+1";"t"_{\delta}
\ar "cdots2";"t+1"_{\id}
\ar "n-2";"cdots2"_{\id}
\ar "n-1";"n-2"_{\alpha}
\ar "n";"n-2"^{\beta}
\end{xy}
& (t\le n-2,\ \epsilon=1)\\
\begin{xy}
(0,5)*[o]+{K}="0", 
(0,-5)*[o]+{0}="1",
(10,0)*[o]+{0}="2",  
(25,0)*[o]+{\cdots}="cdots", 
(40,0)*[o]+{0}="n-2",
(50,5)*[o]+{0}="n-1", 
(50,-5)*[o]+{0}="n",
\ar "0";"2"
\ar "1";"2"
\ar "cdots";"2"
\ar "n-2";"cdots"
\ar "n-1";"n-2"
\ar "n";"n-2"
\end{xy}
& (t=n-1,\ \epsilon=0)\\
\begin{xy}
(0,5)*[o]+{0}="0", 
(0,-5)*[o]+{K}="1",
(10,0)*[o]+{0}="2",  
(25,0)*[o]+{\cdots}="cdots", 
(40,0)*[o]+{0}="n-2",
(50,5)*[o]+{0}="n-1", 
(50,-5)*[o]+{0}="n",
\ar "0";"2"
\ar "1";"2"
\ar "cdots";"2"
\ar "n-2";"cdots"
\ar "n-1";"n-2"
\ar "n";"n-2"
\end{xy}
& (t=n-1,\ \epsilon =1)\\
\end{array}\right.
\]	
(We assume $\alpha=\beta=\id$ if $t=n-2$.)
If $t\le n-3$, then the minimal injective copresentation of $M_0(t)$ (resp. $M_1(t)$) is given by
\[0\to N_0(t)\to I_2\oplus I_{t+1}\to I_1\oplus I_{n-1}\oplus I_n\to 0\ (\text{resp. }0\to N_1(t)\to I_2\oplus I_{t+1}\to I_0\oplus I_{n-1}\oplus I_n\to 0).  \]	
Hence, we have 
\[\begin{array}{lllll}
g^{\tau^{-1}N_0(t)}&=&g^{N_1(t+1)}&=&\bm{e}_1+\bm{e}_{n-1}+\bm{e}_n-\bm{e}_2-\bm{e}_{t+1}\\
g^{\tau^{-1}N_1(t)}&=&g^{N_0(t+1)}&=&\bm{e}_0+\bm{e}_{n-1}+\bm{e}_n-\bm{e}_2-\bm{e}_{t+1}.\\
\end{array}
\]
In particular, we have
\begin{center}
$\tau^{-1}N_0(t)\simeq N_1(t+1)$ and $\tau^{-1}N_1(t)\simeq N_0(t+1).$
\end{center}

If $t=n-2$, then the minimal injective copresentation of $N_0(n-2)$ (resp. $N_1(n-2)$) is given by
\[0\to N_0(n-2)\to I_2\to I_1\to 0\ (\text{resp. }0\to N_1(n-2)\to I_2\to I_0\to 0).  \]	
Hence, we have 
\[\begin{array}{lllll}
g^{\tau^{-1}N_0(n-2)}&=&g^{N_1(n-1)}&=&\bm{e}_1-\bm{e}_2\\
g^{\tau^{-1}N_1(n-2)}&=&g^{N_0(n-1)}&=&\bm{e}_0-\bm{e}_2.\\
\end{array}
\]
In particular, we have
\begin{center}
	$\tau^{-1}N_0(n-2)\simeq N_1(n-1)$ and $\tau^{-1}N_1(n-2)\simeq N_0(n-1).$
\end{center}

Then the assertion follows from $N_0(n-1)\simeq I_0$ and $N_1(n-1)\simeq I_1$.

(2). For each $(s,t)$ satisfying $2\le s<t\le n-2$, we define $N(s,t)$ as follows.
\[
\begin{xy}
(0,5)*[o]+{0}="0", 
(0,-5)*[o]+{0}="1",
(10,0)*[o]+{0}="2",  
(25,0)*[o]+{\cdots}="cdots", 
(40,0)*[o]+{0}="s",
(40,-5)*[o]+{s}="s'",
(50,0)*[o]+{K}="s+1",
(65,0)*[o]+{\cdots}="cdots2",
(80,0)*[o]+{K}="t",
(80,-5)*[o]+{t}="t'",
(93,0)*[o]+{K^2}="t+1",
(105,0)*[o]+{\cdots}="cdots3",
(120,0)*[o]+{K^2}="n-2",
(130,5)*[o]+{K}="n-1", 
(130,-5)*[o]+{K}="n",
\ar "0";"2"
\ar "1";"2"
\ar "cdots";"2"
\ar "s";"cdots"
\ar "s+1";"s"_{\id}
\ar "cdots2";"s+1"_{\id}
\ar "t";"cdots2"_{\id}
\ar "t+1";"t"_{\delta}
\ar "cdots3";"t+1"_{\id}
\ar "n-2";"cdots3"_{\id}
\ar "n-1";"n-2"_{\alpha}
\ar "n";"n-2"^{\beta}
\end{xy}
\]	
It is easy to check that $N(s,t)$ is indecomposable for each $(s,t)$ satisfying $2\le s < t \le n-2$.

If $t\le n-3$, then the minimal injective copresentation of $N(s,t)$ is given by
\[0\to N(s,t)\to I_{s+1}\oplus I_{t+1}\to I_{n-1}\oplus I_n\to 0.\]
Hence, we have
\[
g^{\tau^{-1}N(s,t)}=g^{N(s+1,t+1)}=\bm{e}_{n-1}+\bm{e}_n-\bm{e}_{s+1}-\bm{e}_{t+1}.
\] 
In particular, we obtain
\[\tau^{-1} M(s,t)\simeq M(s+1,t+1).\]  

Then the assertion follows from $N(s,n-2)\simeq I_{s+1}$.

(3). For each $t\in \{1,\dots, n-2\}$ and $\epsilon\in \{0,1\}$, we define $N_{n-\epsilon}(t)$ as follows.
\[\left\{
\begin{array}{cl}
\begin{xy}
(0,5)*[o]+{K}="0", 
(0,-5)*[o]+{K}="1",
(10,0)*[o]+{K}="2",  
(25,0)*[o]+{\cdots}="cdots", 
(40,0)*[o]+{K}="n-2",
(50,5)*[o]+{0}="n-1", 
(50,-5)*[o]+{K}="n",
\ar "0";"2"^{\id}
\ar "1";"2"_{\id}
\ar "cdots";"2"_{\id}
\ar "n-2";"cdots"_{\id}
\ar "n-1";"n-2"
\ar "n";"n-2"^{\id}
\end{xy}
& (t=1,\ \epsilon=0)\\
\begin{xy}
(0,5)*[o]+{K}="0", 
(0,-5)*[o]+{K}="1",
(10,0)*[o]+{K}="2",  
(25,0)*[o]+{\cdots}="cdots", 
(40,0)*[o]+{K}="n-2",
(50,5)*[o]+{K}="n-1", 
(50,-5)*[o]+{0}="n",
\ar "0";"2"^{\id}
\ar "1";"2"_{\id}
\ar "cdots";"2"_{\id}
\ar "n-2";"cdots"_{\id}
\ar "n-1";"n-2"_{\id}
\ar "n";"n-2"
\end{xy}
& (t=1,\ \epsilon=1)\\
\begin{xy}
(0,5)*[o]+{0}="0", 
(0,-5)*[o]+{0}="1",
(10,0)*[o]+{0}="2",  
(25,0)*[o]+{\cdots}="cdots", 
(40,0)*[o]+{0}="t",
(40,-5)*[o]+{t}="t'",
(50,0)*[o]+{K}="t+1",
(65,0)*[o]+{\cdots}="cdots2", 
(80,0)*[o]+{K}="n-2",
(90,5)*[o]+{0}="n-1", 
(90,-5)*[o]+{K}="n",
\ar "0";"2"
\ar "1";"2"
\ar "cdots";"2"
\ar "t";"cdots"
\ar "t+1";"t"
\ar "cdots2";"t+1"_{\id}
\ar "n-2";"cdots2"_{\id}
\ar "n-1";"n-2"
\ar "n";"n-2"^{\id}
\end{xy}
& (t\ge 2,\ \epsilon=0)\\
\begin{xy}
(0,5)*[o]+{0}="0", 
(0,-5)*[o]+{0}="1",
(10,0)*[o]+{0}="2",  
(25,0)*[o]+{\cdots}="cdots", 
(40,0)*[o]+{0}="t",
(40,-5)*[o]+{t}="t'",
(50,0)*[o]+{K}="t+1",
(65,0)*[o]+{\cdots}="cdots2", 
(80,0)*[o]+{K}="n-2",
(90,5)*[o]+{K}="n-1", 
(90,-5)*[o]+{0}="n",
\ar "0";"2"
\ar "1";"2"
\ar "cdots";"2"
\ar "t";"cdots"
\ar "t+1";"t"
\ar "cdots2";"t+1"_{\id}
\ar "n-2";"cdots2"_{\id}
\ar "n-1";"n-2"_{\id}
\ar "n";"n-2"
\end{xy}
& (t\ge 2,\ \epsilon=1)\\
\end{array}\right.
\]
It is easy to check that $N_{n-1}(t)$ and $N_n(t)$ are indecomposable for each $t\in \{1,\dots, n-2\}$.

For any $t\in \{1,\dots,n-3\}$, the minimal injective copresentation of $N_{n-1}(t)$ (resp. $N_n(t)$) is given by
\[0\to N_{n-1}(t)\to I_{t+1}\to I_n \to 0\ (\text{resp. }0\to N_n(t)\to I_{t+1}\to I_{n-1} \to 0).\]
Hence, we obtain
\[
\begin{array}{lllll}
g^{\tau^{-1}N_{n-1}(t)}&=&g^{N_{n}(2)}&=&\bm{e}_n-\bm{e}_{t+1}\\
g^{\tau^{-1}N_{n}(t)}&=&g^{N_{n-1}(2)}&=&\bm{e}_{n-1}-\bm{e}_{t+1}\\
\end{array}
\] 
In particular, we have
\begin{center}
	$\tau^{-1}N_{n-1}(t)\simeq N_n(t+1)$ and $\tau^{-1}N_{n}(t)\simeq N_{n-1}(t+1)$. 
\end{center}
Then the assertion follows from $N_{n-1}(n-2)\simeq I_{n-1}$ and $N_n(n-2)\simeq I_n$.
\end{proof}
For each $t\in \{2,3,\dots, n-2\}$, we define an indecomposable module $R_{t-1}$ as follows. 
\[\begin{xy}
(0,5)*[o]+{K}="0", 
(0,-5)*[o]+{K}="1",
(10,0)*[o]+{K}="2",  
(25,0)*[o]+{\cdots}="cdots", 
(40,0)*[o]+{K}="t",
(40,-5)*[o]+{t}="t'",
(50,0)*[o]+{0}="t+1",
(65,0)*[o]+{\cdots}="cdots2", 
(80,0)*[o]+{0}="n-2",
(90,5)*[o]+{0}="n-1", 
(90,-5)*[o]+{0}="n",
\ar "0";"2"^{\id}
\ar "1";"2"_{\id}
\ar "cdots";"2"_{\id}
\ar "t";"cdots"_{\id}
\ar "t+1";"t"
\ar "cdots2";"t+1"
\ar "n-2";"cdots2"
\ar "n-1";"n-2"
\ar "n";"n-2"
\end{xy}\] 

\begin{lemma}
	\label{nonsincere_r_n-2} 
Let $t\in \{2,\dots, n-2\}$. Then we have $\tau^{n-2}R_{t-1}\simeq R_{t-1}$ and  
$\tau^{-k}R_{t-1}$ $(k\in \{1,\dots, n-3\})$ is isomorphic to $R_{t-1}(k)$ given by the following: 
\[
\left\{
\begin{array}{cl}
\begin{xy}
	(0,5)*[o]+{0}="0", 
	(0,-5)*[o]+{0}="1",
	(10,0)*[o]+{0}="2",  
	(25,0)*[o]+{\cdots}="cdots", 
	(40,0)*[o]+{0}="s",
	(40,-5)*[o]+{k+1}="s'",
	(50,0)*[o]+{K}="s+1",
	(65,0)*[o]+{\cdots}="cdots2",
	(80,0)*[o]+{K}="t",
	(80,-5)*[o]+{t+k}="t'",
	(90,0)*[o]+{0}="t+1",
	(105,0)*[o]+{\cdots}="cdots3",
	(120,0)*[o]+{0}="n-2",
	(130,5)*[o]+{0}="n-1", 
	(130,-5)*[o]+{0}="n",
	\ar "0";"2"
	\ar "1";"2"
	\ar "cdots";"2"
	\ar "s";"cdots"
	\ar "s+1";"s"
	\ar "cdots2";"s+1"_{\id}
	\ar "t";"cdots2"_{\id}
	\ar "t+1";"t"
	\ar "cdots3";"t+1"
	\ar "n-2";"cdots3"
	\ar "n-1";"n-2"
	\ar "n";"n-2"
\end{xy}
& (1\le k\le n-t-2)\\
\begin{xy}
(0,5)*[o]+{0}="0", 
(0,-5)*[o]+{0}="1",
(10,0)*[o]+{K}="2",  
(25,0)*[o]+{\cdots}="cdots", 
(40,0)*[o]+{K}="t",
(40,-5)*[o]+{n-t}="t'",
(53,0)*[o]+{K^2}="t+1",
(65,0)*[o]+{\cdots}="cdots2", 
(80,0)*[o]+{K^2}="n-2",
(90,5)*[o]+{K}="n-1", 
(90,-5)*[o]+{K}="n",
\ar "0";"2"
\ar "1";"2"
\ar "cdots";"2"_{\id}
\ar "t";"cdots"_{\id}
\ar "t+1";"t"_{\delta}
\ar "cdots2";"t+1"_{\id}
\ar "n-2";"cdots2"_{\id}
\ar "n-1";"n-2"_{\alpha}
\ar "n";"n-2"^{\beta}
\end{xy}
& (k=n-t-1)\\
\begin{xy}
(0,5)*[o]+{K}="0", 
(0,-5)*[o]+{K}="1",
(10,0)*[o]+{K^2}="2",  
(25,0)*[o]+{\cdots}="cdots", 
(40,0)*[o]+{K^2}="s",
(40,-5)*[o]+{k+t}="s'",
(38.5,-10)*[o]+{-n+2}="s''",
(53,0)*[o]+{K}="s+1",
(65,0)*[o]+{\cdots}="cdots2",
(79,0)*[o]+{K}="t",
(79,-5)*[o]+{k+1}="t'",
(92,0)*[o]+{K^2}="t+1",
(105,0)*[o]+{\cdots}="cdots3",
(120,0)*[o]+{K^2}="n-2",
(130,5)*[o]+{K}="n-1", 
(130,-5)*[o]+{K}="n",
\ar "0";"2"^{\alpha}
\ar "1";"2"_{\beta}
\ar "cdots";"2"_{\id}
\ar "s";"cdots"_{\id}
\ar "s+1";"s"_{\gamma}
\ar "cdots2";"s+1"_{\id}
\ar "t";"cdots2"_{\id}
\ar "t+1";"t"_{\delta}
\ar "cdots3";"t+1"_{\id}
\ar "n-2";"cdots3"_{\id}
\ar "n-1";"n-2"_{\alpha}
\ar "n";"n-2"^{\beta}
\end{xy}
& (n-t\le k \le n-3)\\
\end{array}\right.
\]
where $\alpha$, $\beta$, $\gamma$ are the $K$-linear maps from $K$ to $K^2$ given by $(1\ 0)$, $(0\ 1)$ and $(1\ 1)$, respectively, and 
 $\delta$ is the $K$-linear map from $K^2$ to $K$ given by $\left(\begin{smallmatrix}1\\1\end{smallmatrix}\right)$.
 (We assume $\alpha=\beta=\id$ if $k=n-3$.)
\end{lemma}
\begin{proof}
We let $R_{t-1}(0)=R_{t-1}(n-2):=R_{t-1}$.
	
We first show $R_{t-1}(k)$ is indecomposable for each $(t,k)\in \{2,\dots,n-2\}\times \{0,\dots, n-3\}$.
It is easy to check that $R_{t-1}(k)$ is indecomposable for $0\le k\le n-t-1$.
Hence we assume $n-t\le k \le n-3$ and $R:=R_{t-1}(k)$. Let $f\in \End_{A}(R)$ and $f_i:Re_i\to Re_i$ be the $K$-linear map induced by $f$.
Then, by the definition of $R$, we have
\[f_2=\cdots=f_{k+t-n+2},\ f_{k+t-n+3}=\cdots =f_{k+1},\ f_{k+2}=\cdots=f_{n-2}.\]

Now assume that $f_0$, $f_1$, $f_{n-1}$, $f_n$, and $f_{k+t-n+3}=\cdots =f_{k+1}$ are given by $a\in K$, $b\in K$, $c \in K$, $d\in K$, and $e\in K$, respectively.
By $f_2\circ\alpha=\alpha\circ f_0$ and $f_2\circ\beta=\beta\circ f_1$, we obtain 
\[f_2=\cdots=f_{k+t-n+2}=\begin{pmatrix}
a& 0\\
0& b\\
\end{pmatrix}.\]
Similarly, we have
\[f_{k+2}=\cdots=f_{n-2}=\begin{pmatrix}
c& 0\\
0& d\\
\end{pmatrix}.\]
Then it follows from $f_{k+t-n+2}\circ \gamma=\gamma\circ f_{k+t-n+3}$ (resp. $f_{k+1}\circ \delta=\delta\circ f_{k+2}$)
we obtain
\[(a\ b)=(e\ e)\ \left(\text{resp. }\begin{pmatrix}e\\e\end{pmatrix}=\begin{pmatrix}c\\d\end{pmatrix}\right).\]
This shows that $a=b=c=d=e$ and $f=a\cdot \id$. In particular, $\dim_K\End_A(R)=1$ and $R$ is indecomposable.	

Then the minimal injective copresentation of $R_{t-1}(k)$ ($0\le k \le n-3$) is given by the following.
\[
\left\{
\begin{array}{cl}
0\to R_{t-1}(k) \to I_{k+2}\to I_{t+k+1}\to 0 & (0\le k\le n-t-3)\\\\
0\to R_{t-1}(n-t-2) \to I_{n-t}\to I_{n-1}\oplus I_n\to 0 & (k=n-t-2)\\\\
0\to R_{t-1}(n-t-1) \to I_2\oplus I_{n-t+1}\to I_0\oplus I_1\oplus I_{n-1}\oplus I_n\to 0 & (t\ge 3,\ k= n-t-1)\\\\
0\to R_{t-1}(k) \to I_2^{\oplus 2}\oplus I_{k+2}\to I_0\oplus I_1\oplus I_{k+t-n+3}\oplus I_{n-1}\oplus I_n\to 0 & (n-t\le k\le  n-4)\\\\
0\to R_1(n-3) \to I_2\to I_0\oplus I_1\to 0 & (t=2,\ k= n-3)\\\\
0\to R_{t-1}(n-3) \to I_2^{\oplus 2}\to  I_0\oplus I_1\oplus I_t\to 0 & (t\ge 3,\  k= n-3)\\
\end{array}
\right.
\]
Therefore, we obtain 
\[
g^{\tau^{-1}R_{t-1}(k)}
=\left\{
\begin{array}{ll}
\bm{e}_{t+k+1}-\bm{e}_{k+2} & (0\le k\le n-t-3)\\
\bm{e}_{n-1}+\bm{e}_n-\bm{e}_{n-2} & (k=n-t-2) \\
\bm{e}_0+\bm{e}_1+\bm{e}_{n-1}+\bm{e}_n-\bm{e}_{2}-\bm{e}_{n-t+1} & (t\ge 3,\ k=n-t-1)\\
\bm{e}_0+\bm{e}_1+\bm{e}_{k+t-n+3}+\bm{e}_{n-1}+\bm{e}_n-2\bm{e}_{2}-\bm{e}_{k+2} & (n-t\le k\le  n-4)\\
\bm{e}_{0}+\bm{e}_1-\bm{e}_{2}  & (t=2,\ k= n-3)\\
\bm{e}_{0}+\bm{e}_1+\bm{e}_t-2\bm{e}_{2} & (t\ge 3)\\
\end{array}
\right.
\]

Then, by comparing the $g$-vectors, we have
\[\tau^{-1}R_{t-1}(k)\simeq R_{t-1}(k+1).
\] 
This implies the assertion.
\end{proof}
\begin{lemma}
	\label{nonsincere_r_2}
Let $L_{0,n-1}$, $L_{0,n}$, $L_{1,n-1}$ and $L_{1,n}$ be indecomposable modules given by
\[\begin{xy}
(0,7)*[o]+{0}="0", 
(0,-7)*[o]+{K}="1",
(6,0)*[o]+{K}="2",  
(16,0)*[o]+{\cdots}="cdots", 
(26,0)*[o]+{K}="n-2",
(32,7)*[o]+{0}="n-1", 
(32,-7)*[o]+{K}="n",
\ar "0";"2"
\ar "1";"2"_{\id}
\ar "cdots";"2"_{\id}
\ar "n-2";"cdots"_{\id}
\ar "n-1";"n-2"
\ar "n";"n-2"^{\id}
\end{xy},
\begin{xy}
(0,7)*[o]+{0}="0", 
(0,-7)*[o]+{K}="1",
(6,0)*[o]+{K}="2",  
(16,0)*[o]+{\cdots}="cdots", 
(26,0)*[o]+{K}="n-2",
(32,7)*[o]+{K}="n-1", 
(32,-7)*[o]+{0}="n",
\ar "0";"2"
\ar "1";"2"_{\id}
\ar "cdots";"2"_{\id}
\ar "n-2";"cdots"_{\id}
\ar "n-1";"n-2"_{\id}
\ar "n";"n-2"
\end{xy},
\begin{xy}
(0,7)*[o]+{K}="0", 
(0,-7)*[o]+{0}="1",
(6,0)*[o]+{K}="2",  
(16,0)*[o]+{\cdots}="cdots", 
(26,0)*[o]+{K}="n-2",
(32,7)*[o]+{0}="n-1", 
(32,-7)*[o]+{K}="n",
\ar "0";"2"^{\id}
\ar "1";"2"
\ar "cdots";"2"_{\id}
\ar "n-2";"cdots"_{\id}
\ar "n-1";"n-2"
\ar "n";"n-2"_{\id}
\end{xy},
\begin{xy}
(0,7)*[o]+{K}="0", 
(0,-7)*[o]+{0}="1",
(6,0)*[o]+{K}="2",  
(16,0)*[o]+{\cdots}="cdots", 
(26,0)*[o]+{K}="n-2",
(32,7)*[o]+{K}="n-1", 
(32,-7)*[o]+{0}="n",
\ar "0";"2"^{\id}
\ar "1";"2"
\ar "cdots";"2"_{\id}
\ar "n-2";"cdots"_{\id}
\ar "n-1";"n-2"^{\id}
\ar "n";"n-2"
\end{xy},
\]
respectively. 
\begin{enumerate}[{\rm (1)}]
		\item We have
\[\begin{xy}
(0,0)*[o]+{L_{0,n-1}}="0,n-1",
(20,0)*[o]+{L_{1,n}}="1,n",
\ar @<3pt> "0,n-1";"1,n"^{\tau}
\ar @<3pt> "1,n";"0,n-1"^{\tau}
\end{xy},\begin{xy}
(0,0)*[o]+{L_{0,n}}="0,n",
(20,0)*[o]+{L_{1,n-1}}="1,n-1",
\ar @<3pt> "0,n";"1,n-1"^{\tau}
\ar @<3pt> "1,n-1";"0,n"^{\tau}
\end{xy}
\]
\item Let $L,L'\in \add \underset{\epsilon,\epsilon'\in \{0,1\}}{\bigoplus}L_{\epsilon,n-\epsilon'}$ be indecomposable modules. Then 
$\Hom_A(L,L')\ne 0$ if and only if $L\simeq L'$.
\end{enumerate}
\end{lemma}
\begin{proof}
	(1). Let $\epsilon,\epsilon'\in \{0,1\}$.  Then the minimal injective copresentation of $L_{\epsilon,n-\epsilon'}$ is given by
	\[
	0\to L_{\epsilon,n-\epsilon'} \to I_2 \to I_{\epsilon}\oplus I_{n-\epsilon'} \to 0.
	\]
	Hence, we have
	\[
	g^{\tau^{-1}L_{\epsilon,n-\epsilon'}}=g^{L_{1-\epsilon,n-(1-\epsilon')}}=\bm{e}_\epsilon+\bm{e}_{n-\epsilon'}-\bm{e}_2.
	\] 
	In particular, we obtain
	\[
	\tau^{-1}L_{\epsilon,n-\epsilon'}\simeq L_{1-\epsilon,n-(1-\epsilon')}.
	\]
	This shows (1).
	
	(2). Let $f\in \Hom_A(L,L')$ and $f_i:L e_i\to L' e_i$ the $K$-linear map induced by $f$.
	We may assume $L=L_{0,n}$ and $L'=L_{\epsilon, n-\epsilon'}$ with $\epsilon,\epsilon'\in \{0,1\}$.
	Then we have
	\[
	f_2=\cdots=f_{n-2}.
	\]
	
	If $(\epsilon, \epsilon')=(0,1)$, then $f_0=f_{n-1}=f_n=0$, $f_1=f_2$, and   
	$f_{n-2}$ factors through $0$. Therefore, we have $f=0$.
	Similarly, we can check $f=0$ for each $(\epsilon, \epsilon')\ne (0,0)$.
	This shows the assertion.
		\end{proof}
Let 
\[\begin{array}{lll}
\mathbb{T}_{n-2}&:=&\{\tau^{-k}R_{t-1}\mid 2\le t \le n-2,\ 0\le k \le n-3\}\\
\mathbb{T}_{2}&:=&\{L_{0,n-1},\; L_{0,n},\; L_{1,n-1},\; L_{1,n}\}\\
\end{array}
\]
Then we have the following proposition.
\begin{proposition}
	\label{nonsincere}
Each nonsincere indecomposable module appears in either Lemma\;\ref{nonsincere_p}, Lemma\;\ref{nonsincere_i}, Lemma\;\ref{nonsincere_r_n-2}, or Lemma\;\ref{nonsincere_r_2}.
Moreover, 
\[\mathbb{T}_{n-2}\cup \mathbb{T}_{2}\] 
gives a complete set of representatives of isomorphism classes of indecomposable regular $\tau$-rigid modules.
\end{proposition}
\begin{proof}
By the classification of indecomposable modules of path algebra of type $\mathbf{D}$, we can check that each indecomposable
module in $\ind A_0\cup \ind A_1\cup \ind A_{n-1}\cup \ind A_n\subset \ind A$ appears in either Lemma\;\ref{nonsincere_p}, Lemma\;\ref{nonsincere_i}, Lemma\;\ref{nonsincere_r_n-2}, or Lemma\;\ref{nonsincere_r_2}. 
Then the assertion follows from Lemma\;\ref{rinmgs}. 
\end{proof}
\subsection{A proof of the inequality $\ell(Q)\le 2n^2-2n-2$}
In this subsection, we show the following inequality.
\[
\ell(Q)\le 2n^2-2n-2.
\]

Take a maximal green sequence 
\[
\omega: T_0\to \cdots \to T_\ell
\]
with length $\ell$ and set
\[
\begin{array}{cll}
\add \omega &=& \add \overset{\ell}{\underset{k=0}{\oplus}} T_k\\
\mathcal{P}_\omega&=&\add \omega \cap \mathcal{P}\\
\mathcal{I}_\omega&=&\add \omega \cap \mathcal{I}\\
\widetilde{\mathcal{I}}_\omega&=&\add \omega \cap \widetilde{I}\\
\mathcal{R}_\omega&=&\add \omega \cap \mathcal{R}\\
t&=&\max\left\{k\mid \num\left(\add T_k\cap \mathcal{P}\right)\ge 1\right\}\\
s&=&\min\{k\mid \num(\add T_k\cap \widetilde{\mathcal{I}})\ge 1\}
\end{array}
\]
Then we assume
\begin{center}
	$\tau^{-p}P_i\in \add T_t$ and $\tau^{q} P_j^-\in \add T_{s}$ with $p,q\ge 0$.
\end{center}
\begin{lemma}
	\label{length:lemma:t>=s}
In the above setting, the following statements hold.
\begin{enumerate}[{\rm (1)}]
	\item $t\ge s$. 
	\item $\tau^{-p-q}P_i\in \mod A_j$. 
	\item $\num\mathcal{P}_\omega\le (n+1)(p+1)+\num\left(\mathcal{P}\cap \ind A_i\right)$.
	\item $\num\mathcal{I}_\omega\le (n+1)q+\num\left(\mathcal{I}\cap \ind A_j\right)$.
	\item $\num\mathcal{R}_\omega\le \num(\mathcal{R}\cap \ind A_i)+\num(\mathcal{R}\cap \ind A_j)-\num\left\{R\in (\mathcal{R}\cap \ind A_i)\mid \tau^{-(p+q+1)} R\in \ind A_j \right\}$.
\end{enumerate}	
\end{lemma}
\begin{proof}
(1). Suppose $t<s$. Then the maximality of $t$ and the minimality of $s$ imply 
	\[
	\num \left(\add T_t \cap \left(\mathcal{P}\cup \widetilde{\mathcal{I}}\right)\right)=1.
	\]
	This contradicts Proposition\;\ref{noregulartilt}.
	
(2). By (1), we have $T_t\le T_s$. Therefore, $\tau^{-p}P_i\in \mathcal{P}$ and $\tau^{q} P_j^-\in \widetilde{\mathcal{I}}$ imply
that $\tau^{-p}P_i\oplus \tau^{q} P_j^-$ is $\tau$-rigid. Then it follows from Lemma\;\ref{mgs:program:relation} that
\[\tau^{-p}P_i\prec \tau^{q} P_j^-.\]
Hence we have $\left(\tau^{-p-q}P_i\right)e_j=0$.

(3) and (4). By Lemma\;\ref{mgs:program:relation} and the maximality of $t$ (resp. the minimality of $s$), each module $\tau^{-r}P_k$ in $\mathcal{P}_{\omega}$
(resp. $\tau^{r}P_k^-$ in $\mathcal{I}_{\omega}$) satisfies
\[
\tau^{-p}P_i\prec \tau^{-r}P_k\ (\text{resp. $\tau^{r}P_k^-\prec \tau^{q}P_j^-$}).
\] 
Hence one of the following conditions holds.
\begin{enumerate}[{\rm (i)}]
	\item $0\le r\le p$ (resp. $1\le r \le q$)
	\item $r>p$ and $\tau^{p+1}(\tau^{-r}P_k)\in \mathcal{P}\cap \ind A_i$ 
	(resp. $r>q$ and $\tau^{-q}(\tau^{r}P_k^-)\in \mathcal{I}\cap \ind A_j$)
\end{enumerate}
Since the condition (ii) is equivalent to $\tau^{-r}P_k\in \tau^{-(p+1)}\left(\mathcal{P}\cap \ind A_i\right)$
(resp. $\tau^{r}P_k^-\in \tau^{q}\left(\mathcal{I}\cap \ind A_j\right)$),
we obtain 
\[
\num\mathcal{P}_\omega\le (n+1)(p+1)+\num\left(\mathcal{P}\cap \ind A_i\right)\ (\text{resp. $\num\mathcal{I}_\omega\le (n+1)q+\num\left(\mathcal{I}\cap \ind A_j\right)$}).
\] 

(5). Assume $R\in \mathcal{R}\cap \add T_k$.
If $k\le t$ (resp. $k\ge s$), then we have $R\oplus \tau^{-p}P_i$ (resp. $R\oplus \tau^{q}P_j^-$) is $\tau$-rigid.
In particular, we have
\[
\begin{array}{cccl}
\tau^{-p}P_i&\prec& R & (k\le t)\\
R&\prec& \tau^{q}P_j^-& (k\ge s)\\
\end{array}
\]
by Lemma\;\ref{mgs:program:relation}.
Therefore, it follows from (1) that
\[
\mathcal{R}_\omega\subset \{R\in \mathcal{R}\mid \tau^{p+1} R\in \mod A_i\}
\cup \{R\in \mathcal{R}\mid \tau^{-q} R\in \mod A_j\}.
\]
Then the assertion follows from the following bijections.
\[
\begin{array}{rll}
\{R\in \mathcal{R}\mid \tau^{p+1} R\in \mod A_i\}&\leftrightarrow& \mathcal{R}\cap \ind A_i\\
\{R\in \mathcal{R}\mid \tau^{-q} R\in \mod A_j\}&\leftrightarrow& \mathcal{R}\cap \ind A_j\\
\{R\in \mathcal{R}\mid \tau^{p+1} R\in \mod A_i\}\cap \{R\in \mathcal{R}\mid \tau^{-q} R\in \mod A_j\}
&\leftrightarrow&\left\{R\in (\mathcal{R}\cap \ind A_i)\mid \tau^{-(p+q+1)} R\in \ind A_j \right\}
\end{array}
\]
This finishes the proof.
\end{proof}

Then we give $\num \left(\mathcal{P} \cap \ind A_i\right)$, $\num \left(\mathcal{I} \cap \ind A_i\right)$, and
$\num \left(\mathcal{R} \cap \ind A_i\right)$ explicitly.

\begin{lemma}
	\label{length:lemma:nonsincereati}
	We have the following equations.
	\begin{enumerate}[{\rm (1)}]
		\item 
		$
		\num \left(\mathcal{P} \cap \ind A_i\right)
		=
		\left\{
		\begin{array}{ll}
		2n-3 & (i\in\{0,1\})\vspace{5pt} \\
		\frac{1}{2}\left(i^2+i-6\right) & (i\in \{2,\dots,n-2\})\vspace{5pt}\\
		\frac{1}{2}\left(n^2+n-10\right) & (i\in \{n-1,n\})
		\end{array}
		\right.
		$\vspace{5pt}
		\item 
		$\num \left(\mathcal{I} \cap \ind A_i\right)
		=
		\left\{
		\begin{array}{ll}
		\frac{1}{2}(n^2-n-4) & (i\in\{0,1\})\vspace{5pt} \\
		\frac{1}{2}(i^2-(2n+1)i+n^2+n+2) & (i\in \{2,\dots, n-2\})\vspace{5pt}\\
		n & (i\in \{n-1,n\})
		\end{array}
		\right.
		$
		\vspace{5pt}
		\item 
		$
		\num \left(\left(\mathbb{T}_{n-2}\cup \mathbb{T}_2\right) \cap \ind A_i\right)
		=
		\left\{
		\begin{array}{ll}
		\frac{1}{2}(n-3)(n-2)+2 & (i\in \{0,1,n-1,n\})\\\\
		i^2-ni+\frac{1}{2}(n^2-3n+4) & ( 2 \le i\le n-2) \\
		\end{array}
		\right.
		$
		
	\end{enumerate}
\end{lemma}
\begin{proof}
	(1). By Lemma\;\ref{nonsincere_p} and Proposition\;\ref{nonsincere}, $\mathcal{P}\cap \ind A_i$ is given by the following. 
	\[
	\left\{
	\begin{array}{ll}
	\begin{array}{ll}
	&\{\tau^{-t+2}P_\epsilon \mid 2\le t \le n-1,\epsilon\in\{0,1\},\epsilon+t\in\odd\}\vspace{3pt}\\
	\sqcup&\{P_t\mid 1<t \le n\}\\ 
	\end{array}
	& (i=0)\\\\
	\begin{array}{ll}
	&\{\tau^{-t+2}P_\epsilon \mid 2\le t \le n-1,\epsilon\in\{0,1\},\epsilon+t\in \even\}\vspace{3pt}\\
	\sqcup&\{P_t\mid 1<t \le n\}\\ 
	\end{array}
	& (i=1)\\\\
	\begin{array}{ll}
	&\{\tau^{-t+2}P_\epsilon\mid 2\le t \le i-1,\epsilon\in\{0,1\}\}\vspace{3pt}\\
	\sqcup&\{\tau^{-s+1}P_{t-s+1}\mid 1\le s<t \le i-1\}
	\end{array}
	& (2\le i\le n-2)\\\\
	\begin{array}{ll}
	&\{\tau^{-t+2}P_\epsilon\mid 2\le t \le n-2,\epsilon\in\{0,1\}\}\vspace{3pt}\\
	\sqcup&\{\tau^{-s+1}P_{t-s+1}\mid 1\le s<t \le n-2\}\vspace{3pt}\\ 
	\sqcup&\{\tau^{-t+1}P_{n-\epsilon}\mid 1\le t \le n-2, \epsilon\in\{0,1\}, \epsilon+t\in \odd\}
	\end{array}
	& (i=n-1)\\\\
	\begin{array}{ll}
	&\{\tau^{-t+2}P_\epsilon\mid 2\le t \le n-2,\epsilon\in\{0,1\}\}\vspace{3pt}\\
	\sqcup&\{\tau^{-s+1}P_{t-s+1}\mid 1\le s<t \le n-2\}\vspace{3pt}\\ 
	\sqcup&\{\tau^{-t+1}P_{n-\epsilon}\mid 1\le t \le n-2, \epsilon\in\{0,1\}, \epsilon+t\in \even\}
	\end{array}
	& (i=n)\\
	\end{array}
	\right.
	\]
	If $i\in \{0,1\}$, then $\num \left(\mathcal{P} \cap \ind A_i\right)$ is given by
	\[
	(n-2)+(n-1)=2n-3.
	\]
	If $i\in \{2,\dots,n-2\}$, then $\num \left(\mathcal{P} \cap \ind A_i\right)$ is given by
	\[
	2(i-2)+\overset{i-1}{\underset{t=2}{\sum}}(t-1)=2i-4+\frac{1}{2}(i-2)(i-1)=\frac{1}{2}(i^2+i-6).
	\]
	If $i\in \{n-1,n\}$, then $\num \left(\mathcal{P} \cap \ind A_i\right)$ is given by
	\[
	2(n-3)+\overset{n-2}{\underset{t=2}{\sum}}(t-1)+(n-2)=3n-8+\frac{1}{2}(n-3)(n-2)=\frac{1}{2}(n^2+n-10).
	\]
	Therefore, the assertion (1) holds.
	
	(2). By Lemma\;\ref{nonsincere_i} and Proposition\;\ref{nonsincere}, $\mathcal{I}\cap \ind A_i$ is given by the following. 
	\[
	\left\{
	\begin{array}{ll}
	\begin{array}{ll}
	&\{\tau^{n-1-t}I_\epsilon \mid 2\le t \le n-1,\epsilon\in\{0,1\},n-1-t+\epsilon\in \odd\}\vspace{3pt}\\
	\sqcup&\{\tau^{n-2-t}I_{n-1-t+s}\mid 2\le s <t \le n-2\}\\
	\sqcup&\{\tau^{n-2-t}I_{n-\epsilon}\mid 2\le t \le n-2, \epsilon\in \{0,1\}\}\\ 
	\end{array}
	& (i=0)\\\\
	\begin{array}{ll}
	&\{\tau^{n-1-t}I_\epsilon \mid 2\le t \le n-1,\epsilon\in\{0,1\},n-1-t+\epsilon\in \even\}\vspace{3pt}\\
	\sqcup&\{\tau^{n-2-t}I_{n-1-t+s}\mid 2\le s <t \le n-2\}\\
	\sqcup&\{\tau^{n-2-t}I_{n-\epsilon}\mid 2\le t \le n-2, \epsilon\in \{0,1\}\}\\ 
	\end{array}
	& (i=1)\\\\
	\begin{array}{ll}
	&\{I_0,I_1\}\vspace{3pt}\\
	\sqcup&\{\tau^{n-2-t}I_{n-1-t+s}\mid i\le s<t \le n-2\}\\
	\sqcup&\{\tau^{n-2-t}I_{n-\epsilon}\mid i\le t \le n-2, \epsilon\in \{0,1\}\}\\ 
	\end{array}
	& (2\le i\le n-2)\\\\
	\begin{array}{ll}
	&\{I_0,I_1\}\vspace{3pt}\\
	\sqcup&\{\tau^{n-2-t}I_{n-\epsilon}\mid 1\le t \le n-2, \epsilon\in\{0,1\}, n-2-t+\epsilon\in \even\}\\
	\end{array}
	& (i=n-1)\\\\
	\begin{array}{ll}
	&\{I_0,I_1\}\vspace{3pt}\\
	\sqcup&\{\tau^{n-2-t}I_{n-\epsilon}\mid 1\le t \le n-2, \epsilon\in\{0,1\}, n-2-t+\epsilon\in \odd\}\\
	\end{array}
	& (i=n)\\
	\end{array}
	\right.
	\]
	If $i\in \{0,1\}$, then $\num \left(\mathcal{I} \cap \ind A_i\right)$ is given by
	\[
	(n-2)+\overset{n-3}{\underset{s=2}{\sum}}(n-2-s)+2(n-3)=3n-8+\frac{1}{2}(n-4)(n-3)=\frac{1}{2}(n^2-n-4).
	\]
	If $i\in \{2,\dots,n-2\}$, then $\num \left(\mathcal{I} \cap \ind A_i\right)$ is given by
	\[
	2+\overset{n-3}{\underset{s=i}{\sum}}(n-2-s)+2(n-1-i)=2n-2i+\frac{1}{2}(n-2-i)(n-1-i)=\frac{1}{2}(i^2-(2n+1)i+n^2+n+2).
	\]
	If $i\in \{n-1,n\}$, then $\num \left(\mathcal{I} \cap \ind A_i\right)$ is given by
	\[
	2+(n-2)=n.
	\]
	Therefore, the assertion (2) holds.
	
	(3). Assume that $\tau^{-k}R_{t-1}\in \mathbb{T}_{n-2}\;(t=2,3,\dots, n-2)$ is in $\mod A_i$.
	
	First we consider the case $2\le i \le n-2$.
	By Lemma\;\ref{nonsincere_r_n-2} and the definition of $R_{t-1}$, we have $0\le k \le n-t-2$ and 
	\[
	i\in [2,k+1]\cup [t+k+1,n-2].
	\]
	Equivalently, we obtain
	\[
	k\in [0,i-t-1]\cup [i-1,n-t-2]. 
	\]
	In particular, the number of elements in $\{R_{t-1}^{-k}\mid 0\le k\le n-3 \}\cap \mod A_i$ is given by
	\[
	\left\{
	\begin{array}{cll}
	(i-t)+(n-t-i) & (t<i, n-t>i)\\	
	i-t & (t<i, n-t\le i)\\
	n-t-i & (t\ge i,\ n-t>i)\\
	0 & (t\ge i,\ n-t\le i)
	\end{array}
	\right.
	\]
	Therefore, we have the following equations.
	\[
	\begin{array}{lll}
	\num\left(\mathbb{T}_{n-2} \cap \ind A_i\right)&=&
		\sum_{t=2}^{i-1}(i-t)+\sum_{t=2}^{n-i-1}(n-i-t)\\\\
	&=&\frac{1}{2}\left((i-2)(i-1)+(n-i-2)(n-i-1)\right)\\\\
	&=&i^2-ni+\frac{1}{2}(n^2-3n+4).
	\end{array}
	\]	
	Next we consider the case $i\in \{0,1,n-1,n\}$.
	By Lemma\;\ref{nonsincere_r_n-2} and the definition of $R_{t-1}$, we obtain
	\[
	k\in 
	\left\{
	\begin{array}{ll}
	\left[1,n-t-1\right] & (i=0,1)\\
	\left[0,n-t-2\right] & (i=n-1,n).\\
	\end{array}
	\right.
	\]
	Thus the number of elements in $\{R_{t-1}^{-k}\mid 0\le k\le n-3 \}\cap \mod A_i$ is given by
	\[
	n-1-t.
	\]
	Therefore, we obtain the following equations.
	\[
	\begin{array}{lll}
	\num\left(\mathbb{T}_{n-2} \cap \ind A_i\right)&=&\overset{n-2}{\underset{t=2}{\sum}}(n-1-t)\vspace{5pt}\\
	&=&\frac{1}{2}(n-3)(n-2).
	\end{array}
	\]
	
	Then the assertion follows from
	\[
	\num\left(\mathbb{T}_2\cap \ind A_i\right)=
	\left\{
	\begin{array}{cl}
	0 & (i\in \{2,\dots, n-2\})\vspace{3pt}\\
	2 & (i\in \{0,1,n-1,n\}).
	\end{array}
	\right.
	\]
	This finishes the proof.
\end{proof}

In the rest of this subsection, we give an upper bound for $\ell$.
\subsubsection{$(2\le i\le n-2$, $2\le j\le n-2)$}
We assume $2\le i\le n-2$ and $2\le j \le n-2$. 
By Lemma\;\ref{nonsincere_p}, Proposition\;\ref{nonsincere}, and Lemma\;\ref{length:lemma:t>=s} (2), there exists $(s,t)$ satisfying the following conditions.
\begin{enumerate}[{\rm (i)}]
	\item $1 \le s <t <j$ 
	\item $\tau^{-p-q}P_i\simeq \tau^{-s+1}P_{t-s+1}$ 
\end{enumerate}
The condition (ii) implies $s=p+q+1$ and $t=i+s-1=i+p+q$.
Thus it follows from (i) that
\[
j>i\text{ and }0\le p+q <j-i.
\]
\begin{lemma}
	\label{length:general:number}
We have the following inequalities.
\begin{enumerate}[{\rm (1)}]
	\item $\num \mathcal{P}_\omega\le \frac{1}{2}(i^2+i-6)+(n+1)(p+1)$\vspace{5pt}
	\item $\num \mathcal{I}_\omega\le \frac{1}{2}(j^2-(2n+1)j+n^2+n+2)+(n+1)q=\frac{1}{2}(n-j)(n-j+1)+(n+1)q+1$\vspace{5pt}
	\item $\num \mathcal{R}_\omega\le i^2+j^2-n(i+j)+n^2-3n+4$
\end{enumerate}	
\end{lemma}
\begin{proof}
These inequalities follow from Lemma\;\ref{length:lemma:t>=s} and Lemma\;\ref{length:lemma:nonsincereati}.
\end{proof}
If we let $j'=n-j$ and $x=i+j'$, we have $0\le p+q<j-i=n-x$ and
\[
\begin{array}{lll}
\ell(Q)&=&\num\mathcal{P}_\omega+\num\mathcal{I}_\omega+\num\mathcal{R}_\omega\vspace{4pt}\\
       &\le& \frac{1}{2}(i^2+i+j'^2+j')+i^2+(n-j')^2-n(i+n-j')+n^2-3n+4+(n+1)(p+q+1)-2\vspace{4pt}\\
       &=& \frac{3}{2}(i^2+j'^2)+\frac{1}{2}(i+j')-2nj'-ni+nj'+n^2-3n+2+(n+1)(p+q+1)\vspace{4pt}\\
       &\le& \frac{3}{2}x^2 -3ij'+\frac{1}{2}x-nx+n^2-3n+2+(n+1)(n-x)\vspace{4pt}\\
       &=& \frac{3}{2}x^2-\frac{1}{2}(4n+1)x+2n^2-2n+2-3ij'\vspace{4pt}\\
       &\le& \frac{3}{2}x^2-\frac{1}{2}(4n+1)x+2n^2-2n-10\vspace{4pt}\\
       &= &\frac{3}{2}\left(x-\frac{1}{6}(4n+1)\right)^2+2n^2-2n-10-\frac{1}{24}(4n+1)^2\vspace{4pt}\\
       &=:&f(x).
\end{array}
\]
Since $n>x=i+j'\ge 4 > 0$, we obtain
\[
\ell(Q)\le \max\{f(0),f(n)\}=\max\left\{2n^2-2n-10,\ \frac{1}{2}(3n^2-5n)-10 \right\}< 2n^2-2n-2.
\]

\subsubsection{$(2\le i\le n-2$, $j\in \{0,1\})$}
We assume $2\le i\le n-2$ and $j\in \{0,1\}$.
By Lemma\;\ref{nonsincere_p}, Proposition\;\ref{nonsincere}, and Lemma\;\ref{length:lemma:t>=s} (2), there exists 
$t\in \{2,3,\dots,n-2\}$ such that
\[
\tau^{-p-q}P_i\simeq P_t.
\]
Hence $p=q=0$ and $i=t$.
\begin{lemma}
	\label{length:lemma:j=0:number}
	We have the following inequalities.
	\begin{enumerate}[{\rm (1)}]
		\item $\num \mathcal{P}_\omega\le \frac{1}{2}(i^2+i-6)+(n+1)$\vspace{5pt}
		\item $\num \mathcal{I}_\omega\le \frac{1}{2}(n^2-n-4)$\vspace{5pt}
		\item $\num \mathcal{R}_\omega\le i^2-ni+n^2-4n+7$
	\end{enumerate}	
\end{lemma}
\begin{proof}
	These inequalities follow from Lemma\;\ref{length:lemma:t>=s} and Lemma\;\ref{length:lemma:nonsincereati}.
\end{proof}
Then we have
\[
\begin{array}{lll}
	\ell(Q)&=&\num\mathcal{P}_\omega+\num\mathcal{I}_\omega+\num\mathcal{R}_\omega\vspace{4pt}\\
	&\le& \frac{3}{2}i^2-\frac{1}{2}(2n-1)i+\frac{1}{2}(3n^2-7n)+3\vspace{4pt}\\
	&= &\frac{3}{2}\left(i-\frac{1}{6}(2n-1)\right)^2+\frac{1}{2}(3n^2-7n)+3-\frac{1}{24}(2n-1)^2\vspace{4pt}\\
	&=:&f(i).
\end{array}
\]
Since $0< i<n $ and $n\ge 5$, we obtain
\[
\ell(Q)<  \max\{f(0),f(n)\}=\max\left\{\frac{3}{2}n^2-\frac{7}{2}n+3,\ 2n^2-3n+3 \right\}\le 2n^2-2n-2.
\]

\subsubsection{$(2\le i\le n-2$, $j\in \{n-1,n\})$}
We assume $2\le i\le n-2$ and $j\in\{n-1,n\}$.
By Lemma\;\ref{nonsincere_p}, Proposition\;\ref{nonsincere}, and Lemma\;\ref{length:lemma:t>=s} (2), there exists $(s,t)$ satisfying the following conditions.
\begin{enumerate}[{\rm (i)}]
	\item $1 \le s <t \le n-2$. 
	\item $\tau^{-p-q}P_i\simeq \tau^{-s+1}P_{t-s+1}$. 
\end{enumerate}
The condition (ii) implies $s=p+q+1$ and $t=i+s-1=i+p+q$.
Thus it follows from (i) that
\[
0\le p+q \le n-2-i.
\]
\begin{lemma}
	\label{length:lemma:j=n:number}
	We have the following inequalities.
	\begin{enumerate}[{\rm (1)}]
		\item $\num \mathcal{P}_\omega\le \frac{1}{2}(i^2+i-6)+(n+1)(p+1)$\vspace{5pt}
		\item $\num \mathcal{I}_\omega\le n+(n+1)q$\vspace{5pt}
		\item $\num \mathcal{R}_\omega\le i^2-ni+n^2-4n+7$
	\end{enumerate}	
\end{lemma}
\begin{proof}
	These inequalities follow from Lemma\;\ref{length:lemma:t>=s} and Lemma\;\ref{length:lemma:nonsincereati}.
\end{proof}
Then we have the following inequalities.
\[
\begin{array}{lll}
\ell(Q)&=&\num\mathcal{P}_\omega+\num\mathcal{I}_\omega+\num\mathcal{R}_\omega\vspace{4pt}\\
&\le& \frac{3}{2}i^2-\frac{1}{2}(2n-1)i+(n+1)(p+q+1)+n^2-3n+4\vspace{4pt}\\
&\le& \frac{3}{2}i^2-\frac{1}{2}(2n-1)i+(n+1)(n-i-1)+n^2-3n+4\vspace{4pt}\\
&=& \frac{3}{2}i^2-\frac{1}{2}(4n+1)i+2n^2-3n+3\vspace{4pt}\\
&= &\frac{3}{2}\left(i-\frac{1}{6}(4n+1)\right)^2+2n^2-3n+3-\frac{1}{24}(4n+1)^2\vspace{4pt}\\
&=:&f(i)
\end{array}
\]
Since $0< i <n $ and $5\le n$, we obtain
\[
\ell(Q)< \max\{f(0),f(n)\}=\max\left\{2n^2-3n+3,\ \frac{3}{2}n^2-\frac{7}{2}n+3 \right\}\le  2n^2-2n-2.
\]

\subsubsection{$(i\in \{0,1\},\ 2\le j\le n-2)$}
We assume $i\in \{0,1\}$ and $2\le j\le n-2$.
Without loss of generality, we may assume $i=0$.
By Lemma\;\ref{nonsincere_p}, Proposition\;\ref{nonsincere}, and Lemma\;\ref{length:lemma:t>=s} (2), there exists $t\in \{2,\dots, j-1\}$ such that
\[
\tau^{-p-q}P_0\simeq \tau^{-t+2}P_0.
\]
Hence, we have $j\ge 3$ and 
\[
0\le t-2=p+q\le j-3.
\]
\begin{lemma}
	\label{length:lemma:i=0:number}
	We have the following inequalities.
	\begin{enumerate}[{\rm (1)}]
		\item $\num \mathcal{P}_\omega\le 2n-3+(n+1)(p+1)$\vspace{5pt}
		\item $\num \mathcal{I}_\omega\le \frac{1}{2}(j^2-(2n+1)j+n^2+n+2)+(n+1)q$\vspace{5pt}
		\item $\num \mathcal{R}_\omega\le j^2-nj+n^2-4n+7$
	\end{enumerate}	
\end{lemma}
\begin{proof}
	These inequalities follow from Lemma\;\ref{length:lemma:t>=s} and Lemma\;\ref{length:lemma:nonsincereati}.
\end{proof}
Then we have the following inequalities.
\[
\begin{array}{lll}
\ell(Q)&=&\num\mathcal{P}_\omega+\num\mathcal{I}_\omega+\num\mathcal{R}_\omega\vspace{4pt}\\
&\le& \frac{3}{2}j^2-\frac{1}{2}(4n+1)j+(n+1)(p+q+1)+\frac{3}{2}n^2-\frac{3}{2}n+5\vspace{4pt}\\
&\le& \frac{3}{2}j^2-\frac{1}{2}(4n+1)j+(n+1)(j-2)+\frac{3}{2}n^2-\frac{3}{2}n+5\vspace{4pt}\\
&=& \frac{3}{2}j^2-\frac{1}{2}(2n-1)j+\frac{3}{2}n^2-\frac{7}{2}n+3\vspace{4pt}\\
&= &\frac{3}{2}\left(j-\frac{1}{6}(2n-1)\right)^2+\frac{3}{2}n^2-\frac{7}{2}n+3-\frac{1}{24}(2n-1)^2\vspace{4pt}\\
&=:&f(j)
\end{array}
\]
Since $0< j <n $ and $5\le n$, we obtain
\[
\ell(Q)< \max\{f(0),f(n)\}=\max\left\{\frac{3}{2}n^2-\frac{7}{2}n+3,\ 2n^2-3n+3 \right\}\le   2n^2-2n-2.
\]
\subsubsection{$(i,j\in \{0,1\})$}
We assume $i,j\in \{0,1\}$.
Without loss of generality, we may assume $i=0$.
By Lemma\;\ref{nonsincere_p}, Proposition\;\ref{nonsincere}, and Lemma\;\ref{length:lemma:t>=s} (2), there exists $t\in \{2,\dots, n-1\}$ such that
\begin{center}
$\tau^{-p-q}P_0\simeq \tau^{-t+2}P_0$ and $t\equiv j+1\;(\mod 2)$.
\end{center}
Hence, we have 
\begin{center}
$0\le p+q=t-2 \le n-3$ and $p+q=t-2\equiv j+1\;(\mod 2)$.
\end{center}
\begin{lemma}
	\label{length:lemma:i=j=0:number}
	Let $r=p+q$. Then we have the following inequalities.
	\begin{enumerate}[{\rm (1)}]
		\item $\num \mathcal{P}_\omega\le 2n-3+(n+1)(p+1)$\vspace{5pt}
		\item $\num \mathcal{I}_\omega\le \frac{1}{2}(n^2-n-4)+(n+1)q$\vspace{5pt}
		\item $\num \mathcal{R}_\omega\le \frac{1}{2}(n^2-3n+4)-\left(r^2+(4-n)r\right)$
	\end{enumerate}	
\end{lemma}
\begin{proof}
(1) and (2). The assertions (1) and (2) follow Lemma\;\ref{length:lemma:t>=s} and Lemma\;\ref{length:lemma:nonsincereati}.

(3). We calculate 
\[
\num\left\{R\in (\mathbb{T}_{n-2}\cap \ind A_0)\mid \tau^{-(r+1)} R\in \ind A_j \right\}.
\]
Assume that $R:=\tau^{-k}R_{u-1}$ ($0\le k \le n-3$, $2\le u \le n-2$) satisfies the following conditions.
\begin{itemize}
	\item $R\in \ind A_0$.
	\item $\tau^{-(r+1)} R\in \ind A_j$. 
\end{itemize} 
Since $r+1\le n-2$ and $R\in \mathbb{T}_{n-2}$, $R$ satisfies the above conditions if and only if
\[
k\in [1, n-2-u-r]\cup[n-2-r, n-1-u].
\]
 Then we have
\[
\begin{array}{rll}
\num[1, n-2-u-r]&=&
\left\{
\begin{array}{ll}
n-2-r-u & (2\le u\le n-2-r)\\
0 & (\text{otherwise})\\
\end{array}
\right.\vspace{5pt}\\
\num[n-2-r, n-u-1]&=&
\left\{
\begin{array}{ll}
r+2-u & (2\le u\le r+1)\\
0 & (\text{otherwise})
\end{array}	
\right.\\
\end{array}
\]
Therefore, we have
\[
\begin{array}{ll}
&\num\left\{R\in (\mathbb{T}_{n-2}\cap \ind A_0)\mid \tau^{-(r+1)} R\in \ind A_j \right\}\vspace{5pt}\\
=&
\overset{n-2-r}{\underset{u=2}{\sum}}(n-2-r-u)+\overset{r+1}{\underset{u=2}{\sum}}(r+2-u)
\vspace{5pt}\\
=&\frac{1}{2}(n-4-r)(n-3-r)+\frac{1}{2}(r)(r+1)\vspace{5pt}\\
=&r^2+(4-n)r+\frac{1}{2}(n^2-7n+12).\\
\end{array}
\]

Furthermore, it follows from Lemma\;\ref{nonsincere_r_2} and $r+1\equiv j\;(\mod 2)$, we have
\[
\num\left\{R\in (\mathbb{T}_{2}\cap \ind A_0)\mid \tau^{-(r+1)} R\in \ind A_j \right\}=2.
\]

Therefore, it follows from  Lemma\;\ref{length:lemma:t>=s} and Lemma\;\ref{length:lemma:nonsincereati} that
\[
\begin{array}{lll}
\num \mathcal{R}_\omega&\le& n^2-5n+10-\num\left\{R\in (\mathcal{R}\cap \ind A_0)\mid \tau^{-(r+1)} R\in \ind A_j \right\}\vspace{5pt}\\
&=&\frac{1}{2}(n^2-3n+4)-\left(r^2+(4-n)r\right).
\end{array}
\] 
This finishes the proof.
\end{proof}
	
Let $r=p+q$.
Then $0\le r \le n-3$ and we have the following inequalities.
\[
\begin{array}{lll}
\ell(Q)&=&\num\mathcal{P}_\omega+\num\mathcal{I}_\omega+\num\mathcal{R}_\omega\vspace{4pt}\\
&\le& n^2-3+(n+1)(r+1)-\left(r^2+(4-n)r\right)\vspace{4pt}\\
&=& n^2+n-2-\left(r^2-(2n-3)r\right)\vspace{4pt}\\
&= & n^2+n-2-\left(\left(r-(n-\frac{3}{2})\right)^2-\left(n^2-3n+\frac{9}{4}\right)\right)\vspace{4pt}\\
&=:&f(r)
\end{array}
\]
Since $r \le n-3 $, we obtain
\[
\ell(Q)\le f(n-3)=  2n^2-2n-2.
\]
\subsubsection{$(i\in\{0,1\},\ j\in \{n-1,n\})$}
We assume $i\in \{0,1\}$ and $j\in \{n-1,n\}$.
Without loss of generality, we may assume $i=0$ and $j=n$.
By Lemma\;\ref{nonsincere_p}, Proposition\;\ref{nonsincere}, and Lemma\;\ref{length:lemma:t>=s} (2), there exists $t\in \{2,\dots, n-2\}$ such that
\begin{center}
	$\tau^{-p-q}P_0\simeq \tau^{-t+2}P_0$.
\end{center}
Hence we have
\[
0\le p+q=t-2\le n-4.
\]

\begin{lemma}
	\label{length:lemma:i=0j=n:number}
	We have the following inequalities.
	\begin{enumerate}[{\rm (1)}]
		\item $\num \mathcal{P}_\omega\le 2n-3+(n+1)(p+1)$\vspace{5pt}
		\item $\num \mathcal{I}_\omega\le n+(n+1)q$\vspace{5pt}
		\item $\num \mathcal{R}_\omega\le n^2-5n+10$
	\end{enumerate}	
\end{lemma}
\begin{proof}
	These inequalities follow from Lemma\;\ref{length:lemma:t>=s} and Lemma\;\ref{length:lemma:nonsincereati}.
\end{proof}
Let $r=p+q$.
Then $0\le r \le n-4$. Since $45\le n$, we have the following inequalities.
\[
\begin{array}{lll}
\ell(Q)&=&\num\mathcal{P}_\omega+\num\mathcal{I}_\omega+\num\mathcal{R}_\omega\vspace{4pt}\\
&\le& n^2-2n+7+(n+1)(r+1)\vspace{4pt}\\
&\le& n^2-2n+7+(n+1)(n-3)\vspace{4pt}\\
&= & 2n^2-4n+4\vspace{4pt}\\
&< & 2n^2-2n-2
\end{array}
\]
\subsubsection{$(i\in\{n-1,n\},\ j\in\{2,\dots,n-2\})$}
By Lemma\;\ref{nonsincere_p} and Proposition\;\ref{nonsincere}, we have 
\[
\{\tau^{-t}P_{i}\mid i\in \{n-1,n\},\  t\in \Z_{\ge 0}\}\cap \left(\mod A_j\right)=\emptyset
\]
for all $j\in \{2,\dots,n-2\}$.
Then it follows from Lemma\;\ref{length:lemma:t>=s} (2) that
 \[(i,j)\not \in \{n-1,n\}\times \{2,\dots,n-2\}.\]

\subsubsection{$(i\in \{n-1,n\},\ j\in \{0,1\})$}
We assume $i\in \{n-1,n\}$ and $j\in \{0,1\}$.
Without loss of generality, we may assume $i=n$ and $j=0$.
By Lemma\;\ref{nonsincere_p}, Proposition\;\ref{nonsincere}, and Lemma\;\ref{length:lemma:t>=s} (2), we have
\begin{center}
	$\tau^{-p-q}P_n\simeq P_n$.
\end{center}
Hence we have
\[
p=q=0.
\]

\begin{lemma}
	\label{length:lemma:i=nj=0:number}
	We have the following inequalities.
	\begin{enumerate}[{\rm (1)}]
		\item $\num \mathcal{P}_\omega\le \frac{1}{2}(n^2+n-10)+(n+1)$\vspace{5pt}
		\item $\num \mathcal{I}_\omega\le \frac{1}{2}(n^2-n-4)$\vspace{5pt}
		\item $\num \mathcal{R}_\omega\le n^2-5n+10$
	\end{enumerate}	
\end{lemma}
\begin{proof}
	These inequalities follow from Lemma\;\ref{length:lemma:t>=s} and Lemma\;\ref{length:lemma:nonsincereati}.
\end{proof}
Since $5\le n$, we have the following inequalities.
\[
\begin{array}{lll}
\ell(Q)&=&\num\mathcal{P}_\omega+\num\mathcal{I}_\omega+\num\mathcal{R}_\omega\vspace{4pt}\\
&\le& 2n^2-4n+4\vspace{4pt}\\
&<& 2n^2-2n-2
\end{array}
\]

\subsubsection{$(i,j\in \{n-1,n\})$}
We assume $i,j\in \{n-1,n\}$.
Without loss of generality, we may assume $i=n$.
By Lemma\;\ref{nonsincere_p}, Proposition\;\ref{nonsincere}, and Lemma\;\ref{length:lemma:t>=s} (2), there exists $t\in \{1,\dots, n-2\}$ such that
\begin{center}
	$\tau^{-p-q}P_n\simeq \tau^{-t+1}P_n$ and $t\equiv n-j\;(\mod 2)$.
\end{center}
Hence, we have 
\begin{center}
	$0\le p+q=t-1 \le n-3$ and $p+q=t-1\equiv n-j+1\;(\mod 2)$.
\end{center}
\begin{lemma}
	\label{length:lemma:i=j=n:number}
	Let $r=p+q$. Then we have the following inequalities.
	\begin{enumerate}[{\rm (1)}]
		\item $\num \mathcal{P}_\omega\le \frac{1}{2}(n^2+n-10)+(n+1)(p+1)$\vspace{5pt}
		\item $\num \mathcal{I}_\omega\le n+(n+1)q$\vspace{5pt}
		\item $\num \mathcal{R}_\omega
		                     \le \frac{1}{2}n^2-\frac{3}{2}n+2-\left(r^2+(4-n)r\right)$

	\end{enumerate}	
\end{lemma}
\begin{proof}
	(1) and (2). The assertions (1) and (2) follow from Lemma\;\ref{length:lemma:t>=s} and Lemma\;\ref{length:lemma:nonsincereati}.
	
%
	(3). We calculate 
		\[
		\num\left\{R\in (\mathbb{T}_{n-2}\cap \ind A_n)\mid \tau^{-(r+1)} R\in \ind A_j) \right\}.
		\]
				
		Assume that $R:=\tau^{-k}R_{u-1}$ ($0\le k \le n-3$, $2\le u \le n-2$) satisfies the following conditions.
		\begin{itemize}
			\item $R\in \ind A_n$
			\item $\tau^{-(r+1)} R\in \ind A_j$ 
		\end{itemize} 
		Since $r+1\le n-2$ and $R\in \mathbb{T}_{n-2}$, $R$ satisfies the above conditions if and only if
		\[
		k\in [0, n-u-2-(r+1)]\cup[n-2-(r+1), n-u-2]=[0, n-3-r-u]\cup[n-3-r, n-2-u].
		\]
		Then we have
		\[
		\begin{array}{lll}
		\num[0, n-3-r-u]&=&
		\left\{
		\begin{array}{ll}
		n-2-r-u & (2\le u\le n-2-r)\\
		0 & (\text{otherwise})\\
		\end{array}
		\right.\vspace{5pt}\\
		\num[n-3-r, n-2-u]&=&
		\left\{
		\begin{array}{ll}
		r+2-u & (2\le u\le r+1)\\
		0 & (\text{otherwise})
		\end{array}	
		\right.\\
		\end{array}
		\]
		
			Then we have
		\[
		\begin{array}{ll}
		&\num\left\{R\in (\mathbb{T}_{n-2}\cap \ind A_n)\mid \tau^{-(r+1)} R\in \ind A_j \right\}\vspace{5pt}\\
		=&
		\overset{n-2-r}{\underset{u=2}{\sum}}(n-2-r-u)+\overset{r+1}{\underset{u=2}{\sum}}(r+2-u)
		\vspace{5pt}\\
		=&\frac{1}{2}(n-4-r)(n-3-r)+\frac{1}{2}(r)(r+1)\vspace{5pt}\\
		=&r^2+(4-n)r+\frac{1}{2}(n^2-7n+12).\\
		\end{array}
		\]

	Furthermore, it follows from Lemma\;\ref{nonsincere_r_2} and $r+1\equiv n-j\;(\mod 2)$, we have
\[
\num\left\{R\in (\mathbb{T}_{n-2}\cap \ind A_n)\mid \tau^{-(r+1)} R\in \ind A_j \right\}
=2.
\]
	Therefore, it follows from Lemma\;\ref{length:lemma:t>=s} and Lemma\;\ref{length:lemma:nonsincereati} that
		\[
		\begin{array}{lll}
		\num \mathcal{R}_\omega&\le& n^2-5n+10-\num\left\{R\in (\mathcal{R}\cap \ind A_n)\mid \tau^{-(r+1)} R\in \ind A_j \right\}\vspace{5pt}\\
		&=&\frac{1}{2}n^2-\frac{3}{2}n+2-\left(r^2+(4-n)r\right).
		\end{array}
		\] 
		This finishes the proof.
\end{proof}
Let $r=p+q$.
Then $0\le r \le n-3$ and we have the following inequalities.
\[
\begin{array}{lll}
\ell(Q)&=&\num\mathcal{P}_\omega+\num\mathcal{I}_\omega+\num\mathcal{R}_\omega\vspace{4pt}\\
&\le& n^2-3+(n+1)(r+1)-\left(r^2+(4-n)r\right)\vspace{4pt}\\
&=& n^2+n-2-\left(r^2-(2n-3)r\right)\vspace{4pt}\\
&= & n^2+n-2-\left(\left(r-(n-\frac{3}{2})\right)^2-\left(n^2-3n+\frac{9}{4}\right)\right)\vspace{4pt}\\
&=:&f(r)
\end{array}
\]
Since $r \le n-3 $, we obtain
\[
\ell(Q)\le f(n-3)=  2n^2-2n-2.
\]
\subsection{Construction of a maximal green sequence for $Q$ with length $2n^2-2n-2$}
In this subsection, we construct a maximal green sequence for $Q$ with length $2n^2-2n-2$ as follows.
\begin{enumerate}
	\item Construct a path $\omega^{(0)}$ with length $n^2-2n-3$ such that $s(\omega^{(0)})=A$.
	\item Construct a path $\omega^{(1)}$ with length $\frac{1}{2}(n^2+n-12)$ such that $s(\omega^{(1)})=t(\omega^{(0)})$.
	\item Construct a path $\omega^{(2)}$ with length $n-2$ such that $s(\omega^{(2)})=t(\omega^{(1)})$.
	\item Construct a path $\omega^{(3)}$ with length $6$ such that $s(\omega^{(3)})=t(\omega^{(2)})$.
	\item Construct a path $\omega^{(4)}$ with length $\frac{1}{2}(n^2-7n+12)$ such that $s(\omega^{(4)})=t(\omega^{(3)})$.
	\item Construct a path $\omega^{(5)}$ with length $n-3$ such that $s(\omega^{(5)})=t(\omega^{(4)})$.
	\item Construct a path $\omega^{(6)}$ with length $n-3$ such that $s(\omega^{(6)})=t(\omega^{(5)})$.
	\item Construct a path $\omega^{(7)}$ with length $3$ such that $s(\omega^{(7)})=t(\omega^{(6)})$ and $t(\omega^{(7)})=A^-$.
	\item Concatenate above paths.
\end{enumerate}
In this subsection, for a module $T=T'\oplus X$, we often write $\dfrac{T}{X}=T'$.
\subsubsection{Some ($\tau$-)rigid modules in $\add \mathcal{R}$}
\label{slice}
Before construct $\omega^{(0)},\dots,\omega^{(7)}$, we define some $\tau$-rigid modules in $\add \mathcal{R}$.

Denote by $C$ the following quiver.
\[\begin{array}{lrl}
C_0&:=&\{1,\dots,n-3\}\times \left(\Z/(n-2)\Z\right)\\
C_1&:=&\{a_{(t,k)}:(t,k)\to (t+1,k)\mid (t,k)\in C_0 \}\\
&\cup&\{b_{(t,k)}:(t,k)\to (t-1,k+1)\mid (t,k)\in C_0\}\\
\end{array}
\]
We denote by $\mathcal{S}$ the set of connected full subquivers of $C$ defined as follows: $\Sigma\in \mathcal{S}$ if and only if 
for any $t\in \{1,\dots, n-3\}$, there exists a unique $k\in \Z/(n-2)\Z$ such that $(t,k)\in \Sigma_0$.

For $\Sigma\in \mathcal{S}$, we define a basic module $R(\Sigma)\in \add \mathbb{T}_{n-2}$ as follows.
\[R(\Sigma):=\bigoplus_{(t,k)\in \Sigma_0}\tau^{-k}R_t.\]
(Since $\tau^{n-2}R_t\simeq R_t$, we can naturally define $\tau^{-k} R_t$ for $k\in \Z/(n-2)\Z$.) 

Then we have the following statement. 
\begin{lemma}
	\label{silce module}
	$R(\Sigma)$ is $\tau$-rigid for any $\Sigma\in \mathcal{S}$.
\end{lemma}

\begin{proof}
	(1). Let $\Sigma\in \mathcal{S}$ and $(t,k), (t',k')\in \Sigma_0$. 
	We show 
	\begin{center}
		$\tau^{-k}R_t\oplus \tau^{-k'}R_{t'}$ is $\tau$-rigid. 
	\end{center}	
	
	We may assume that $t<t'$ and denote by $\Sigma'$ the full subquiver of $\Sigma$
	such that 
	\begin{center} $(t'',k'')\in \Sigma'\Leftrightarrow (t'',k'')\in \Sigma$ and $ t\le t'' \le t'$ \end{center}
	(or equivalently,
	$\Sigma'$ is the minimal connected full subquiver of $\Sigma$ containing $(t,k)$ and $(t',k')$).
	
	We set 
	\[\begin{array}{lll}
	n_a&:=&\num(\mathcal{A}\cap \Sigma'_1)\\
	n_b&:=&\num(\mathcal{B}\cap \Sigma'_1)\\ 
	\end{array}
	\]
	where, $\mathcal{A}:=\{a_{(t,k)}\mid (t,k)\in C_0 \}$ and $\mathcal{B}:=\{b_{(t,k)}\mid (t,k)\in C_0\}$.
	Then $t'-t=n_a+n_b$, $k-n_b=k'$, and there is the following full subquiver of $C$.
	\[(t,k)\stackrel{b_{n_b}}{\longleftarrow}\cdots \stackrel{b_1}{\longleftarrow}
	(t+n_b, k-n_b)\stackrel{a_1}{\longrightarrow}\cdots \stackrel{a_{n_a}}{\longrightarrow} (t+n_a+n_b, k-n_b)=(t',k'). \]
	For indecomposable regular modules $R$ and $R'$, the direct sum $R\oplus R'$ is $\tau$-rigid if and only if $\tau R \oplus \tau R'$ is $\tau$-rigid. Therefore, we
	may assume $k'=0$, $k=n_b$, and $t'=t+n_a+n_b$. Then it is sufficient to show $\tau^{-n_b}R_t\oplus R_{t+n_a+n_b}$ is $\tau$-rigid.

	It is easy to verify that $\Hom_A(\tau^{-1}R_s, R_{s'})=0$ for any $1\le s,s'\le n-3$.
	In particular, $R_s\oplus R_{s'}$ is $\tau$-rigid for any $1\le s,s'\le n-3$. Therefore, we may assume that $n_b\ge 1$.
	
	Since $1\le t+n_a+n_b=t'\le n-3$, we have
	$n-(t+1)-2-n_b=n-t-n_b-3\ge n_a\ge 0.$
	Hence we obtain
	\[1\le n_b \le n-(t+1)-2.\]
	Then, by Lemma\;\ref{nonsincere_r_n-2}, $\tau^{-n_b} R_t$ is given by the following quiver representation.
	\[
	\begin{xy}
	(0,5)*[o]+{0}="0", 
	(0,-5)*[o]+{0}="1",
	(10,0)*[o]+{0}="2",  
	(25,0)*[o]+{\cdots}="cdots", 
	(40,0)*[o]+{0}="s",
	(40,-5)*[o]+{n_b+1}="s'",
	(50,0)*[o]+{K}="s+1",
	(65,0)*[o]+{\cdots}="cdots2",
	(80,0)*[o]+{K}="t",
	(80,-5)*[o]+{t+1+n_b}="t'",
	(90,0)*[o]+{0}="t+1",
	(105,0)*[o]+{\cdots}="cdots3",
	(120,0)*[o]+{0}="n-2",
	(130,5)*[o]+{0}="n-1", 
	(130,-5)*[o]+{0}="n",
	\ar "0";"2"
	\ar "1";"2"
	\ar "cdots";"2"
	\ar "s";"cdots"
	\ar "s+1";"s"
	\ar "cdots2";"s+1"_{\id}
	\ar "t";"cdots2"_{\id}
	\ar "t+1";"t"
	\ar "cdots3";"t+1"
	\ar "n-2";"cdots3"
	\ar "n-1";"n-2"
	\ar "n";"n-2"
	\end{xy}
	\] 
	
	It is sufficient to show 
	\[\Hom_A(R_{t+n_a+n_b},\tau(\tau^{-n_b} R_t))=0=\Hom_A(\tau^{-1}(\tau^{-n_b}R_t), R_{t+n_a+n_b}).\] 
	It follows from Lemma\;\ref{nonsincere_r_n-2} that
	\[\begin{array}{ccl}
	\tau^{-n_b+1}R_t&\simeq& \left\{
	\begin{array}{cl}
	R_t& (n_b=1)	\\
	\begin{xy}
	(3,5)*[o]+{0}="0", 
	(3,-5)*[o]+{0}="1",
	(10,0)*[o]+{0}="2",  
	(22,0)*[o]+{\cdots}="cdots", 
	(34,0)*[o]+{0}="s",
	(34,-5)*[o]+{n_b}="s'",
	(43,0)*[o]+{K}="s+1",
	(57,0)*[o]+{\cdots}="cdots2",
	(72,0)*[o]+{K}="t",
	(72,-5)*[o]+{t+n_b}="t'",
	(82,0)*[o]+{0}="t+1",
	(94,0)*[o]+{\cdots}="cdots3",
	(107,0)*[o]+{0}="n-2",
	(114,5)*[o]+{0}="n-1", 
	(114,-5)*[o]+{0}="n",
	\ar "0";"2"
	\ar "1";"2"
	\ar "cdots";"2"
	\ar "s";"cdots"
	\ar "s+1";"s"
	\ar "cdots2";"s+1"_{\id}
	\ar "t";"cdots2"_{\id}
	\ar "t+1";"t"
	\ar "cdots3";"t+1"
	\ar "n-2";"cdots3"
	\ar "n-1";"n-2"
	\ar "n";"n-2"
	\end{xy}&
	(n_b\ge 2)\\	
	\end{array}
	\right.\\
	\tau^{-n_b-1}R_t&\simeq& \left\{
	\begin{array}{cl}
	\begin{xy}
	(3,5)*[o]+{0}="0", 
	(3,-5)*[o]+{0}="1",
	(10,0)*[o]+{K}="2",  
	(25,0)*[o]+{\cdots}="cdots", 
	(40,0)*[o]+{K}="t",
	(40,-5)*[o]+{n-t-1}="t'",
	(53,0)*[o]+{K^2}="t+1",
	(65,0)*[o]+{\cdots}="cdots2", 
	(80,0)*[o]+{K^2}="n-2",
	(87,5)*[o]+{K}="n-1", 
	(87,-5)*[o]+{K}="n",
	\ar "0";"2"
	\ar "1";"2"
	\ar "cdots";"2"_{\id}
	\ar "t";"cdots"_{\id}
	\ar "t+1";"t"_{\delta}
	\ar "cdots2";"t+1"_{\id}
	\ar "n-2";"cdots2"_{\id}
	\ar "n-1";"n-2"_{\alpha}
	\ar "n";"n-2"^{\beta}
	\end{xy}
	&
	(t+n_b=n-3)\\
	\begin{xy}
	(3,5)*[o]+{0}="0", 
	(3,-5)*[o]+{0}="1",
	(10,0)*[o]+{0}="2",  
	(22,0)*[o]+{\cdots}="cdots", 
	(34,0)*[o]+{0}="s",
	(34,-5)*[o]+{n_b+2}="s'",
	(43,0)*[o]+{K}="s+1",
	(57,0)*[o]+{\cdots}="cdots2",
	(72,0)*[o]+{K}="t",
	(72,-5)*[o]+{t+n_b+2}="t'",
	(82,0)*[o]+{0}="t+1",
	(94,0)*[o]+{\cdots}="cdots3",
	(107,0)*[o]+{0}="n-2",
	(114,5)*[o]+{0}="n-1", 
	(114,-5)*[o]+{0}="n",
	\ar "0";"2"
	\ar "1";"2"
	\ar "cdots";"2"
	\ar "s";"cdots"
	\ar "s+1";"s"
	\ar "cdots2";"s+1"_{\id}
	\ar "t";"cdots2"_{\id}
	\ar "t+1";"t"
	\ar "cdots3";"t+1"
	\ar "n-2";"cdots3"
	\ar "n-1";"n-2"
	\ar "n";"n-2"
	\end{xy}&
	(t+n_b\le n-4)\\	
	\end{array}
	\right.\\
	\end{array}
	\] 
	
	Note that $R_{t+n_a+n_b}$ is given by the following quiver representation.
	\[\begin{xy}
	(0,5)*[o]+{K}="0", 
	(0,-5)*[o]+{K}="1",
	(10,0)*[o]+{K}="2",  
	(25,0)*[o]+{\cdots}="cdots", 
	(40,0)*[o]+{K}="t",
	(40,-5)*[o]+{t+1}="t'",
	(38,-10)*[o]+{+n_a+n_b}="t''",
	(50,0)*[o]+{0}="t+1",
	(65,0)*[o]+{\cdots}="cdots2", 
	(80,0)*[o]+{0}="n-2",
	(90,5)*[o]+{0}="n-1", 
	(90,-5)*[o]+{0}="n",
	\ar "0";"2"^{\id}
	\ar "1";"2"_{\id}
	\ar "cdots";"2"_{\id}
	\ar "t";"cdots"_{\id}
	\ar "t+1";"t"
	\ar "cdots2";"t+1"
	\ar "n-2";"cdots2"
	\ar "n-1";"n-2"
	\ar "n";"n-2"
	\end{xy}\] 
	Then it follows from $t+n_a+n_b+1>t+n_b$ that
	\[\Hom_A(R_{t+n_a+n_b},\tau(\tau^{-n_b} R_t))=0.\]
	
	Thus we want to show $\Hom_A(\tau^{-1}(\tau^{-n_b}R_t), R_{t+n_a+n_b})=0$.
	If $t+n_b\le n-4$, then we can easily check the assertion.
	Therefore, we assume $t+n_b=n-3$. In this case, we have $n_a=0$ and $t+1+n_a+n_b=n-2$.
	Let $f\in \Hom_{A}(\tau^{-n_b-1}R_t, R_{n-3})$. For each $i\in Q_0$, we set $f_i:(\tau^{-n_b-1}R_t)e_i\to (R_{n-3})e_i$ the $K$-linear map induced by $f$.
	
	If $t=1$, then 
	we have 
	\[f_0=f_1=f_{n-1}=f_n=0,\ f_{2}=\cdots=f_{n-2},\ f_{n-2}\circ \id =0.\]
	This shows $f=0$.
	
	If $t\ge 2$, then we have 
	\begin{itemize}
		\item $f_0=f_1=f_{n-1}=f_n=0$.
		\item $f_{2}=\cdots=f_{n-t-1}$.
		\item $f_{n-t-1}\circ \delta=f_{n-t}$.
		\item $f_{n-t}=\cdots =f_{n-2}$.
		\item $f_{n-2}\circ \alpha=0,\ f_{n-2}\circ \beta=0$.
	\end{itemize}
	By the fourth and last equations, we obtain 
	\[f_{n-t}=\cdots=f_{n-2}=0.\] 
	Then the second and third equations imply
	\[f_2=\cdots=f_{n-t-1}=0.\]
	Hence $f=0$. 
\end{proof}
The following lemma is a direct consequence of Lemma\;\ref{tublar_family_is_standard}.
\begin{lemma}
	\label{rigid:t2tn-1}
We have $\Hom_A(R, L)=0=\Hom_A(L,R)$ for each $(R,L)\in \mathbb{T}_{n-2}\times \mathbb{T}_2$.	
In particular, if $X\in \add\mathbb{T}_{n-2}$ and $Y\in \add\mathbb{T}_2$, then 
$X\oplus Y$ is $\tau$-rigid if and only if $X$ and $Y$ are $\tau$-rigid.
\end{lemma}
\subsubsection{Construct $\omega^{(0)}$}
Consider a sink mutation sequence $(i_1,i_2,\dots,i_n,i_{n+1})=(2,3,\dots,n-2,0,1,n-1,n)$ of $Q$.
We define $\mu_{i_k}\cdots \mu_{i_1}A\in \mod A$ ($1\le k \le n+1$) by
\[
\tau^{-1}\left(P_{i_1}\oplus \cdots \oplus P_{i_k}\right)\oplus P_{i_{k+1}}\oplus \cdots \oplus P_{i_{n+1}}. 
\] 
Then it is well-known that
\[T^{(0)}_{r,k}:=\tau^{-r}(\mu_{i_k}\cdots \mu_{i_1}A)\in \sttilt A\]
 for each $(r,k)\in \Z_{\ge 0} \times \{1,\dots, n+1\}$ (see \cite[Section\;4.2]{R} for example).
Then we construct $\omega^{(0)}$ as follows.
\[
\omega^{(0)}:\left[
\begin{array}{ccccccccccccccc}
 A &  \to& T^{(0)}_{0, 1} &\to& \cdots&\to  & T^{(0)}_{0, n+1}\\
 & \to & T^{(0)}_{1,1}&\to& \cdots&\to  & T^{(0)}_{1, n+1} \\
&\vdots \\
& \to & T^{(0)}_{r,1}&\to& \cdots&\to  & T^{(0)}_{r, n+1} \\
&\vdots \\
& \to & T^{(0)}_{n-4,1}&\to& \cdots&\to  & T^{(0)}_{n-4, n+1} \\
\\
\end{array}
\right]
\] 

\subsubsection{Construct $\omega^{(1)}$}
For each $(a,b)$ such that $0\le a \le b \le n-3$, we define a module $T^{(1)}_{a,b}$ as follows.
\[T^{(1)}_{a,b}=\tau^{-(n-3)}\left(\left(\overset{a}{\underset{k=0}{\oplus}} \tau^{-k}P_{n-\delta_{k,\rm{odd}}}\right)
\oplus \left(\overset{b}{\underset{k=a}{\oplus}} \tau^{-a}P_{n-1-k}\right)
\oplus \left(\overset{n-4}{\underset{k=b}{\oplus}} \tau^{-(a+1)}P_{n-2-k}\right)
\oplus \tau^{-a}\left(P_1\oplus P_0\right)\right).
\]

We also define $T^{(1)}_{a,a,+},T^{(1)}_{a,a,++}$ ($0\le a\le n-4$) as follows.
\[
\begin{array}{lll}
T^{(1)}_{a,a,+}&=&\dfrac{T^{(1)}_{a,a}}{\tau^{-(n-3+a)}P_1}\oplus \tau^{-(n-2+a)}P_1\vspace{5pt}\\
T^{(1)}_{a,a,++}&=&\dfrac{T^{(1)}_{a,a,+}}{\tau^{-(n-3+a)}P_0}\oplus \tau^{-(n-2+a)}P_0\vspace{5pt}\\
\end{array}
\]

\begin{lemma}
	\label{omega1}
We have the following statements.
\begin{enumerate}[{\rm (1)}]
	\item $T^{(1)}_{a,b}\in \sttilt A$ for any $(a,b)$.
	\item For each $(a,b)$ satisfying $a < b$, there is an arrow $T^{(1)}_{a,b}\to T^{(1)}_{a,b-1}$ in $\Hasse(\sttilt A)$.
	\item For each $a\le n-4$, there is a path $T^{(1)}_{a,a}\to T^{(1)}_{a,a,+}\to T^{(1)}_{a,a,++}\to T^{(1)}_{a+1,n-3}$ in $\Hasse(\sttilt A)$.
\end{enumerate} 
\end{lemma}
\begin{proof}
(1). We set 
\[
\begin{array}{lll}
	M_1&:=&\overset{a}{\underset{k=0}{\oplus}} \tau^{-k}P_{n-\delta_{k,\rm{odd}}}\\
	M_2&:=&\left(\overset{b}{\underset{k=a}{\oplus}} \tau^{-a}P_{n-1-k}\right)\oplus \tau^{-a}\left(P_1\oplus P_0\right)\\
	M_3&:=&\overset{n-4}{\underset{k=b}{\oplus}} \tau^{-(a+1)}P_{n-2-k}\\
\end{array}	
\]
	Since $T^{(1)}_{a,b}=\tau^{-(n-3)}(M_1\oplus M_2\oplus M_3)$, it is sufficient to check
	$M_1\oplus M_2\oplus M_3\in \sttilt A$.
Then the assertion (1) follows from Lemma\;\ref{mgs:program:relation} (3) and Lemma\;\ref{nonsincere_p}.

(2) and (3). 
By definition, we have the following equations.
\[
\begin{array}{lll}
T^{(1)}_{a,b-1}&=&\dfrac{T^{(1)}_{a,b}}{\tau^{-(n-3+a)}P_{n-1-b}}\oplus \tau^{-(n-2+a)}P_{n-1-b}\vspace{5pt}\\
T^{(1)}_{a,a,+}&=&\dfrac{T^{(1)}_{a,a}}{\tau^{-(n-3+a)}P_1}\oplus \tau^{-(n-2+a)}P_1\vspace{5pt}\\
T^{(1)}_{a,a,++}&=&\dfrac{T^{(1)}_{a,a,+}}{\tau^{-(n-3+a)}P_0}\oplus \tau^{-(n-2+a)}P_0\vspace{5pt}\\
T^{(1)}_{a+1,n-3}&=&\dfrac{T^{(1)}_{a,a,++}}{\tau^{-(n-3+a)}P_{n-1-a}}\oplus \tau^{-(n-2+a)}P_{n-\delta_{a+1,{\rm odd}}}\vspace{5pt}\\
\end{array}
\]
Then the assertions (2) and (3) follow from Lemma\;\ref{mgs:program:relation} (3).
\end{proof}
By Lemma\;\ref{omega1}, we obtain a path
\[
\omega^{(1)}:\left[
\begin{array}{ccccccccccccccc}
T^{(0)}_{n-4,n+1}&=& T^{(1)}_{0,n-3} &\to& T^{(1)}_{0, n-4} &\to& \cdots&\to  & T^{(1)}_{0, 0}&\to  & T^{(1)}_{0, 0,+}& \to  & T^{(1)}_{0, 0,++}\\\\
&\to & T^{(1)}_{1,n-3}&\to& \cdots&\to  & T^{(1)}_{1, 1}&\to  & T^{(1)}_{1, 1,+}&\to  & T^{(1)}_{1, 1,++} \\
&\vdots \\
&\to & T^{(1)}_{a,n-3} & \to & \cdots&\to & T^{(1)}_{a, a}& \to &T^{(1)}_{a, a,+}&\to  & T^{(1)}_{a, a,++}\\
&\vdots \\
&\to & T^{(1)}_{n-4,n-3}& \to &T^{(1)}_{n-4,n-4} & \to &T^{(1)}_{n-4, n-4,+}&\to  & T^{(1)}_{n-4, n-4,++} \\\\
&\to &T^{(1)}_{n-3,n-3}\\
\end{array}
\right]
\]
in $\overrightarrow{\mathcal{H}}(\sttilt A)$.
\subsubsection{Construct $\omega^{(2)}$}

For each $c\in\{1,\dots,n-2\}$, we define a module $T^{(2)}_c$ as follows.
\[
T^{(2)}_c=
\tau^{-n+3}P_n\oplus \underset{k=c}{\overset{n-2}{\bigoplus}}\tau^{-(n-3+k)}(P_{n-\delta_{k,\rm{odd}}})
\oplus \tau^{-2n+6}(P_1\oplus P_0)
\oplus R_1\oplus \cdots \oplus R_{c-1}.
\]

\begin{lemma}
	\label{t'istilting}
	$T^{(2)}_c\in \sttilt A$  for any $c\in \{1,\dots,n-2\}$.
\end{lemma}
\begin{proof}	
	Let $M=P_n\oplus \underset{k=c}{\overset{n-2}{\bigoplus}}\tau^{-k}(P_{n-\delta_{k,\rm{odd}}})
	\oplus \tau^{-n+3}(P_1\oplus P_0)$ and $R=R_1\oplus \cdots \oplus R_{c-1}$.
	Then we have
	 \[
	 T_c^{(2)}=\tau^{-n+3}M\oplus R.
	 \]
	Then, by Lemma\;\ref{mgs:program:relation} and Lemma\;\ref{nonsincere_p}, $M$ is $\tau$-rigid. In particular, we have
	$\tau^{-n+3}M$ is $\tau$-rigid.
	By Lemma\;\ref{silce module}, $R$ is also $\tau$-rigid. Therefore, by Lemma\;\ref{mgs:program:relation} (2), it remains to check the following conditions.
	\begin{enumerate}[{\rm (i)}]
		\item $\tau^{-(n-3)}P_n\prec R_{t}$ for any $t\in \{1,\dots,c-1\}$.
		\item $\tau^{-(n-3+k)}P_{n-\delta_{k,\rm{odd}}}\prec R_{t}$ for any $(k,t)\in \{c,\dots, n-2\}\times\{1,\dots, c-1\}$.
		\item $\tau^{-(2n-6)}P_0\prec R_{t}$ and $\tau^{-(2n-6)}P_1\prec R_{t}$  for any $t \in \{1,\dots, c-1\}$.
	\end{enumerate}
Note that the condition (i) (resp. (ii),  (iii)) is equivalent to the condition (i') (resp. (ii'), (iii')) below.
\begin{enumerate}[{\rm (i')}]
\item $\tau^{n-2}R_{t}e_n=R_t e_n=0$ for any $t\in \{1,\dots,c-1\}$.
\item $\tau^{n-2+k} R_{t}e_{n-\delta_{k,\rm{odd}}}=\tau^{-(n-2-k)}R_{t}e_{n-\delta_{k,\rm{odd}}}=0$ for any $(k,t)\in \{c,\dots, n-2\}\times\{1,\dots, c-1\}$.
\item $\tau^{2n-5}R_te_{\epsilon}=\tau^{-1}R_t e_{\epsilon}=0$  for any $(\epsilon, t) \in \{0,1\}\times \{1,\dots, c-1\}$.
\end{enumerate} 
Then the assertion follows from Lemma\;\ref{nonsincere_r_n-2}. In fact, (i') and (i''') are obvious and (ii') follows from
$0\le n-2-k\le n-(t+1)-2$. 
\end{proof}
Note that
\[
\begin{array}{lll}
T^{(2)}_1&=&\dfrac{T^{(1)}_{n-3,n-3}}{\tau^{-(2n-6)}P_2}\bigoplus \tau^{-(2n-5)}P_{n-\delta_{n,\rm{odd}}}\vspace{5pt}\\
T^{(2)}_{c+1}&=&\dfrac{T^{(2)}_c}{\tau^{-(n-3+c)}P_{n-\delta_{c,\rm{odd}}}}\bigoplus R_c.\\
\end{array}
\]	 
Then it follows from Lemma\;\ref{mgs:program:relation} and Lemma\;\ref{t'istilting} that
there exists a path
\[
\omega^{(2)}:T_{n-3,n-3}\to T^{(2)}_1\to\cdots\to T^{(2)}_{n-2}
\]
in $\overrightarrow{\mathcal{H}}(\sttilt A)$.
\subsubsection{Construct $\omega^{(3)}$}
Let $R=R_1\oplus \cdots \oplus R_{n-4}\oplus R_{n-3}$ and define $T^{(3)}_d$ $(d=1,2,3,4,5,6)$ as follows.
\[
\begin{array}{lllll}
T^{(3)}_1&=&\dfrac{T^{(2)}_{n-2}}{\tau^{-2(n-3)}P_0}\oplus L_{0,n-\delta_{n,\rm{odd}}}&=&\tau^{-(n-3)}P_n\oplus \tau^{-(2n-5)}P_{n-\delta_{n,\rm{odd}}}\oplus L_{0,n-\delta_{n,\rm{odd}}}\oplus \tau^{-2(n-3)}P_1\oplus R\vspace{5pt} \\
T^{(3)}_2&=&\dfrac{T^{(3)}_{1}}{\tau^{-2(n-3)}P_1}\oplus L_{1,n-\delta_{n,\rm{odd}}}&=&\tau^{-(n-3)}P_n\oplus \tau^{-(2n-5)}P_{n-\delta_{n,\rm{odd}}}\oplus L_{0,n-\delta_{n,\rm{odd}}}\oplus L_{1,n-\delta_{n,\rm{odd}}}\oplus R\vspace{5pt}\\
T^{(3)}_3&=&\dfrac{T^{(3)}_{2}}{\tau^{-2n+5}P_{n-\delta_{n,\rm{odd}}}}\oplus P^-_{n-\delta_{n,\rm{odd}}} &=&\tau^{-(n-3)}P_n\oplus P_{n-\delta_{n,\rm{odd}}}^-\oplus L_{0,n-\delta_{n,\rm{odd}}}\oplus L_{1,n-\delta_{n,\rm{odd}}}\oplus R\vspace{5pt}\\
T^{(3)}_4&=& \dfrac{T^{(3)}_{3}}{\tau^{-(n-3)}P_{n}}\oplus \tau^{n-3}I_n &=&\tau^{n-3}I_n\oplus P_{n-\delta_{n,\rm{odd}}}^-\oplus L_{0,n-\delta_{n,\rm{odd}}}\oplus L_{1,n-\delta_{n,\rm{odd}}}\oplus R\vspace{5pt}\\
T^{(3)}_5&=&\dfrac{T^{(3)}_{4}}{L_{0,n-\delta_{n,\rm{odd}}}}\oplus I_0 &=&\tau^{n-3}I_n\oplus P_{n-\delta_{n,\rm{odd}}}^-\oplus I_0\oplus L_{1,n-\delta_{n,\rm{odd}}}\oplus R\vspace{5pt}\\
T^{(3)}_6&=& \dfrac{T^{(3)}_{5}}{L_{1,n-\delta_{n,\rm{odd}}}}\oplus I_1 &=&\tau^{n-3}I_n\oplus P_{n-\delta_{n,\rm{odd}}}^-\oplus I_0\oplus I_1\oplus R\vspace{5pt}\\
\end{array}
\]
Then we have the following lemma.
\begin{lemma}
	\label{lemma:length:R3}
	For any $d\in \{1,2,3,4,5,6\}$, $T_{d}^{(3)}\in \sttilt A$ and there is an arrow $T^{(3)}_{d-1}\to T^{(3)}_d$ in $\Hasse(\sttilt A)$,
	where we put $T^{(3)}_0=T^{(2)}_{n-2}$. 
\end{lemma}
\begin{proof}
	We put $\delta=\delta_{n,\rm{odd}}$ and $\delta'=\delta_{n,\rm{even}}$.
	
$(d=1)$. We show $T^{(3)}_1\in \sttilt A$ and there is an arrow $T^{(2)}_{n-2}\to T^{(3)}_1$ in $\Hasse(\sttilt A)$. By Lemma\;\ref{mgs:program:relation} and Lemma\;\ref{rigid:t2tn-1}, it is sufficient to check the following conditions.
\begin{enumerate}[{\rm (i)}]
	\item $L_{0,n-\delta}\prec \tau^{-(n-3)}P_n$, $\tau^{-(2n-5)}P_{n-\delta}$, $\tau^{-2(n-3)}P_1$.
	\item $\tau^{-(n-3)}P_n$, $\tau^{-(2n-5)}P_{n-\delta}$, $\tau^{-2(n-3)}P_1\prec L_{0,n-\delta}$.
	\item $L_{0,n-\delta}\prec \tau^{-2(n-3)}P_0$. 
\end{enumerate}	
Then (i) and (iii) follow from the definition of $\prec$ and (ii) follows from Lemma\;\ref{nonsincere_r_2}. 

$(d=2)$. We show $T^{(3)}_2\in \sttilt A$ and there is an arrow $T^{(3)}_{1}\to T^{(3)}_2$ in $\Hasse(\sttilt A)$. By Lemma\;\ref{mgs:program:relation}, Lemma\;\ref{nonsincere_r_2}, and Lemma\;\ref{rigid:t2tn-1}, it is sufficient to check the following conditions.
\begin{enumerate}[{\rm (i)}]
	\item $L_{1,n-\delta}\prec \tau^{-(n-3)}P_n$, $\tau^{-(2n-5)}P_{n-\delta}$.
	\item $\tau^{-(n-3)}P_n$, $\tau^{-(2n-5)}P_{n-\delta}\prec L_{1,n-\delta}$.
	\item $L_{1,n-\delta}\prec \tau^{-2(n-3)}P_1$. 
\end{enumerate}	
Then (i) and (iii) follow from the definition of $\prec$ and (ii) follows from Lemma\;\ref{nonsincere_r_2}. 
	
$(d=3)$. We show $T^{(3)}_{3}\in \sttilt A$ and there is an arrow $T^{(3)}_{2}\to T^{(3)}_3$ in $\Hasse(\sttilt A)$.
By Lemma\;\ref{nonsincere_p}, Lemma\;\ref{nonsincere_i}, Lemma\;\ref{nonsincere_r_n-2}, and Lemma\;\ref{nonsincere_r_2},
we have \[\tau^{-(n-3)}P_n\oplus L_{0,n-\delta}\oplus L_{1,n-\delta}\oplus R\in \mod A_{n-\delta}.\]
In particular, $T_3^{(3)}\in \sttilt A$ and we have an arrow $T_2^{(3)}\to T_3^{(3)}$ in $\Hasse(\sttilt A)$.

$(d=4)$. We show $T^{(3)}_4\in \sttilt A$ and there is an arrow $T^{(3)}_{3}\to T^{(3)}_4$ in $\Hasse(\sttilt A)$. By Lemma\;\ref{mgs:program:relation} and Lemma\;\ref{rigid:t2tn-1}, it is sufficient to check the following conditions.
\begin{enumerate}[{\rm (i)}]
	\item $L_{0,n-\delta}$, $L_{1,n-\delta} \prec \tau^{n-3}I_n$.
	\item $R_t\prec \tau^{n-3}I_n$ for each $t\in \{1,\dots,n-3\}$.
	\item $\tau^{n-3}I_n\prec L_{0,n-\delta}$, $L_{1,n-\delta}$.
	\item $\tau^{n-3}I_n\prec R_t$ for each $t\in \{1,\dots,n-3\}$.
	\item $\tau^{n-3}I_n\in \mod A_{n-\delta}$. 
	\item $\tau^{n-3}I_n\prec \tau^{-(n-3)}P_n$.
\end{enumerate}	
Then (iii), (iv), and (vi) follow from the definition of $\prec$, (v) follows from Lemma\;\ref{nonsincere_i}. 
Further, (ii) follows from Lemma\;\ref{nonsincere_r_n-2} and (i) follows from Lemma\;\ref{nonsincere_r_2}.

$(d=5)$. We show $T^{(3)}_5\in \sttilt A$ and there is an arrow $T^{(3)}_{4}\to T^{(3)}_5$ in $\Hasse(\sttilt A)$. By Lemma\;\ref{mgs:program:relation} and Lemma\;\ref{rigid:t2tn-1}, it is sufficient to check the following conditions.
\begin{enumerate}[{\rm (i)}]
	\item $\tau^{n-3}I_n \prec I_0$.
	\item $R_t\prec I_0$ for each $t\in \{1,\dots,n-3\}$.
	\item $L_{1,n-\delta}\prec I_0$.
	\item $I_0\prec \tau^{n-3}I_n$.
	\item $I_0\prec L_{1,n-\delta}$.
	\item $I_0\prec R_t$ for each $t\in \{1,\dots, n-3\}$.
	\item $I_0\in \mod A_{n-\delta}$.
	\item $I_0\prec L_{0,n-\delta}$. 
\end{enumerate}	
Then (iv), (v), (vi), and (viii) follow from the definition of $\prec$, and (vii) follows from Lemma\;\ref{nonsincere_i}. 
Further, (i) follows from Lemma\;\ref{nonsincere_i}, (ii) follows from Lemma\;\ref{nonsincere_r_n-2} ,and (iii) follows from Lemma\;\ref{nonsincere_r_2}.

$(d=6)$. We show $T^{(3)}_6\in \sttilt A$ and there is an arrow $T^{(3)}_{5}\to T^{(3)}_6$ in $\Hasse(\sttilt A)$. By Lemma\;\ref{mgs:program:relation} and Lemma\;\ref{rigid:t2tn-1}, it is sufficient to check the following conditions.
\begin{enumerate}[{\rm (i)}]
	\item $\tau^{n-3}I_n \prec I_1$.
	\item $R_t\prec I_1$ for each $t\in \{1,\dots,n-3\}$.
	\item $I_1\prec \tau^{n-3}I_n$.
	\item $I_1\prec I_0 \prec I_1$.
	\item $I_1\prec R_t$ for each $t\in \{1,\dots, n-3\}$.
	\item $I_1\in \mod A_{n-\delta}$.
	\item $I_1\prec L_{1,n-\delta}$. 
\end{enumerate}	
Then (iii), (iv), (v), and (vii) follow from the definition of $\prec$, and (vi) follows from Lemma\;\ref{nonsincere_i}. 
Further, (i) follows from Lemma\;\ref{nonsincere_i} and (ii) follows from Lemma\;\ref{nonsincere_r_n-2}.
\end{proof}	
By Lemma\;\ref{lemma:length:R3}, there is a path
		\[
		\omega^{(3)}:T^{(2)}_{n-2}\to T^{(3)}_1\to T^{(3)}_2\to T^{(3)}_3\to T^{(3)}_4\to T^{(3)}_5\to T^{(3)}_6.
		\]  

\subsubsection{Construct $\omega^{(4)}$}
For each $(e,f)\in \Z\times \Z$ satisfying 
\[
1\le e \le n-4,\ 1\le f \le n-4,\ e+f\le n-3,
\]
we define $R(e,f)$ as follows.
\[
R(e,f):=\tau^{-e}(R_1\oplus\cdots \oplus  R_f)\oplus \tau^{-e+1}(R_{f+1}\oplus \cdots\oplus R_{n-2-e})\oplus \left(\bigoplus_{t=n-1-e}^{n-3}\tau^{-(n-3-t)}R_{t}\right)
\]
If we denote by $\Sigma_{e,f}$ the subquiver of $C$ (defined in Section\;\ref{slice}) given by
 \[(1,e)\stackrel{a}{\rightarrow} \cdots \stackrel{a}{\rightarrow} (f,e)\stackrel{b}{\leftarrow} (f+1,e-1)
 \stackrel{a}{\rightarrow}
  \cdots \stackrel{a}{\rightarrow} (n-2-e,e-1)\stackrel{b}{\leftarrow} (n-1-e,e-2)\stackrel{b}{\leftarrow}
   \cdots \stackrel{b}{\leftarrow} (n-3,0),\]
then $R(e,f)=R(\Sigma_{e,f})$ and it is $\tau$-rigid by Lemma\;\ref{silce module}. 

\begin{lemma}
	\label{lemma:length:R4}
Let $T^{(4)}_{e,f}:=R(e,f)\oplus \tau^{n-3}I_n\oplus I_1\oplus I_0\oplus P_{n-\delta_{n,{\rm odd}}}^-$.
\begin{enumerate}[{\rm (1)}]
	\item $T^{(4)}_{e,f}\in \sttilt A$.
	\item We have the following statements.
	\begin{enumerate}[{\rm (i)}]
		\item There is an arrow $T:=T^{(3)}_6\to T^{(4)}_{1,1}=:T'$ in $\ORA{\mathcal{H}}(\sttilt A)$.
		\item For each $(e,f)$ with $e+f\le n-4$, there is an arrow $T:=T^{(4)}_{e,f}\to T^{(4)}_{e,f+1}=:T'$ in $\ORA{\mathcal{H}}(\sttilt A)$.
		\item For each $(e,n-3-e)$ with $e\le n-5$, there is an arrow $T:=T^{(4)}_{e,n-3-e}\to T^{(4)}_{e+1,1}=:T'$ in $\ORA{\mathcal{H}}(\sttilt A)$.
	\end{enumerate}
\end{enumerate}
\end{lemma}
\begin{proof}
	(1). Note that $R(e,f)$ and $\tau^{n-3}I_n\oplus I_1\oplus I_0\oplus P_{n-\delta_{n,\rm{odd}}}^-\in \add T^{(3)}_{6}$ are $\tau$-rigid.
	Since $R(e,f)\in \add\mathcal{R}$, we have	
	\[\Hom_A((\tau^{n-3}I_n\oplus I_1\oplus I_0), \tau R(e,f))=0.\]
	Therefore, it is sufficient to check that $R(e,f)$ satisfies the following conditions.  
	\begin{enumerate}[{\rm (i)}]
		\item $\Hom_A(R(e,f),\tau (\tau^{n-3}I_n\oplus I_1\oplus I_0))=0$.
		\item $R(e,f)e_{n-\delta_{n,\rm{odd}}}=0$.
	\end{enumerate}

    Since $\tau^{-(n-2)}R(e,f)\simeq R(e,f)$, it is sufficient to check 
    \begin{center}
    $R(e,f)e_{n-\delta_{n,\rm{odd}}}+R(e,f)e_n=0$ and $\tau^{-1}R(e,f)(e_0+e_1)=0$.
    \end{center}
    Note that $\tau^{-1}R(e,f)=\tau^{-(e+1)}(R_1\oplus\cdots \oplus  R_f)\oplus \tau^{-e}(R_{f+1}\oplus \cdots\oplus R_{n-2-e})\oplus \left(\bigoplus_{t=n-1-e}^{n-3}\tau^{-(n-2-t)}R_{t}\right)$.
    Then the assertion (1) follows from the hypothesis for $(e,f)$ and Lemma\;\ref{nonsincere_r_n-2}.

    (2). 
         Note that there are $M$ with $|M|=n-1$ and $R\in \mathbb{T}_{n-2}$ such that
\begin{center}
	$T=M\oplus R$ and $T'=M\oplus \tau^{-1}R$.
\end{center}
    Therefore, either $T\to T'$ or $T'\to T$ holds. Suppose that 
    there is an arrow $T'\to T$ in $\ORA{\mathcal{H}}(\sttilt A)$. Then
    we have $T'\ge T$. In particular, we have
    \[
    0=\Hom_A(R,\tau (\tau^{-1} R))=\Hom_A(R,R).
    \]
    This is a contradiction. Therefore, we have the assertion.
    \end{proof}

By Lemma\;\ref{lemma:length:R4}, we can construct $\omega^{(4)}$ as follows.
\[
\omega^{(4)}:\left[
\begin{array}{ccccccccccccccc}
T^{(3)}_6 &\to & T^{(4)}_{1,1} &\to& T_{1, 2} &\to& \cdots&\to  & T_{1, n-4}\\
&\vdots \\
&\to & T_{e,1} & \to & \cdots&\to & T_{e, n-3-e}\\
&\vdots \\
&\to & T_{n-4,1}\\
\end{array}
\right]
\]

\subsubsection{Construct $\omega^{(5)}$}
For each $s\in \{1,2,\dots,n-3\}$, we define $T^{(5)}_s$ as follows.
\begin{itemize}
	\item $T^{(5)}_1=\dfrac{T^{(4)}_{n-4,1}}{R_{n-3}}\oplus \tau^{n-4}I_{n-1}$.
	\item $T^{(5)}_s=\dfrac{T^{(5)}_{s-1}}{\tau^{-(s-1)}R_{n-2-s}}\oplus \tau^{n-3-s}I_{n-\delta_{s,{\rm odd}}}$ ($2\le s \le n-3$).
\end{itemize}
In other word, $T^{(5)}_s$ has the following form.
\[
\left(\overset{n-4}{\underset{k=s}{\bigoplus}}\tau^{-k} R_{n-3-k}\right)\oplus 
\left(\overset{s}{\underset{j=0}{\bigoplus}}\tau^{n-3-j} I_{n-\delta_{j,\mathrm{odd}}}\right)\oplus I_1\oplus I_0 \oplus P_{n-\delta_{n,\rm{odd}}}^-
\] 
\begin{lemma}
	\label{lemma:length:R5}
	For each $s\in \{1,\dots,n-3\}$, $T_s^{(5)}\in \sttilt A$ and there is an arrow
	$T_{s-1}^{(5)}\to T_s^{(5)}$ in $\Hasse(\sttilt A)$, where we put $T_{0}^{(5)}=T_{n-4,1}^{(4)}$.

\end{lemma}
\begin{proof}
	By Lemma\;\ref{mgs:program:relation} and Lemma\;\ref{lemma:length:R4}, 	
	it is sufficient to check the following conditions.
	\begin{enumerate}[{\rm (i)}]
		\item $\tau^{-k}R_{n-3-k}\prec \tau^{n-3-j} I_{n-\delta_{j,{\rm odd}}}\prec \tau^{-k}R_{n-3-k}$ for each $(j,k)\in \{0,\dots,s\}\times \{s,\dots,n-4\}$.
		\item $\tau^{n-3-j}I_{n-\delta_{j,{\rm odd}}}\prec \tau^{n-3-j'} I_{n-\delta_{j',{\rm odd}}}$ for each $0\le j, j'\le s$.
		\item $I_\epsilon\prec \tau^{n-3-j}I_{n-\delta_{j,{\rm odd}}}\prec I_\epsilon$ for each $(j,\epsilon)\in \{0,\dots, s\}\times\{0,1\}$.
		\item $\left(\tau^{n-3-j}I_{n-\delta_{j,{\rm odd}}}\right) e_{n-\delta_{n,\rm{odd}}}=0$ for each $0\le j \le s$.
		\item $\tau^{n-3-s}I_{n-\delta_{s,\rm{odd}}}\prec \tau^{-(s-1)}R_{n-2-s}$.
	\end{enumerate}
    Note that $\tau^{r}I_i=\tau^{r+1}P_i^-$.
    
	(i). By the definition of $\prec$, we have $\tau^{n-3-j} I_{n-\delta_{j,{\rm odd}}}\prec \tau^{-k}R_{n-3-k}$.
	By Lemma\;\ref{nonsincere_r_n-2}, we have 
	\[
	(\tau^{-(k+n-2-j)}R_{n-3-k})(e_{n-1}+e_n)=(\tau^{-(k-j)}R_{n-3-k})(e_{n-1}+e_n)=0.
	\]
	This shows $\tau^{-k}R_{n-3-k}\prec \tau^{n-3-j} I_{n-\delta_{j,{\rm odd}}}$ and the condition (i) holds.
	
	(ii). If $j\ge j'$, the the assertion follows from the definiton of $\prec$.
	Thus, we may assume $j<j'$. In this case, it follows from $0\le j'-j-1 \le n-4$, $(j'-j-1)+\delta_{j,\rm{odd}}\equiv j'-1\; (\mod 2)$, and Lemma\;\ref{nonsincere_i} that
	\[
	\tau^{j'-j-1} (I_{n-\delta_{j,\rm{odd}}})(e_{n-\delta_{j',\rm{odd}}})=0.
	\]  
	Thus the condition (ii) holds.
	
	(iii). If $j=n-3$, then the assertion follows from the definition of $\prec$. Therefore, we may assume	$j<n-3$.	
	In this case, we have $0\le n-4-j \le n-4$ and the condition (iii) follows from Lemma\;\ref{nonsincere_i}. 
	
	(iv). Since $n-3-j+\delta_{j,\rm{odd}}\equiv n-1\; (\mod 2)$, the condition (iv) follows from Lemma\;\ref{nonsincere_i}.
	
	(v). The condition (v) follows from the definition of $\prec$.

\end{proof}
By Lemma\;\ref{lemma:length:R5}, we can construct $\omega^{(5)}$ as follows.
\[
\omega^{(5)}:T^{(4)}_{n-4,1}\to T^{(5)}_1\to \cdots\to T^{(5)}_{n-3}.
\]
\subsubsection{Construct $\omega^{(6)}$}
Let $M^{(6)}_t=\left(\overset{n-3}{\underset{k=t-1}{\bigoplus}}\tau^{n-3-k} I_{n-\delta_{k,\mathrm{odd}}}\right)\oplus I_1\oplus I_0 \oplus P_{n-\delta_{n,\rm{odd}}}^-$ for each $1\le t \le n-2$.
Note that $M^{(6)}_t\in \add T^{(5)}_{n-2}$. In particular, it is $\tau$-rigid by Lemma\;\ref{lemma:length:R5}.
We also note that if $t-1 \le k \le n-3$, then 
$\left(\tau^{n-3-k} I_{n-\delta_{k,\mathrm{odd}}}\right)e_j=0$ for each $j\in \{2,\dots,t\}$.	
In fact, it follows from Lemma\;\ref{nonsincere_i} (3) that
\[
\{2,\dots,t\}\subset\{2,\dots, k+1\}\subset Q_0\setminus \supp\left(\tau^{n-3-k} I_{n-\delta_{k,\mathrm{odd}}}\right).  
\] 
Since $(I_1\oplus I_0)(e_2+\cdots+ e_{n-2})=0$, we have
\[
T^{(6)}_t:=M^{(6)}_t \oplus P_2^-\oplus \cdots\oplus P_t^-\in \sttilt A. 
\]
Moreover, we have a path
\[
\omega^{(6)}:T^{(5)}_{n-2}=T^{(6)}_1\to \cdots \to T^{(6)}_{n-2}.
\]
in $\ORA{\mathcal{H}}(\sttilt A)$.
\subsubsection{Construct $\omega^{(7)}$}
Note that $T^{(6)}_{n-2}=I_{n-\delta_{n-3,{\rm odd}}}\oplus I_1\oplus I_0 \oplus P_{n-\delta_{n,{\rm odd}}}^-\oplus P_2^- \oplus \cdots \oplus P_{n-2}^-$.
Therefore, we have a path  
\[
\omega^{(7)}:T^{(6)}_{n-2}\to I_1\oplus I_0\oplus\left(\overset{n}{\underset{k=2}{\bigoplus}} P_k^-\right)\to I_0\oplus
 \left(\overset{n}{\underset{k=1}{\bigoplus}} P_k^-\right)\to A^-
\]
in $\ORA{\mathcal{H}}(\sttilt A)$.

\subsubsection{Construct a path with length $2n^2-2n-2$}
Now let $\omega$ be a path in $\ORA{\mathcal{H}}(\sttilt A)$ given by connecting $\omega^{(1)}$, $\omega^{(2)}$, $\omega^{(3)}$, $\omega^{(4)}$,
$\omega^{(5)}$, $\omega^{(6)}$, and $\omega^{(7)}$.
Then we have
\[
\begin{array}{lll}
\ell(\omega)&=&\sum_{k=1}^{7}\ell(\omega_k)\\
            &=&(n+1)(n-3)+\left(\dfrac{(n-2)(n-1)}{2}+2(n-3)-1\right)+(n-2)+6+\left(\dfrac{(n-4)(n-3)}{2}\right)+(n-3)+(n-3)+3\\
            &=&2n^2-2n-2
\end{array}
\]
This finishes the proof.

\end{document}